\documentclass[a4paper,11pt,reqno]{amsart}

\usepackage[utf8]{inputenc}
\usepackage{amsmath}
\usepackage{amssymb}
\usepackage{amsthm}
\usepackage{hyperref}
\usepackage{color}
\usepackage{a4wide}
\usepackage{dsfont}
\usepackage{mathdots}
\usepackage{tikz-cd}
\usepackage{enumitem}
\usepackage{comment}
\usepackage{euscript}
\usepackage{appendix}
\usepackage[normalem]{ulem}

\usepackage{mathrsfs}

\usepackage[new]{old-arrows}

\makeatletter

\newcommand{\fraka}{\mathfrak{a}}
\newcommand{\frakg}{\mathfrak{g}}
\newcommand{\frakh}{\mathfrak{h}}

\newcommand{\frakm}{\mathfrak{m}}
\newcommand{\frakn}{\mathfrak{n}}

\newcommand{\frakp}{\mathfrak{p}}

\newcommand{\fraku}{\mathfrak{u}}

\newcommand{\CC}{\mathbb{C}}
\newcommand{\DD}{\mathbb{D}}

\newcommand{\NN}{\mathbb{N}}

\newcommand{\RR}{\mathbb{R}}
\newcommand{\ZZ}{\mathbb{Z}}

\newcommand{\calD}{\mathcal{D}}

\newcommand{\calF}{\mathcal{F}}

\newcommand{\calL}{\mathcal{L}}

\newcommand{\calP}{\mathcal{P}}

\newcommand{\N}{\mathbb{N}}
\newcommand{\Z}{\mathbb{Z}}

\newcommand{\R}{\mathbb{R}}
\newcommand{\C}{\mathbb{C}}
\newcommand{\F}{\mathbb{F}}

\DeclareMathOperator{\Cay}{Cay}

\DeclareMathOperator{\Ind}{Ind}

\DeclareMathOperator{\ad}{ad}
\DeclareMathOperator{\Ad}{Ad}
\DeclareMathOperator{\End}{End}
\DeclareMathOperator{\Hom}{Hom}

\DeclareMathOperator{\Span}{span}
\DeclareMathOperator{\Pol}{Pol}

\DeclareMathOperator{\id}{id}

\DeclareMathOperator{\diag}{diag}

\renewcommand\Im{\operatorname{Im}}

\newcommand{\frakb}{{\mathfrak b}}
\newcommand{\frakl}{{\mathfrak l}}
\newcommand{\fraky}{{\mathfrak y}}

\newcommand{\abs}[1]{\left\vert#1\right\vert}

\newcommand{\ord}{{\mathrm{ord}}}
\newcommand{\irredsupp}{\mathrm{irred\text{-}supp}}

\newcommand{\irred}{\mathrm{irred}}

\DeclareMathOperator{\supp}{supp}
\DeclareMathOperator{\Supp}{Supp}

\DeclareMathOperator{\Diff}{Diff}
\DeclareMathOperator{\Sol}{Sol}

\newcommand{\fraksl}{{\mathfrak{sl}}}

\newcommand{\wJ}{\widetilde{J}}

\newcommand{\valpha}{\vec{\alpha}}

\newcommand{\Mb}{M^{\frakg}_{\frakb}}
\newcommand{\Mp}{M^{\frakg}_{\frakp}}
\newcommand{\Mpp}{M^{\frakg'}_{\frakp'}}
\newcommand{\Mbsl}{M^{\frakg_{\fraksl}}_{\frakb_{\fraksl}}}

\DeclareMathOperator{\To}{\longrightarrow}

\DeclareMathOperator{\Symb}{Symb}
\DeclareMathOperator{\Rest}{Rest}
\DeclareMathOperator{\sym}{sym}

\DeclareMathOperator{\Aut}{Aut}

\DeclareMathOperator{\wSol}{\widetilde{\Sol}}

\DeclareMathOperator{\Ker}{Ker}

\theoremstyle{plain}
\newtheorem{theorem}[equation]{Theorem}
\newtheorem{proposition}[equation]{Proposition}
\newtheorem{lemma}[equation]{Lemma}
\newtheorem{corollary}[equation]{Corollary}

\newtheorem{fact}[equation]{Fact}
\newtheorem{thmalph}{Theorem}

\theoremstyle{definition}
\newtheorem{definition}[equation]{Definition}
\newtheorem{example}[equation]{Example}

\newtheorem{remark}[equation]{Remark}

\numberwithin{equation}{section}

\theoremstyle{definition}
\newtheorem{problem}{Problem}

\title[differential symmetry breaking operators]{The ordered F-system for the F-method and differential symmetry breaking operators for $(GL(3,\mathbb{R}), GL(2,\mathbb{R}))$ }

\author{Jonathan Ditlevsen}
\address{Graduate School of Mathematical Sciences, The University of Tokyo, 3-8-1 Komaba, Meguro-ku, Tokyo 153-8914, Japan}
\email{JonathanDitlevsen@gmail.com}

\author{Toshihisa Kubo}
\address{Faculty of Economics, 
Ryukoku University, 67 Tsukamoto-cho, Fukakusa, Fushimi-ku, Kyoto 612-8577, Japan
}
\email{toskubo@econ.ryukoku.ac.jp}

\author{V{\' i}ctor P{\'e}rez-Vald{\'e}s}
\address{Graduate School of Mathematical Sciences, The University of Tokyo, 3-8-1 Komaba, Meguro-ku, Tokyo 153-8914, Japan}
\email{perez-valdes@g.ecc.u-tokyo.ac.jp}

\begin{document}

 %Orthogonal polynomials and functions of hypergeometric type 


\date{\today}

%%%%%%%%%%%%%%%%%%%%%%%%%%%%%%%%%%%%%%%%%%
\begin{abstract} 
In this paper, we introduce a new aspect of the F-system for the F-method arising in the case where the nilpotent radical is not necessarily abelian.
For this, we define the \emph{symmetrization operator}
on the space of polynomial functions on the dual $\frakg^\vee$ 
of a finite-dimensional Lie algebra $\frakg$. 
In the context of the F-method,
the conjugation by this operator 
of the F-system yields 
a new system, which we call the \emph{ordered F-system}.
These two techniques allow one to consider the symmetrized form 
(with symmetrization) and ordered form (without symmetrization) of differential symmetry breaking operators (DSBOs) uniformly. 

As an application of this theory,
we classify and construct all the DSBOs 
between principal series representations
for $(GL(3,\R), GL(2,\R))$ in both symmetrized and ordered forms.
Here, we consider all embeddings $GL(2,\R) \hookrightarrow GL(3,\R)$ 
corresponding to the positive roots of $\mathfrak{gl}(3,\R)$. Furthermore, we utilize the DSBOs for the above pair to investigate differential intertwining operators (DIOs) for $GL(3,\R)$ and DSBOs for $(SL(3,\R), SL(2,\R))$. The branching laws of the corresponding Verma modules are also studied to support our results of DSBOs.
\end{abstract}

%%%%%%%%%%%%%%%%%%%%%%%%%%%%%%%%%%%%%%%%%%

\maketitle

%%%%%%%%%%%%%%%%%%%%%%%%%%%%%%%%%%%%%%%%%%
\setcounter{tocdepth}{1}
\tableofcontents

%%%%%%%%%%%%%%%%%%%%%%%%%%%%%%%%%%%%%%%%%
\section{Introduction}

Intertwining operators have played a central role in the representation theory of reductive groups since the origin of the field. 
Nevertheless, the classification and explicit construction of such operators remain difficult problems in general. One natural setting for these questions is that of branching problems, where one considers a pair of reductive groups $(G,G')$ with $G'\subseteq G$ (see \cite{Ko24} for an overview on branching problems). In this context, one may consider 
$G'$-intertwining operators in the space
\[
\Hom_{G'}(\pi|_{G'},\tau),
\]
where $\pi$ (resp.\ $\tau$) is a representation of $G$ (resp.\ $G'$). Such operators are often referred to as \emph{symmetry breaking operators} (SBOs). 
Furthermore, the SBOs that can be realized as differential operators 
are called
\emph{differential symmetry breaking operators} (DSBOs).

The DSBOs are of particular interest in this paper. More precisely,
we consider the classification and construction problems of DSBOs between principal
series representations (induced from minimal parabolic subgroups) for the following pairs of groups  $(G,G')$:
\begin{equation}\label{eqn:GGprime}
(GL(3,\R), GL(2,\R)),
\qquad
(GL(3,\R), GL(3,\R)),
\qquad
(SL(3,\R), SL(2,\R)).
\end{equation}
Moreover, we also investigate the branching law of 
the Verma module of $\frakg$ restricted to $\frakg'$
for the following pairs of 
complex Lie algebras $(\frakg, \frakg')$:
\begin{equation}\label{eqn:ggprime}
(\mathfrak{gl}(3,\C), \mathfrak{gl}(2,\C)),
\qquad
(\mathfrak{sl}(3,\C), \mathfrak{sl}(2,\C)).
\end{equation}
We remark that we consider all embeddings 
$\iota \colon G' \hookrightarrow G$ 
corresponding to the positive roots of $\mathfrak{gl}(3,\R)$
for $(G, G')=(GL(3,\R), GL(2,\R)),(SL(3,\R), SL(2,\R))$
and $(\frakg, \frakg')$ in \eqref{eqn:ggprime}.

Observe that the unipotent radical of the  minimal parabolic subgroup of $G$
in \eqref{eqn:GGprime} is non-abelian.
To deal with the formulas of DSBOs for the non-abelian setting uniformly,
we introduce two new techniques, namely, 
the \emph{symmetrization operator} and the \emph{ordered F-system}.
In what follows, we shall discuss our results in order.

%%%%%%%%%%%%%%%%%%%%%%%%%%%%%%%%%%%%%%%%%
\subsection{Symmetrization operator and ordered F-system}

In general, the classification and construction of DSBOs are challenging problems, even in the classical case $G'=G$. A major breakthrough for these problems was the introduction of the \emph{F-method} by T.~Kobayashi \cite{Kobayashi13}.
This is a powerful technique for the classification and construction of DSBOs between representations induced from parabolic subgroups. The fundamental idea is to translate 
the problem of classifying and constructing DSBOs into that of finding polynomial solutions to a certain system of partial differential equations. This is done via the algebraic Fourier transform $F_c$ of 
generalized Verma modules (GVMs)
and the duality map $\EuScript{D}$ between
the homomorphisms between GVMs and DSBOs. 
The resulting system of PDEs is referred to as 
the \emph{F-system} \cite[Fact 3.3]{KoKuPe16}.
The diagram \eqref{eqn:IntroSymb-triangle} below illustrates a brief idea of the F-method.
We remark that the F-method in principle works for any unipotent radicals
(not necessarily abelian).
\begin{equation}
\label{eqn:IntroSymb-triangle}
\begin{tikzcd}[column sep=1cm]
	& {\text{\fbox{Polynomial solutions to the F-system}}} &\\
	{\text{\fbox{Hom(GVMs)}}}&& {\text{\fbox{DSBOs}}}
	\arrow["{\text{algebraic Fourier transform $F_c$}}","{\sim}"', from=2-1, to=1-2]
	\arrow["{\text{duality map $\EuScript{D}$}}"',"{\sim}", from=2-1, to=2-3]
	\arrow["{\EuScript{D} \circ F_c^{-1}}","{\sim}"', dashed,from=1-2, to=2-3]
\end{tikzcd}
\end{equation}
Here ``Hom(GVMs)'' in \eqref{eqn:IntroSymb-triangle} 
means the homomorphisms between GVMs.

Better yet,
in \cite{KoPe16a,KoPe16b}, T.~Kobayashi and M.~Pevzner further
showed that if the unipotent radical of the parabolic subgroup in concern is abelian, then the composition $\EuScript{D} \circ F_c^{-1}$ 
in \eqref{eqn:IntroSymb-triangle}
can be realized as
$\Rest \circ \Symb^{-1}$,
where $\Rest$ is the restriction operator with respect to the inclusion $G'\hookrightarrow G$
and $\Symb^{-1}$ is the inverse of the symbol map $\Symb$ of differential operators with constant coefficients. The diagram \eqref{eqn:IntroSymb} below 
illustrates this procedure. (For simplicity, we omit the restriction operator 
$\Rest$.)
\begin{equation}
\label{eqn:IntroSymb}
\begin{tikzcd}[column sep=1cm]
%	{\text{(the F-system)}} &
	{\fbox{Polynomial solutions to the F-system}} && {\fbox{DSBOs}}
	&{\text{(abelian case)}} 
%	\arrow[from=1-1, to=1-2]
	\arrow["{\operatorname{Symb}^{-1}}","{\sim}"', from=1-1, to=1-3]
\end{tikzcd}
\end{equation}
This viewpoint revealed deep connections between DSBOs and orthogonal polynomials such as Gegenbauer polynomials and Jacobi polynomials. 
For instance, the celebrated Rankin--Cohen brackets can be realized by using 
Jacobi polynomials via $\Symb^{-1}$, see \cite[Thm.\ 8.1]{KoPe16b}.
For relevant work on DSBOs, see, for instance,
\cite{Kobayashi13, Kobayashi14, KoKuPe16, KoPe16a, KoPe16b, KOSS15, KuOr26, Nakahama23,PV23, PV26b,  PV26a,Somberg25} among others.

In general, unless the unipotent radical is abelian, the inverse $F_c^{-1}$ of the algebraic Fourier transform is not explicit. Thus, it had been unclear how to generalize the procedure in \eqref{eqn:IntroSymb} to the non-abelian setting.
Recently, the last two authors established 
this generalization
through the introduction of the \emph{truncated symbol map} $\Symb_0$ 
of left-invariant differential operators \cite{KPV25}.
For the details of this map, see \cite[Sect.~3]{KPV25}. 
In the case where the unipotent radical is not necessarily abelian, 
any DSBO can be obtained via 
the inverse $\Symb_0^{-1}$ of the truncated symbol map as follows.
\begin{equation}\label{eqn:IntroSymb2}
\begin{tikzcd}[column sep=1cm]
%	{\text{(the F-system)}} &
	{\fbox{Polynomial solutions to the F-system}} && {\fbox{DSBOs}}
	&{\text{(general case)}} 
%	\arrow[from=1-1, to=1-2]
	\arrow["{\operatorname{Symb}_0^{-1}}", "{\sim}"', from=1-1, to=1-3]
\end{tikzcd}
\end{equation}

A key idea for overcoming the obstacle of non-commutativity is to apply symmetrization to the left-invariant differential operators corresponding to polynomial solutions. Roughly speaking,
if $\zeta_1\zeta_2$ is a polynomial solution  
in the source side of \eqref{eqn:IntroSymb2}, then $\Symb_0^{-1}(\zeta_1\zeta_2)$ is given as
\begin{equation}\label{eqn:IntroD1D2}
\Symb_0^{-1}(\zeta_1\zeta_2) = \frac{1}{2}(D_1D_2+D_2D_1),
\end{equation}
where $D_1, D_2$ are the left-invariant differential operators
corresponding to $\zeta_1, \zeta_2$, respectively.
Therefore, the symmetrization appears in the explicit formula of 
a DSBO obtained through $\Symb_0^{-1}$.
We call such an expression of a DSBO the \emph{symmetrized form}.
Hence, we have the following.
\begin{equation}\label{eqn:IntroSymb3}
\begin{tikzcd}[column sep=1cm]
	{\fbox{Polynomial solutions to the F-system}} && {\fbox{DSBOs in symmetrized form}}
	\arrow["{\operatorname{Symb}_0^{-1}}","{\sim}"', from=1-1, to=1-3]
\end{tikzcd}
\end{equation}

On the other hand, DSBOs can also be expressed in terms of
the PBW-basis elements induced from a fixed ordered basis
of the nilpotent radical of the parabolic subalgebra in concern.
We call such an expression
the \emph{ordered form} with respect to the fixed ordered basis.
For instance, let $D_3$ denote the left-invariant differential operator 
such that $[D_1,D_2]=-D_3$ for  $D_1,D_2$ 
introduced in \eqref{eqn:IntroD1D2}. 
If we impose an ordering on $\{D_1,D_2\}$ in which $D_1$ comes
before $D_2$, then $\Symb_0^{-1}(\zeta_1\zeta_2)$ can also be given as
\begin{equation*}
\Symb_0^{-1}(\zeta_1\zeta_2) 
= \frac{1}{2}(D_1D_2+D_2D_1)=D_1D_2+\frac{1}{2}D_3.
\end{equation*}
The second expression is the ordered form with respect to the given ordering.

Therefore, DSBOs admit two types of expressions, namely, symmetrized form (with symmetrization) and  ordered form (without symmetrization).
See Definition \ref{def:symmetrized-ordered-forms} for 
the definitions of these forms.
Then one may consider the following problem.

\begin{problem}\label{prob:ordered}
Given a PBW-basis,
construct the ordered form of a DSBO directly from a polynomial
via a certain linear map, say $\Symb_{\ord}^{-1}$.
\end{problem}

One may understand Problem \ref{prob:ordered} 
as finding a suitable system of PDEs and a linear map 
$\Symb_{\ord}^{-1}$ such that the following diagram holds.
\begin{equation*}
\begin{tikzcd}
	{\fbox{Polynomial solutions to certain PDEs}} && {\fbox{DSBOs in ordered form}}
	\arrow["{\operatorname{Symb}_{\ord}^{-1}}", "{\sim}"',from=1-1, to=1-3]
\end{tikzcd}
\end{equation*}

In this paper, we solve Problem \ref{prob:ordered} 
by introducing the \emph{symmetrization operator} 
$S$ on the space $\Pol(\frakg^\vee)$
of polynomial functions on the dual $\frakg^\vee$ of a
finite-dimensional Lie algebra $\frakg$.
This can be considered as the pull-back of the symmetrization $\sym_{\ord}$
on the universal enveloping algebra $U(\frakg)$ via a certain linear isomorphism
$\Upsilon_{\ord}\colon \Pol(\frakg^\vee) \stackrel{\sim}{\to} U(\frakg)$. 
See Section \ref{sec:nonsym} for the precise  definition.
Quite surprisingly, this operator turns out to be a differential operator of infinite order 
if $\frakg$ is a Heisenberg Lie algebra.
For instance, if $\frakg$ is the three-dimensional Heisenberg Lie algebra $\frakh_{3}$, then
the symmetrization operator $S$ can be realized as
\begin{equation*}
S=e^J:=\sum_{\ell=0}^\infty \frac{J^\ell}{\ell!},
\end{equation*}
where
\begin{equation*}
J:=\frac{\zeta_3}{2}\frac{\partial^2}{\partial\zeta_1\partial\zeta_2}.
\end{equation*}
The symmetrization operator $S$ further relates the ``non-commutativity'' of $\frakh_3$ to a ${}_2F_0$-hypergeometric polynomial, also known as the 
generalized Bessel polynomial.
For the details,
see Section \ref{sec:Upn} for the three-dimensional case $\frakh_3$ and 
Appendix \ref{appendix:Heis} for 
Heisenberg Lie algebras of arbitrary dimension. 

Let $D_F$ denote the given F-system arising from the F-method. 
Then, the key idea for solving Problem \ref{prob:ordered} 
is to consider another system of PDEs obtained by conjugating $D_F$ by $S$, that is, $\Ad(S)D_F:=S \circ D_F \circ S^{-1}$.
We call this new system $\Ad(S)D_F$ 
the \emph{ordered F-system}.
The next theorem explains the motivation for the term; for the precise statement,
see Theorem \ref{thm:symb_ord}.

\begin{thmalph}[{{See Theorem \ref{thm:symb_ord}}}]\label{Intro:A}
Let $D_F$ be the given F-system.
Also, let $\Sol_{D_F}$ and $\Sol_{\Ad(S)D_F}$ denote the space of polynomial
solutions to $D_F$ and $\Ad(S)D_F$, respectively. Then, there exists a linear 
isomorphism $\Symb_{\ord}^{-1}$ such that the following 
diagram commutes.
\begin{equation}\label{eqn:A}
\begin{tikzcd}[row sep=1cm, column sep=1.5cm]
	{\Sol_{\Ad(S)D_F}} 
	& {\fbox{\textnormal{DSBOs in ordered form}}}
	\arrow[dl, pos=0.5, phantom, "\circlearrowleft"]\\
		{\Sol_{D_F}} 
	& {\fbox{\textnormal{DSBOs in symmetrized form}}}
	\arrow["{\sim}" sloped,"{S^{-1}}"', from=1-1, to=2-1]
	\arrow["{\id}", "{\sim}"' sloped, from=1-2, to=2-2]
	\arrow["{\operatorname{Symb}_{\ord}^{-1}}","{\sim}"', from=1-1, to=1-2]
	\arrow[ "{\sim}","{\operatorname{Symb}_{0}^{-1}}"',from=2-1, to=2-2]
\end{tikzcd}
\end{equation}
\end{thmalph}

In Section \ref{sec:Ordered_Fsystem}, we define $\Symb_{\ord}$ and 
$\Symb_{\ord}^{-1}$ explicitly. We refer to $\Symb_{\ord}$ as 
the \emph{ordered symbol map}.

Interestingly, it turns out that the ordered F-system $\Ad(S)D_F$ 
is strongly related to the generalized hypergeometric differential equation 
for ${}_3F_1$ and ${}_2F_0$ for $(G,G')$ in \eqref{eqn:GGprime}.
On the other hand,  certain tridiagonal determinants $\Cay_n(a,b)$ 
called \emph{Cayley continuants} occur as the coefficients of 
the polynomial solutions to the original F-system $D_F$.
Therefore, the symmetrization operator $S$ allows us to uncover
an intriguing relationship between 
hypergeometric polynomials and Cayley continuants
that appears to be new in the literature.
We next discuss the classification in more detail.

%%%%%%%%%%%%%%%%%%%%%%%%%%%%%%%%%%%%%%%%%
\subsection{Classification and construction of DSBOs}

In this paper, we consider the following problem for 
the classification and construction of DSBOs
between principal series representations 
(induced from minimal parabolic subgroups) for $(G,G')$ in \eqref{eqn:GGprime}.
(For the maximal parabolic case,  see, for instance, \cite{FW19, Kubo26}.)

\begin{problem}\label{prob:class_const}
Do the following.

\begin{enumerate}[label = \normalfont{(\arabic*)}]

\item Classify the parameters such that the space of DSBOs is non-zero.
\vskip 0.1in

\item Determine the dimension of the space of DSBOs.
\vskip 0.1in

\item Give generators of the space of DSBOs explicitly.
\end{enumerate}

\end{problem}

%%%%%%%%%%%%%%%%%%%%%%%%%%%%%%%%%%%%%%%%%
%\subsubsection{Three embeddings $\iota\colon G' \hookrightarrow G$}

In order to work on Problem \ref{prob:class_const}, one first needs to choose 
an embedding $\iota\colon G' \hookrightarrow G$ unless $G'=G$.
In this paper, we consider 
all embeddings $\iota\colon G' \hookrightarrow G$
corresponding to the positive roots of $\mathfrak{gl}(3,\R)$ 
for the case
$(G,G') = (GL(3,\R), GL(2,\R)),  (SL(3,\R), SL(2,\R))$. 
In the standard realization, the positive roots for $\mathfrak{gl}(3,\R)$ may be realized as
$\varepsilon_1-\varepsilon_2$, $\varepsilon_2-\varepsilon_3$ 
and $\varepsilon_1-\varepsilon_3$, where $\{\varepsilon_1, \varepsilon_2,\varepsilon_3\}$ is the dual basis
of the standard basis $\{e_1,e_2,e_3\}$ for $\R^3$. We write $\iota_1, \iota_2, \iota_3$ 
for the embeddings corresponding to
$\varepsilon_1-\varepsilon_2, \varepsilon_2-\varepsilon_3,
\varepsilon_1-\varepsilon_3$, respectively.
Then we put 
\begin{equation}\label{eqn:IntroEmb}
G'_i := \iota_i(G') \subset G
\end{equation}
for $G'=GL(2,\R), SL(2,\R)$.
See \eqref{eqn:embedding} for the concrete realization of the embeddings $\iota_i$.

We remark that the minimal parabolic subalgebra (Borel subalgebra) 
of $\mathfrak{gl}(3,\C)$ 
is \emph{compatible} with 
 the complexified Lie algebra $\frakg'_i$ of $G'_i$ 
for $i=3$, but not for $i=1,2$ 
(for the definition of compatibility, see Definition \ref{def:compatible}). 
The compatibility condition
affects the branching laws of Verma modules 
and, consequently, the DSBOs (we shall discuss them in Section \ref{sec:GVM}). 
In fact,
the results for DSBOs in the case $i=1,2$ are essentially the same,
whereas those for $G'_3$ are quite different.
For instance, the dimension of the space of DSBOs is at most one for 
$i=1,2$, but could be two for $i=3$ . 

We next present a detailed account of our results for Problem \ref{prob:class_const}
for $(G,G')$ in \eqref{eqn:GGprime}. 

%%%%%%%%%%%%%%%%%%%%%%%%%%%%%%%%%%%%%%%%%
\subsubsection{Case \texorpdfstring{$(G,G')=(GL(3,\R), GL(2,\R))$}{(G,G')=(GL(3,R),GL(2,R))}}
\label{sec:Intro_GL3GL2}

Let $I(\xi,\lambda)$ $(\xi\in (\ZZ/2\ZZ)^3,\lambda\in \CC^3)$, respectively 
$J(\eta,\nu)$ $(\eta\in (\ZZ/2\ZZ)^2,\nu\in \CC^2)$, denote 
the principal series representations
induced from the minimal parabolic subgroup of upper triangular matrices inside $G=GL(3,\RR)$, respectively $G'=GL(2,\RR)$, see Section \ref{sec:GL(3)} for more details. As in \eqref{eqn:IntroEmb}, we write $G'_i = \iota_i(G')\subset G$. 
Via the embedding $\iota_i \colon G' \hookrightarrow G$, 
we regard $J(\eta,\nu)$ as a representation of $G'_i$.

We put 
\begin{equation}\label{eqn:IntroDiff}
\Diff_{G_i^\prime}(I(\xi,\lambda),J(\eta,\nu)) 
:=\{\text{DSBOs $\DD \colon I(\xi,\lambda) \to J(\eta,\nu)$}\}.
\end{equation}
For $i\in \{1,2, 3\}$ and $j \in \{a ,b ,c \}$, let 
\begin{equation*}
\Supp_i(\DD),\, \Supp_{3,j}(\DD) \subset (\ZZ/2\ZZ)^5\times \CC^5
\end{equation*}
be the parameter sets 
given in Sections \ref{sec:Emb3} and \ref{sec:Emb12}.

Since the results for the case $i=3$ are considerably  different from 
those for the cases $i=1,2$, we consider them separately.
Moreover, it follows that the cases $i=1,2$ can be thought of as a degenerate case of $i=3$.
Therefore, we state the results for $i=3$ before those for $i=1,2$.

The classification of parameters is given as follows.

\begin{thmalph}[$i=3$, see Theorem \ref{thm:class3}]
\label{Intro:B}
The following conditions on $(\xi, \eta; \lambda, \nu) \in (\Z/2\Z)^5 \times \C^5$ are equivalent.
\begin{enumerate}[label = \normalfont{(\roman*)}]
    \item $\Diff_{G_3^\prime}\left(I(\xi, \lambda), J(\eta, \nu)\right) \neq \{0\}$.
    
    \item $\dim_\C \Diff_{G_3^\prime}\left(I(\xi, \lambda), J(\eta, \nu)\right) \in \{1, 2\}$.

    \item $(\xi, \eta; \lambda, \nu) \in \Supp_3(\DD)$.
\end{enumerate}
Moreover, the dimension of the space above is two if and only if $(\xi, \eta; \lambda, \nu) \in \Supp_{3,c}(\mathbb{D})$.
\end{thmalph}

\begin{thmalph}[$i=1,2$, see Theorem \ref{thm:class1}]
\label{Intro:C}
The following conditions on $(\xi, \eta; \lambda, \nu) \in (\Z/2\Z)^5 \times \C^5$ are equivalent.
\begin{enumerate}[label = \normalfont{(\roman*)}]
    \item $\Diff_{G_i^\prime}\left(I(\xi, \lambda), J(\eta, \nu)\right) \neq \{0\}$.
    
    \item $\dim_\C \Diff_{G_i^\prime}\left(I(\xi, \lambda), J(\eta, \nu)\right) = 1$.

    \item $(\xi, \eta; \lambda, \nu) \in \Supp_i(\DD)$. 
\end{enumerate}
\end{thmalph}

We next give the construction of the DSBOs in terms of left-invariant differential
operators on the opposite nilradical $\frakn_-(\R)=\Span_\R\{N_1^-,N_2^-,N_3^-\} \simeq \RR^3$. 
Consider  $D_i:=dR(N_i^-)$ for $i=1,2,3$, where $dR$ denotes the infinitesimal
right translation.
Then, for $\alpha=(\alpha_1,\alpha_2) \in \NN^2$
and 
$q(t)=\sum_{n=0}^{\min(\alpha)}c_n t^n \in \Pol_{\min(\alpha)}[t]$, 
let 
\begin{equation*}
(T^{\mathrm{sym}}_{\alpha} q)(D_1, D_2, D_3)
\quad
\text{and}
\quad
(T^\ord_{\alpha} q)(D_1, D_2, D_3)
\end{equation*}
denote 
the \emph{symmetrized operator} 
and
\emph{ordered operator}, respectively, defined in 
Section \ref{sec:const3}.
The operator
$(T^{\mathrm{sym}}_{\alpha} q)(D_1, D_2, D_3)$
(resp.\ $(T^\ord_{\alpha} q)(D_1, D_2, D_3)$)
represents the symmetrized form (resp.\ ordered form) of a DSBO.
Moreover, we define a polynomial
\begin{equation*}
p^{(\alpha)}_{a,b}(t)
:=\sum_{n=0}^{\min(\alpha)}\frac{(-\alpha_1)_n(-\alpha_2)_n}{2^n n!(-b)_n}
\Cay_n(a,b)\,t^\ell,
\end{equation*}
where $\Cay_n(a,b)$ is
the Cayley continuant mentioned above 
(see Appendix \ref{appendix:Cayely} for some details).
We write $\Rest_i$ for the restriction operator associated with the 
embedding $\iota_i\colon G' \hookrightarrow G$ for $i=1,2,3$.

\begin{thmalph}[$i=3$, see Theorem \ref{thm:cons3}]
\label{Intro:D}
We have
\begin{equation*}
\Diff_{G_3^\prime}\left(I(\xi, \lambda), J(\eta, \nu)\right)
=
\begin{cases}
\CC\, \DD^{(1)}_3(\lambda, \nu) & 
\textnormal{if $(\xi, \eta; \lambda, \nu) \in \Supp_{3,a}(\DD)$},\\[3pt]
\CC\, \DD^{(2)}_3(\lambda, \nu) & 
\textnormal{if $(\xi, \eta; \lambda, \nu) \in \Supp_{3,b}(\DD)$},\\[3pt]
\CC\, \DD^{(1)}_3(\lambda, \nu) 
\oplus
\CC\, \DD^{(2)}_3(\lambda, \nu)
& \textnormal{if $(\xi, \eta; \lambda, \nu) \in \Supp_{3,c}(\DD)$},\\[3pt]
\{0\} &\textnormal{otherwise},
\end{cases}
\end{equation*}
where $\DD^{(j)}_3(\lambda, \nu)$ for $j=1,2$ are given as follows.
\begin{align*}
\DD^{(1)}_3(\lambda,\nu)
&=\Rest_3\circ
\big(T^{\mathrm{sym}}_{\alpha} 
p^{(\alpha)}_{a,b}\big)
(D_1, D_2, D_3)\\[5pt]
&=\Rest_3\circ 
T^\ord_{\alpha}{}_3F_1\left [\begin{matrix}
    -\alpha_1 &-\alpha_2 & c\\
     &d&  
\end{matrix}; D_1, D_2, D_3 \right ],\\[10pt]
\DD^{(2)}_3(\lambda,\nu)
&=\Rest_3\circ \big(T^{\mathrm{sym}}_{\alpha} 
t^{m+1}p^{(\alpha)}_{a',-(m+2)}\big)(D_1, D_2, D_3) \\[5pt]
&=\Rest_3\circ  D_3^{m+1}
T^\ord_{\alpha}{}_3F_1\left [\begin{matrix}
    -\alpha_1' & -\alpha_2' & c'\\
    & m+2& 
\end{matrix}; D_1, D_2, D_3\right ].
\end{align*}
Here, $\alpha_1,\alpha_2,\alpha_1',\alpha_2',m\in \NN$ and $a,b,c,d,a',c'\in \CC$ are given in terms of $\lambda$ and $\nu$.
\end{thmalph}

\begin{thmalph}[$i=1,2$, see Theorem \ref{thm:cons1}]
\label{Intro:E}
We have 
\begin{equation*}
\Diff_{G_i^\prime}\left(I(\xi, \lambda), J(\eta, \nu)\right)
=
\begin{cases}
\CC\, \DD_i(\lambda, \nu) & 
\textnormal{if $(\xi, \eta; \lambda, \nu) \in \Supp_{i}(\DD)$},\\[3pt]
\{0\} & \textnormal{otherwise},
\end{cases}
\end{equation*}
where $\DD_i(\lambda, \nu)$ for $i=1,2$ are given as follows.
\vspace{5pt}
\begin{align*}
\DD_1(\lambda,\nu)
&=\Rest_1\circ
\big(T^{\mathrm{sym}}_{\alpha}
p_{a,\alpha_2}^{(\alpha)}\big)(D_1,D_2,D_3)\\[5pt]
&=\Rest_1\circ 
T^\ord_{\alpha}{}_2F_0
[-\alpha_1, c;  D_1, D_2, D_3],\\[10pt]
\DD_2(\lambda,\nu)
&=\Rest_2\circ 
\big(T^{\mathrm{sym}}_{\alpha'}
p_{a',\alpha_2'}^{(\alpha')}\big)(D_1,D_2,D_3) \\[5pt]
&=\Rest_2\circ  
T^\ord_{\alpha'}
{}_2F_0[-\alpha_1', c';D_1, D_2, D_3].
\end{align*}
Here, $\alpha=(\alpha_1,\alpha_2),\alpha^\prime=(\alpha_1',\alpha_2')\in \NN^2$ and $a,a',c,c'\in \CC$ are given in terms of $\lambda$ and $\nu$.
\end{thmalph}

We remark that the symmetrization operator $S$ plays an essential role to prove 
Theorems \ref{Intro:B}--\ref{Intro:E}. Here is an outline of our strategy.
As mentioned above, we start with the case $i=3$, thereby proving
Theorems \ref{Intro:B} and \ref{Intro:D}.
In this case, 
we first solve the ordered F-system $\Ad(S)D_F$, since it can be realized by using
the hypergeometric differential equation for ${}_3F_1$. After that, in order to solve the original F-system $D_F$,
we apply the inverse of the symmetrization operator $S^{-1}$ to the solution to the ordered F-system 
$\Ad(S)D_F$.
The desired symmetrized and ordered forms of DSBOs 
are then obtained through $\Symb_{0}^{-1}$ and $\Symb_{\ord}^{-1}$, respectively
(cf.\ \eqref{eqn:A}).
It is remarked that we use a suitable T-saturation to solve 
both the original F-system $D_F$ and the ordered F-system $\Ad(S)D_F$.

To prove Theorems \ref{Intro:C} and \ref{Intro:E}, we begin with the case $i=1$.
In this case, we reduce the computations to the case $i=3$.
The ${}_3F_1$-hypergeometric polynomials in Theorem \ref{Intro:D} 
are degenerated to the ${}_2F_0$-hypergeometric polynomials
in Theorem \ref{Intro:E}.
The results for $i=2$ follow from those for $i=1$ via a certain duality 
(see Section \ref{sec:duality}). 

%%%%%%%%%%%%%%%%%%%%%%%%%%%%%%%%%%%%%%%%%
\subsubsection{Case \texorpdfstring{$(G,G')=(GL(3,\R), GL(3,\R))$}{(G,G')=(GL(3,R),GL(2,R))}}

In the case $G=G'$, the DSBOs are often referred to as 
\emph{differential intertwining operators} (DIOs).
As in \eqref{eqn:IntroDiff}, we write 
$\Diff_{G}\left(I(\xi, \lambda), I(\xi', \lambda')\right)$
for the space of DIOs $\calD \colon  I(\xi, \lambda) \to I(\xi', \lambda')$.

Similar to the previous cases in Section \ref{sec:Intro_GL3GL2}, 
let $\Supp(\calD)$ be the parameter set defined in Section \ref{sec:DIO}.

\begin{thmalph}[See Theorem \ref{thm:class_DIOs}]
\label{Intro:F}
The following conditions on  $(\xi, \xi'; \lambda, \lambda') \in (\Z/2\Z)^6 \times \C^6$ are equivalent.
\begin{enumerate}[label = \normalfont{(\roman*)}]
    \item $\Diff_{G}\left(I(\xi, \lambda), I(\xi', \lambda')\right) \neq \{0\}$.
    
    \item $\dim_\C \Diff_{G}\left(I(\xi, \lambda), I(\xi', \lambda')\right) = 1$.

    \item $(\xi, \xi'; \lambda, \lambda') \in \Supp(\calD)$.
\end{enumerate}
\end{thmalph}

\begin{thmalph}[See Theorem \ref{thm:cons_DIOs}]
\label{Intro:G}
We have
\begin{equation*}
\Diff_{G}\left(I(\xi, \lambda), I(\xi', \lambda')\right)
=
\begin{cases}
\CC\, \calD(\lambda, \lambda') & 
\textnormal{if $(\xi, \xi'; \lambda, \lambda') \in \Supp(\calD)$},\\[3pt]
\{0\} & \textnormal{otherwise},
\end{cases}
\end{equation*}
where $\calD(\lambda, \lambda')$ 
is given as follows.
\begin{align*}
\calD(\lambda,\lambda')
&=T^{\mathrm{sym}}_{\alpha}
p_{a,\alpha_2}^{(\alpha)}
(D_1, D_2, D_3)\\[3pt]
&=T^\ord_{\alpha}
{}_2F_0\Big [-\alpha_1, c; D_1, D_2, D_3\Big ].
\end{align*}
Here, $\alpha=(\alpha_1,\alpha_2)\in \NN^2$ and $a,c\in \CC$ are given in terms of $\lambda$ and $\lambda'$.
\end{thmalph}

Theorems \ref{Intro:F} and \ref{Intro:G} are obtained by considering the simultaneous polynomial solutions to the original F-system $(i=1,2)$ 
and those to the ordered F-system $(i=1,2)$ in Section \ref{sec:Intro_GL3GL2}.
We remark that the ${}_2F_0$-hypergeometric polynomial 
appears also in the 
symmetrized form of the DIOs $\calD(\lambda, \lambda')$ for certain parameters. 
See Theorem \ref{thm:cons_DIOs} for more detailed formulas of 
$\calD(\lambda, \lambda')$.

%%%%%%%%%%%%%%%%%%%%%%%%%%%%%%%%%%%%%%%%%
\subsubsection{Case \texorpdfstring{$(G,G')=(SL(3,\R), SL(2,\R))$}{(G,G')=(SL(3,R),SL(2,R))}}

To distinguish the $SL$-case from the $GL$-case, 
we write $G_{SL}:=SL(3,\RR)$ and $G'_{SL}:=SL(2,\RR)$. 
Then, as in the $GL$-case, 
we denote by  $I(\delta,u)_{SL}$ $(\delta \in (\ZZ/2\ZZ)^2,u\in \CC^2)$ 
and $J(\sigma,v)_{SL}$ $(\sigma \in \ZZ/2\ZZ, v\in \CC)$
the principal series representations
of $G_{SL}$ and $G'_{SL}$, respectively. 
See Section \ref{sec:SL} for more details. 

We put $G'_{SL,i} := \iota_i(G'_{SL})\subset G_{SL}$ and 
write
$\Diff_{G'_{SL,3}}(I(\delta,u)_{SL}, J(\sigma,v)_{SL})$
for the space of 
DSBOs $\DD \colon I(\delta,u)_{SL} \to J(\sigma,v)_{SL}$.
For $i=1,2,3$,  let $\Supp_{SL,i}(\DD)$ be the parameter set defined in Section \ref{sec:class-const-SL}.

We start with the classification theorems. It turns out that the classification of DSBOs
in the $SL$-case is significantly different from that for the $GL$-case; 
for instance, the space of DSBOs for the $SL$-case is not uniformly bounded.
Since the results for $i=1,2$ and $i=3$ are also substantially different, 
we consider the two cases separately.

\begin{thmalph}
[{$i=1,2$, see Theorem \ref{thm:classSL12}}]
\label{Intro:H}
The following conditions on $(\delta,\sigma; u, v)  
\in (\Z/2\Z)^3\times \C^3$ are equivalent.
\begin{enumerate}[label=\normalfont{(\roman*)}]
\item $\Diff_{G'_{SL,i}}(I(\delta,u)_{SL}, J(\sigma,v)_{SL})\neq \{0\}$.
\item $\dim_\C\Diff_{G'_{SL,i}}(I(\delta,u)_{SL}, J(\sigma,v)_{SL})=\infty$.
\item $(\delta,\sigma; u, v) \in \Supp_{SL,i}(\DD)$.
\end{enumerate}
\end{thmalph}

\begin{thmalph}
[{$i=3$, see Theorem \ref{thm:classSL3}}]
\label{Intro:I}
The following conditions on 
$(\delta,\sigma; u, v)  \in (\Z/2\Z)^3\times \C^3$ are equivalent.
\begin{enumerate}[label=\normalfont{(\roman*)}]
\item $\Diff_{G'_{SL,3}}(I(\delta,u)_{SL}, J(\sigma,v)_{SL})\neq \{0\}$.
\item $(\delta,\sigma; u, v) \in \Supp_{SL,3}(\DD)$.
\end{enumerate}
Moreover, for $(\delta,\sigma; u, v)  \in \Supp_{SL,3}(\DD)$ with $j=-(u_1+u_2-v)\in\N$,
we have
\begin{equation*}
\dim_\C\Diff_{G'_{SL,3}}(I(\delta,u)_{SL}, J(\sigma,v)_{SL}) \geq j+1.
\end{equation*}
\end{thmalph}

We next consider the construction of DSBOs. 

\begin{thmalph}
[{See Theorems \ref{thm:consSL12} and \ref{thm:consSL3}}]
\label{Intro:J}
Any DSBO $\DD \in \Diff_{G'_{SL,i}}(I(\delta,u)_{SL}, J(\sigma,v)_{SL})$
can be given as a linear combination of DSBOs 
for the $GL$-case for all $i=1,2,3$.
\end{thmalph}

In Theorems \ref{thm:consSL12} and \ref{thm:consSL3}, 
we give the DSBOs for the $GL$-case contributing to 
$\DD \in \Diff_{G'_{SL,i}}(I(\delta,u)_{SL}, J(\sigma,v)_{SL})$ explicitly.
A key idea for the proofs of Theorems \ref{Intro:H}, \ref{Intro:I} and \ref{Intro:J} is 
to make use of the classification of polynomial solutions for $i=1,2,3$ in
Section \ref{sec:Intro_GL3GL2}. We shall discuss the details in Section \ref{sec:SL}.

%%%%%%%%%%%%%%%%%%%%%%%%%%%%%%%%%%%%%%%%%
\subsection{Branching laws of Verma modules}

In addition to the DSBOs for $(G,G')$ in \eqref{eqn:GGprime}, we also consider 
the branching laws of Verma modules in the Grothendieck group for the pair 
$(\frakg,\frakg')$ in \eqref{eqn:ggprime} with the embeddings $\iota_i$ for $i=1,2,3$.
Although explicit branching laws are obtained, 
we state only multiplicity results here.

Let $(\frakg,\frakg')=(\mathfrak{gl}(3,\C), \mathfrak{gl}(2,\C))$.
We write $m_i(\nu,\lambda)$ for the multiplicity of the Verma module
$[M^{\frakg'_i}_{\frakb'_i}(\nu)]$
for $\frakg_i':=\iota_i(\mathfrak{gl}(2,\C))$ in the branching law
for $[\Mb(\lambda)\vert_{\frakg_i'}]$ in the Grothendieck group.
Here $\frakb$ and $\frakb'_i$ are the Borel subalgebras of $\frakg$ and $\frakg'_i$
corresponding to the minimal parabolic subgroups for $G$ and $G'_i$, respectively.

In the $\mathfrak{gl}$-case, the multiplicity $m_i(\nu,\lambda)$ behaves nicely as follows.

\begin{thmalph}
[{See Theorem \ref{thm:BranchingEmb}}]
\label{Intro:K}
We have $m_i(\nu,\lambda) \leq 1$ for all $i=1,2,3$.
\end{thmalph}

Next, let $(\frakg,\frakg')=(\mathfrak{sl}(3,\C), \mathfrak{sl}(2,\C))$.
To distinguish this case from the $\mathfrak{gl}$-case, we write
$\frakg_{\fraksl} := \mathfrak{sl}(3,\C)$ and $\frakg_{\fraksl,i}':=\iota_i(\mathfrak{sl}(2,\C))$.
We then denote by
$m_{\fraksl,i}(v;u)$ the multiplicity of the Verma module
$[M^{\frakg'_{\fraksl,i}}_{\frakb'_{\fraksl,i}}(v)]$ for $\frakg_{\fraksl,i}'$ 
in the branching law $[\Mbsl(u)\vert_{\frakg_{\fraksl,i}'}]$ in the Grothendieck group.
As in the $\mathfrak{gl}$-case, let  $\frakb_{\fraksl}$ (resp.\ $\frakb'_{\fraksl,i}$) denote
the Borel subalgebra of $\frakg_{\fraksl}$ (resp.\ $\frakg'_{\fraksl,i}$).

We consider the cases $i=1,2$ and $i=3$, separately.

\begin{thmalph}
[{$i=1,2$, see Theorem \ref{thm:BranchingEmb_SL12}}]
\label{Intro:L}
Given $u \in \C^2$, the following conditions on $v \in \C$ are equivalent.
\begin{enumerate}[label=\normalfont{(\roman*)}]
\item $m_{\fraksl,i}(v;u)\neq 0$.
\item $m_{\fraksl,i}(v;u)=\infty$.
\end{enumerate}
\end{thmalph}

\begin{thmalph}
[{$i=3$, see Theorem \ref{thm:BranchingEmb_SL3}}]
\label{Intro:M}
Given $u \in \C^2$, for every $k \in \N_+$, there exists $v_k \in \C$ such that
$m_{\fraksl,3}(v_k;u)=k$.
\end{thmalph}

Theorems \ref{Intro:K}, \ref{Intro:L} and \ref{Intro:M} are obtained by
Kobayashi's character formula of generalized Verma modules
(see Theorem \ref{thm:bGVM}).
In addition to the branching laws in the Grothendieck group,
we also give the branching laws for actual Verma modules 
for irreducible parameters (see Corollaries \ref{cor:BranchingEmb} and
\ref{cor:BranchingEmb_SL3}).
At the end of Section \ref{sec:GVM}, we further discuss the relations between
the results of DSBOs and Verma modules.

%%%%%%%%%%%%%%%%%%%%%%%%%%%%%%%%%%%%%%%%%
\subsection{Relation to other work}

This paper provides the first systematic treatment of the symmetrized and ordered forms of DSBOs arising from the non-abelian structure of the nilpotent radical. 
However, our work overlaps to some extent with the existing literature.
For instance, 
the results of Theorems \ref{Intro:B}--\ref{Intro:E}  
were obtained by the first author with Q.\ Labriet in \cite{DiLa26}, where DSBOs are given in the ordered form in terms of distributional kernels on the opposite unipotent radical $N_-$. The $SL(3,\RR)$-analogue of Theorems \ref{Intro:F} and \ref{Intro:G}
was also previously obtained by the last two authors in its symmetrized form in \cite{KPV25}. See Table~\ref{Table:1} for an overview and comparison of the various realizations. 
We remark that even in the overlapping cases, the strategies to construct DSBOs
are significantly different. In \cite{DiLa26, KPV25}, certain recurrence relations
are solved to obtain DSBOs, whereas we investigate 
the differential equations themselves and apply our theory of the ordered F-system in this paper.

\begin{table}[t]\label{Table:1}
\caption{Relationship to Ditlevsen--Labriet and Kubo--P{\'e}rez-Vald{\'e}s}
\begin{center}
\renewcommand{\arraystretch}{1.2}
{
\begin{tabular}{|c|c|c|c|}
\hline
& Differential operators & Symmetrized form & Ordered form\\
 \hline
 This paper & DSBOs and DIOs & Yes & Yes\\
 \hline
Ditlevsen--Labriet \cite{DiLa26} & DSBOs & No & Yes\\
\hline
Kubo--P{\'e}rez-Vald{\'e}s \cite{KPV25} & DIOs &  Yes & No\\
\hline
\end{tabular}
}
\end{center}

\end{table}%

We also mention the work of K{\v{r}}i{\v{z}}ka--Somberg \cite{KrSom17}, who studied the classification of differential symmetry breaking operators for the pair
$(SL(n+2,\CC),SL(n+2-r,\CC))$ with Heisenberg parabolic subgroups
under the condition $n-r>2$.
Their assumptions nevertheless exclude the low-rank cases
considered in the present paper.

%%%%%%%%%%%%%%%%%%%%%%%%%%%%%%%%%%%%%%%%%%
\subsection{Organization of the paper}
This paper is organized into twelve sections including this introduction, followed by three appendices.
The aim of Section \ref{sec:Fmethod} is to give an overview of the F-method. 
We in particular review the inverse of the truncated symbol map
$\Symb_0^{-1}$. Then we introduce the symmetrization operator $S$ and the ordered F-system $\Ad(S)D_F$ in Section \ref{sec:nonsym}. We first discuss the symmetrization operator $S$ in full generality, namely, for an arbitrary finite-dimensional Lie algebra $\frakg$ over $\F$, where $\F=\R$ or $\C$. 
Then, the discussion is specialized to 
the case of the nilpotent radical $\frakn_-$ of a parabolic subalgebra $\frakp$
over $\C$. 
The ordered F-system
$\Ad(S)D_F$ and the ordered symbol map $\Symb_{\ord}$ are also defined.
At the end of this section, we provide the recipe of the F-method using the ordered F-system.

Sections \ref{sec:GL(3)}--\ref{sec:DIO} are devoted to solving 
Problem \ref{prob:class_const} for the $GL$-case; namely, 
$G=GL(3,\R)$ and 
$G'=GL(2,\R),GL(3,\R)$. 
In Section \ref{sec:GL(3)},
we fix some notation and normalizations to prepare for the problem. 
The symmetrization operator $S$ and the ordered F-system $\Ad(S)D_F$ are 
then studied for $G=GL(3,\R)$ in Section \ref{sec:ordered-GL(3)}. 
We consider Problem \ref{prob:class_const}  for 
$(G,G')=(GL(3,\R),GL(2,\R))$
in Sections \ref{sec:Emb3}--\ref{sec:proof2} for $i=3$ and 
in Section \ref{sec:Emb12} for $i=1,2$.
As mentioned above, for $i=3$, we solve the ordered F-system
$\Ad(S)D_F$ first, and then solve the original F-system $D_F$ by applying  
the inverse of the symmetrization operator  $S^{-1}$ (see \eqref{eqn:A}).
In Section \ref{sec:Emb12}, we utilize the results for $i=3$ to obtain DSBOs for $i=1,2$. 
After obtaining all DSBOs for $i=1,2$, we consider Problem \ref{prob:class_const} for $G=G'=GL(3,\R)$ in Section \ref{sec:DIO}. The key idea is to investigate the simultaneous solutions to the F-system for $i=1,2$.

The aim of Section \ref{sec:SL} is to study DSBOs for $(G,G')=(SL(3,\R), SL(2,\R))$.
In this section, we in particular show that any DSBO for the $SL$-case 
can be given as a linear combination of those for the $GL$-case.

It is known that there is a certain duality between the space of DSBOs and 
that of homomorphisms between generalized Verma modules. 
With this in mind, we discuss the branching laws of Verma modules for $(\frakg, \frakg')$ in \eqref{eqn:ggprime} in Section \ref{sec:GVM}. 
At the end of the section 
we verify our results for DSBOs from the viewpoint of Verma modules.

The last three appendices are devoted to topics 
that might otherwise interrupt the flow of the main arguments,
namely, the symmetrization operator for Heisenberg Lie algebras of arbitrary 
dimension, the proof of Proposition \ref{Prop:SolGeneralDiffEq1} 
and an alternative proof 
of Corollary \ref{prop:SolD1}. 

We discuss the symmetrization operator
of Heisenberg Lie algebras in Appendix \ref{appendix:Heis}. 
It will be shown that the symmetrization operator can be 
given as a product of copies of the symmetrization operator of the 
three-dimensional Heisenberg Lie algebra.

The main purpose of Appendix \ref{appendix:Cayely} is to give the proof of 
Proposition \ref{Prop:SolGeneralDiffEq1}, which concerns the symmetrized forms
of DSBOs for $i=3$. Since the Cayley continuants $\Cay_n(a,b)$ are involved in
the arguments, we first review some basic properties of these tridiagonal
determinants. We then give the proof by 
applying the inverse of the symmetrization operator 
$S^{-1}$ to the ${}_3F_1$-hypergeometric polynomials.

In Appendix \ref{appendix:2F0}, which is the last material of this paper, 
we give an alternative proof of Corollary \ref{cor:Sol2F0}, 
which provides
the polynomial solutions to the ordered F-system $\Ad(S)D_F$ for $i=1$. 
In Section \ref{sec:Thms12B}, we obtain the assertion 
by applying the symmetrization operator $S$
to the polynomial solutions to the original F-system $D_F$ (see \eqref{eqn:A}).
In the appendix we carefully observe 
the hypergeometric differential equation for ${}_2F_0$
to solve the ordered F-system $\Ad(S)D_F$ directly.

\medskip
%%%%%%%%%%%%%%%%%%%%%%%%%%%%%%%%%%%%%%%%%%
\noindent
\textbf{Notation}: 
$\N:=\{0, 1, 2, 3, \ldots\}$ and $\N_+:=\{1, 2, 3, \ldots \}$
\vspace{5pt}

%%%%%%%%%%%%%%%%%%%%%%%%%%%%%%%%%%%%%%%%%%
\noindent
\textbf{Convention}:
For $b < a$, we regard $[a,b]$ as $[a,b]=\emptyset$.
\vspace{5pt}

%%%%%%%%%%%%%%%%%%%%%%%%%%%%%%%%%%%%%%%%%%
\section{Brief overview of the F-method}
\label{sec:Fmethod}

The aim of this section is to overview the F-method and the inverse 
of the truncated symbol map $\Symb_0^{-1}$ for DSBOs $\DD$. For the details of the F-method, see, for instance, \cite{KoPe16a, KPV25}. 

%%%%%%%%%%%%%%%%%%%%%%%%%%%%%%%%%%%%%%%%%%
\subsection{Notation}\label{sec:prelim}

Let $G$ be a real reductive Lie group and $P=MAN_+ $ a Langlands decomposition of a parabolic subgroup $P$ of $G$. We denote by $\frakg(\R)$ and 
$\frakp(\R) = \frakm(\R) \oplus \fraka(\R) \oplus \frakn_+(\R)$ the Lie algebras of $G$ and 
$P=MAN_+$, respectively.

For a real Lie algebra $\mathfrak{y}(\R)$, we write $\mathfrak{y}$
and $U(\mathfrak{y})$ for its complexification and the universal enveloping algebra of 
$\fraky$, respectively. 
For instance, $\frakg, \frakp, \frakm, \fraka$, and $\frakn_+$ are the complexifications of $\frakg(\R), \frakp(\R), \frakm(\R), \fraka(\R)$, and $\frakn_+(\R)$, 
respectively.

For $\lambda \in \fraka^* \simeq \Hom_\R(\fraka(\R),\C)$,
we denote by $\C_\lambda$ 
the one-dimensional representation of $A$ defined by 
$a\mapsto a^\lambda:=e^{\lambda(\log a)}$. 
For a finite-dimensional irreducible 
representation $(\sigma, V)$ of $M$ and $\lambda \in \fraka^*$,
we consider  the outer tensor product representation 
$\sigma_\lambda:=\sigma \boxtimes \C_\lambda$
on $V$,
namely, 
$\sigma_\lambda \colon
ma \mapsto a^\lambda\sigma(m)$. By letting $N_+$ act trivially, we regard 
$\sigma_\lambda$ as a representation of $P$. Let $\mathcal{V}:=G \times_P V \to G/P$
be the $G$-equivariant vector bundle over the real flag variety $G/P$
 associated with the 
representation $(\sigma_\lambda, V)$ of $P$. We identify the Fr{\' e}chet space 
$C^\infty(G/P, \mathcal{V})$ of smooth sections with 
\begin{equation*}
C^\infty(G, V)^P:=\{f \in C^\infty(G,V) : 
f(gp) = \sigma_\lambda^{-1}(p)f(g)
\;\;
\text{for any $p \in P$}\},
\end{equation*}
the space of $P$-invariant smooth functions on $G$.
Then, via the left regular representation $L$ of $G$ on $C^\infty(G)$,
we realize the parabolically induced representation 
\begin{equation*}
\pi_{(\sigma, \lambda)} = I(\sigma,\lambda)
:=\Ind_{P}^G(\sigma\boxtimes \CC_\lambda)
\end{equation*}
on $C^\infty(G/P, \mathcal{V})$.

Let $G'$ be a reductive subgroup of $G$ and $P'=M'A'N_+'$ 
a parabolic subgroup of $G'$
with $P' \subset P$
so that there exists a natural morphism $G'/P' \to G/P$
from the real flag variety $G'/P'$ to $G/P$.
We further assume that 
\begin{equation}\label{eqn:MAN}
M'A' \subset MA \quad \text{and} \quad N_+' \subset N_+.
\end{equation}

As for $G/P$, given a finite-dimensional irreducible
representation $(\delta_\nu, W)$ of $M'A'$ with 
$\delta_\nu:=\delta \boxtimes \CC_{\nu}$,
we define
the induced representation 
\begin{equation*}
\pi'_{(\delta, \nu)}=J(\delta, \nu):=\Ind_{P'}^{G'}(\delta \boxtimes \CC_{\nu})
\end{equation*}
on the space $C^\infty(G'/P', \mathcal{W})$ of smooth sections for the 
$G'$-equivariant vector bundle $\mathcal{W}:=G'\times_{P'}W \to G'/P'$.

Via the morphism $G'/P' \to G/P$, we define 
differential operators $\DD\colon C^\infty(G/P, \mathcal{V}) \to C^\infty(G'/P', \mathcal{W})$
in the sense of \cite[Def.\ 2.1]{KoPe16a}.
As $C^\infty(G/P, \mathcal{V})$ is a $G$-representation and $G' \subset G$,
the space $C^\infty(G/P, \mathcal{V})$ is also a $G'$-representation.
We then denote by
\begin{equation*}
\Diff_{G'}(I(\sigma,\lambda),J(\delta,\nu))
\end{equation*}
the space of 
\emph{differential symmetry breaking operators} 
(differential $G'$-intertwining operators)
$\DD \colon C^\infty(G/P, \mathcal{V}) \to C^\infty(G'/P', \mathcal{W})$.

Let 
$\mathfrak{g}(\R)=\mathfrak{n}_-(\R) \oplus \mathfrak{m}(\R) 
\oplus \mathfrak{a}(\R) \oplus \mathfrak{n}_+(\R)$ 
be the Gelfand--Naimark decomposition of $\mathfrak{g}(\R)$
compatible with the Langlands decomposition
$\frakp(\R) = \frakm(\R) \oplus \fraka(\R) \oplus \frakn_+(\R)$,
and write $N_- = \exp(\frakn_-(\R))$. We identify $N_-$ with the 
open Bruhat cell $N_-P$ of $G/P$ via the embedding 
$\iota\colon N_- \hookrightarrow G/P$, $\bar{n} \mapsto \bar{n}P$.
By restricting the vector bundle $\mathcal{V} \to G/P$ to the open Bruhat cell
$N_-\stackrel{\iota}{\hookrightarrow} G/P$,
we regard $C^\infty(G/P,\mathcal{V})$ as a subspace of 
$C^\infty(N_-) \otimes V$.

Likewise,
let 
$\mathfrak{g}'(\R)=\mathfrak{n}_-'(\R) \oplus 
\mathfrak{m}'(\R) \oplus \mathfrak{a}'(\R) \oplus \mathfrak{n}_+'(\R)$ 
be the Gelfand--Naimark decomposition of $\mathfrak{g}'(\R)$
so that
$\mathfrak{p}'(\R) = \mathfrak{m}'(\R) \oplus \mathfrak{a}'(\R) \oplus 
\mathfrak{n}_+'(\R)$ is the Langlands decomposition of 
the parabolic subalgebra $\mathfrak{p}'(\R)$ of $P'$ with respect to $P'=M'A'N_+'$. 
Write $N_-' = \exp(\frakn_-'(\R))$.
As for $C^\infty(G/P,\mathcal{V})$, we regard $C^\infty(G'/P',\mathcal{W})$
as a subspace of $C^\infty(N_-')\otimes W$.

We often view differential symmetry breaking operators
$\DD \colon
C^\infty(G/P,\mathcal{V})
\to C^\infty(G'/P',\mathcal{W})$
as differential operators 
\begin{equation*}
\widetilde{\DD} \colon C^\infty(N_-) \otimes V
\to C^\infty(N_-') \otimes W
\end{equation*}
such that
the restriction $\widetilde{\DD}\vert_{C^\infty(G/P,\mathcal{V})}$
to $C^\infty(G/P,\mathcal{V})$ is a map
\begin{equation*}
\widetilde{\DD}\vert_{C^\infty(G/P,\mathcal{V})}\colon
C^\infty(G/P,\mathcal{V})
\to C^\infty(G'/P',\mathcal{W}).
\end{equation*}
In particular, we regard $\Diff_{G'}(I(\sigma,\lambda),J(\delta,\nu))$ as
\begin{equation}\label{eqn:space_DifferenitalOperators}
\Diff_{G'}(I(\sigma,\lambda),J(\delta,\nu))
\subset \Diff_\C(C^\infty(N_-)\otimes V, C^\infty(N_-')\otimes W).
\end{equation}

%%%%%%%%%%%%%%%%%%%%%%%%%%%%%%%%%%%%%%%%%%
\subsection{Fourier transformed representation \texorpdfstring{$\widehat{d\pi_{(\sigma,\lambda)^*}}$}{ }}
\label{sec:dpi2}

For a representation $\mu$ of $G$, we denote by $d\mu$ 
the infinitesimal representation of $\frakg(\R)$. 
For instance,  $dL$ and $dR$ denote the infinitesimal representations of $\frakg(\R)$ of
the left and right regular representations $L$ and $R$ of $G$ on $C^\infty(G)$, respectively.
As usual, we naturally extend representations of $\frakg(\R)$ 
to ones of the universal enveloping algebra $U(\frakg)$ of its complexification $\frakg$.
The same convention is applied for closed subgroups of $G$.

For $g \in N_-MAN_+$, we write
\begin{equation*}
g = p_-(g)p_0(g)p_+(g),
\end{equation*}
where $p_\pm(g) \in N_{\pm}$ and $p_0(g) \in MA$. 
Similarly, for $Y \in \frakg = \frakn_- \oplus \frakl\oplus \frakn_+$ with $\frakl= \frakm \oplus \fraka$,
we write
\begin{equation*}
Y=Y_{\frakn_-} + Y_{\frakl} + Y_{\frakn_+},
\end{equation*}
where $Y_{\frakn_{\pm}} \in \frakn_{\pm}$ and $Y_\frakl \in \frakl$.

For $2\rho\equiv 2\rho(\frakn_+)= \mathrm{Trace}(\ad\vert_{\frakn_+})\in \mathfrak{a}^*$,
we denote by  $\C_{2\rho}$ the one-dimensional representation of $P$
defined by
$p \mapsto \chi_{2\rho}(p)=
\abs{\mathrm{det}(\mathrm{Ad}(p)\colon 
\mathfrak{n}_+ \to \mathfrak{n}_+)}$.
For the contragredient representation
$((\sigma_\lambda)^\vee, V^\vee)$ of $(\sigma_\lambda,V)$,
we consider the outer tensor product representation
$\sigma^\vee \boxtimes \C_{2\rho-\lambda}$.
As for $\sigma \boxtimes \CC_\lambda$, 
we regard $\sigma^\vee \boxtimes \C_{2\rho-\lambda}$ 
as a representation of $P$.
Define the induced representation
$\pi_{(\sigma, \lambda)^*} 
= \mathrm{Ind}_P^G(\sigma^\vee\boxtimes\CC_{2\rho-\lambda})$
on the space $C^\infty(G/P,\mathcal{V}^*)$ of smooth sections 
for the vector bundle $\mathcal{V}^*=G\times_P (V^\vee\otimes \C_{2\rho})$
associated with $\sigma^\vee \boxtimes \C_{2\rho-\lambda}$, 
which is isomorphic to the tensor bundle
of the dual vector bundle $\mathcal{V}^\vee = G\times_PV^\vee$ and the bundle
of densities over $G/P$.
Then, the integration on $G/P$ gives a 
$G$-invariant non-degenerate bilinear form
$
\mathrm{Ind}_P^G(\sigma\boxtimes\CC_{\lambda})
\times 
\mathrm{Ind}_P^G(\sigma^\vee\boxtimes\CC_{2\rho-\lambda}) \to \C.
$

As for $C^\infty(G/P,\mathcal{V})$,
the space $C^\infty(G/P,\mathcal{V}^*)$ can be regarded as 
a subspace of $C^\infty(N_-) \otimes V^\vee$.
Then the infinitesimal representation $d\pi_{(\sigma,\lambda)^*}(X)$
on $C^\infty(N_-)\otimes V^\vee$
for $X \in \frakg$ is given as
\begin{equation}\label{eqn:dpi3}
d\pi_{(\sigma,\lambda)^*}(X)f(\bar{n})
=d\sigma_\lambda^*((\Ad(\bar{n}^{-1})X)_\frakl)f(\bar{n})
-\left(dR((\Ad(\cdot^{-1})X)_{\frakn_-})f\right)(\bar{n}).
\end{equation}
(For the details, see, for instance, \cite[Sect.\ 2]{KuOr25a}.)
Via the exponential map $\exp\colon \frakn_-(\R) \xrightarrow{\sim}  N_-$, one can regard
$d\pi_{(\sigma,\lambda)^*}(X)$ as a representation 
on $C^\infty(\mathfrak{n}_-(\R)) \otimes V^\vee$.

Let  
\begin{equation}\label{eqn:ord-n}
\ord:=\{X_1,\ldots, X_n\}
\end{equation}
be an ordered basis of $\frakn(\R)$
with $n:=\dim \frakn_-(\R)$.
We write $(x_1, \ldots, x_n)$ for the coordinates on $\frakn_-(\R)$ with respect to 
the ordered basis $\ord$.
It then follows from \eqref{eqn:dpi3} that 
$d\pi_{(\sigma,\lambda)^*}$ gives a Lie algebra homomorphism
\begin{equation*}
d\pi_{(\sigma,\lambda)^*}\colon \mathfrak{g} \To 
\C[\frakn_-(\R);x, \tfrac{\partial}{\partial x}]  \otimes \mathrm{End}(V^\vee).
\end{equation*}
We regard $\ord=\{X_1,\ldots, X_n\}$ also as an ordered basis of $\frakn_-$.
Let  $(z_1, \ldots, z_n)$ denote the coordinates on $\frakn_-$
with respect to the basis $\ord$.
Then $d\pi_{(\sigma,\lambda)^*}$ can be considered as a homomorphism
\begin{equation}\label{eqn:dpi-complex-g}
d\pi_{(\sigma,\lambda)^*}\colon \mathfrak{g} \To 
\C[\frakn_-;z, \tfrac{\partial}{\partial z}]  \otimes \mathrm{End}(V^\vee).
\end{equation}

Now we fix a non-degenerate $\Ad$-invariant symmetric bilinear form $\kappa$ on
$\frakg$. Via $\kappa$, we identify 
$\frakn_+$  with the dual space $\frakn_-^\vee$ of 
$\frakn_-$. 
Write $(\zeta_1, \ldots, \zeta_n)$ for the dual coordinates on 
$\frakn_+\simeq \frakn_-^\vee$ corresponding to $(z_1, \ldots, z_n)$ on $\frakn_-$.
Then
the algebraic Fourier transform $\widehat{\;\;\cdot\;\;}$
of Weyl algebras \cite[Def.\ 3.1]{KoPe16a} gives a Weyl algebra isomorphism
\begin{equation*}
\widehat{\;\;\cdot\;\;}\;\colon
\C[\frakn_-;z, \tfrac{\partial}{\partial z}] 
\stackrel{\sim}{\To} 
\C[\frakn_+;\zeta, \tfrac{\partial}{\partial \zeta}].
\end{equation*}
In particular, it gives a Lie algebra homomorphism
\begin{equation}\label{eqn:hdpi}
\widehat{d\pi_{(\sigma,\lambda)^*}}\colon \mathfrak{g} \To 
\C[\frakn_+;\zeta, \tfrac{\partial}{\partial \zeta}]\otimes \mathrm{End}(V^\vee).
\end{equation}

Now consider the generalized Verma module 
\begin{equation*}
\Mp(V^\vee) := U(\frakg)\otimes_{U(\frakp)}V^\vee
\end{equation*}
for the pair $(\frakg, \frakp)$ induced from $V^\vee$. 
By letting the parabolic subgroup $P$ acting on $\Mp(V^\vee)$ diagonally, we regard
$\Mp(V^\vee)$ as a $(\frakg, P)$-module. It follows from \cite[Sect.\ 3.4]{KoPe16a} that 
the Lie algebra homomorphism \eqref{eqn:hdpi} gives rise to a $(\frakg,P)$-isomorphism 
\begin{equation}\label{eqn:Fc}
F_c\colon \Mp(V^\vee) 
\stackrel{\sim}{\To}
\Pol(\frakn_+) \otimes V^\vee,
\quad u\otimes v^\vee \longmapsto 
\widehat{d\pi_{(\sigma,\lambda)^*}}(u)(1\otimes v^\vee),
\end{equation}
where $\Pol(\frakn_+)$ is the space of polynomial functions on $\frakn_+$.
The isomorphism $F_c$ is called
the \emph{algebraic Fourier transform of the generalized Verma module 
$\Mp(V^\vee)$}.

%%%%%%%%%%%%%%%%%%%%%%%%%%%%%%%%%%%%%%%%%%
\subsection{The F-method}
\label{sec:Fmethod2}
Now, 
for the outer tensor product representations 
$\sigma \boxtimes \C_\lambda$ and $\delta \boxtimes \C_\nu$
on $V$ and $W$, respectively, we put
\begin{equation}\label{eqn:Sol2a}
\mathrm{Sol}(\sigma,\delta;\lambda,\nu):=
\operatorname{Hom}_{M'A'}(W^\vee,
(\mathrm{Pol}(\mathfrak{n}_+) 
\otimes V^\vee)^{\widehat{d\pi_{(\sigma,\lambda)^*}}(\frakn_+')}),
\end{equation}
where $M'A'$ acts on $\Pol(\frakn_+)$ by
\begin{equation}\label{eqn:sharp}
\Ad_{\#}(l) \colon p(\cdot) \mapsto p(\Ad(l^{-1})\cdot)
\;\; \text{for $l \in M'A'$}.
\end{equation}
Via the identification
$\textrm{Hom}_{M'A'}(W^\vee, \Pol(\frakn_+)\otimes V^\vee)
\simeq
\big((\Pol(\frakn_+)\otimes V^\vee) \otimes W\big)^{M'A'}$,
we have
\begin{align}\label{eqn:Sol}
&\mathrm{Sol}(\sigma,\delta;\lambda,\nu)\nonumber \\[3pt]
&=\{ \psi \in \Hom_{M'A'}(W^\vee, \Pol(\frakn_+)\otimes V^\vee): 
\text{
$\psi$ solves the system \eqref{eqn:Fsys} of PDEs below.}\}
\end{align}
\begin{equation}\label{eqn:Fsys}
(\widehat{d\pi_{(\sigma,\lambda)^*}}(C) \otimes \id_W)\psi=0
\,\,
\text{for all $C \in \frakn_+'$}.
\end{equation}
The system \eqref{eqn:Fsys} of PDEs is referred to
as the \emph{F-system}.

The F-method asserts that there exists a linear isomorphism
from the solution space $\mathrm{Sol}(\sigma,\delta;\lambda,\nu)$ to
$\Diff_{G'}(I(\sigma,\lambda),J(\delta,\nu))$.
To describe the linear isomorphism explicitly, we first observe 
the inverse of the algebraic Fourier transform $F_c^{-1}$.

For $X_1^{r_1}\cdots X_n^{r_n} \in U(\frakn_-)$ with 
$|\mathbf{r}|:= r_1+\cdots+r_n$,
we write $\sym_{|\mathbf{r}|}(X_1^{r_1}\cdots X_n^{r_n})$ 
for the symmetrization of $X_1^{r_1}\cdots X_n^{r_n}$, namely,
\begin{equation}\label{eqn:symm}
\sym_{|\mathbf{r}|}(X_1^{r_1}\cdots X_n^{r_n}):=
\frac{1}{|\mathbf{r}|!}\sum_{\tau \in \mathfrak{S}_{|\mathbf{r}|}}
X_{j_{\tau(1)}} \cdots X_{j_{\tau(|r|)}},
\end{equation}
for $X_1^{r_1}\cdots X_n^{r_n}=X_{j_1}\cdots X_{j_{|\mathbf{r}|}}$.
Here, $\mathfrak{S}_{|\mathbf{r}|}$ denotes the symmetric group on 
$|\mathbf{r}|$ letters. 

\begin{proposition}[{\cite[Cor.\ 3.30]{KPV25}}]
\label{prop:Fc}
The inverse 
\begin{equation*}
F_c^{-1}\colon \Pol(\frakn_+)\otimes V^\vee
\stackrel{\sim}{\To} 
 \Mp(V^\vee)
\end{equation*}
is given by
\begin{equation*}
F_c^{-1}(\zeta_1^{r_1}\cdots \zeta_n^{r_n}\otimes v^\vee) =
\sym_{|\mathbf{r}|}(X_1^{r_1}\cdots X_n^{r_n})\otimes v^\vee.
\end{equation*}
\end{proposition}

Now, let $\Diff_{N_-}(N_-)$ denote the space of left-invariant differential operators
on $N_-$. It is known that there exists a linear isomorphism
\begin{equation}\label{eqn:U_Diff_N}
dR\colon U(\frakn_-) \stackrel{\sim}{\To} \Diff_{N_-}(N_-),
\quad
u \mapsto dR(u).
\end{equation}
Then, by using $F_c^{-1}$,  we define a linear isomorphism
\begin{equation*}
\Symb_0^{-1} \colon\Pol(\frakn_+)\otimes \Hom_\CC(V,W) \stackrel{\sim}{\To} 
\Diff_{N_-}(N_-) \otimes \Hom_\CC(V,W)
\end{equation*}
as
\begin{equation}\label{eqn:invSymb}
\Symb_0^{-1} := dR \circ (F_c^{-1}\otimes \id_{W}).
\end{equation}

\begin{remark}
The linear isomorphism $\Symb_0^{-1}$ is indeed the inverse of a certain linear map $\Symb_0$ called the \emph{truncated symbol map}. For the details, see \cite{KPV25}.
\end{remark}

The following theorem is our main tool to classify and construct 
DSBOs $\DD$. It is first showed in \cite{KoPe16a} for 
$\frakn_+$ abelian. The general case is then proved in \cite{KPV25}.

\begin{theorem}[{\cite[Thm.\ 4.1]{KoPe16a} and \cite[Thm.\ 3.32]{KPV25}}]
\label{thm:symb}
The map $\Rest \circ \Symb_0^{-1}$ provides a linear isomorphism
\begin{equation}\label{eqn:RestSymb0}
\Rest \circ \Symb_0^{-1} \colon 
\mathrm{Sol}(\sigma,\delta;\lambda,\nu) \stackrel{\sim}{\To} 
\Diff_{G'}(I(\sigma,\lambda),J(\delta,\nu)),
\end{equation}
where $\Rest$ denotes the restriction map 
from $C^\infty(G/P',\mathcal{W})$ to $C^\infty(G'/P',\mathcal{W})$.
\end{theorem}

The inverse of the truncated symbol map $\Symb_0^{-1}$ can be given as follows.
First, take $\psi(\zeta) \in  \Pol(\frakn_+)\otimes \Hom_\C(V,W)$ such that
\begin{equation*}
\psi(\zeta)=\sum_{j}\psi_{j}(\zeta) \otimes T_j,
\end{equation*}
with $\psi_{j}(\zeta) \in \Pol(\frakn_+)$ and $T_j \in \Hom_\C(V,W)$.
Write
\begin{equation*}
\psi_{j}(\zeta)=\sum_{\mathbf{r} \in \NN^n}a_\mathbf{r} \, 
\zeta_1^{r_1}\cdots \zeta_n^{r_n}
\end{equation*}
for $\mathbf{r} = (r_1, \ldots, r_n)$. 
Then we have
\begin{equation}\label{eqn:Symb0-inverse}
\Symb_0^{-1}(\psi)=\sum_{j}\psi_{j}(\sym_{|\mathbf{r}|}(dR(X))) \otimes T_j
\in \Diff_{N_-}(N_-)\otimes \Hom_\C(V,W),
\end{equation}
where
\begin{equation*}
\psi_{j}(\sym_{|\mathbf{r}|}(dR(X)))
=\sum_{\mathbf{r} \in \NN^n}a_\mathbf{r}\, 
\sym_{|\mathbf{r}|}(dR(X_1)^{r_1}\cdots dR(X_n)^{r_n}).
\end{equation*}

It follows from 
Theorem \ref{thm:symb} and \eqref{eqn:Symb0-inverse} that
any DSBO $\mathbb{D}$ obtained 
through $\Rest \circ \Symb_0^{-1}$
in \eqref{eqn:RestSymb0} contains $\sym_{|\mathrm{r}|}$
in its expression. On the other hand, any DSBO $\mathbb{D}$ can also be expressed 
in terms of the PBW-basis elements of $U(\frakn_-)$ 
induced from a given ordered basis of $\frakn_-$.

In order to define the two types of expressions of a DSBO $\DD$
more explicitly,
recall from \eqref{eqn:ord-n}
that $\ord=\{X_1, X_2, \ldots, X_n\}$ is the ordered basis 
of $\frakn_-(\R)$ used to define 
the Lie algebra homomorphism \eqref{eqn:dpi-complex-g}. 
We regard $\ord$ also as an ordered basis of 
$\frakn_-$.
The corresponding PBW-basis $\mathrm{PBW}_\ord$ of $U(\frakn_-)$ is given as
\begin{equation}\label{eqn:PBW}
\mathrm{PBW}_\ord =\{X_1^{r_1}X_2^{r_2}\cdots X_n^{r_n}: r_j\geq 0\}.
\end{equation}
For $\tau \in \mathfrak{S}_n$, 
we define another ordered basis
\begin{equation}\label{eqn:ord-tau}
\ord(\tau):=\{X_{\tau(1)}, X_{\tau(2)},\ldots, X_{\tau(n)}\}.
\end{equation}
The PBW-basis $\mathrm{PBW}_{\ord(\tau)}$ induced from
$\ord(\tau)$ is
\begin{equation*}
\mathrm{PBW}_{\ord(\tau)} =\{X_{\tau(1)}^{r_1}X_{\tau(2)}^{r_2}\cdots 
X_{\tau(n)}^{r_n}: r_j\geq 0\}.
\end{equation*}
Then we define the two types of expressions of $\DD$ as follows.

\begin{definition}\label{def:symmetrized-ordered-forms}
We say that a DSBO $\mathbb{D}$ is in 
\emph{symmetrized form} if $\mathbb{D}$ is expressed as  a linear combination of 
$\sym_{|\mathbf{r}|}(dR(X_1)^{r_1}\cdots dR(X_n)^{r_n})$.
If $\mathbb{D}$ is expressed as a linear combination of
$dR(X_{\tau(1)})^{r_1}\cdots dR(X_{\tau(n)})^{r_n}$
for the PBW-basis elements $X_{\tau(1)}^{r_1}\cdots X_{\tau(n)}^{r_n}
\in \mathrm{PBW}_{\ord(\tau)}$ for some $\tau \in \mathfrak{S}_n$,
then we say that $\mathbb{D}$ is in \emph{ordered form}.
\end{definition}

In the next section we discuss the relationship between these two expressions
in detail.
%%%%%%%%%%%%%%%%%%%%%%%%%%%%%%%%%%%%%%%%%
\section{Symmetrization operator and ordered F-system}
\label{sec:nonsym}

The aim of this section is to formulate a general framework that 
relates the symmetrized form of a DSBO $\DD$
to the ordered form with respect to a given ordered basis. 
To this end, we define a new  operator on the space $\Pol(\frakg^\vee)$ 
of polynomial functions on the dual $\frakg^\vee$ of $\frakg$, which we call the symmetrization operator.
We then use this operator to define a new system of PDEs, called
the ordered F-system.
In Theorem \ref{thm:ImFc2}, we show a relationship between
the symmetrized form and the ordered form via the ordered F-system.
We shall apply this theorem to construct DSBOs $\DD$ in later sections.

Throughout this section, let $\F:=\R$ or $\C$, and we assume that 
$\frakg$ is a finite-dimensional Lie algebra over $\F$ 
with $\dim_\F \frakg =n$, unless otherwise specified. 

%%%%%%%%%%%%%%%%%%%%%%%%%%%%%%%%%%%%%%%%%
\subsection{Symmetrization operator}

Let  
$\ord:=\{X_1,X_2,\ldots, X_n\}$ 
be an ordered basis of $\frakg$.
As in \eqref{eqn:PBW},
the PBW-basis $\mathrm{PBW}_\ord$ of $U(\frakg)$ induced from
$\ord$ is 
\begin{equation*}
\mathrm{PBW}_\ord =\{X_1^{r_1}X_2^{r_2}\cdots X_n^{r_n}: r_j\geq 0\}.
\end{equation*}

In order to define the symmetrization operator, one first needs to define
the \emph{symmetrization} $\sym_{\ord}$ on $U(\frakg)$ and the
\emph{ordering map} $\Upsilon_\ord$ from $\Pol(\frakg^\vee)$ to $U(\frakg)$.
By using the fixed PBW-basis $\mathrm{PBW}_\ord$,
the symmetrization $\sym_{\ord}$ on $U(\frakg)$ is defined as follows.

\begin{definition}
The  \emph{symmetrization $\sym_{\ord}$ on $U(\frakg)$ with respect to 
$\ord$} is a linear isomorphism on $U(\frakg)$ defined by
\begin{equation}\label{eqn:sym-ord}
\sym_{\ord}\colon U(\frakg) \stackrel{\sim}{\To} U(\frakg),
\quad
u\longmapsto
\sum_{\mathbf{r} \in \NN^n} 
a_\mathbf{r}\, \sym_{|\mathbf{r}|}(X^{\mathbf{r}}),
\end{equation}
where  $u=\sum_{\mathbf{r} \in \NN^n} 
a_\mathbf{r}\, X^{\mathbf{r}} \in U(\frakg)$
with $\mathbf{r} :=(r_1,\ldots, r_n)$ and 
$X^{\mathbf{r}} := X_1^{r_1}\cdots X_n^{r_n}\in \mathrm{PBW}_\ord$.
\end{definition}

\begin{remark}
Since any element $u \in U(\frakg)$ can be uniquely expressed 
as a linear combination of  
$\sym_{|\mathbf{r}|}(X_1^{r_1}\cdots X_n^{r_n})$ 
for some $(r_1,\dots, r_n) \in \NN^n$ (cf.\ \cite[p.\ 225]{Knapp02}), 
the map \eqref{eqn:sym-ord} is indeed  a linear isomorphism.
\end{remark}

As opposed to the symmetrization map
from the symmetric algebra $S(\frakg)$ to $U(\frakg)$ (cf.\ \cite[p.\ 225]{Knapp02}), 
in general, we have $\sym_{\ord}(X_1X_2) \neq \sym_{\ord}(X_2X_1)$, 
as the following example shows.

\begin{example}\label{example:symmetrization}
Let $\mathfrak{g}=\Span_\F\{X_1, X_2, X_3\}$ 
with $[X_1, X_2]=-X_3$ and $[X_3, X_j]=0$ for $j=1,2$;
that is,
$\frakg$ is the three-dimensional Heisenberg Lie algebra.
Fix the ordered basis $\ord:=\{X_1, X_2, X_3\}$.
Then, $\sym_\ord(X_2X_1)$ is given  as
\begin{align*}
\sym_\ord(X_2X_1)
&=\sym_\ord(X_1X_2+X_3)\\
&=\sym_2(X_1X_2)+\sym_1(X_3) \\
&= \frac{1}{2}(X_1X_2+X_2X_1) + X_3,
\end{align*}
whereas we have 
\begin{align*}
\sym_\ord(X_1X_2)
&=\sym_2(X_1X_2)\\
&= \frac{1}{2}(X_1X_2+X_2X_1).
\end{align*}
\end{example}

Now, observe that, for each  $X_{j_1}X_{j_2}\cdots X_{j_m} \in U(\frakg)$,
there exists $\tau \in \mathfrak{S}_m$ such that 
$j_{\tau(k)} \leq j_{\tau(k+1)}$ for all $k=1,\ldots, m-1$.
In particular, we have 
\begin{equation*}
X_{j_{\tau(1)}}X_{j_{\tau(2)}} \cdots X_{j_{\tau(m)}} \in \mathrm{PBW}_\ord.
\end{equation*}
If $\tau' \in \mathfrak{S}_{m}$ also satisfies
$j_{\tau'(k)} \leq j_{\tau'(k+1)}$ for all $k=1,\ldots, m-1$,
then 
\begin{equation*}
X_{j_{\tau'(1)}}X_{j_{\tau'(2)}} \cdots X_{j_{\tau'(m)}}=
X_{j_{\tau(1)}}X_{j_{\tau(2)}} \cdots X_{j_{\tau(m)}}.
\end{equation*}
Thus, for such $\tau \in \mathfrak{S}_m$, we write
\begin{equation}\label{eqn:X-ord}
(X_{j_1}X_{j_2}\cdots X_{j_m})_\ord
:=X_{j_{\tau(1)}}X_{j_{\tau(2)}} \cdots X_{j_{\tau(m)}} \in \mathrm{PBW}_\ord.
\end{equation}

\begin{example}\label{example:arrangement1}
Let $\ord:=\{X_1,X_2, X_3,X_4\}$ be an ordered basis
of a Lie algebra $\frakg$. Then
we have 
\begin{equation*}
(X_4X_2X_1X_3^2)_\ord=X_1X_2X_3^2X_4.    
\end{equation*}
\end{example}

Let $(\zeta_1,\zeta_2,\ldots, \zeta_n)$ denote 
the coordinates on $\frakg^\vee$ with respect to 
the dual basis of $\ord$. 
Now the ordering map $\Upsilon_{\ord}$ from $\Pol(\frakg^\vee)$ to $U(\frakg)$
is defined as follows.

\begin{definition}
The \emph{ordering map $\Upsilon_{\ord}$ from 
$\Pol(\frakg^\vee)$ to $U(\frakg)$ 
with respect to $\ord$} is a linear isomorphism defined by
\begin{equation}\label{eqn:Upsilon_ord}
\Upsilon_{\ord} \colon \Pol(\frakg^\vee) \stackrel{\sim}{\To} U(\frakg), \quad 
\zeta_{j_1}\zeta_{j_2}\cdots \zeta_{j_m} 
\longmapsto
(X_{j_1}X_{j_2}\cdots X_{j_m})_\ord.
\end{equation}
\end{definition}

By definition, we have
\begin{equation}\label{eqn:Uu}
\Upsilon_{\ord} 
(\zeta_1^{r_1} \zeta_2^{r_2}\cdots\zeta_n^{r_n}) 
= X_1^{r_1}X_2^{r_2}\cdots X_n^{r_n}.
\end{equation}

\begin{example}
Let $\ord$ be the ordered basis defined in 
Example \ref{example:arrangement1}.
Then $\Upsilon_{\ord}(\zeta_4\zeta_2\zeta_1\zeta_3^2)$
is given as
\begin{equation*}
\Upsilon_{\ord}(\zeta_4\zeta_2\zeta_1\zeta_3^2) 
=(X_4X_2X_1X_3^2)_\ord
=  X_1 X_2 X_3^2X_4.
\end{equation*}
\end{example}

By using the symmetrization $\sym_\ord \colon U(\frakg) \stackrel{\sim}{\to} U(\frakg)$and ordering map $\Upsilon_\ord \colon \Pol(\frakg^\vee) \stackrel{\sim}{\to} U(\frakg)$,
we define the symmetrization operator 
$S_{\ord}$ on $\Pol(\frakg^\vee)$ as follows.

\begin{definition}
The \emph{symmetrization operator
$S_{\ord}$ on $\Pol(\frakg^\vee)$
with respect to $\ord$} is a linear isomorphism 
$S_{\ord}\colon 
\Pol(\frakg^\vee)\stackrel{\sim}{\to}\Pol(\frakg^\vee)$ 
satisfying the following identity:
\begin{equation*}
\Upsilon_{\ord} \circ S_{\ord}=
\sym_{\ord} \circ \Upsilon_{\ord}.
\end{equation*}
Equivalently, the following diagram commutes.
\begin{equation}\label{eqn:orderization}
\begin{tikzcd}
	{\Pol(\frakg^\vee)}& {U(\frakg)}
	\arrow[dl, pos=0.5, phantom, "\circlearrowleft"] \\
	{\Pol(\frakg^\vee)} & {U(\frakg)}
	\arrow["{\Upsilon_{\ord}}", "\sim"' sloped, from=1-1, to=1-2]
	\arrow["{S_{\ord}}","\sim"' sloped,  from=2-1, to=1-1]
	\arrow["\sim" sloped, "\sym_{\ord}"',  from=2-2, to=1-2]
	\arrow["\sim" sloped,"{\Upsilon_{\ord}}"', from=2-1, to=2-2]
\end{tikzcd}
\end{equation}
\end{definition}

Here are some examples of $S_{\ord}$.

\begin{example}
If $\frakg$ is abelian, then
the symmetrization $\sym$ on $U(\frakg)$
is the identity map $\sym = \id_{U(\frakg)}$.
Therefore, the symmetrization operator $S_{\ord}$ is given as 
$S_{\ord}=\id_{\Pol(\frakg^\vee)}$.
\end{example}

\begin{example}
Let $\mathfrak{g}=\Span_\F\{X_1, X_2, X_3\}$ 
be the three-dimensional Heisenberg Lie algebra considered 
in Example \ref{example:symmetrization}.
Since the symmetrization $\sym_\ord(X_1 X_2)$
with respect to the ordered basis $\ord=\{X_1, X_2, X_3\}$
is given as
\begin{align*}
\sym_\ord(X_1 X_2)
&=\sym_2(X_1 X_2) \\
&=\frac{1}{2}(X_1 X_2 + X_2 X_1) \\
&= X_1X_2 +\frac{1}{2}X_3,
\end{align*}
we have
\begin{equation*}
S_{\ord}(\zeta_1 \zeta_2) 
= \zeta_1\zeta_2 + \frac{1}{2}\zeta_3.
\end{equation*}

We shall give a general formula for $S_{\ord}$
in this three-dimensional Heisenberg case in Proposition \ref{prop:eJU} (also, see
Proposition \ref{Prop:SymmetrizationMapIdentity}).
Moreover, the symmetrization operator $S_{\ord}$ 
for the Heisenberg Lie algebras of arbitrary dimension is discussed in
Appendix \ref{appendix:Heis}.
\end{example}

\begin{example}
Let $\mathfrak{g}=\Span_\F\{X_1, X_2, X_3\}$ with $[X_1,X_2] = X_3,\, [X_1,X_3]=-2X_1$ and $[X_2,X_3]=2X_2$; that is, $\mathfrak{g} = \fraksl(2,\F)$.
With respect to the ordered basis $\ord:=\{X_1, X_2, X_3\}$, one can show the following:
\begin{alignat*}{3}
S_{\ord}(\zeta_1 \zeta_2) 
&= \zeta_1\zeta_2 - \frac{1}{2}\zeta_3,
\quad
&&S_{\ord}(\zeta_1\zeta_2\zeta_3) 
&&= \zeta_1\zeta_2\zeta_3 - \frac{1}{2}\zeta_3,\\
S_{\ord}(\zeta_1 \zeta_3) 
&= \zeta_1\zeta_3 + \zeta_1,
\quad
&&S_{\ord}(\zeta_1\zeta_3^2) 
&&= \zeta_1\zeta_3^2 + 2\zeta_1\zeta_3 + \frac{4}{3}\zeta_1.
\\
S_{\ord}(\zeta_2 \zeta_3) 
&= \zeta_2\zeta_3 - \zeta_2,
\end{alignat*}
\end{example}

Now let $\Aut(\frakg)$ denote the group of automorphisms
on $\frakg$. For later convenience, we give the following proposition.

\begin{proposition}\label{prop:G_orderization}
The symmetrization operator $S_{\ord}$ 
is an $\Aut(\frakg)$-isomorphism on $\Pol(\frakg^\vee)$.
\end{proposition}

\begin{proof}
It is readily seen that the symmetrization 
$\sym_{\ord} \colon U(\frakg) \stackrel{\sim}{\to} U(\frakg)$
and the ordering map $\Upsilon_{\ord} \colon \Pol(\frakg) \stackrel{\sim}{\to} U(\frakg)$ are both $\Aut(\frakg)$-isomorphisms. 
The commutative diagram \eqref{eqn:orderization} then concludes 
the desired assertion.
\end{proof}

Let $H$ be a subgroup of $\Aut(\frakg)$ and $V$ a finite-dimensional 
representation of $H$. For later convenience, 
we consider the dual $V^\vee$ rather than $V$ itself.
Via the diagonal action, we regard $\Pol(\frakg^\vee) \otimes V^\vee$
as an $H$-representation.

For an $H$-subrepresentation $U \subset \Pol(\frakg^\vee) \otimes V^\vee$,
we denote by $E(U)$ the $U$-isotypic component in 
$\Pol(\frakg^\vee) \otimes V^\vee$, namely,
\begin{equation*}
E(U):=\bigoplus_{R}  R,
\end{equation*}
where the direct sum runs over 
$H$-subrepresentations $R \subset \Pol(\frakg^\vee) \otimes V^\vee$
such that $R$ is equivalent to $U$ as an $H$-representation.

\begin{corollary}\label{cor:G_orderization}
The $U$-isotypic component $E(U) \subset \Pol(\frakg^\vee) \otimes V^\vee$
is  an $S_{\ord}\otimes \id_{V^\vee}$-invariant subspace of $\Pol(\frakg^\vee)\otimes V^\vee$.
\end{corollary}

\begin{proof}
It follows from Proposition \ref{prop:G_orderization} that 
$S_{\ord}\otimes \id_{V^\vee}$ is an $H$-isomorphism on 
$\Pol(\frakg^\vee)\otimes V^\vee$.
In particular, we have $\Im(S_{\ord}\otimes \id_{V^\vee})\vert_{U} \simeq U$ 
as $H$-representations. 
Thus, $\Im(S_{\ord}\otimes \id_{V^\vee})\vert_{U}  \subset E(U)$.
This concludes the corollary.
\end{proof}

By abuse of notation, for $T \in \{S_{\ord}, \Upsilon_{\ord}, \sym_\ord\}$,
we simply write
\begin{equation*}
T\equiv T \otimes \id_{V^\vee}.
\end{equation*}

%%%%%%%%%%%%%%%%%%%%%%%%%%%%%%%%%%%%%%%%%
\subsection{Conjugation \texorpdfstring{$\Ad(S)$}{Ad(S)}}
\label{sec:ConT}

Suppose that $E$ is an
$S_\ord$-invariant
subspace of $\Pol(\frakg^\vee)\otimes V^\vee$.
Then the following diagram commutes.
\begin{equation}\label{eqn:RS}
\begin{tikzcd}
	{E}& {U(\frakg)\otimes V^\vee}
	\arrow[dl, pos=0.5, phantom, "\circlearrowleft"] \\
	{E} & {U(\frakg)\otimes V^\vee}
	\arrow["{\Upsilon_{\ord}}", hook, from=1-1, to=1-2]
	\arrow["{S_\ord}", "\sim"' sloped,  from=2-1, to=1-1]
	\arrow["\sim" sloped, "\sym_{\ord}"',  from=2-2, to=1-2]
	\arrow["{\Upsilon_{\ord}}"', hook, from=2-1, to=2-2]
\end{tikzcd}
\end{equation}
In other words, we have 
\begin{equation*}
\sym_{\ord} \circ \Upsilon_{\ord}\vert_{E} = \Upsilon_{\ord} 
\circ S_\ord\vert_{E}.
\end{equation*}

For $A \in \End_\F(E)$ we put
\begin{align}
\Ad(S)A&:= S_\ord\circ A \circ S_\ord^{-1}, \label{eqn:adA}\\
\Sol_A(E)&:=\{\psi(\zeta) \in E : A \psi(\zeta)=0\}.\nonumber
\end{align}
As $\Ad(S)A \in \End_\F(E)$, one has
\begin{equation*}
 \Sol_{\Ad(S)A}(E)
=\{\varphi(\zeta) \in E : (\Ad(S)A)\varphi(\zeta)=0\}.
\end{equation*}

\begin{proposition}\label{prop:SolA}
The symmetrization operator $S_\ord$ 
gives rise to a linear isomorphism
\begin{equation*}
 \Sol_{\Ad(S)A}(E)
 \stackrel{\sim}{\To} 
 \Sol_A(E),
\quad
\varphi(\zeta) \longmapsto S_\ord^{-1}\varphi(\zeta).
\end{equation*}
\end{proposition}

\begin{proof}
The injectivity and surjectivity of the proposed map
$\varphi(\zeta) \mapsto S_\ord^{-1}\varphi(\zeta)$ simply follow from  
the fact that $S_\ord$ is an isomorphism on $E$. 
Thus, to prove the proposition, 
it suffices to show that $S_\ord^{-1}\varphi(\zeta) \in \Sol_A(E)$.
Indeed, 
for $\varphi(\zeta) \in  \Sol_{\Ad(S)A}(E)$, we have
\begin{equation*}
(\Ad(S)A) \varphi(\zeta)=0,
\end{equation*}
which is, by \eqref{eqn:adA}, 
\begin{equation}\label{eqn:SAS}
(S_\ord \circ A \circ S_\ord^{-1}) \varphi(\zeta)=0.
\end{equation}
Now apply $S_\ord^{-1}$ to both sides of \eqref{eqn:SAS} 
to obtain $AS_\ord^{-1}\varphi(\zeta) = 0$, 
which is $S_\ord^{-1}\varphi(\zeta) \in \Sol_A(E)$.
\end{proof}

\begin{corollary}\label{cor:SolAU}
The following diagram commutes.
\begin{equation}\label{eqn:RS2}
\begin{tikzcd}
	{\Sol_{\Ad(S)A}(E)}& {U(\frakg)\otimes V^\vee}
	\arrow[dl, pos=0.5, phantom, "\circlearrowleft"] \\
	{\Sol_A(E)} & {U(\frakg)\otimes V^\vee}
	\arrow["{\Upsilon_{\ord}}", hook, from=1-1, to=1-2]
	\arrow["\sim" sloped, "{S_{\ord}^{-1}}"',   from=1-1, to=2-1]
	\arrow["\sim" sloped, "\sym_{\ord}"',  from=2-2, to=1-2]
	\arrow["{\Upsilon_{\ord}}"', hook, from=2-1, to=2-2]
\end{tikzcd}
\end{equation}
\end{corollary}

\begin{proof}
For $\varphi(\zeta) \in  \Sol_{\Ad(S)A}(E)$,
by Proposition \ref{prop:SolA}, we have $S_\ord^{-1}\varphi(\zeta) \in \Sol_{A}(E)$.
Then, as $\sym_{\ord} \circ \Upsilon_{\ord} 
= \Upsilon_{\ord} \circ S_\ord$, we have 
\begin{equation*}
(\sym_{\ord} \circ \Upsilon_{\ord} \circ S_\ord^{-1})\varphi(\zeta)
=(\Upsilon_{\ord} \circ S_\ord \circ S_\ord^{-1})\varphi(\zeta)
=\Upsilon_{\ord}(\varphi(\zeta)).\qedhere
\end{equation*}
\end{proof}

%%%%%%%%%%%%%%%%%%%%%%%%%%%%%%%%%%%%%%%%%
\subsection{Further identification of \texorpdfstring{$\Sol_A(E)$}{Sol_A(E)}}
\label{sec:ImFc}

Suppose that there exist a linear isomorphism 
$T\colon Q \stackrel{\sim}{\to} E$ from 
some vector space $Q$ to 
an $S_\ord$-invariant
subspace $E$ of $\Pol(\frakg^\vee)\otimes V^\vee$. 
For $A\in \End_\F(E)$, we define 
$T^\sharp A\in \End_\F(Q)$ by
\begin{equation*}
T\circ T^\sharp A=A\circ T, 
\end{equation*}
so that the following diagram commutes.
\[\begin{tikzcd}
	{Q} & {E} 	\arrow[dl, pos=0.5, phantom, "\circlearrowleft"] \\
	{Q} & {E}
	\arrow["{T}", "\sim"', from=1-1, to=1-2]
	\arrow["{T^\sharp A}", from=2-1, to=1-1]
	\arrow["A"', from=2-2, to=1-2]
	\arrow["\sim", "{T}"', from=2-1, to=2-2]
\end{tikzcd}\]

For $A \in \End_\F(E)$, we write 
\begin{equation*}
\Sol_{T^\sharp A}(Q)
:=\{q \in Q : (T^\sharp A)q=0\}.
\end{equation*}
Then the restriction $T\vert_{\Sol_{T^\sharp A}(Q)}$ of $T$ on 
$\Sol_{T^\sharp A}(Q)$ gives rise to a linear isomorphism
\begin{equation*}
T\vert_{\Sol_{T^\sharp A}(Q)}: \Sol_{T^\sharp A}(Q)
\stackrel{\sim}{\To} 
\Sol_A(E).
\end{equation*}
Since $S_\ord \in GL(E)$, we have
$T^\sharp S_\ord \colon Q \stackrel{\sim}{\to} Q$.
Thus, for $B \in \End_\F(Q)$, one can define
\begin{equation}\label{eqn:adB}
\Ad(T^\sharp S_\ord)B
:= 
T^\sharp S \circ B \circ (T^\sharp S)^{-1}.
\end{equation}
To simplify the notation, we write
\begin{equation*}
T^\sharp S  \equiv T^\sharp S_\ord. 
\end{equation*}
As
\begin{equation*}
\Ad(T^\sharp S)T^\sharp A = T^\sharp \Ad(S)A,
\end{equation*}
one can extend the commutative diagram \eqref{eqn:RS2} as follows.
\begin{equation}\label{eqn:SSSS}
\begin{tikzcd}
	{\Sol_{\Ad(T^\sharp S)T^\sharp A }(Q)} & 
	{\Sol_{\Ad(S)A}(E)} 	\arrow[dl, pos=0.5, phantom, "\circlearrowleft"] 	
	& {U(\frakg)\otimes V^\vee}
	\arrow[dl, pos=0.5, phantom, "\circlearrowleft"]\\
	{\Sol_{T^\sharp A}(Q)} & {\Sol_{A}(E)} & {U(\frakg)\otimes V^\vee}
	\arrow["{T}", "\sim"' , from=1-1, to=1-2]
	\arrow["{\Upsilon_{\ord}}", hook, from=1-2, to=1-3]
	\arrow["\sim" sloped, "(T^\sharp S)^{-1}"', from=1-1, to=2-1]
	\arrow["S_\ord^{-1}", "\sim"' sloped,  from=1-2, to=2-2]
	\arrow["\sim" sloped, "\sym_{\ord}"',  from=2-3, to=1-3]
	\arrow["{\Upsilon_{\ord}}"', hook, from=2-2, to=2-3]
	\arrow["\sim" , "{T}"',  from=2-1, to=2-2]
\end{tikzcd}
\end{equation}

%%%%%%%%%%%%%%%%%%%%%%%%%%%%%%%%%%%%%%%%%
\subsection{Case of the nilpotent radicals \texorpdfstring{$\frakn_\pm$}{n_±}}

We now apply the preceding arguments to
$\frakg=\frakn_-$ and $\frakg^\vee=\frakn_+$,
where $\frakn_\pm$ are nilpotent radicals of a parabolic subalgebra.
Throughout the rest of this section, we resume the notation from Section \ref{sec:Fmethod}. In particular, the base field $\F$ is $\F=\C$.

Recall from Section \ref{sec:Fmethod} that 
$P=MAN_+$ (resp.\ $P'=M'A'N_+'$) is a parabolic subgroup of $G$ (resp.\ $G'$),
and $\ord(\tau)=\{X_{\tau(1)},\ldots, X_{\tau(n)}\}$ is the ordered basis
of $\frakn_-$
defined in \eqref{eqn:ord-tau} for some $\tau \in \mathfrak{S}_n$.
For a finite-dimensional $P$-representation $V^\vee$, 
by abuse of notation, we write
\begin{equation*}
S_\ord\equiv S_{\ord(\tau)}\otimes \id_{V^\vee}, \quad
\Upsilon_\ord\equiv \Upsilon_{\ord(\tau)}\otimes \id_{V^\vee}, \quad
\sym_\ord \equiv \sym_{\ord(\tau)}\otimes \id_{V^\vee}.
\end{equation*}
Then, via the identification $\frakn_-^\vee \simeq \frakn_+$,
we consider the following situation.
\begin{equation*}
\begin{tikzcd}
	{\Pol(\frakn_+)\otimes V^\vee}& {U(\frakn_-)\otimes V^\vee}
	\arrow[dl, pos=0.5, phantom, "\circlearrowleft"] \\
	{\Pol(\frakn_+)\otimes V^\vee} & {U(\frakn_-)\otimes V^\vee}
	\arrow["{\Upsilon_{\ord}}", "\sim"' sloped, from=1-1, to=1-2]
	\arrow["{S_{\ord}}","\sim"' sloped,  from=2-1, to=1-1]
	\arrow["\sim" sloped, "\sym_{\ord}"',  from=2-2, to=1-2]
	\arrow["\sim" sloped,"{\Upsilon_{\ord}}"', from=2-1, to=2-2]
\end{tikzcd}
\end{equation*}

Observe that the subgroup $MA$ acts on $\frakn_{\pm}$ 
via the adjoint action $\Ad\vert_{\frakn_{\pm}}$.
 Since $M'A'$ is assumed to be $M'A' \subset MA$ 
in \eqref{eqn:MAN}, the nilpotent radicals $\frakn_{\pm}$ are also 
$M'A'$-representations. Moreover, one can regard the $P$-representation
$V^\vee$ as an $M'A'$-representation.
Therefore, 
the symmetrization operator
\begin{equation*}
S_{\ord}
\colon \Pol(\frakn_+)\otimes V^\vee \stackrel{\sim}{\To} \Pol(\frakn_+)\otimes V^\vee
\end{equation*}
is an $M'A'$-isomorphism.

Recall from Proposition \ref{prop:Fc} that
the inverse $F_c^{-1}$ of the algebraic Fourier transform is given by
\begin{equation}\label{eqn:Fc2}
F_c^{-1}\colon \Pol(\frakn_+)\otimes V^\vee 
\stackrel{\sim}{\To} U(\frakn_-)\otimes V^\vee,
\quad 
\zeta_1^{r_1}\cdots \zeta_n^{r_n}\otimes v^\vee
\mapsto
\sym_{|\mathbf{r}|}(X_1^{r_1}\cdots X_n^{r_n}) \otimes v^\vee.
\end{equation}
Let $E$ be an $S_\ord$-invariant subspace 
of $\Pol(\frakn_+)\otimes V$.

\begin{proposition}\label{prop:FcU}
The inverse $F_c^{-1}\vert_{E}$ on $E$ is factored as 
\begin{equation*}
F_c^{-1}\vert_{E} = \Upsilon_{\ord}  \circ S_\ord\vert_{E}.
\end{equation*}
Equivalently, the following diagram commutes.
\begin{equation}\label{eqn:diagram2}
\begin{tikzcd}[row sep=1cm, column sep=1cm]
E
\arrow[r, "\Upsilon_{\ord}", hook]
&
U(\frakn_-)\otimes V^\vee
\\
E
\arrow[u, "S_\ord","\sim"' sloped]
\arrow[ur, "F_c^{-1}"', hook]
&
\arrow[ul, pos=0.8, phantom, "\circlearrowleft"]
\end{tikzcd}
\end{equation}
\end{proposition}

\begin{proof}
By \eqref{eqn:Uu} and \eqref{eqn:Fc2},
we have $F_c^{-1} = \sym_{\ord} \circ \Upsilon_{\ord}$.
The commutative diagram \eqref{eqn:RS} then concludes the proposition.
\end{proof}

For $A \in \End_\C(E)$, we  write 
\begin{equation*}
\Im F_c^{-1}(A):=\{F_c^{-1}(\psi(\zeta)) \in U(\frakn_-)\otimes V^\vee: 
\psi(\zeta) \in \Sol_A(E)\},
\end{equation*}
namely, the image of $F_c^{-1}$ 
on $\Sol_A(E) \subset \Pol(\frakn_+)\otimes V^\vee$.

\begin{theorem}\label{thm:ImFc2}
Suppose that $E$ is an 
$S_{\ord}$-invariant subspace of $\Pol(\frakn_+)\otimes V^\vee$.
Then, for  $A \in \End_\C(E)$, we have
\begin{align*}
\Im F_c^{-1}(A)
&=\{(\Upsilon_{\ord} \circ T)q
: q \in \Sol_{\Ad(T^\sharp S)T^\sharp A }(Q)\}\\[3pt]
&=\{(F_c^{-1}\circ T \circ (T^\sharp S)^{-1})q : 
q \in \Sol_{\Ad(T^\sharp S)T^\sharp A }(Q)\}.
\end{align*}
Equivalently, the following diagram commutes.
\begin{equation}\label{eqn:SolSolU}
\begin{tikzcd}
	{\Sol_{\Ad(T^\sharp S)T^\sharp A }(Q)} & 
	{\Sol_{\Ad(S)A}(E)} 	\arrow[dl, pos=0.5, phantom, "\circlearrowleft"]
	&{U(\frakn_-)\otimes V^\vee} \\
	{\Sol_{T^\sharp A}(Q)} & {\Sol_{A}(E)} &
	\arrow[ul, pos=0.85, phantom, "\circlearrowleft"]
	\arrow[r, "\Upsilon_{\ord}", hook, from=1-2, to=1-3]
	\arrow["F_c^{-1}"', hook, from=2-2, to=1-3]
	\arrow["{T}", "\sim"' , from=1-1, to=1-2]
	\arrow["\sim" sloped, "(T^\sharp S)^{-1}"', from=1-1, to=2-1]
	\arrow["S^{-1}_{\ord}", "\sim"' sloped,  from=1-2, to=2-2]
	\arrow["\sim" , "{T}"',  from=2-1, to=2-2]
\end{tikzcd}
\end{equation}
\end{theorem}

\begin{proof}
This is a direct consequence of 
\eqref{eqn:SSSS} and  Proposition \ref{prop:FcU}.
\end{proof}

\begin{remark}\label{rem:ImFc2}
It is clear that Theorem \ref{thm:ImFc2} holds 
also for  a system of linear operators in $\End_\C(E)$. 
Indeed, suppose that $A$ is the system of linear operators
$A_1, A_2, \ldots, A_m \in \End_\C(E)$. Then put 
\begin{equation*}
\Sol_{A}(E):=\bigcap_{j=1}^m\Sol_{A_j}(E).
\end{equation*}
Theorem \ref{thm:ImFc2} holds for the space $\Sol_{A}(E)$ 
of simultaneous solutions.
\end{remark}

\begin{remark}
For $\sum_{j} u_j\otimes v_j^\vee =(\Upsilon_{\ord} \circ T)q = 
 (F_c^{-1}\circ T \circ (T^\sharp S)^{-1})q
\in \Im F_c^{-1}(A)$,  write
$\calD = \sum_{j} dR(u_j)\otimes v_j^\vee
\in \Diff_{N_-}(N_-)\otimes V^\vee$. Then the formula 
\begin{equation*}
\calD =  dR \circ (F_c^{-1}\circ T \circ (T^\sharp S)^{-1})q 
\end{equation*}
contains symmetrization, whereas
\begin{equation*}
\calD=  dR \circ (\Upsilon_{\ord} \circ T)q 
\end{equation*}
does not.
\end{remark}

%%%%%%%%%%%%%%%%%%%%%%%%%%%%%%%%%%%%%%%%%
\subsection{Ordered F-system}
\label{sec:Ordered_Fsystem}

Now we consider the solution of the F-system in \eqref{eqn:Fsys}: 
\begin{equation}\label{eqn:Sol-C}
\mathrm{Sol}(\sigma,\delta;\lambda,\nu)
=\{ \psi \in \Hom_{M'A'}(W^\vee, \Pol(\frakn_+)\otimes V^\vee): 
\text{$\psi$ solves \eqref{eqn:Fsys-C}}\}
\end{equation}
\begin{equation}\label{eqn:Fsys-C}
(\widehat{d\pi_{(\sigma,\lambda)^*}}(C) \otimes \id_W)\psi=0
\,\,
\text{for all $C \in \frakn_+'$}.
\end{equation}
Let $E(W^\vee)$ denote 
the $W^\vee$-isotypic component of $\Pol(\frakn_+)\otimes V^\vee$.
Then $E(W^\vee)$ is $S_\ord$-invariant by Corollary \ref{cor:G_orderization}),
and we have 
\begin{equation*}
\mathrm{Sol}(\sigma,\delta;\lambda,\nu)
=\{ \psi \in \Hom_{M'A'}(W^\vee, E(W^\vee)): 
\text{$\psi$ solves  \eqref{eqn:Fsys}}\}.
\end{equation*}

Suppose that there exist $D_{F,1}, \ldots, D_{F,m} \in \End_\C(E(W^\vee))$
such that the system \eqref{eqn:Fsys-C} 
is equivalent to
\begin{equation}\label{eqn:Fsys-DF}
(D_{F,j} \otimes \id_W)\psi=0
\,\,
\text{for all $j=1,2,\ldots,m$},
\end{equation}
in the sense that $\psi$ solves \eqref{eqn:Fsys-C} if and only if 
it solves \eqref{eqn:Fsys-DF}.
Let $D_F$ denote the system \eqref{eqn:Fsys-DF} and put
\begin{equation*}
\mathrm{Sol}_{D_F}(\sigma,\delta;\lambda,\nu)
:=\{ \psi \in \Hom_{M'A'}(W^\vee, E(W^\vee)): 
\text{$\psi$ solves \eqref{eqn:Fsys-DF}}\}.
\end{equation*}
By definition, we have 
\begin{equation*}
\mathrm{Sol}_{D_F}(\sigma,\delta;\lambda,\nu)=
\mathrm{Sol}(\sigma,\delta;\lambda,\nu).
\end{equation*}

Then we give the following definition.

\begin{definition}\label{def:modified-Fsystem}
We refer to the system $D_F$ in \eqref{eqn:Fsys-DF} 
as a \emph{modified F-system}.    
\end{definition}

It follows from Corollary \ref{cor:G_orderization} that 
the $W^\vee$-isotypic component $E(W^\vee)$ is an 
$S_\ord$-invariant subspace of $\Pol(\frakn_+) \otimes V^\vee$.
Since $D_{F,j} \in \End_{\C}(E(W^\vee))$, 
the conjugation $\Ad(S)D_{F,j}$ by the symmetrization operator 
$S_{\ord}$ is a well-defined operator in $E(W^\vee)$.
Now we define the ordered F-system $\Ad(S)D_F$ of $D_F$
as follows.

\begin{definition}
The \emph{ordered F-system $\Ad(S)D_F$ of $D_F$}
is the system of differential equations
\begin{equation}\label{eqn:Ad(S)DF}
(\Ad(S)D_{F,j} \otimes \id_W)\psi=0
\,\,
\text{for all $j=1,2,\ldots,m$},
\end{equation}
where $\psi \in \Hom_{M'A'}(W^\vee, E(W^\vee))$.
\end{definition}

We put
\begin{equation*}
\mathrm{Sol}_{\Ad(S)D_F}(\sigma,\delta;\lambda,\nu) 
=\{ \psi \in \Hom_{M'A'}(W^\vee, E(W^\vee)): 
\text{$\psi$ solves \eqref{eqn:Ad(S)DF}}\}.
\end{equation*}

Recall from \eqref{eqn:U_Diff_N} that there exists a linear isomorphism
\begin{equation*}
dR\colon U(\frakn_-)\otimes V^\vee \stackrel{\sim}{\To} 
\Diff_{N_-}(N_-)\otimes V^\vee,
\quad
u \otimes v^\vee\mapsto dR(u) \otimes v^\vee.
\end{equation*} 
Then, by using the ordering map 
$\Upsilon_{\ord}\colon \Pol(\frakn_+)\otimes V^\vee 
\stackrel{\sim}{\to} U(\frakn_-)\otimes V^\vee$,
we define the \emph{ordered symbol map $\Symb_{\ord}$ 
with respect to $\ord$} by 
\begin{equation*}
\Symb_{\ord}\colon 
\Diff_{N_-}(N_-)
\otimes W\stackrel{\sim}{\To} 
\Pol(\frakn_+)\otimes W,
\quad
dR(u)\otimes v^\vee \otimes w 
\longmapsto \Upsilon_{\ord}^{-1}(u) \otimes v^\vee \otimes w.
\end{equation*}
Its inverse $\Symb_{\ord}^{-1}$ is given as 
\begin{equation*}
\Symb_{\ord}^{-1}=dR \circ 
(\Upsilon_{\ord}\otimes \id_{W}) .
\end{equation*}
On the other hand,
recall from \eqref{eqn:invSymb} that we have 
\begin{equation*}
\Symb_0^{-1}=dR\circ (F_c^{-1}\otimes \id_{W}).
\end{equation*}

The following theorem shows that the ordered F-system 
$\Ad(S)D_F$  indeed yields the ordered form of 
DSBOs arising from the F-system $D_F$.

\begin{theorem}\label{thm:symb_ord}
For a modified F-system $D_F$,
the following diagram commutes.
\begin{equation}\label{eqn:RestSymbOrd}
\begin{tikzcd}[row sep=1cm, column sep=1.5cm]
\mathrm{Sol}_{\Ad(S)D_F}(\sigma,\delta;\lambda,\nu)
\arrow[r, "\Symb_{\ord}^{-1}", "\sim"']
\arrow[d, "\sim" sloped,"S^{-1}_\ord"']
&
\Diff_{G'}(I(\sigma,\lambda),J(\delta,\nu))
\\
\mathrm{Sol}_{D_F}(\sigma,\delta;\lambda,\nu)
\arrow[ur,"\sim" sloped, "\Symb_0^{-1}"']
&
\arrow[ul, pos=0.8, phantom, "\circlearrowleft"]
\end{tikzcd}
\end{equation}
Here, for simplicity, the restriction map $\Rest$ is omitted 
for $\Symb_{\ord}^{-1}$ and $\Symb_0^{-1}$.
\end{theorem}

\begin{proof}
The assertion is a straightforward consequence of Theorem \ref{thm:ImFc2}
and Remark \ref{rem:ImFc2} for the case that $A=D_F$.
\end{proof}

%%%%%%%%%%%%%%%%%%%%%%%%%%%%%%%%%%%%%%%%%%
\subsection{Recipe for the F-method using the ordered F-system}
\label{sec:recipe}

By Theorem \ref{thm:symb_ord},
one may classify and construct DSBOs
$\DD \in \Diff_{G'}(I(\sigma,\lambda),J(\delta,\nu))$
for $(\sigma_\lambda, V)$ and $(\delta_\nu,W)$
by computing 
$\varphi \in \mathrm{Sol}_{\Ad(S)D_F}(\sigma,\delta;\lambda,\nu)$
and 
$\psi \in\mathrm{Sol}_{D_F}(\sigma,\delta;\lambda,\nu)$
as follows.
\vskip 0.1in

\begin{enumerate}
\item[Part 1] Preperation
\vskip 0.1in

\begin{enumerate}

\item[Step 1]
Determine the $E(W^\vee)$-isotypic component of
$\Pol(\frakn_+)\otimes V^\vee$.
\vskip 0.1in

\item[Step 2]
Fix an ordered basis $\ord$ and
compute the symmetrization operator $S_\ord$ on $\Pol(\frakn_+)$.
\vskip 0.1in

\item[Step 3]
Compute $d\pi_{(\sigma,\lambda)^*}(C)$ and 
$\widehat{d\pi_{(\sigma,\lambda)^*}}(C)$
for $C \in \frakn_+'$.
\vskip 0.1in

\item[Step 4]
Find a modified F-system $D_F$.
\end{enumerate}

\vskip 0.2in

\item[Part 2] Solving the F-systems

\vskip 0.1in

\begin{enumerate}

\item[Step 5]
Classify and construct
\begin{equation*}
\psi \in \Hom_{M'A'}(W^\vee, \Pol(\frakn_+)\otimes V^\vee)
=\Hom_{M'A'}(W^\vee, E(W^\vee)).    
\end{equation*}
\vskip 0.1in

\item[Step 6]
Apply the conjugation $\Ad(S)$ by $S_\ord$ to the modified F-system $D_F$ 
to obtain the ordered F-system 
$\Ad(S)D_F$.

\vskip 0.1in

\item[Step 7]
Compute $\mathrm{Sol}_{D_F}(\sigma,\delta;\lambda,\nu)$ 
and 
$\mathrm{Sol}_{\Ad(S)D_F}(\sigma,\delta;\lambda,\nu)$.
\vskip 0.1in

\item[Step 8]
Do the following.
\vspace{3pt}

\begin{enumerate}
\item 
Apply $\Rest \circ \Symb_0^{-1}$
to $\psi \in \mathrm{Sol}_{D_F}(\sigma,\delta;\lambda,\nu)$.
\vspace{5pt}

\item
Apply $\Rest \circ\Symb_{\ord}^{-1}$
to  $\varphi \in \mathrm{Sol}_{\Ad(S)D_F}(\sigma,\delta;\lambda,\nu)$.

\end{enumerate}

\end{enumerate}

\end{enumerate}

\vskip 0.1in

\begin{remark}
Since
\begin{equation*}
\Hom_{M'A'}(W^\vee, \Pol(\frakn_+)\otimes V^\vee)
\simeq 
\Hom_{M'A'}(V, \Pol(\frakn_+)\otimes W),
\end{equation*}
one may, alternatively, choose to find the $E(V)$-isotypic component of 
$\Pol(\frakn_+)\otimes W$ instead of the $E(W^\vee)$-isotypic component in Step 1. 
\end{remark}

From the next section onward, we shall carry out the recipe for the F-method
in the case 
\begin{equation*}
(G,G')=(GL(3,\R), GL(2,\R))
\end{equation*}
 with minimal
parabolic subgroups $P$ and $P'$ and $\dim_\C V=\dim_\C W =1$.

%%%%%%%%%%%%%%%%%%%%%%%%%%%%%%%%%%%%%%%%%
\section{Setup and notation for \texorpdfstring{$G=GL(3,\RR)$}{G=GL(3,R)} and \texorpdfstring{$G'=GL(2,\RR)$}{G'=GL(2,R)}}
\label{sec:GL(3)}

In this short section, we first introduce some 
notation and conventions used in the rest of the paper.
We then compute the isotypic decomposition of $\Pol(\frakn_+)$
as Step 1 of the recipe in Section \ref{sec:recipe}.

As in Section \ref{sec:Fmethod},
for a real Lie algebra $\mathfrak{y}(\R)$, we write $\mathfrak{y}$
and $U(\mathfrak{y})$ for its complexification and the universal enveloping algebra of $\fraky$, respectively. 

%%%%%%%%%%%%%%%%%%%%%%%%%%%%%%%%%%%%%%%%%
\subsection{Notation for \texorpdfstring{$G=GL(3,\RR)$}{G=GL(3,R)}}
\label{sec:notation}

Let $G=GL(3,\RR)$ with Lie algebra $\mathfrak{g}(\RR)=\mathfrak{gl}(3,\RR)$ and set 
\begin{alignat}{3}
   N_1^+&=E_{1,2}, \quad 
    N_2^+&&=E_{2,3}, \quad  N_3^+&&=E_{1,3},\label{eqn:N+}\\[3pt]
   N_1^-&=E_{2,1}, \quad N_2^-&&=E_{3,2}, \quad  N_3^-&&=E_{3,1},\nonumber\\[3pt]
H_1&=E_{1,1}, \quad H_2&&=E_{2,2}, \quad  H_3&&=E_{3,3},\nonumber
\end{alignat}
where $E_{i,j}$ are the standard matrix units. We set $\fraka(\RR)=\Span_\RR\{H_1, H_2, H_3\}$ and $\frakn_\pm(\RR)=\Span_\RR\{N_1^\pm,N_2^\pm,N_3^\pm\}$.  Then $\frakg(\RR)=\frakn_-(\RR)\oplus \fraka(\RR)\oplus\frakn_+(\RR)$ is a Gelfand--Naimark decomposition of $\frakg(\RR)$ and 
\begin{equation}\label{eqn:minimal_parabolic}
\frakp(\RR):=\fraka(\RR)\oplus\frakn_+(\RR)
= \big\{
\begin{pmatrix} 
* & * & * \\
0 & * & * \\
0 & 0 & *
\end{pmatrix}\in \frakg(\RR)\big\}
\end{equation}
 is a minimal parabolic subalgebra of $\frakg(\RR)$. We let $N_\pm =\exp(\frakn_\pm(\RR))$ and $A=\exp(\fraka(\RR))$ and fix a maximal compact subgroup $K=O(3)$. Then $G=KAN_+$ is an Iwasawa decomposition of $G$. 
We write $M=Z_K(\fraka(\RR))$, namely,
\[
M=\big\{\diag(m_1,m_2,m_3) :  m_1,m_2,m_3\in \{-1,1\}\big\} \simeq (\Z/2\Z)^3,
\]
so that $P=MAN_+$ is a minimal parabolic subgroup of $G$ with Lie algebra $\fraka(\RR)\oplus\frakn_+(\RR)$. Here, the subgroup $MA$ is given by
\begin{equation*}
MA=\{\diag(\ell_1,\ell_2,\ell_3): \ell_j \in \R^{\times} \; (j=1,2,3)\}.
\end{equation*}

For $\xi=(\xi_1,\xi_2,\xi_3)\in (\ZZ/2\ZZ)^3$, we define a one-dimensional representation $\CC_\xi=(\chi_{M,\xi},\CC)$ of $M$ by 
\begin{equation}\label{eqn:characterM}
\chi_{M,\xi}\colon
\diag(m_1,m_2,m_3) \mapsto m_1^{\xi_1}m_2^{\xi_2}m_3^{\xi_3}.
\end{equation}
Up to equivalence, all irreducible representations of $M$ are of the form $\CC_\xi$ for some $\xi\in (\ZZ/2\ZZ)^3$. Similarly, for $\lambda=(\lambda_1,\lambda_2,\lambda_3)\in \CC^3$, we define a one-dimensional representation $\CC_\lambda=(\chi_{A,\lambda},\CC)$ of $A$ by
\begin{equation}\label{eqn:characterA}
\chi_{A,\lambda}\colon
\exp(t_1H_1+t_2H_2+t_3H_3)\mapsto \exp(\lambda_1t_1+\lambda_2t_2+\lambda_3t_3).
\end{equation}
Then all irreducible representations of $A$ are, up to equivalence,  
of the form $\CC_\lambda$ for some $\lambda\in \CC^3$.
By extending trivially to $N_+$, the representations $\CC_\xi\boxtimes \CC_\lambda$ constitute all irreducible representations of $P=MAN_+$. 

We let 
\[
I(\xi,\lambda)=\Ind_P^G(\CC_\xi\boxtimes\CC_\lambda),\qquad (\xi\in (\ZZ/2\ZZ)^3,\lambda\in \CC^3)
\]
be the unnormalized parabolically induced representations, also known as the principal series representations of $G$.

%%%%%%%%%%%%%%%%%%%%%%%%%%%%%%%%%%%%%%%%%
\subsection{Notation for \texorpdfstring{$G'=GL(2,\RR)$}{G'=GL(2,R)}}

We wish to consider the subgroup $G'=GL(2,\RR)$ of $G$.
However, there are three different ways of embedding $G'$ into $G$ such that
a parabolic subgroup $P'=M'A'N_+'$ of $G'$ satisfies
$M'A' \subset MA$ and $N_+'\subset N_+$ for $P=MAN_+$.
So first, we fix some notation for $G'$ and then we consider these three embeddings in $G$. 

We introduce notation for $G'$ similar to the one we introduced for $G$. 
Write $\frakg'(\RR) = \mathfrak{gl}(2,\RR)$ and put 
\[
H'_1=E_{1,1},\qquad H'_2=E_{2,2}.
\]
We then set 
$\fraka'(\RR)=\Span_\RR\{H_1',H_2'\}$, 
$\frakn'_{+}(\RR)=\RR E_{1,2}$ and $\frakn'_-(\RR)=\RR E_{2,1}$.
Then $\frakg'(\RR)=\frakn'_-(\RR)\oplus \fraka'(\RR)\oplus \frakn_+'(\RR)$ is a Gelfand--Naimark decomposition and 
$\frakp'(\RR):=\fraka'(\RR)\oplus \frakn_+'(\RR)$ is a minimal parabolic subalgebra of $\frakg'(\RR)$. Let $A'=\exp(\fraka'(\RR))$ and $N'=\exp(\frakn_+'(\RR))$ and let $K'=O(2)$ such that $G'=K'A'N_+'$ is an Iwasawa decomposition of $G'$. We write $M'=Z_{K'}(\fraka'(\RR))$ such that $P'=M'A'N'$ is a minimal parabolic subgroup of $G'$ with Lie algebra $\fraka'(\RR)\oplus \frakn_+'(\RR)$.

The subgroup $M'$ is given by
\[
M'=\big\{\diag(m_1',m_2') :  m_1',m_2'\in \{-1,1\}\big\}.
\]
Similarly to the case for $P$, we can parameterize all irreducible finite-dimensional representations of $P'$ by $(\eta,\nu)\in (\ZZ/2\ZZ)^2\times \CC^2$ and define the principal series representations 
\[
J(\eta,\nu)=\Ind_{P'}^{G'}(\CC_{\eta}\boxtimes \CC_\nu),\qquad (\eta\in (\ZZ/2\ZZ)^2,\nu\in \CC^2)
\]
as unnormalized parabolic induction. Here $\CC_\eta=(\chi_{M',\eta},\CC)$ and $\CC_\nu=(\chi_{A',\nu},\CC)$ are given by 
\begin{alignat}{2}
&\chi_{M',\eta}\colon \diag(m_1,m_2)&&\mapsto m_1^{\eta_1}m_2^{\eta_2},
\label{eqn:characterMprime} \\[3pt]
&\chi_{A',\nu}\colon 
\exp(t_1H_1'+t_2H_2')&&\mapsto \exp(t_1\nu_1+t_2\nu_2).
\label{eqn:characterAprime}
\end{alignat}

%%%%%%%%%%%%%%%%%%%%%%%%%%%%%%%%%%%%%%%%%
\subsection{Embeddings of \texorpdfstring{$G'$}{G'} in \texorpdfstring{$G$}{G}}
\label{sec:Gprime}

We consider the following three embeddings of $G'$ in $G$ given by
\begin{equation}\label{eqn:embedding}
\iota_1: g'\mapsto \begin{pmatrix}
    g' & 0\\
     0& 1
\end{pmatrix},\qquad
\iota_2: g'\mapsto \begin{pmatrix}
    1 &0 \\
    0& g'
\end{pmatrix},\qquad
\iota_3:\begin{pmatrix}
    *_1 & *_2\\
    *_3 & *_4
\end{pmatrix}\mapsto \begin{pmatrix}
    *_1 & 0 & *_2\\
    0 & 1 & 0\\
    *_3 & 0 & *_4
\end{pmatrix},
\end{equation}
and similarly on the Lie algebra level we can embed $\frakg'(\RR)$ into $\frakg(\RR)$ by
\begin{equation}\label{eqn:embedding2}
\iota_1:X\mapsto \begin{pmatrix}
    X &0 \\
    0&0 \\
\end{pmatrix},\qquad \iota_2:X\mapsto \begin{pmatrix}
    0 & 0\\
    0 & X
\end{pmatrix},\qquad \iota_3:\begin{pmatrix}
    *_1 & *_2\\
    *_3 & *_4
\end{pmatrix}\mapsto \begin{pmatrix}
    *_1 & 0 & *_2\\
    0 & 0 & 0\\
    *_3 & 0 & *_4
\end{pmatrix},
\end{equation}
where we abuse notation and do not make a distinction in the notation of the maps on Lie group and Lie algebra level. 

We then let $\frakg_i'(\RR)=\iota_i(\frakg'(\RR))$,  $\fraka_i'(\RR)=\iota_i(\fraka'(\RR))$ and 
$\frakn_{i,\pm}'(\RR)=\iota_i(\frakn'_\pm(\RR))$ such that 
\[
\fraka_1'(\RR)=\Span_\RR\{H_1,H_2\},\qquad \fraka_2'(\RR)=\Span_\RR\{H_2,H_3\},\qquad \fraka_3'(\RR)=\Span_\RR\{H_1,H_3\},
\]
and
\[
\frakn_{i,\pm}'(\RR)=\RR N_i^{\pm},\qquad (i=1,2,3).
\]
We can then set $M_i'=\iota_i(M')$, $A'_i=\exp(\fraka_i'(\RR))$ and $N_{i,\pm}'=\exp(\frakn_{i,\pm}'(\RR))$ such that 
\[
P_i':=M_i'A_i'N_{i,+}'\subset P
\]
with $M_i'\subset M$, $A_i'\subset A$ and $N_{i,\pm}'\subset N_{\pm}$. We also write $G^\prime\simeq  \iota_i(G^\prime)=: G^\prime_i\subset G$.

Then our aim is to classify and construct DSBOs
\begin{equation*}
\DD \in \Diff_{G_i^\prime}(I(\xi, \lambda), J(\eta,\nu))
\end{equation*}
for $i=1,2,3$, and differential intertwining operators (DIOs)
\begin{equation*}
\mathcal{D} \in \Diff_{G}(I(\xi, \lambda), I(\xi',\lambda')).
\end{equation*}

%%%%%%%%%%%%%%%%%%%%%%%%%%%%%%%%%%%%%%%%%
\subsection{DSBOs \texorpdfstring{$\DD$}{D} in the non-compact picture}

As in \eqref{eqn:space_DifferenitalOperators},
for each embedding $\iota_i \colon G^\prime \hookrightarrow G$,
we understand DSBOs
$\DD \in \Diff_{G_i^\prime}(I(\xi, \lambda), J(\eta,\nu))$
as 
$\DD\in \Diff_\CC(C^\infty(\R^3), C^\infty(\R))$
via the diffeomorphisms
\begin{alignat}{2}\label{eqn:coordinate}
\R^{3} 
&\stackrel{\sim}{\To} N_-, \quad (x_1, x_2, x_3) 
&&\mapsto \exp(x_1 N^-_1 + x_2 N^-_2+x_3N^-_3),
\end{alignat}
and
\begin{alignat}{2}
\R 
&\stackrel{\sim}{\To} N_{1,-}', \quad (x_1,0, 0) 
&&\mapsto \exp(x_1 N^-_1), \label{eqn:coordinate1}\\[3pt]
\R
&\stackrel{\sim}{\To} N_{2,-}', \quad (0,x_2,0) 
&&\mapsto \exp(x_2 N^-_2), \label{eqn:coordinate2}\\[3pt]
\R
&\stackrel{\sim}{\To} N_{3,-}', \quad (0,0, x_3) 
&&\mapsto \exp(x_3 N^-_3). \label{eqn:coordinate3}
\end{alignat}

%%%%%%%%%%%%%%%%%%%%%%%%%%%%%%%%%%%%%%%%%
\subsection{Isotypic decomposition of \texorpdfstring{$\Pol(\frakn_+)$}{Pol(n_+)}}

Put $M_0'A_0':=MA$. 
In the current setting, we consider the space
\begin{equation}\label{eqn:VW}
\Hom_{M'_iA'_i}\big(W^\vee,\Pol(\frakn_+)\otimes V^\vee\big)
\simeq 
\Hom_{M'_iA'_i}\big(V,\Pol(\frakn_+)\otimes W\big)
\end{equation}
of $M'_iA'_i$-maps ($i=0,1,2,3$)  for the following $(V,W)$.
\vskip 0.1in

\begin{enumerate}[label = \normalfont{(\arabic*)}]

\item $i=0$: 
$(V, W) = (\CC_\xi\boxtimes \CC_{\lambda},\CC_{\xi'}\boxtimes \CC_{\lambda'})$
for $(\xi, \xi'; \lambda, \lambda')\in (\Z/2\Z)^6 \times \C^6$.

\vskip 0.1in

\item $i=1,2,3$: 
$(V,W) = (\CC_\xi\boxtimes \CC_{\lambda},\CC_{\eta}\boxtimes \CC_{\nu})$
for $(\xi,\eta;\lambda,\nu) \in (\ZZ/2\ZZ)^5 \times \CC^5$.
\vskip 0.1in

\end{enumerate}

For a character $\chi_i$ of $M'_iA_i'$ for $i= 0,1,2,3$,
we denote by $E(\chi_i)$ the 
$\chi_i$-isotypic component of $\Pol(\frakn_+)$. 
Since $\dim_\C V= \dim_\C W=1$, we have 
\begin{equation}\label{eqn:HomMA}
\Hom_{M'_iA'_i}\big(V,\Pol(\frakn_+)\otimes W\big)
= \bigoplus_{\text{characters $\chi_i$}} \Hom_{M'_iA'_i}\big(V,E(\chi_i)\otimes W\big),
\end{equation}
where the direct sum runs over all characters $\chi_i$ of $M'_iA'_i$.
Then we end this section by investigating the 
isotypic components of $\Pol(\frakn_+)$ as Step 1 in the recipe in 
Section \ref{sec:recipe}.

As in Section \ref{sec:dpi2}, let $(\zeta_1, \zeta_2, \zeta_3)$ denote
the dual coordinates on $\frakn_+$ such that $\zeta_i(N_j^+)=\delta_{i,j}$
for the basis $\{N_1^+, N_2^+, N_3^+\}$ of $\frakn_+$.
One may understand
$\psi(\zeta) \equiv \psi(\zeta_1,\zeta_2,\zeta_3) \in \Pol(\frakn_+)=\C[\zeta_1,\zeta_2,\zeta_3]$ as 
\[
\psi(\zeta_1,\zeta_2,\zeta_3)=\psi(\zeta_1N_1^++\zeta_2N_2^++\zeta_3N_3^+).
\]
For $\alpha=(\alpha_1,\alpha_2)\in \NN^2$ and 
$k = 0,1,\dots,\min(\alpha)$, put
\begin{equation}\label{eqn:q_alpha_k}
\psi_{\alpha,k}(\zeta) := \zeta_1^{\alpha_1-k}\zeta_2^{\alpha_2-k}\zeta_3^k.
\end{equation}
Then we have the following.
\begin{lemma}\label{lem:MApol}
The one-dimensional space $\C \psi_{\alpha,k}(\zeta)$ is a representation of $MA$.
In fact, the following hold.
\begin{equation*}
\C \psi_{\alpha, k}(\zeta) \simeq 
\begin{cases}
\C_{(\alpha_1,\, \alpha_1 - \alpha_2,\, \alpha_2)}
\boxtimes 
\C_{(-\alpha_1,\, \alpha_1 - \alpha_2,\, \alpha_2)} 
&\textnormal{as an $MA$-module},\\[3pt]
\hspace{15.7pt}
\C_{(\alpha_1\, \alpha_1 - \alpha_2)}
\boxtimes 
\C_{(-\alpha_1,\, \alpha_1 - \alpha_2)} 
& \textnormal{as an $M_1'A_1'$-module},\\[3pt]
\hspace{13.5pt}
\C_{(\alpha_1 - \alpha_2,\, \alpha_2)}
\boxtimes 
\C_{(\alpha_1 - \alpha_2,\, \alpha_2)} 
& \textnormal{as an $M_2'A_2'$-module},\\[3pt]
\hspace{30pt}
\C_{(\alpha_1,\, \alpha_2)} 
\boxtimes\C_{(-\alpha_1,\, \alpha_2)} & \textnormal{as an  $M_3'A_3'$-module}.\\[3pt]
\end{cases}
\end{equation*}
\end{lemma}

\begin{proof}
The polynomial $\psi_{\alpha,k}(\zeta)$ enjoys the following transformation laws:
\begin{align*}
    &\Ad_\#(\diag(t,1,1))\psi_{\alpha,k}(\zeta)=t^{-\alpha_1}\psi_{\alpha,k}(\zeta),\\[3pt]
    &\Ad_\# (\diag(1,t,1))\psi_{\alpha,k}(\zeta)=t^{\alpha_1-\alpha_2}\psi_{\alpha,k}(\zeta),\\[3pt]
    &\Ad_\#(\diag(1,1,t))\psi_{\alpha,k}(\zeta)=t^{\alpha_2}\psi_{\alpha,k}(\zeta).
\end{align*}
This shows the lemma.
\end{proof}

Now put
\begin{equation}\label{eq:DefinitionPol(alpha)}
\Pol(\alpha) := \Span_\C\{\psi_{\alpha,k}(\zeta): k=0, \ldots, \min(\alpha)\},
\end{equation}
so that 
\begin{equation}\label{eqn:Pol}
\Pol(\frakn_+) = \bigoplus_{\alpha \in \N^2} \Pol(\alpha).
\end{equation}

In the following lemma, by the $\C \psi_{\alpha,k}$-isotypic component,
we mean the isotypic component corresponding to
the one-dimensional representation of $M_i'A_i'$ on $\C \psi_{\alpha,k}$ shown in Lemma \ref{lem:MApol}.

\begin{lemma}\label{lem:MApol2}
The subspace
$\Pol(\alpha)$ is the $\C \psi_{\alpha,k}$-isotypic 
component of $\Pol(\frakn_+)$ as a representation of $M_i'A_i'$ (i=0,1,2,3).
Equivalently, the $M_i'A_i'$-module 
$\Pol(\alpha)\big\vert_{M_i'A_i'}$  can be characterized as 
\begin{equation}\label{eqn:Pol_alpha_MjAj}
\Pol(\alpha)\big\vert_{M_i'A_i'}
=\{\psi(\zeta) \in \Pol(\frakn_+):
\text{$\C \psi(\zeta) \simeq \C \psi_{\alpha,k}$ as $M_i'A_i'$-modules}\}.
\end{equation}
Moreover, the subspaces $\Pol(\alpha)$ are all 
isotypic components in $\Pol(\frakn_+)$ as $M_i'A_i'$-modules.
\end{lemma}

\begin{proof}
The lemma readily follows from Lemma \ref{lem:MApol}  and the $M'_iA'_i$-decomposition \eqref{eqn:Pol}.
\end{proof}

It follows from Lemma \ref{lem:MApol2} that,
for $(V,W) = (\CC_\xi\boxtimes \CC_{\lambda},\CC_{\xi'}\boxtimes \CC_{\lambda'})$
for $i=0$ and 
$(V,W) = (\CC_\xi\boxtimes \CC_{\lambda},\CC_{\eta}\boxtimes \CC_{\nu})$
for $i=1,2,3$, the decomposition in  \eqref{eqn:HomMA} is given as follows.

\vskip 0.1in

\begin{enumerate}[label = \normalfont{(\arabic*)}]
\item $i=0$: We have
\begin{align}
&\Hom_{MA}\big(\CC_\xi\boxtimes \CC_{\lambda},\Pol(\frakn_+)\otimes (\CC_{\xi'}\boxtimes \CC_{\lambda'})\big) \nonumber \\[5pt]
&\hspace{4cm}
=\bigoplus_{\alpha \in \N^2}
\Hom_{MA}\big(\CC_\xi\boxtimes \CC_{\lambda},\Pol(\alpha)\otimes (\CC_{\xi'}\boxtimes \CC_{\lambda'})\big). \label{eqn:Homalpha2}
\end{align}

\item $i=1,2,3$: We have 
\begin{align}
&\Hom_{M_i'A_i'}\big((\CC_\xi\boxtimes \CC_{\lambda})\big\rvert_{M_i'A_i'},\Pol(\frakn_+)\otimes (\CC_\eta\boxtimes \CC_{\nu})\big) \nonumber \\[5pt]
&\hspace{4cm}
=\bigoplus_{\alpha \in \N^2}
\Hom_{M_i'A_i'}\big((\CC_\xi\boxtimes \CC_{\lambda})\big\rvert_{M_i'A_i'}, \Pol(\alpha)\otimes (\CC_\eta\boxtimes \CC_{\nu})\big). \label{eqn:Homalpha}
\end{align}
\end{enumerate}

We now focus on the case $i=1,2,3$;
the case $i=0$ (the case $G=G'$) is considered carefully in Section \ref{sec:DIO}.
Let $\Sol_i(\xi,\eta; \lambda,\nu)$ denote the space of solutions to the F-system for $i=1,2,3$.
Since  $\frakn_{i,+}'=\Span_\C\{N_i^+\}$ for $i=1,2,3$,
the space $\Sol_i(\xi,\eta; \lambda,\nu)$ is given as follows.
\begin{align}
&\Sol_i(\xi,\eta; \lambda,\nu) \nonumber\\
&\hspace{50pt}= \{ \psi \in 
\Hom_{M_i'A_i'}\big((\CC_\xi\boxtimes \CC_{\lambda})\big\rvert_{M_i'A_i'},\Pol(\frakn_+)\otimes (\CC_\eta\boxtimes \CC_{\nu})\big) : 
\text{\eqref{eqn:PDE_SymmetryBreaking} holds}\}
\end{align}
\begin{equation}\label{eqn:PDE_SymmetryBreaking}
\widehat{d\pi_{(\xi,\lambda)^*}}(N_i^+)\psi = 0
\end{equation}

Moreover, we define
\begin{align}
\Sol_i^{(\alpha)}(\xi,\lambda) 
:= \{ \psi \in \Hom_{M_i'A_i'}\big((\CC_\xi\boxtimes \CC_{\lambda})\big\rvert_{M_i'A_i'},\Pol(\alpha)\otimes (\CC_\eta\boxtimes \CC_{\nu})\big) : \text{\eqref{eqn:PDE_SymmetryBreaking} holds}\}, \label{eqn:Solalpha}
\end{align}
so that
\begin{equation}\label{eqn:SolDecomp}
\Sol_i(\xi,\eta; \lambda,\nu) 
= \bigoplus_{\alpha \in \NN^2}\Sol_i^{(\alpha)}(\xi,\lambda).
\end{equation}

In the next section, we discuss the symmetrization operator 
and the ordered F-system to the PDE \eqref{eqn:PDE_SymmetryBreaking}.

%%%%%%%%%%%%%%%%%%%%%%%%%%%%%%%%%%%%%%%%%
\section{The ordered F-system for \texorpdfstring{$G=GL(3,\RR)$}{G=GL(3,R)} and \texorpdfstring{$G'=GL(2,\RR)$}{G'=GL(2,R)}}
\label{sec:ordered-GL(3)}

The aim of this section is to proceed with Steps 2--4 in the recipe 
in Section \ref{sec:recipe}.
In particular, we compute the commutative diagram
\eqref{eqn:RestSymbOrd} in the present setting.
This is done in Theorem \ref{eqn:conji}. 

As part of Steps 2--4, 
we also consider the following commutative diagram
from  \eqref{eqn:SolSolU}.
\begin{equation*}
\begin{tikzcd}
	{\Sol_{\Ad(T^\sharp S)T^\sharp A }(Q)} & 
	{\Sol_{\Ad(S)A}(E)} 	\arrow[dl, pos=0.5, phantom, "\circlearrowleft"]
	&{U(\frakn_-)\otimes V^\vee} \\
	{\Sol_{T^\sharp A}(Q)} & {\Sol_{A}(E)} &
	\arrow[ul, pos=0.85, phantom, "\circlearrowleft"]
	\arrow[r, "\Upsilon_{\ord}", hook, from=1-2, to=1-3]
	\arrow["F_c^{-1}"', hook, from=2-2, to=1-3]
	\arrow["{T}", "\sim"' , from=1-1, to=1-2]
	\arrow["\sim" sloped, "(T^\sharp S)^{-1}"', from=1-1, to=2-1]
	\arrow["S^{-1}_{\ord}", "\sim"' sloped,  from=1-2, to=2-2]
	\arrow["\sim" , "{T}"',  from=2-1, to=2-2]
\end{tikzcd}
\end{equation*}

In the current setting,
the $S_\ord$-invariant subspace $E \subset \Pol(\frakn_+)$ 
is taken to be the isotypic component $\Pol(\alpha)$ 
(cf.\  Corollary \ref{cor:G_orderization}).
Then we do the following.
\vspace{3pt}

\begin{itemize}

\item
Compute the symmetrization operator $S_\ord$.
\vskip 0.1in

\item
Find a vector space $Q$ with a linear map 
$T\colon Q \stackrel{\sim}{\to} \Pol(\alpha)$.
\vskip 0.1in

\item
Compute a linear operator $A \in \End_\C(\Pol(\alpha))$
such that $A$ yields a modified F-system $D_F$ (see Definition \ref{def:modified-Fsystem}).
\vskip 0.1in

\end{itemize}

We remark that the vector space $Q$ with the linear operator 
$T\colon Q \stackrel{\sim}{\to} \Pol(\alpha)$
provides the T-saturation (\cite[Sec.\ 3.2]{KoPe16b}) of $A$, which will reduce the modified F-system to an ordinary differential equation.

At the end of the section,
we revisit the second part of the recipe 
in Section \ref{sec:recipe} using the T-saturation.

%%%%%%%%%%%%%%%%%%%%%%%%%%%%%%%%%%%%%%%%%
\subsection{Symmetrization operator for the nilpotent radical}
\label{sec:Upn}

Fix the ordered basis 
\begin{equation}\label{eqn:ord_nilpotent}
\ord:=\{N_1^-, N_2^-, N_3^-\}
\end{equation}
of $\mathfrak{n}_-$. Then the PBW-basis $\mathrm{PBW}_\ord$
of $U(\frakn_-)$ induced from $\ord$ is 
\begin{equation*}
\mathrm{PBW}_\ord=\{(N_1^-)^{r_1}(N_2^-)^{r_2}(N_3^-)^{r_3}: r_j \geq 0\}.
\end{equation*}
As in \eqref{eqn:X-ord}, for each $(N_{j_1}^-)(N_{j_2}^-)\cdots (N_{j_m}^-) \in U(\frakn_-)$ with $j_k=1,2,3$, we write
\begin{equation*}
\big( (N_{j_1}^-)(N_{j_2}^-)\cdots (N_{j_m}^-) \big)_\ord
:=
(N_{j_{\tau(1)}}^-)(N_{j_{\tau(2)}}^-)\cdots (N_{j_{\tau(m)}}^-)
\in \mathrm{PBW}_\ord
\end{equation*}
for $\tau \in \mathfrak{S}_m$ such that $j_{\tau(k)} \leq j_{\tau(k+1)}$
for all $k=1,2,\ldots, m-1$.

Write $(\zeta_1, \zeta_2, \zeta_3)$ for the coordinates 
on $\frakn_+\simeq (\frakn_-)^\vee$ with respect to the 
dual basis of $\ord$.
Then, by \eqref{eqn:Upsilon_ord},
the ordering map $\Upsilon_\ord$ from $\Pol(\frakn_+)$ to 
$U(\mathfrak{n}_-)$ is given by
\begin{equation}\label{eqn:upsilon_ord2}
\Upsilon_{\ord} \colon \Pol(\frakg^\vee) \stackrel{\sim}{\To} U(\frakg), \quad 
\zeta_{j_1}\zeta_{j_2}\cdots \zeta_{j_m} 
\longmapsto
\big((N_{j_1}^-)(N_{j_2}^-)\cdots (N_{j_m}^-) \big)_\ord.
\end{equation}

To compute the symmetrization operator $S_{\ord}$ on $\Pol(\frakn_+)$,
we put
\begin{equation*}
J:=\frac{\zeta_3}{2}\frac{\partial^2}{\partial\zeta_1\partial\zeta_2}
\end{equation*}
and formally define
\begin{equation*}
e^{J}:=\sum_{\ell=0}^\infty \frac{J^\ell}{\ell!}.
\end{equation*}
A direct computation shows that
\begin{equation}\label{eqn:JL}
J^\ell
\zeta_1^i\zeta_2^j\zeta_3^k = 
\frac{(-i)_\ell (-j)_\ell}{2^\ell}\zeta_1^{i-\ell}\zeta_2^{j-\ell}\zeta_3^{k+\ell},
\end{equation}
where $(a)_n:=a(a+1)(a+2)\cdots (a+n-1) = \frac{\Gamma(a+n)}{\Gamma(a)}$ is the Pochhammer symbol.
As $(-i)_\ell(-j)_\ell$ vanish for $\ell>\min(i,j)$, it follows that $e^{J}$ is a well-defined operator on $\Pol(\frakn_+)= \C[\zeta_1, \zeta_2, \zeta_3]$.

We claim that $e^{J}=S_{\ord}$.
For the proof, we first consider the following identity about
the universal enveloping algebra
$U(\frakh_3)$ of the Heisenberg Lie algebra 
$\frakh_3 = \Span\{X_1, X_2, X_3\}$ with $[X_1, X_2]=-X_3$.

\begin{proposition}\label{Prop:SymmetrizationMapIdentity}
For $i+j+k =n$,
we have
\begin{align}\label{eqn:HeisSym}
\sym_n(X_1^iX_2^jX_3^k)
&=\sum_{\ell=0}^{\min(i,j)}\frac{(-i)_\ell(-j)_\ell}{2^\ell \ell!}
X_1^{i-\ell}X_2^{j-\ell}X_3^{k+\ell}.
\end{align}
\end{proposition}

\begin{proof}
Observe that, as $X_3$ is a central element, we have 
\begin{equation*}
\sym_n(X_1^iX_2^jX_3^k)=\sym_{n-k}(X_1^iX_2^j)X_3^k.
\end{equation*}
It thus suffices to show \eqref{eqn:HeisSym} for  $k=0$.

Take $s, t \in \CC$. Since $[X_1, X_2]=-X_3$, 
the Baker--Campbell--Hausdorff formula implies that 
\begin{equation}\label{Eq:BCH-formula}
\exp(sX_1)\exp(tX_2)\exp(\tfrac{st}{2}X_3)=\exp(sX_1+tX_2).
\end{equation}
By expanding the left-hand side of \eqref{Eq:BCH-formula} we obtain
\begin{align}
\exp(sX_1)\exp(tX_2)\exp(\tfrac{st}{2}X_3)
&=
\sum_{p,q,\ell=0}^\infty\frac{s^{p+\ell}t^{q+\ell}}{p!q!\ell! 2^\ell}
X_1^pX_2^qX_3^\ell \nonumber \\[3pt]
&=
\sum_{i,j=0}^\infty s^it^j\bigg(\sum_{\ell=0}^{\min(i,j)}\frac{X_1^{i-\ell}X_2^{j-\ell}X_3^\ell}{(i-\ell)!(j-\ell)!\ell!2^\ell}\bigg).\label{eqn:expSym1}
\end{align}
Likewise, by expanding the right-hand side of \eqref{Eq:BCH-formula} we get
\begin{align}
\exp(sX_1+tX_2)
&=\sum_{\ell=0}^\infty \frac{(sX_1+tX_2)^\ell}{\ell!}\nonumber\\[3pt]
&=\sum_{i,j=0}^\infty\frac{s^it^j}{i!j!}\sym_{i+j}(X_1^iX_2^j),\label{eqn:expSym2}
\end{align}
where the following expansion of $(sX_1+tX_2)^\ell$ is applied from the first line 
to the second:
\begin{equation*}
(sX_1+tX_2)^\ell
=\sum_{i+j=\ell}\frac{\ell!}{i!j!}\sym_{i+j}(X_1^iX_2^j).
\end{equation*}
The desired identity is now obtained by comparing the coefficients of 
$s^it^j$ in \eqref{eqn:expSym1} with that in \eqref{eqn:expSym2}.
\end{proof}

\begin{remark}
If we choose other ordered basis, namely,
$\ord':=\{X_2, X_1, X_3\}$, 
for the Heisenberg Lie algebra $\frakh_3$
with $[X_2, X_1]=X_3$,
then 
\begin{equation*}
\sym_n(X_1^iX_2^jX_3^k)
=\sum_{\ell=0}^{\min(i,j)}(-1)^\ell\frac{(-i)_\ell(-j)_\ell}{2^\ell \ell!}
X_2^{i-\ell}X_1^{j-\ell}X_3^{k+\ell}.
\end{equation*}
\end{remark}
 
\begin{proposition}\label{prop:eJU}
The symmetrization operator $S_\ord$ with respect to the ordered basis 
$\ord$ in \eqref{eqn:ord_nilpotent} is given as 
\begin{equation*}
S_{\ord}=e^{J}.
\end{equation*}
\end{proposition}

\begin{proof}
It follows from \eqref{eqn:JL} that
\begin{equation}\label{eqn:ijk}
e^{J}\zeta_1^i\zeta_2^j\zeta_3^k 
=\sum_{\ell=0}^{\min(i,j)} 
\frac{(-i)_\ell (-j)_\ell}{2^\ell \ell!}\zeta_1^{i-\ell}\zeta_2^{j-\ell}\zeta_3^{k+\ell}.
\end{equation}
As $\frakn_-$ is the Heisenberg Lie algebra such that
$\frakn_-=\Span_\CC\{N_1^-, N_2^-, N_3^-\}$
 with $[N_1^-, N_2^-]=-N_3^-$, 
Proposition \ref{Prop:SymmetrizationMapIdentity} concludes the proof.
\end{proof}

\begin{proposition}\label{prop:symU}
The subspace $\Pol(\alpha)$ is $e^{J}$-invariant.
\end{proposition}

\begin{proof}
It follows from Lemma \ref{lem:MApol2} that
$\Pol(\alpha)$ is the $\C \psi_{\alpha,k}$-isotypic 
component of $\Pol(\frakn_+)$.
Since $e^J$ is the symmetrization operator on $\Pol(\frakn_+)$,
Corollary \ref{cor:G_orderization} concludes the desired proposition.
\end{proof}

We write
\begin{equation*}
J(\alpha)
:=\frac{\zeta_3}{2}\frac{\partial^2}{\partial\zeta_1\partial\zeta_2}\, \bigg\vert_{\Pol(\alpha)}.
\end{equation*}
Then, by Propositions \ref{prop:eJU} and \ref{prop:symU}, we have 
\begin{equation}\label{eqn:symU}
\sym_{\ord} \circ \Upsilon_{\ord}\vert_{\Pol(\alpha)} 
= \Upsilon_{\ord} \circ e^{J(\alpha)},
\end{equation}
that is, the following diagram commutes.
\begin{equation}\label{eqn:diagram0}
\begin{tikzcd}[row sep=1cm, column sep=1cm]
\Pol(\alpha)
\arrow[r, "\Upsilon_{\ord}", hook]
& 
U(\frakn_-)
\arrow[dl, pos=0.5, phantom, "\circlearrowleft"]
\\
\Pol(\alpha)
\arrow[u, "e^{J(\alpha)}","\sim"' sloped]
\arrow[r, "\Upsilon_{\ord}"', hook]
& U(\frakn_-)
\arrow[u, "\sim" sloped, "\sym_{\ord}"']
\end{tikzcd}
\end{equation}

%%%%%%%%%%%%%%%%%%%%%%%%%%%%%%%%%%%%%%%%%
\subsection{T-saturation}
\label{sec:T-saturation}

Observe that one may identify $\Pol(\alpha)$ with the space $\Pol_{\min(\alpha)}[t]$ of polynomials $p(t)$ of one variable $t$ with $\deg p(t) \leq \min(\alpha)$ via the linear map
\begin{equation}\label{eqn:Tsat}
T_\alpha\colon
\Pol_{\min(\alpha)}[t]\xrightarrow{\sim}\Pol(\alpha),\quad 
p(t)\mapsto \zeta_1^{\alpha_1}\zeta_2^{\alpha_2}p\Big(\frac{\zeta_3}{\zeta_1\zeta_2}\Big).
\end{equation}
We call $T_{\alpha}$ a \emph{saturation map}.
As in Section \ref{sec:ImFc}, 
given $A\in \End_\C(\Pol(\alpha))$, we define 
$T_\alpha^\sharp A\in \End_\C(\Pol_{\min(\alpha)}[t])$
of $A$ by
\begin{equation}\label{eqnTsharp}
T_\alpha\circ T_\alpha^\sharp A=A\circ T_\alpha, 
\end{equation}
so that the following diagram commutes.
\[\begin{tikzcd}
	{\Pol_{\min(\alpha)}[t]} & {\Pol(\alpha)}
	\arrow[dl, pos=0.5, phantom, "\circlearrowleft"] \\
	{\Pol_{\min(\alpha)}[t]} & {\Pol(\alpha)}
	\arrow["{T_\alpha}", "\sim"', from=1-1, to=1-2]
	\arrow[ "{T_\alpha^\sharp A}"', from=1-1, to=2-1]
	\arrow["A",  from=1-2, to=2-2]
	\arrow["\sim", "{T_\alpha}"', from=2-1, to=2-2]
\end{tikzcd}\]
Following \cite[Sec.\ 3.2]{KoPe16b}, 
we call $T_\alpha^\sharp A$ the \emph{T-saturation} of $A$.

For $\alpha=(\alpha_1,\alpha_2) \in \NN^2$, we put
\begin{equation}\label{eqn:Da}
\wJ(\alpha):=\frac{t}{2}(\vartheta_t-\alpha_1)(\vartheta_t-\alpha_2).
\end{equation}
For $k,\ell\in \N$, a direct computation shows that 
\begin{equation}\label{eqn:D1}
\wJ(\alpha)^\ell t^k 
=\frac{(k-\alpha_1)_\ell(k-\alpha_2)_\ell}{2^\ell}
t^{k+\ell},
\end{equation}
where $(x)_\ell:=x(x+1) \cdots (x+\ell-1) = \frac{\Gamma(x+\ell)}{\Gamma(x)}$.
In particular, we have 
$\wJ(\alpha) \in \End_\CC(\Pol_{\min(\alpha)}[t])$.
As for $e^{J(\alpha)}$, one can define 
$e^{\wJ(\alpha)} \in GL(\Pol_{\min(\alpha)}[t])$ by
\begin{equation*}
e^{\wJ(\alpha)}
:=\sum_{\ell=0}^\infty \frac{\wJ(\alpha)^\ell}{\ell!}.
\end{equation*}

\begin{proposition}\label{prop:DD}
We have 
\begin{equation*}
e^{\wJ(\alpha)}=T_\alpha^\sharp e^{J(\alpha)}.
\end{equation*}
\end{proposition}

\begin{proof}
It suffices to show
\begin{equation*}
\wJ(\alpha)=T_\alpha^\sharp J(\alpha).
\end{equation*}
The lemma then immediately follows from \eqref{eqn:JL} and \eqref{eqn:D1}.
\end{proof}

It follows from Proposition \ref{prop:DD} that the commutative diagram \eqref{eqn:diagram0}
can be extended as follows.
\begin{equation}\label{eqn:diagram01}
\begin{tikzcd}[row sep=1cm, column sep=1cm]
\Pol_{\min(\alpha)}[t]
\arrow[r, "T_\alpha", "\sim"']
& 
\Pol(\alpha)
\arrow[dl, pos=0.5, phantom, "\circlearrowleft"]
\arrow[r, "\Upsilon_{\ord}", hook]
& 
U(\frakn_-)
\arrow[dl, pos=0.5, phantom, "\circlearrowleft"]
\\
\Pol_{\min(\alpha)}[t]
\arrow[u, "e^{\wJ(\alpha)}","\sim"' sloped]
\arrow[r,  "\sim", "T_\alpha"']
& 
\Pol(\alpha)
\arrow[u, "e^{J(\alpha)}","\sim"' sloped]
\arrow[r, "\Upsilon_{\ord}"', hook]
& U(\frakn_-)
\arrow[u, "\sim" sloped, "\sym_{\ord}"']
\end{tikzcd}
\end{equation}

\begin{remark}\label{rem:Heis_2F0}
Via the commutative diagram \eqref{eqn:diagram01}, one may relate
the symmetrization in \eqref{eqn:HeisSym}
to the generalized hypergeometric polynomial ${}_2F_0[-n, -m;t]$.
Indeed, take $1 \in \Pol_{\min(\alpha)}[t]$ at the lower left corner of \eqref{eqn:diagram01}.
For $\alpha=(\alpha_1,\alpha_2) \in \NN^2$, 
it follows from \eqref{eqn:D1} that 
\begin{equation}\label{eqn:eJ1}
e^{\wJ(\alpha)} 1 
= \sum_{\ell=0}^{\min(\alpha)}\frac{(-\alpha_1)_\ell(-\alpha_2)_\ell}{2^\ell \ell!}t^\ell
= {}_2F_0[-\alpha_1, -\alpha_2;\tfrac{t}{2}].
\end{equation}
Then the images of $1$ under the maps in \eqref{eqn:diagram01} are traced as 
follows.
\begin{equation*}
\begin{tikzcd}[row sep=1cm, column sep=1cm]
{}_2F_0[-\alpha_1, -\alpha_2;\tfrac{t}{2}]
\arrow[r, "T_\alpha", mapsto]
& 
T_\alpha({}_2F_0[-\alpha_1, -\alpha_2;\tfrac{t}{2}])
\arrow[dl, pos=0.5, phantom, "\circlearrowleft"]
\arrow[r, "\Upsilon_{\ord}", mapsto]
& 
\sym_{\alpha_1+\alpha_2}((N_1^-)^{\alpha_1}(N_2^-)^{\alpha_2})
\arrow[dl, pos=0.5, phantom, "\circlearrowleft"]
\\
1
\arrow[u, "e^{\wJ(\alpha)}", mapsto]
\arrow[r, "T_\alpha"', mapsto]
& 
\zeta_1^{\alpha_1}\zeta_2^{\alpha_2}
\arrow[u, "e^{J(\alpha)}",mapsto]
\arrow[r, "\Upsilon_{\ord}"', mapsto]
& (N_1^-)^{\alpha_1}(N_2^-)^{\alpha_2}
\arrow[u,  "\sym_{\ord}"', mapsto]
\end{tikzcd}
\end{equation*}
At the upper right corner, we have
\begin{align*}
\sym_{\alpha_1+\alpha_2}((N_1^-)^{\alpha_1}(N_2^-)^{\alpha_2})
&= \sum_{\ell=0}^{\min(\alpha)}\frac{(-\alpha_1)_\ell(-\alpha_2)_\ell}{2^\ell \ell!}
(N_1^-)^{\alpha_1-\ell}(N_2^-)^{\alpha_2-\ell}(N_3^-)^\ell\\[3pt]
&=(\Upsilon_{\ord}\circ T_\alpha)({}_2F_0[-\alpha_1, -\alpha_2;\tfrac{t}{2}]).
\end{align*}
As $\sym_{i+j+k}((N_1^-)^i(N_2^-)^j(N_3^-)^k) 
= (N_3^-)^k\sym_{i+j}\big((N_1^-)(N_2^-)\big)$, the general case also follows from the 
above arguments.
\end{remark}

\begin{remark}
One can readily show that the symmetrization operator for
Heisenberg Lie algebras of arbitrary dimension is 
a product of copies of the symmetrization operator for 
the three-dimensional Heisenberg Lie algebra. 
Since the generalization is not the main objective of this section, we will defer it to Appendix \ref{appendix:Heis}.
\end{remark}

For $A \in \End_\C(\Pol(\alpha))$, we write
\begin{align*}
\Sol_{A}^{(\alpha)}&=\{\psi(\zeta) \in \Pol(\alpha) : A\psi(\zeta) = 0\},\\
\wSol_{T_\alpha^\sharp A}^{(\alpha)}&=
\{q(t) \in \Pol_{\min(\alpha)}[t] : (T_\alpha^\sharp A)q(t)=0\}.
\end{align*}
Then, by Theorem \ref{thm:ImFc2}, we have
\begin{align*}
\Im F_c^{-1}(A)
&=\{(\Upsilon_{\ord} \circ T)q(t)
: q(t) \in \wSol_{\Ad(e^{\wJ(\alpha)})T_\alpha^\sharp A}^{(\alpha)}\}\\
&=\{(F_c^{-1}\circ T \circ e^{-\wJ(\alpha)})q(t) : p(t) \in 
\wSol_{\Ad(e^{\wJ(\alpha)})T_\alpha^\sharp A}^{(\alpha)}\}.
\end{align*}
Equivalently, the following diagram commutes.
\begin{equation}\label{eqn:cdA}
\begin{tikzcd}
	{\wSol_{\Ad(e^{\wJ(\alpha)})T_\alpha^\sharp A}^{(\alpha)}} & 
	\Sol_{\Ad(e^{J(\alpha)})A}^{(\alpha)}  	
	\arrow[dl, pos=0.5, phantom, "\circlearrowleft"]
	&{U(\frakn_-)} \\
	{\wSol_{T_\alpha^\sharp A}^{(\alpha)}}  & {\Sol_{A}^{(\alpha)}} &
	\arrow[ul, pos=0.886, phantom, "\circlearrowleft"]
	\arrow[r, "\Upsilon_{\ord}", hook, from=1-2, to=1-3]
	\arrow["F_c^{-1}"', hook, from=2-2, to=1-3]
	\arrow["{T_\alpha}", "\sim"' , from=1-1, to=1-2]
	\arrow["\sim" sloped, "e^{-\wJ(\alpha)}"', from=1-1, to=2-1]
	\arrow["e^{-J(\alpha)}", "\sim"' sloped,  from=1-2, to=2-2]
	\arrow["\sim" , "{T_\alpha}"',  from=2-1, to=2-2]
\end{tikzcd}
\end{equation}

%%%%%%%%%%%%%%%%%%%%%%%%%%%%%%%%%%%%%%%%%%
\subsection{The case \texorpdfstring{$A=-\zeta_i\widehat{d\pi_{(\xi,\lambda)^*}}(N_i^+)$}{A=-ζdπ^}}
\label{sec:dpihat2}

We now specialize $A \in \End_{\CC}(\Pol(\alpha))$ 
to a partial differential operator on $\Pol(\alpha)$.
More precisely, we show 
that $-\zeta_i\widehat{d\pi_{(\xi,\lambda)^*}}(N_i^+) \in \End_\C(\Pol(\alpha))$
for $i=1,2,3$.
Since the character $\chi_\xi$ of $M$ does not contribute to solving the F-system, we simply write $d\pi_{\lambda^*} \equiv d\pi_{(\xi,\lambda)^*}$, where
\[
d\chi_\lambda^*=d\chi_{2\rho - \lambda}=(2-\lambda_1,-\lambda_2,-\lambda_3-2).
\]
Here, we identify $\fraka^* \simeq \C^3$ such that 
$2\rho = (2, 0, -2)$.

\begin{proposition}\label{prop:frakn-mod}
We have 
$-\zeta_i\widehat{d\pi_{\lambda^*}}(N_i^+) \in \End_\C(\Pol(\alpha))$
for all $i=1,2,3$.
\end{proposition}

\begin{proof}
Observe that $MA$ acts on $\zeta_i$ and $N_i^+$ as characters.
Since $\zeta_i$ and $N_i^+$ are dual to each other, it implies that
$MA$ acts on $-\zeta_i\widehat{d\pi_{\lambda^*}}(N_i^+)$ trivially.
Thus, for $\psi_{\alpha,k}(\zeta) \in \Pol(\frakn_+)$, we have 
\begin{equation*}
\C (-\zeta_i\widehat{d\pi_{\lambda^*}}(N_i^+) \psi_{\alpha,k}(\zeta)) 
\simeq \C \psi_{\alpha,k}(\zeta)
\quad
\text{as $MA$-modules}.
\end{equation*}
Now \eqref{eqn:Pol_alpha_MjAj} completes the proof.
\end{proof}

For later use, we next compute
$-\zeta_i\widehat{d\pi_{\lambda^*}}(N_i^+)$.
To do so, we start by computing 
$d\pi_{\lambda^*}(N_i^+)$ for $i=1,2,3$ 
for $I(\xi,\lambda)$.

\begin{proposition}
\label{prop:dpi}
We have
\begin{align}
\label{eq:dpimu_N1}
d\pi_{\lambda^*}(N_1^+)&=x_1\Big(\lambda_2-\lambda_1+2+x_1\frac{\partial}{\partial x_1}\Big)+(x_3-\tfrac{1}{2}x_1x_2)\frac{\partial}{\partial x_2}+\frac{x_1}{2}(x_3+\tfrac{1}{2}x_1x_2)\frac{\partial}{\partial x_3},\\[3pt]
\label{eq:dpimu_N2}
d\pi_{\lambda^*}(N_2^+)&=x_2\Big(\lambda_3-\lambda_2+2+x_2\frac{\partial}{\partial x_2}\Big)-(x_3+\tfrac{1}{2}x_1x_2)\frac{\partial}{\partial x_1}+\frac{x_2}{2}(x_3-\tfrac{1}{2}x_1x_2)\frac{\partial}{\partial x_3},\\[3pt]
d\pi_{\lambda^*}(N_3^+)&=(x_3+\tfrac{1}{2}x_1x_2)\Big(\lambda_2-\lambda_1+2+x_1\frac{\partial}{\partial x_1}\Big) \nonumber\\
&\enspace +(x_3-\tfrac{1}{2}x_1x_2)\Big(\lambda_3-\lambda_2+2+x_2\frac{\partial}{\partial x_2}\Big)
+(x_3^2+\tfrac{1}{4}x_1^2x_2^2)\frac{\partial}{\partial x_3}.\label{eq:dpimu_N3}
\end{align}
\end{proposition}

\begin{proof}
The formulas \eqref{eq:dpimu_N1} and \eqref{eq:dpimu_N2}
are computed in {\cite[Prop.\ 5.10]{KPV25}}.
Regarding the third formula \eqref{eq:dpimu_N3}, 
since $N_3^+=[N_1^+, N_2^+]$, the operator $d\pi_{\lambda*}(N_3^+)$ can be computed as $d\pi_{\lambda^*}(N_1^+)d\pi_{\lambda*}(N_2^+)-d\pi_{\lambda*}(N_2^+)d\pi_{\lambda*}(N_1^+)$. 
The expression follows now from the first two formulas.
\end{proof}

Write $\vartheta_i=\zeta_i\frac{\partial}{\partial \zeta_i}$ for $i=1,2,3$.

\begin{proposition}
\label{prop:Fdpi}
We have
\begin{align*}
-\zeta_1\widehat{d\pi_{\lambda^*}}(N_1^+) &=(\lambda_1-\lambda_2)\vartheta_1+(\vartheta_1^2-\vartheta_1)-\frac{1}{2}\vartheta_1\vartheta_2+\frac{1}{2}\vartheta_1\vartheta_3+\frac{\zeta_1\zeta_2}{\zeta_3}\vartheta_3+\frac{1}{4}\frac{\zeta_3}{\zeta_1\zeta_2}(\vartheta_1^2-\vartheta_1)\vartheta_2,\\[3pt]
-\zeta_2\widehat{d\pi_{\lambda^*}}(N_2^+)&=(\lambda_2-\lambda_3)\vartheta_2+(\vartheta_2^2-\vartheta_2)-\frac{1}{2}\vartheta_1\vartheta_2+\frac{1}{2}\vartheta_2\vartheta_3-\frac{\zeta_1\zeta_2}{\zeta_3}\vartheta_3-\frac{1}{4}\frac{\zeta_3}{\zeta_1\zeta_2}\vartheta_1(\vartheta_2^2-\vartheta_2),\\[3pt]
-\zeta_3\widehat{d\pi_{\lambda^*}}(N_3^+)&=(\lambda_1-\lambda_3-1+\vartheta_1+\vartheta_2+\vartheta_3)\vartheta_3+\frac{1}{2}\frac{\zeta_3}{\zeta_1\zeta_2}(\lambda_1-2\lambda_2+\lambda_3+\vartheta_1-\vartheta_2)\vartheta_1\vartheta_2\\
&\enspace +\frac{1}{4}\Big(\frac{\zeta_3}{\zeta_1\zeta_2}\Big)^2(\vartheta_1^2-\vartheta_1)(\vartheta_2^2-\vartheta_2).
\end{align*}
\end{proposition}

\begin{proof}
The first two formulas are computed in {\cite[Props.\ 5.11]{KPV25}}.
The third one follows from a direct application of the algebraic Fourier transform
to  $d\pi_{\lambda*}(N_3^+)$ in \eqref{eq:dpimu_N3}.
\end{proof}

\begin{remark}
The assertion in Proposition \ref{prop:frakn-mod} can also be shown  
directly from the explicit formulas of $-\zeta_i\widehat{d\pi_{\lambda^*}}(N_i^+)$
in Proposition \ref{prop:Fdpi}.
\end{remark}

Observe that the differential equation
$\widehat{d\pi_{\lambda^*}}(N_i^+)\psi = 0$ in 
\eqref{eqn:PDE_SymmetryBreaking}
is equivalent to
\begin{equation}\label{PDE_SymmetryBreaking_123}
-\zeta_i\widehat{d\pi_{\lambda^*}}(N_i^+)\psi = 0.
\end{equation}
Thus, the space $\Sol_i(\xi,\eta; \lambda,\nu)$ of the solutions to 
$\widehat{d\pi_{\lambda^*}}(N_i^+)\psi = 0$ can be given as follows.
\begin{align}
&\Sol_i(\xi,\eta; \lambda,\nu) \nonumber\\
&\hspace{30pt}= \{ \psi \in 
\Hom_{M_i'A_i'}\big((\CC_\xi\boxtimes \CC_{\lambda})\big\rvert_{M_i'A_i'},\Pol(\frakn_+)\otimes (\CC_\eta\boxtimes \CC_{\nu})\big) : 
\text{\eqref{PDE_SymmetryBreaking_123} holds}\}.
\end{align}
Since $-\zeta_i\widehat{d\pi_{\lambda^*}}(N_i^+) \in \End_\C(\Pol(\alpha))$
for isotypic components $\Pol(\alpha)$, the differential equation \eqref{PDE_SymmetryBreaking_123} is a modified F-system (see Definition \ref{def:modified-Fsystem}).
We then consider the ordered F-system to \eqref{PDE_SymmetryBreaking_123}.

We put
\begin{align}
\Sol_{\Ad(e^{J(\alpha)}),\, i}^{(\alpha)}(\xi,\lambda)
&:=\{\varphi(\zeta) \in \Pol(\alpha) : 
\Ad(e^{J(\alpha)})(-\zeta_i\widehat{d\pi_{\lambda^*}}(N_i^+))\varphi(\zeta)=0\},
\label{eqn:SolalphaAd}\\
\Sol_{\Ad(e^{J}),\, i}(\xi,\eta; \lambda,\nu) 
&:= \bigoplus_{\alpha \in \NN^2}\Sol_{\Ad(e^{J(\alpha)}),\, i}^{(\alpha)}(\xi,\lambda).
\end{align}

The following theorem is a specialization of Theorem \ref{thm:symb_ord} 
in the case under consideration.

\begin{theorem}
The following commutative diagram holds for $i=1,2,3$.
\begin{equation}\label{eqn:conji}
\begin{tikzcd}[column sep=1.5cm]
	{\Sol_{\Ad(e^{J}),\, i}(\xi,\eta; \lambda,\nu)}  	
	&{\Diff_{G_i^\prime}\left(I(\xi, \lambda), J(\eta, \nu)\right)} \\
	 {\Sol_i(\xi,\eta; \lambda,\nu)} &
	\arrow[ul, pos=0.886, phantom, "\circlearrowleft"]
	\arrow["\Symb_{\ord}^{-1}", "\sim"' sloped,  from=1-1, to=1-2]
	\arrow["\sim" sloped, "\Symb_0^{-1}"', from=2-1, to=1-2]
	\arrow["\sim" sloped, "e^{-J(\alpha)}"',  from=1-1, to=2-1]
\end{tikzcd}
\end{equation}
Here, for simplicity, the restriction map $\Rest$ is omitted 
for $\Symb_{\ord}^{-1}$ and $\Symb_0^{-1}$.
\end{theorem}

\begin{proof}
It follows from Theorem \ref{thm:symb} and \eqref{eqn:cdA} that
for $\alpha \in \NN^2$, we have
\begin{equation}\label{eqn:conji2}
\begin{tikzcd}[column sep=1.5cm]
	{ \Sol_{\Ad(e^{J(\alpha)}),\,i}^{(\alpha)}(\xi,\lambda)}  	
	&{\Diff_{G_i^\prime}\left(I(\xi, \lambda), J(\eta, \nu)\right)} \\
	 {\Sol_i^{(\alpha)}(\xi,\lambda)} &
	\arrow[ul, pos=0.886, phantom, "\circlearrowleft"]
	\arrow["\Symb_{\ord}^{-1}",  from=1-1, to=1-2]
	\arrow[ "\Symb_0^{-1}"', from=2-1, to=1-2]
	\arrow["\sim" sloped, "e^{-J(\alpha)}"',  from=1-1, to=2-1]
\end{tikzcd}
\end{equation}
Now the theorem follows from \eqref{eqn:SolDecomp}.
\end{proof}

%%%%%%%%%%%%%%%%%%%%%%%%%%%%%%%%%%%%%%%%%%
\subsection{Second part of the recipe using the T-saturation}
\label{sec:recipe2}

With \eqref{eqn:cdA} and \eqref{eqn:conji2} in mind,
we shall proceed with following steps to carry out the second part of the 
recipe for the F-method in Section \ref{sec:recipe2}
to classify and construct DSBOs 
$\DD \in \Diff_{G_i^\prime}(I(\xi,\lambda),J(\eta,\nu))$ 
with respect to the embedding $\iota_i \colon G^\prime \hookrightarrow G$ for $i=1,2,3$.
\vspace{3pt}

\begin{enumerate}

\item[Step 1]
Classify and construct
$\psi \in \Hom_{M'_iA'_i}((\CC_\xi\boxtimes \CC_{\lambda})\big\rvert_{M_i'A_i'}, 
\Pol(\mathfrak{n}_+)\otimes (\CC_{\eta}\boxtimes\CC_{\nu}))$.
\vskip 0.1in

\item[Step 2]
Compute the T-saturation
$T_\alpha^\sharp \big(-\zeta_i\widehat{d\pi_{\lambda^*}}(N_i^+)\big)$. 
\vskip 0.1in

\item[Step 3]
Compute
$\wSol_{\Ad(e^{\wJ(\alpha)})T_\alpha^\sharp A}^{(\alpha)}$
and
$\wSol_{T_\alpha^\sharp A}^{(\alpha)}$
in  \eqref{eqn:cdA} for 
$A=-\zeta_i\widehat{d\pi_{(\xi,\lambda)^*}}(N_i^+)$ 
to determine
$\Sol_{i}(\xi,\eta; \lambda,\nu)$
and 
$\Sol_{\Ad(e^{J}),\, i}(\xi,\eta; \lambda,\nu)$.

\vskip 0.1in

\item[Step 4]
Do the following.
\vspace{3pt}

\begin{enumerate}
\item 
Apply $\Rest \circ \Symb_0^{-1}$
to $\psi \in \Sol_{i}(\xi,\eta; \lambda,\nu)$.
\vspace{5pt}

\item
Apply $\Rest \circ\Symb_{\ord}^{-1}$
to  $\varphi \in \Sol_{\Ad(e^{J(\alpha)}),\, i}(\xi,\eta; \lambda,\nu)$.

\end{enumerate}

\end{enumerate}

%%%%%%%%%%%%%%%%%%%%%%%%%%%%%%%%%%%%%%%%%
\section{DSBOs \texorpdfstring{$\DD_3(\lambda,\nu)$}{D₃(λ,ν)} for \texorpdfstring{$(GL(3,\R), G_3')$}{(GL(3,R),G₃')}}
\label{sec:Emb3}

The aim of this short section is to give the main results about
the classification and construction of DSBOs $\mathbb{D}$ with
respect to the third embedding 
$\iota_3 \colon GL(2,\R) \hookrightarrow GL(3,\R)$ 
in \eqref{eqn:embedding2}. The proofs will be discussed in Sections 
\ref{sec:proof1} and \ref{sec:proof2}.
The other two embeddings $\iota_i$ for $i=1,2$ are 
studied in Section \ref{sec:Emb12}.

%%%%%%%%%%%%%%%%%%%%%%%%%%%%%%%%%%%%%%%%%
\subsection{Classification of DSBOs \texorpdfstring{$\DD$}{ } for \texorpdfstring{$(GL(3,\R), G'_3)$}{(GL(3,R),G₃')}}
\label{sec:ClassThmsDSBOs}

We start with the classification of DSBOs $\DD$.
Let $\Phi_{3}$ be the subset of $\C^3\times \C^2 \simeq \C^5$ defined by
\begin{equation}\label{eqn:Phi3}
\Phi_{3} := \{(\lambda, \nu) \in \C^5: \nu_1-\lambda_1,\, \lambda_3-\nu_2 \in \N\}.\\
\end{equation}
We put
\begin{align*}
    \Phi_{3,a} &:= \{(\lambda, \nu) \in \Phi_3: \eqref{eq:condition_set_3a} \text{ holds}\},\\[3pt]
    \Phi_{3,b} &:= \{(\lambda, \nu) \in \Phi_3: \eqref{eq:condition_set_3b} \text{ holds}\},\\[3pt]
    \Phi_{3,c} &:= \{(\lambda, \nu) \in \Phi_3: \eqref{eq:condition_set_3c} \text{ holds}\},
\end{align*}
\begin{equation}\label{eq:condition_set_3a}
    \nu_1-\nu_2-1 \notin [1, \min(\nu_1-\lambda_1, \lambda_3-\nu_2)]\cap \Z,
\end{equation}
\begin{alignat}{2}
    \nu_1-\nu_2-1 &\in  [1, \min(\nu_1-\lambda_1, \lambda_3-\nu_2)]\cap \Z
     \quad \text{and} \quad
     \lambda_2 -\nu_2 &\notin [1, \nu_1-\nu_2-1]\cap \Z, 
     \label{eq:condition_set_3b}\\[3pt]
    \nu_1-\nu_2-1 &\in  [1, \min(\nu_1-\lambda_1, \lambda_3-\nu_2)]\cap \Z
     \quad \text{and} \quad
     \lambda_2 -\nu_2 &\in [1, \nu_1-\nu_2-1]\cap \Z.
          \label{eq:condition_set_3c}
\end{alignat}
We have 
\begin{equation*}
\Phi_{3} = \Phi_{3,a}\cup\Phi_{3,b}\cup\Phi_{3,c}.
\end{equation*}

\begin{remark}
Since $[a,b]$ for $b < a$ is regarded as $[a,b]=\emptyset$, 
the parameters 
$(\lambda,\nu)$ with $\min(\nu_1-\lambda_1,\lambda_3-\nu_2) =0$ are in $\Phi_{3,a}$. 
\end{remark}

Now, for $\omega \in \{a, b, c\}$ we define 
$\Supp_{3,\omega}(\DD)$ and $\Supp_3(\DD)$ 
as follows:
\begin{align*}
\Supp_{3,\omega}(\DD)&:= \{(\xi, \eta; \lambda, \nu) \in (\Z/2\Z)^5 \times \Phi_{3,\omega}: \eqref{eq:condition_full_set_3} \text{ holds}\},
\end{align*}
\begin{equation}
\label{eq:condition_full_set_3}
\xi_1+\eta_1 \equiv \nu_1-\lambda_1,\,\,  \xi_3+\eta_2 \equiv \lambda_3-\nu_2 \bmod{2},
\end{equation}
\begin{equation*}
    \Supp_3(\DD) := \Supp_{3,a}(\DD)\cup \Supp_{3,b}(\DD)\cup \Supp_{3,c}(\DD).
\end{equation*}

\begin{theorem}
\label{thm:class3}
The following conditions on the tuple $(\xi, \eta; \lambda, \nu) \in (\Z/2\Z)^5 \times \C^5$ are equivalent.
\begin{enumerate}[label = \normalfont{(\roman*)}]
    \item $\Diff_{G_3^\prime}\left(I(\xi, \lambda), J(\eta, \nu)\right) \neq \{0\}$.
    
    \item $\dim_\C \Diff_{G_3^\prime}\left(I(\xi, \lambda), J(\eta, \nu)\right) \in \{1, 2\}$.

    \item $(\xi, \eta; \lambda, \nu) \in \Supp_3(\DD)$.
\end{enumerate}
Moreover, the dimension of the space above is two if and only 
if $(\xi, \eta; \lambda, \nu) \in \Supp_{3,c}(\DD)$.
\end{theorem}

%%%%%%%%%%%%%%%%%%%%%%%%%%%%%%%%%%%%%%%%%
\subsection{Construction of DSBOs \texorpdfstring{$\DD$}{ } for \texorpdfstring{$(GL(3,\R), G'_3)$}{(GL(3,R),G₃')}} 
\label{sec:const3}

We next show the explicit formulas of DSBOs $\DD$ in both symmetrized and ordered forms.
Write
\begin{equation*}
D_1:=dR(N_1^-)=\frac{\partial}{\partial x_1}+\frac{x_2}{2}\frac{\partial}{\partial x_3},\quad D_2:=dR(N_2^-)=\frac{\partial}{\partial x_2}-\frac{x_1}{2}\frac{\partial}{\partial x_3},\quad D_3=:dR(N_3^{-})=\frac{\partial}{\partial x_3}.
\end{equation*}

By slight abuse of notation, we write 
$\ord=\{D_1, D_2, D_3\}$.
Then, for $\alpha=(\alpha_1,\alpha_2) \in \NN^2$
and 
$q(t)=\sum_{n=0}^{\min(\alpha)}c_n t^n \in \Pol_{\min(\alpha)}[t]$, 
we define the \emph{symmetrized operator} 
$(T^{\mathrm{sym}}_{(\alpha_1, \alpha_2)} q)(D_1, D_2, D_3)$
and 
\emph{ordered operator}
$(T^\ord_{(\alpha_1, \alpha_2)} q)(D_1, D_2, D_3)$ 
with respect to $\ord$
as follows.
\begin{align}
(T^{\mathrm{sym}}_{(\alpha_1, \alpha_2)} q)(D_1, D_2, D_3)
&:=\sum_{n=0}^{\min(\alpha)}c_n\, \sym_{\alpha_1+\alpha_2-n}(D_1^{\alpha_1-n}D_2^{\alpha_2-n}D_3^n),
\label{eqn:symop}\\[3pt]
(T^\ord_{(\alpha_1, \alpha_2)} q)(D_1, D_2, D_3)
&:=\sum_{n=0}^{\min(\alpha)}c_n\, D_1^{\alpha_1-n}D_2^{\alpha_2-n}D_3^n.
\label{eqn:nonsymop}
\end{align}
Furthermore, define a polynomial $p^{(\alpha_1, \alpha_2)}_{a,b}(t) \in \Pol_{\min(\alpha)}[t]$ as 
\begin{equation}\label{eqn:pab}
p^{(\alpha_1, \alpha_2)}_{a,b}(t)
:=\sum_{n=0}^{\min(\alpha)}\frac{(-\alpha_1)_n(-\alpha_2)_n}{2^n n!(-b)_n}
\Cay_n(a,b)\,t^n,
\end{equation}
where $\Cay_n(a,b)$ denotes a certain tridiagonal determinant called 
the Cayley continuant (see Appendix \ref{appendix:Cayely} for some details).

For 
$\upsilon \in \{\sym, \ord\}$ and 
the hypergeometric polynomial
${}_pF_q[\pmb{\gamma};\pmb{\beta};t]$ with appropriate parameters $\pmb{\gamma}$ and $\pmb{\beta}$, 
we also write
\begin{equation}\label{eqn:Tupsilon}
T^\upsilon_{(\alpha_1,\alpha_2)}{}_pF_q[\pmb{\gamma};\pmb{\beta};D_1, D_2, D_3]
:=\big(T^\upsilon_{(\alpha_1,\alpha_2)}{}_pF_q[\pmb{\gamma};\pmb{\beta};t]\big)(D_1, D_2, D_3).
\end{equation}
The symmetrized form
$T^{\sym}_{(\alpha_1,\alpha_2)}{}_pF_q[\pmb{\gamma};\pmb{\beta};D_1, D_2, D_3]$
of hypergeometric polynomial ${}_pF_q[\pmb{\gamma};\pmb{\beta};t]$
will appear only in Theorem \ref{thm:cons_DIOs}.

Let $\Rest_3:=\Rest_{x_1=x_2=0}$ denote
the restriction map with respect to the coordinates
\eqref{eqn:coordinate} and  \eqref{eqn:coordinate3}.

\begin{theorem}\label{thm:cons3}
We have
\begin{equation*}
\Diff_{G_3^\prime}\left(I(\xi, \lambda), J(\eta, \nu)\right)
=
\begin{cases}
\CC\, \DD^{(1)}_3(\lambda, \nu) & 
\textnormal{if $(\xi, \eta; \lambda, \nu) \in \Supp_{3,a}(\DD)$},\\[3pt]
\CC\, \DD^{(2)}_3(\lambda, \nu) & 
\textnormal{if $(\xi, \eta; \lambda, \nu) \in \Supp_{3,b}(\DD)$},\\[3pt]
\CC\, \DD^{(1)}_3(\lambda, \nu) 
\oplus
\CC\, \DD^{(2)}_3(\lambda, \nu)
& \textnormal{if $(\xi, \eta; \lambda, \nu) \in \Supp_{3,c}(\DD)$},\\[3pt]
\{0\} &\textnormal{otherwise}.
\end{cases}
\end{equation*}
Here $\DD^{(j)}_3(\lambda, \nu)$ for $j=1,2$ 
are given as follows.
\vspace{5pt}
\begin{align*}
\DD^{(1)}_3(\lambda,\nu)
&=\Rest_3\circ
\big(T^{\mathrm{sym}}_{(\nu_1-\lambda_1, \,\lambda_3-\nu_2)} 
p^{(\nu_1-\lambda_1, \,\lambda_3-\nu_2)}_{\nu_1+\nu_2-2\lambda_2, \,\nu_1-\nu_2-2}\big)
(D_1, D_2, D_3)\\[5pt]
&=\Rest_3\circ 
T^\ord_{(\nu_1-\lambda_1,\,\lambda_3-\nu_2)}{}_3F_1
\left [\begin{matrix}
    \lambda_1-\nu_1 &\nu_2-\lambda_3 & \nu_2-\lambda_2+1\\
    & \nu_2-\nu_1+2 & 
\end{matrix}; D_1, D_2, D_3 \right],\\[10pt]
\DD^{(2)}_3(\lambda,\nu)
&=\Rest_3\circ  D_3^{\nu_1-\nu_2-1}\big(T^{\mathrm{sym}}_{(\nu_2-\lambda_1+1,\,\lambda_3-\nu_1+1)} p^{(\nu_2-\lambda_1+1, \,\lambda_3-\nu_1+1)}_{\nu_1+\nu_2-2\lambda_2, \,\nu_2-\nu_1}\big)(D_1, D_2, D_3) \\[5pt]
&=\Rest_3\circ (\star),
\end{align*}
where
\begin{equation*}
(\star):=
 D_3^{\nu_1-\nu_2-1}
T^\ord_{(\nu_2-\lambda_1+1,\,\lambda_3-\nu_1+1)}{}_3F_1\left [\begin{matrix}
    \lambda_1-\nu_2-1 & \nu_1-\lambda_3-1 & \nu_1-\lambda_2\\
    & \nu_1-\nu_2& 
\end{matrix}; D_1, D_2, D_3\right].
\end{equation*}
\end{theorem}

As differential operators,
the orders 
of the DSBOs $\DD^{(j)}_3(\lambda,\nu)$ are
$\nu_1-\lambda_1+\lambda_3 - \nu_2$ for $j=1$
and 
$\lambda_3-\lambda_1+1$ for $j=2$.

We shall proceed with the modified recipe in Section \ref{sec:recipe2}
to prove Theorems \ref{thm:class3} and \ref{thm:cons3}.

%%%%%%%%%%%%%%%%%%%%%%%%%%%%%%%%%%%%%%%%%
\section{Proofs of  Theorems \texorpdfstring{\ref{thm:class3}}{6.7} and \texorpdfstring{\ref{thm:cons3}}{6.12}: Steps 1 and 2}
\label{sec:proof1}

The aim of the next two sections is to prove Theorems \ref{thm:class3} and \ref{thm:cons3},
that is, to determine 
\begin{align*}
\Sol_3(\xi,\eta; \lambda,\nu) := \{ \psi \in \Hom_{M_3'A_3'}\big((\CC_\xi\boxtimes \CC_{\lambda})\big\rvert_{M_3'A_3'},\Pol(\frakn_+)\otimes (\CC_\eta\boxtimes \CC_{\nu})\big) : \eqref{PDE_SymmetryBreaking_3} \text{ holds}\},
\end{align*}
\begin{equation}\label{PDE_SymmetryBreaking_3}
\widehat{d\pi_{\lambda^*}}(N_{3}^+)\psi = 0,
\end{equation}
for $(\xi,\eta; \lambda,\nu) \in (\ZZ/2\ZZ)^5\times \CC^5$.

We follow the recipe of the F-method in Section \ref{sec:recipe2}.
In this section, we handle the former two steps, namely Steps 1 and 2;
the latter two steps will be considered in Section \ref{sec:proof2}.

%%%%%%%%%%%%%%%%%%%%%%%%%%%%%%%%%%%%%%%%%
\subsection{Step 1: Characterization of \texorpdfstring{$M_3'A_3'$}{M₃'A₃'}-invariant elements}

We start by characterizing the space 
\begin{align*}
&\Hom_{M_3'A_3'}\big((\CC_\xi\boxtimes \CC_{\lambda})\big\rvert_{M_3'A_3'},\Pol(\frakn_+)\otimes (\CC_\eta\boxtimes \CC_{\nu})\big)
\end{align*}
of $M_3'A_3'$-invariant elements. 
For $\alpha=(\alpha_1,\alpha_2) \in \N^2$, put
\begin{equation*}
L_3(\alpha):=\{(\xi, \eta; \lambda, \nu) \in (\ZZ/2\ZZ)^5\times \CC^5: 
\text{\eqref{eqn:ch31} and \eqref{eqn:ch32} hold}  \},
\end{equation*}
\begin{alignat}{3}
\xi_1+\eta_1&\equiv \alpha_1 \bmod{2}, \quad &&\xi_3+\eta_2&&\equiv\alpha_2 \bmod{2},
\label{eqn:ch31}\\[3pt]
\nu_1-\lambda_1&=\alpha_1, \quad &&\lambda_3-\nu_2&&=\alpha_2.
\label{eqn:ch32}
\end{alignat}

\begin{proposition}\label{prop:Invariants_SB3} 
The following conditions on $(\xi, \eta; \lambda, \nu)\in (\ZZ/2\ZZ)^5\times \CC^5$ are equivalent.
\begin{enumerate}[label=\normalfont{(\roman*)}]
    \item $\Hom_{M_3'A_3'}\big((\CC_\xi\boxtimes \CC_{\lambda})\big\rvert_{M_3'A_3'},\Pol(\frakn_+)\otimes (\CC_\eta\boxtimes \CC_{\nu})\big)\neq \{0\}$.
    \item There exists some $\alpha\in \NN^2$ such that 
    $(\xi, \eta; \lambda, \nu) \in L_3(\alpha)$.
\end{enumerate}
\end{proposition}

\begin{proof}
It follows from Lemma \ref{lem:MApol} that
\begin{equation*}
\Pol(\alpha) \simeq \bigoplus_{k=0}^{\min(\alpha)} 
\C_{(\alpha_1,\, \alpha_2)} 
\boxtimes \C_{(-\alpha_1, \alpha_2)} \qquad (\text{as $M_3'A'_3$-modules}).
\end{equation*}
Since 
$(\CC_\xi\boxtimes \CC_{\lambda})\big\rvert_{M_3'A_3'}
=\CC_{(\xi_1,\xi_3)}\boxtimes \CC_{(\lambda_1,\lambda_3)}$,
the decomposition \eqref{eqn:Homalpha} then concludes the proposition.
\end{proof}

%%%%%%%%%%%%%%%%%%%%%%%%%%%%%%%%%%%%%%%%%
\subsection{Step 2: 
Compute the T-saturation
\texorpdfstring{$T_\alpha^\sharp \big(-\zeta_3\widehat{d\pi_{\lambda^*}}(N_3^+)\big)$}{ }}
\label{sec:Step2}

We next reduce the PDE $-\zeta_3\widehat{d\pi_{\lambda^*}}(N_3^+)\psi(\zeta) = 0$ to an ODE via the saturation map $T_\alpha$ in \eqref{eqn:Tsat}
for $-\zeta_3\widehat{d\pi_{\lambda^*}}(N_3^+)$ in Proposition \ref{prop:Fdpi}.
The following formulas are useful to compute 
the pullbacks 
\begin{equation*}
T_\alpha^\sharp D:=
T_\alpha \circ D \circ T_\alpha^{-1}
\in \End_\C(\Pol_{\min(\alpha)}[t])
\end{equation*}
of differential operators $D$ on $\Pol(\alpha)$.
We write $\vartheta_t=t\frac{d}{dt}$ for the Euler operator in $t$. 

\begin{lemma}\label{lem:T-saturation}
For any $\alpha\in \NN^2$, we have
\begin{enumerate}[label=\normalfont{(\arabic*)}]
    \item $\vartheta_1T_\alpha=T_\alpha \circ (\alpha_1-\vartheta_t)$,
    \vspace{3pt}
    \item $\vartheta_2T_\alpha=T_\alpha \circ (\alpha_2-\vartheta_t)$,
    \vspace{3pt}
    \item $\vartheta_3T_\alpha=T_\alpha \circ\vartheta_t$,
    \vspace{3pt}
    \item $\frac{\zeta_3}{\zeta_1\zeta_2}T_\alpha p(\zeta)=T_\alpha (tp)(\zeta)$,
    \vspace{3pt}
    \item $\frac{\zeta_1\zeta_2}{\zeta_3}T_\alpha p(\zeta)=T_\alpha(t^{-1}p)(\zeta)$.    
\end{enumerate}
\end{lemma}
\begin{proof}
This follows directly from the definition of $T_\alpha$. The first identity can be computed in the following way:
\begin{align*}
\vartheta_1 (T_\alpha p)(\zeta_1,\zeta_2,\zeta_3)&=\zeta_1\frac{\partial}{\partial\zeta_1}\Big( \zeta_1^{\alpha_1}\zeta_2^{\alpha_2}p\Big(\frac{\zeta_3}{\zeta_1\zeta_2}\Big)\Big)\\
&=\alpha_1\zeta_1^{\alpha_1}\zeta_2^{\alpha_2}p\Big(\frac{\zeta_3}{\zeta_1\zeta_2}\Big)-\zeta_1^{\alpha_1}\zeta_2^{\alpha_2}\frac{\zeta_3}{\zeta_1\zeta_2}p'\Big(\frac{\zeta_3}{\zeta_1\zeta_2}\Big)\\
&=T_\alpha\big((\alpha_1-\vartheta_t)p\big)(\zeta_1,\zeta_2,\zeta_3).
\end{align*}
The rest can be shown similarly.
\end{proof}

For $\alpha=(\alpha_1,\alpha_2)\in \NN^2$ and $a,b\in \CC$, 
let $\calL^{(\alpha)}(a,b)$ be the following differential operator:
\begin{align}
\calL^{(\alpha)}(a,b) 
:= \; &\vartheta_t(b+1-\vartheta_t)+\frac{a}{2}t(\vartheta_t-\alpha_1)
(\vartheta_t-\alpha_2) \nonumber \\
\enspace &+\frac{1}{4}t^2(\vartheta_t-\alpha_1+1)
(\vartheta_t-\alpha_1)(\vartheta_t-\alpha_2+1)(\vartheta_t-\alpha_2).\label{eqn:Lalpha}\end{align}
We put
\begin{equation}\label{eqn:ab}
\begin{aligned}
a_0&=a_0(\lambda;\alpha_1,\alpha_2):= \lambda_1+\lambda_3-2\lambda_2+\alpha_1-\alpha_2, \\[3pt]
b_0&=b_0(\lambda;\alpha_1,\alpha_2):= \lambda_1-\lambda_3-2+\alpha_1+\alpha_2.
\end{aligned}
\end{equation} 

\begin{proposition}\label{prop:ODEdiffoperator3}
We have
\begin{equation*}
T_\alpha^\sharp \big(-\zeta_3\widehat{d\pi_{\lambda^*}}(N_3^+)\big)
=\calL^{(\alpha)}(a_0,b_0\big).
\end{equation*}
\end{proposition}

\begin{proof}
This follows from a direct computation using Lemma \ref{lem:T-saturation}.
\end{proof}

Recall from \eqref{eqn:Solalpha} and \eqref{eqn:SolalphaAd} that we have
\begin{align*}
\Sol_3^{(\alpha)}(\xi,\lambda) &= \{ \psi \in \Hom_{M_3'A_3'}\big((\CC_\xi\boxtimes \CC_{\lambda})\big\rvert_{M_3'A_3'},\Pol(\alpha)\otimes (\CC_\eta\boxtimes \CC_{\nu})\big) : \eqref{PDE_SymmetryBreaking_3} \text{ holds}\},\\
\Sol_{\Ad(e^{J(\alpha)}),\,3}^{(\alpha)}(\xi,\lambda)
&=\{\varphi(\zeta) \in \Pol(\alpha) : 
\Ad(e^{J(\alpha)})(-\zeta_3\widehat{d\pi_{\lambda^*}}(N_3^+))\varphi(\zeta)=0\}.
\end{align*}
We then put
\begin{align*}
\wSol_\calL^{(\alpha)}(\xi,\lambda) &:= \{p(t) \in \Pol_{\min(\alpha)}[t] : 
\calL^{(\alpha)}(a_0,b_0)p(t) = 0\},\\
\wSol_{\Ad(e^{\wJ(\alpha)}) \calL}^{(\alpha)}(\xi,\lambda)
&:=\{q(t) \in \Pol_{\min(\alpha)}[t] : 
\Ad(e^{\wJ(\alpha)})\calL(a_0,b_0)q(t)=0\}.
\end{align*}
Then the commutative diagram \eqref{eqn:cdA} 
for the case for $A = -\zeta_3\widehat{d\pi_{\lambda^*}}(N_3^+)$
with $(\xi,\eta,\lambda, \nu)\in L_3(\alpha)$
can be given as follows up to the restriction map $\Rest_3$.
\begin{equation}\label{eqn:cdSol3}
\begin{tikzcd}[column sep=1.5cm]
	{\wSol_{\Ad(e^{\wJ(\alpha)}) \calL}^{(\alpha)}(\xi,\lambda)} & 
	\Sol_{\Ad(e^{J(\alpha)}),\,3}^{(\alpha)}(\xi,\lambda)  	
	\arrow[dl, pos=0.5, phantom, "\circlearrowleft"]
	&{\Diff_{G_3^\prime}\left(I(\xi, \lambda), J(\eta, \nu)\right)} \\
	{\wSol_\calL^{(\alpha)}(\xi,\lambda)}  & {\Sol_3^{(\alpha)}(\xi,\lambda)} &
	\arrow[ul, pos=0.886, phantom, "\circlearrowleft"]
	\arrow["\Symb_{\ord}^{-1}", from=1-2, to=1-3]
	\arrow["\Symb_0^{-1}"', from=2-2, to=1-3]
	\arrow["{T_\alpha}", "\sim"' , from=1-1, to=1-2]
	\arrow["\sim" sloped, "e^{-\wJ(\alpha)}"', from=1-1, to=2-1]
	\arrow["e^{-J(\alpha)}", "\sim"' sloped,  from=1-2, to=2-2]
	\arrow["\sim" , "{T_\alpha}"',  from=2-1, to=2-2]
\end{tikzcd}
\end{equation}

In the next section we identify the conjugated differential operator
$\Ad(e^{\wJ(\alpha)})\calL^{(\alpha)}(a_0,b_0)$ as 
a certain hypergeometric differential operator
to classify the space $\wSol_{\Ad(e^{\wJ(\alpha)}) \calL}^{(\alpha)}(\xi,\lambda)$.

%%%%%%%%%%%%%%%%%%%%%%%%%%%%%%%%%%%%%%%%%
\section{Proofs of  Theorems \texorpdfstring{\ref{thm:class3}}{6.7} and \texorpdfstring{\ref{thm:cons3}}{6.12}: Steps 3 and 4}
\label{sec:proof2}

The aim of this section is to carry out Steps 3 and 4 of the recipe in Section \ref{sec:recipe2}. 
We handle Step 3 in Section \ref{sec:Step3}. The proof will then be completed in 
Section \ref{sec:Step4}. 

%%%%%%%%%%%%%%%%%%%%%%%%%%%%%%%%%%%%%%%%%
\subsection{Step 3: Solve the differential equation}
\label{sec:Step3}

Our primary goal is to solve the differential equation
\begin{equation*}
\Ad(e^{\wJ(\alpha)})\calL^{(\alpha)}(a_0,b_0)q(t)=0.
\end{equation*}
To do so, we first identify it as a differential equation of hypergeometric type.
Recall from \eqref{eqn:Da} that 
we have 
\begin{equation*}
\wJ(\alpha)=\frac{t}{2}(\vartheta_t-\alpha_1)(\vartheta_t-\alpha_2).
\end{equation*}
It follows from \eqref{eqn:D1} that,
for $k,\ell\in \N$, we have
\begin{equation*}
\wJ(\alpha)^\ell t^k 
=\left(\frac{t}{2}\right)^\ell(k-\alpha_1)_\ell
(k-\alpha_2)_\ell t^k.
\end{equation*}
Then, by \eqref{eqn:Lalpha}, the operator 
$\calL^{(\alpha)}(a,b)$ on $\Pol_{\min(\alpha)}[t]$ can be given by
\begin{equation}\label{eqn:Lalpha2}
\calL^{(\alpha)}(a,b) = 
\vartheta_t(b+1-\vartheta_t)+a\wJ(\alpha)+\wJ(\alpha)^2.
\end{equation}
Hereafter, we always consider $\calL^{(\alpha)}(a,b)$ on $\Pol_{\min(\alpha)}[t]$
so that \eqref{eqn:Lalpha2} holds.

\medskip

The following lemma plays a key role.

\begin{lemma}\label{lem:eD}
We have 
\begin{equation*}
\Ad(e^{\wJ(\alpha)})\vartheta_t 
=\vartheta_t-\wJ(\alpha).
\end{equation*}
\end{lemma}

\begin{proof}
A direct computation shows that
\begin{equation*}
\vartheta_t \wJ(\alpha)^k = \wJ(\alpha)^k(\vartheta_t + k).
\end{equation*}
Therefore,
\begin{align}
\vartheta_t e^{-\wJ(\alpha)} &= \sum_{k=0}^\infty \frac{(-\wJ(\alpha))^k}{k!} (\vartheta_t + k) \nonumber \\[3pt]
&= e^{-\wJ(\alpha)}(\vartheta_t -\wJ(\alpha)). \label{eqn:ewD}
\end{align}
Now apply $e^{\wJ(\alpha)}$ to both
sides of \eqref{eqn:ewD} to obtain the identity.
\end{proof}

For $\pmb{\beta}:=\beta_1,\ldots, \beta_q$, 
$\pmb{\gamma}:=\gamma_1,\ldots, \gamma_p \in \CC$,
put
\begin{equation*}
{}_p\calF_q(\pmb{\gamma};\pmb{\beta};t):=
\vartheta_t(\vartheta_t+\beta_{1}-1)\cdots(\vartheta_t+\beta_{q}-1)-
t(\vartheta_t+\gamma_{1})%
\cdots(\vartheta_t+\gamma_{p})
\end{equation*}
so that 
\begin{equation*}
{}_p\calF_q(\pmb{\gamma};\pmb{\beta};t) f(t) =0
\end{equation*}
is the hypergeometric equation that $f(t) = {}_pF_q[\pmb{\gamma};\pmb{\beta};t]$
satisfies. (See, for instance, \cite[(16.8.3)]{DLMF}.)

\begin{proposition}\label{prop:AdL}
We have
\begin{equation*}
-\Ad(e^{\wJ(\alpha)})\calL^{(\alpha)}(a,b)
={}_3\calF_1(-\alpha_1,-\alpha_2,\tfrac{a-b}{2};-b;t),
\end{equation*}
that is,
\begin{equation*}
-\Ad(e^{\wJ(\alpha)})\calL^{(\alpha)}(a,b)
=
\vartheta_t(\vartheta_t-b-1)
-t(\vartheta_t-\alpha_1)(\vartheta_t-\alpha_2)(\vartheta_t+\tfrac{a-b}{2}).
\end{equation*}
\end{proposition}

\begin{proof}
It follows from Lemma \ref{lem:eD} that 
\begin{equation*}
\Ad(e^{\wJ(\alpha)})\vartheta_t = \vartheta_t - \wJ(\alpha).
\end{equation*}
Further, we have $\Ad(e^{\wJ(\alpha)})\wJ(\alpha)^k=\wJ(\alpha)^k$ for all $k \in \N$.
Now, by \eqref{eqn:Lalpha2}, the desired equation follows from a direct computation.
\end{proof}

It follows from Proposition \ref{prop:AdL} that 
$\wSol_{\Ad(e^{\wJ(\alpha)}) \calL}^{(\alpha)}(\xi,\lambda)$ can be described as
\begin{equation}\label{eqn:AdSol3F1}
\wSol_{\Ad(e^{\wJ(\alpha)}) \calL}^{(\alpha)}(\xi,\lambda)
=\{q(t) \in \Pol_{\min(\alpha)}[t] : 
{}_3\calF_1(-\alpha_1,-\alpha_2,\tfrac{a_0-b_0}{2};-b_0;t)q(t)=0\}
\end{equation}
with $(a_0, b_0)$ in \eqref{eqn:ab}.

We are now ready to classify 
$\wSol_{\Ad(e^{\wJ(\alpha)}) \calL}^{(\alpha)}(\xi,\lambda)$.
Let us first recall a general fact on the polynomial solutions 
to the equation
${}_3\calF_1(-n,-m,c;d;t)q(t)=0$ for $m, n \in \NN$ and $c, d \in\CC$.
Observe that, for $d \notin [1-\min(n,m),0]\cap \Z$, we have
\begin{equation*}
{}_3F_1\left [\begin{matrix}
    -n& -m & c \\
   & d &  
\end{matrix};t\right ]=\sum_{k=0}^{\min(n,m)}\frac{(-n)_k(-m)_k(c)_k}{(d)_kk!}t^k
\in \Pol_{\min(n,m)}[t].
\end{equation*}
We then consider the following 
three conditions on $(n,m,c, d) \in \N^2\times \C^2$:
\vspace{2pt}
\begin{enumerate}[label=\normalfont{(\arabic*)}]
\item[(C1)$'$] $d \notin [1-\min(n,m),0]\cap \Z$;
\vspace{2pt}
\item[(C2)$'$] $d \in [1-\min(n,m),0]\cap \Z$ and 
$c \notin [d, 0] \cap \Z$;
\vspace{2pt}
\item[(C3)$'$] $d\in [1-\min(n,m),0]\cap \Z$ 
and $c\in [d, 0] \cap \Z$.
\end{enumerate}

\begin{remark}\label{rem:zero-case}
If $\min(n,m)=0$, then only (C1)$'$ holds
as  $[a,b]=\emptyset$ for $b< a$.
\end{remark}

Furthermore, to simplify the presentation of the next statement, we write
\begin{align*}
\text{(F1)$'$}
&:=\CC\, 
{}_3F_1\left [\begin{matrix}
    -n& -m & c \\
   & d &  
\end{matrix};t\right ],
\\[3pt]
\text{(F2)$'$}
&:=
\CC\, 
t^{1-d}{}_3F_1\left [\begin{matrix}
    1-d-n & 1-d-m & 1-d+c \\
   & 2-d &  
\end{matrix};t\right ],
\\[3pt]
\text{(F3)$'$}
&:=
\CC\, 
{}_3F_1\left [\begin{matrix}
    -n& -m & c \\
   & d &  
\end{matrix};t\right ]
\oplus
\CC\, 
t^{1-d}{}_3F_1\left [\begin{matrix}
    1-d-n & 1-d-m & 1-d+c \\
   & 2-d &  
\end{matrix};t\right ].
\end{align*}
For $m, n \in \NN$ and $c, d \in\CC$, we put
\begin{equation*}
\Sol^{(n,m)}_{3\calF1}(c,d):=\{q(t) \in \Pol_{\min{(n,m)}}[t]: 
{}_3\calF_1(-n,-m,c;d;t)q(t)=0\}.
\end{equation*}

\begin{fact}[{cf. \cite[Eq. 16.8.6]{DLMF}}]
\label{fact:3F1-sols}
We have
$\Sol^{(n,m)}_{3\calF1}(c,d)\neq \{0\}$
for any $(n,m,c, d) \in \N^2\times \C^2$.
Moreover,
the space $\Sol^{(n,m)}_{3\calF1}(c,d)$
can be classified as follows.
\begin{equation*}
\Sol^{(n,m)}_{3\calF1}(c,d)
=
\begin{cases}
\textnormal{(F1)$'$}
& \textnormal{if (C1)$'$ holds},\\[3pt]
\textnormal{(F2)$'$}
& \textnormal{if (C2)$'$ holds},\\[3pt]
\textnormal{(F3)$'$}
& \textnormal{if (C3)$'$ holds}.
\end{cases}
\end{equation*}
\end{fact}

\begin{remark}
We regard the generalized hypergeometric polynomial
$
{}_3F_1\left [\begin{matrix}
-n& -m & c \\
& -m &  
\end{matrix};t\right ]
$ with $m,n \in \NN$ and $c \in \CC$ as
\begin{equation*}
{}_3F_1\left [\begin{matrix}
-n& -m & c \\
& -m &  
\end{matrix};t\right ]
=
\sum_{k=0}^{\min(n,m)}
\frac{(-n)_k(c)_k}{k!}t^k,
\end{equation*}
so that 
$
{}_3F_1\left [\begin{matrix}
-n& -m & c \\
& -m&  
\end{matrix};t\right ]
$
is a solution to the equation
${}_3\calF_1(-n,-m,c;-m;t)q(t)=0$.
In general, we have
$
{}_3F_1\left [\begin{matrix}
-n& -m & c \\
& -m &  
\end{matrix};t\right ]
\neq
{}_2F_0[-n, c ;t]
$
For instance,
\begin{equation*}
{}_3F_1\left[\begin{matrix}
-2 & -1 & -2\\
& -1 &
\end{matrix};
t\right] = \sum_{k=0}^{1} \frac{(-2)_k(-2)_k}{k!}t^k = 1+4t,
\end{equation*}
whereas
\begin{equation*}
{}_2F_0[-2, -2; t] = 1+4t+2t^2.
\end{equation*}
\end{remark}

\begin{remark}\label{rem:3F1_2F0}
If $d=-m \in -\Z_{\geq 0}$ with $d< -n$ or $d < c$ with $c \in -\N$, then
\begin{equation*}
{}_3F_1\left [\begin{matrix}
-n& d & c \\
& d &  
\end{matrix};t\right ]
=\sum_{k=0}^{n}
\frac{(-n)_k(c)_k}{k!}t^k
={}_2F_0[-n, c ;t].
\end{equation*}
\end{remark}

Now, to classify 
the solution space $\wSol_{\Ad(e^{\wJ(\alpha)}) \calL}^{(\alpha)}(\xi,\lambda)$,
we define the conditions (C1), (C2) and (C3) on 
$(\alpha_1, \alpha_2, a_0, b_0) \in \N^2\times \C^2$
with $(a_0, b_0)$ in \eqref{eqn:ab}  as follows:
\vspace{2pt}
\begin{enumerate}[label=\normalfont{(\arabic*)}]
\item[(C1)] $b_0 \notin [0,\min(\alpha)-1]\cap \Z$;
\vspace{2pt}
\item[(C2)] $b_0 \in [0,\min(\alpha)-1]\cap \Z$ and 
$\frac{a_0-b_0}{2} \notin [-b_0,0]\cap \Z$;
\vspace{2pt}
\item[(C3)] $b_0\in [0,\min(\alpha)-1]\cap \Z$ 
and $\frac{a_0-b_0}{2}\in [-b_0,0]\cap \Z$.
\end{enumerate}
Here, as in Remark \ref{rem:zero-case},
only (C1) holds if $\min(\alpha)=0$.

As for Fact \ref{fact:3F1-sols},
we write
\begin{align*}
\text{(F1)}
&:=\CC\, 
{}_3F_1\left [\begin{matrix}
-\alpha_1& -\alpha_2 & \frac{a_0-b_0}{2} \\
& -b_0 &  
\end{matrix};t\right ], \\[3pt]
\text{(F2)}
&:=
\CC\, 
t^{1+b_0}{}_3F_1
\left [\begin{matrix}
b_0+1-\alpha_1& b_0+1-\alpha_2 & \frac{a_0+b_0}{2}+1 \\
& 2+b_0 &  
\end{matrix};t\right ], \\[3pt]
\text{(F3)}
&:=
\CC\, 
{}_3F_1\left [\begin{matrix}
-\alpha_1& -\alpha_2 & \frac{a_0-b_0}{2} \\
& -b_0 &  
\end{matrix};t\right ] 
\oplus
\CC\, 
t^{1+b_0}{}_3F_1
\left [\begin{matrix}
b_0+1-\alpha_1& b_0+1-\alpha_2 & \frac{a_0+b_0}{2}+1 \\
& 2+b_0 &  
\end{matrix};t\right ].
\end{align*}

\begin{proposition}\label{prop:SolDsec7}
We have
$\wSol_{\Ad(e^{\wJ(\alpha)}) \calL}^{(\alpha)}(\xi,\lambda)\neq \{0\}$
for any $(\xi, \eta; \lambda, \nu) \in L_3(\alpha)$.
Moreover, for $(\xi, \eta; \lambda, \nu) \in L_3(\alpha)$,
the solution space $\wSol_{\Ad(e^{\wJ(\alpha)}) \calL}^{(\alpha)}(\xi,\lambda)$ 
is classified as follows.
\begin{align*}
\wSol_{\Ad(e^{\wJ(\alpha)}) \calL}^{(\alpha)}(\xi,\lambda)
=
\begin{cases}
\textnormal{(F1)}
& \textnormal{if (C1) holds},\\[3pt]
\textnormal{(F2)}
& \textnormal{if (C2) holds},\\[3pt]
\textnormal{(F3)}
& \textnormal{if (C3) holds}.
\end{cases}
\end{align*}
\end{proposition}

\begin{proof}
By  \eqref{eqn:AdSol3F1},
this is a straightforward consequence of Fact \ref{fact:3F1-sols}.
 \end{proof}

Next, we classify the solution space $\wSol_\calL^{(\alpha)}(\xi,\lambda)$.
Recall from \eqref{eqn:pab} that,
for $\alpha=(\alpha_1,\alpha_2) \in \N^2$ and $a,b\in \C$, we denote by
$p^{(\alpha_1, \alpha_2)}_{a,b}(t)$ the polynomial defined as
\begin{equation}\label{eqn:pab1}
p_{a,b}^{(\alpha_1, \alpha_2)}(t)=\sum_{n=0}^{\min(\alpha)} \frac{(-\alpha_1)_n(-\alpha_2)_n}{2^nn!(-b)_n}\Cay_n(a,b)t^n.
\end{equation}

\begin{proposition}\label{Prop:SolGeneralDiffEq1}
We have
$\wSol_\calL^{(\alpha)}(\xi,\lambda)\neq \{0\}$
for any $(\xi, \eta; \lambda, \nu) \in L_3(\alpha)$.
Moreover, 
for $(\xi, \eta; \lambda, \nu) \in L_3(\alpha)$,
the solution space $\wSol_\calL^{(\alpha)}(\xi,\lambda)$ is classified as follows.
\begin{align*}
\wSol_\calL^{(\alpha)}(\xi,\lambda)
=
\begin{cases}
\CC\,
p_{a_0,b_0}^{(\alpha_1, \alpha_2)}(t)
& \textnormal{if (C1) holds},\\[10pt]
 \CC\,
t^{1+b_0}p_{a_0,-b_0-2}^{(\alpha_1-b_0-1,\alpha_2-b_0-1)}(t)
& \textnormal{if (C2) holds},\\[10pt]
\CC\, 
p_{a_0,b_0}^{(\alpha_1, \alpha_2)}(t)
\oplus \CC\,
t^{1+b_0}p_{a_0,-b_0-2}^{(\alpha_1-b_0-1,\alpha_2-b_0-1)}(t)
& \textnormal{if (C3) holds}.
\end{cases}
\end{align*}
\end{proposition}

\begin{proof}
By the commutative diagram \eqref{eqn:cdSol3}
and 
Proposition \ref{prop:SolDsec7}, it suffices to show the following identities:
\begin{align*}
p_{a_0,b_0}^{(\alpha_1,\alpha_2)}(t)
&=
e^{-\wJ(\alpha)}
{}_3F_1\left [\begin{matrix}
-\alpha_1& -\alpha_2 & \frac{a_0-b_0}{2} \\
& -b_0 &  
\end{matrix};t\right ],\\[5pt]
t^{1+b_0}p_{a_0,-b_0-2}^{(\alpha_1-b_0-1,\alpha_2-b_0-1)}(t)
&=
e^{-\wJ(\alpha)}
t^{1+b_0}{}_3F_1
\left [\begin{matrix}
b_0+1-\alpha_1& b_0+1-\alpha_2 & \frac{a_0+b_0}{2}+1 \\
& 2+b_0 &  
\end{matrix};t\right ].
\end{align*}
Since the proof of these identities requires some 
lengthy computations, we give it as 
Proposition \ref{prop:Cayley_and_3F1_relation} in 
Appendix \ref{appendix:Cayely}.
\end{proof}

\medskip

We now classify the solution spaces 
$\Sol_3(\xi, \eta; \lambda, \nu)$ and 
$\Sol_{\Ad(e^{J(\alpha)}),\, 3}(\xi,\eta; \lambda,\nu)$; that is, the spaces \eqref{eqn:Solalpha} and \eqref{eqn:SolalphaAd} for $i=3$.
We put
\begin{align*}
\psi_3^{(1)}(\zeta) &:= T_{\nu_1-\lambda_1, \,\lambda_3-\nu_2}
(p^{(\nu_1-\lambda_1, \,\lambda_3-\nu_2)}_{\nu_1+\nu_2-2\lambda_2,\, \nu_1-\nu_2-2}(t))(\zeta),\\[3pt]
\psi_3^{(2)}(\zeta) 
&:= T_{\nu_1-\lambda_1,\, \lambda_3-\nu_2}(t^{\nu_1-\nu_2-1}p^{(\nu_2-\lambda_1+1, \,\lambda_3-\nu_1+1)}_{\nu_1+\nu_2-2\lambda_2, \,\nu_2-\nu_1}(t))(\zeta)\\[3pt]
&= \zeta_3^{\nu_1-\nu_2-1}
T_{(\nu_2-\lambda_1+1,\,\lambda_3-\nu_1+1)}(p^{(\nu_2-\lambda_1+1, \,\lambda_3-\nu_1+1)}_{\nu_1+\nu_2-2\lambda_2, \,\nu_2-\nu_1}(t))(\zeta),
\end{align*}
and
\begin{align*}
\varphi_3^{(1)}(\zeta) &:= T_{\nu_1-\lambda_1,\, \lambda_3-\nu_2}
\Big({}_3F_1\left [\begin{matrix}
\lambda_1 - \nu_1& \nu_2 - \lambda_3 & \nu_2 - \lambda_2 + 1 \\
& \nu_2 - \nu_1 + 2 &  
\end{matrix};t\right ]\Big)(\zeta),\\[3pt]
\varphi_3^{(2)} (\zeta)
&:= T_{\nu_1-\lambda_1, \,\lambda_3-\nu_2}\Big(t^{\nu_1-\nu_2-1}{}_3F_1
\left [\begin{matrix}
\lambda_1 - \nu_2 - 1& \nu_1 - \lambda_3 - 1 & \nu_1 - \lambda_2 \\
& \nu_1 - \nu_2 &  
\end{matrix};t\right ]\Big)(\zeta)\\[3pt]
&=  \zeta_3^{\nu_1-\nu_2-1}
T_{\nu_2-\lambda_1+1,\,\lambda_3-\nu_1+1}
\Big({}_3F_1
\left [\begin{matrix}
\lambda_1 - \nu_2 - 1& \nu_1 - \lambda_3 - 1 & \nu_1 - \lambda_2 \\
& \nu_1 - \nu_2 &  
\end{matrix};t\right ]\Big)(\zeta).
\end{align*}
Now, we have the following.

\begin{proposition}\label{prop:Sol_3_characterization}
The solution spaces 
$\Sol_3(\xi, \eta; \lambda, \nu)$ and 
$\Sol_{\Ad(e^{J(\alpha)}),\, 3}(\xi,\eta; \lambda,\nu)$ are classified as follows.
\begin{align}
\Sol_3(\xi, \eta; \lambda, \nu) 
&= 
\begin{cases}
\C\psi_3^{(1)}(\zeta) & \textnormal{if $(\xi, \eta; \lambda, \nu) \in \Supp_{3,a}(\DD)$},\\[3pt]
\C\psi_3^{(2)}(\zeta) & \textnormal{if $(\xi, \eta; \lambda, \nu) \in \Supp_{3,b}(\DD)$},\\[3pt]
\C\psi_3^{(1)}(\zeta) \oplus \C\psi_3^{(2)}(\zeta) & 
\textnormal{if $(\xi, \eta; \lambda, \nu) \in \Supp_{3,c}(\DD)$},\\[3pt]
\{0\}, & \textnormal{otherwise,}\\
\end{cases}\label{eqn:Sol3}\\[5pt]
\Sol_{\Ad(e^{J(\alpha)}),\, 3}(\xi,\eta; \lambda,\nu)
&= 
\begin{cases}
\C\varphi_3^{(1)}(\zeta) & \textnormal{if $\xi, \eta, \lambda, \nu) \in \Supp_{3,a}(\DD)$},\\[3pt]
\C\varphi_3^{(2)}(\zeta) & \textnormal{if $(\xi, \eta; \lambda, \nu) \in \Supp_{3,b}(\DD)$},\\[3pt]
\C\varphi_3^{(1)}(\zeta) \oplus \C\varphi_3^{(2)}(\zeta) &
\textnormal{if $(\xi, \eta; \lambda, \nu) \in \Supp_{3,c}(\DD)$},\\[3pt]
\{0\}, & \textnormal{otherwise}.\\
\end{cases}\label{eqn:SolAd3}
\end{align}
\end{proposition}

\begin{proof}
It follows from Proposition \ref{prop:Invariants_SB3}  that 
(C1), (C2) and (C3) correspond to 
\eqref{eq:condition_set_3a}, 
\eqref{eq:condition_set_3b} and 
\eqref{eq:condition_set_3c}, respectively, 
through the conditions
$\alpha_1=\nu_1-\lambda_1$ and $\alpha_2 = \lambda_3-\nu_2$.
The preceding proposition also shows that 
\eqref{eq:condition_full_set_3} is necessary for the existence of non-zero solutions.
Propositions \ref{prop:SolDsec7} and \ref{Prop:SolGeneralDiffEq1}
then lead to  \eqref{eqn:Sol3} and \eqref{eqn:SolAd3}, respectively.
\end{proof}

%%%%%%%%%%%%%%%%%%%%%%%%%%%%%%%%%%%%%%%%%
\subsection{Step 4: Applying \texorpdfstring{$\Symb_0^{-1}$}{Symb_0^{-1}} and
\texorpdfstring{$\Symb_{\ord}^{-1}$}{Symb_ord^{-1}}}
\label{sec:Step4}

We are now proceeding with the last step of the recipe given in 
Section \ref{sec:recipe2} to complete the proof of Theorems \ref{thm:class3} and \ref{thm:cons3}. 

\begin{proof}[Proofs of  Theorems \ref{thm:class3} and \ref{thm:cons3}]
The equivalence on $(\xi, \eta; \lambda, \nu)$ in Theorem  \ref{thm:class3}
can be readily read off from Proposition \ref{prop:Sol_3_characterization} as
$\Supp_3(\DD) = \bigcup_{\omega=a, b, c} \Supp_{3,\omega}(\DD)$.
About Theorem \ref{thm:cons3}, 
apply $\Symb_0^{-1}$ to $\psi^{(j)}_3(\zeta)$ 
in Proposition \ref{prop:Sol_3_characterization} to obtain the symmetrized form
of $\DD_{3}^{(j)}(\lambda,\nu)$. Analogously, the ordered form is obtained by applying 
 $\Symb_{\ord}^{-1}$ to $\varphi^{(j)}_3(\zeta)$  for $j=1,2$.
 This completes the proof. 
\end{proof}

%%%%%%%%%%%%%%%%%%%%%%%%%%%%%%%%%%%%%%%%%
\section{DSBOs \texorpdfstring{$\DD_i(\lambda,\nu)$}{Di(λ,ν)} for \texorpdfstring{$(GL(3,\R), G_i')$}{(GL(3,R),Gi')}, \texorpdfstring{$i=1,2$}{i=1,2}}
\label{sec:Emb12}

In this section, we apply the F-method to determine the space of DSBOs for the first and second embeddings, i.e., we characterize the following space:
\begin{equation*}
\Diff_{G_i'}(I(\xi, \lambda), J(\eta, \nu)),
\end{equation*}
for $i=1,2$. We follow the same steps as in the case of $G_3'$.

%%%%%%%%%%%%%%%%%%%%%%%%%%%%%%%%%%%%%%%%%
\subsection{Classification of DSBOs \texorpdfstring{$\DD_i(\lambda,\nu)$}{D_i(λ,ν)} 
for \texorpdfstring{$(GL(3, \R), G'_i)$}{(GL(3,R),G_i')} for \texorpdfstring{$i=1,2$}{i=1,2}}
\label{sec:ClassThmsDSBOs2}

We start with the classification of DSBOs $\mathbb{D}$.
For $i=1,2$, define $\Phi_{i,a}, \Phi_{i,b} \subset \C^3\times \C^2 \simeq \C^5$ 
as follows:
\begin{align}
    \Phi_{1,a} &:= \{(\lambda, \nu) \in \C^5: \nu_1-\lambda_1,\, \nu_2-\lambda_2 \in \N\},
    \label{eqn:Phi1a}\\
    \Phi_{1,b} &:= \{(\lambda, \nu) \in \C^5: \nu_1-\lambda_1,\, -(\nu_2-\lambda_2) \in \N_+ \, \text{ such that } \eqref{eq:condition_set_1b} \text{ holds}\},
\end{align}
\begin{equation}\label{eq:condition_set_1b}
    \nu_1-\lambda_1 + \nu_2 - \lambda_2 \in \N, \enspace \nu_1-\lambda_2 \in [1, \nu_1-\lambda_1 + \nu_2 - \lambda_2+1]\cap \Z
\end{equation}
\begin{align*}
    \Phi_{2,a} &:= \{(\lambda, \nu) \in \C^5: \lambda_3-\nu_2,\, \lambda_2-\nu_1 \in \N\},\\
    \Phi_{2,b} &:= \{(\lambda, \nu) \in \C^5: \lambda_3-\nu_2,\, -(\lambda_2 - \nu_1)\in \N_+ \text{ such that } \eqref{eq:condition_set_2b} \text{ holds}\},
\end{align*}
\begin{equation}\label{eq:condition_set_2b}
\lambda_3 - \nu_2 + \lambda_2-\nu_1 \in \N, \enspace\lambda_2 - \nu_2 \in [1, \lambda_3 - \nu_2 + \lambda_2-\nu_1 + 1]\cap \Z
\end{equation}
We set 
\begin{equation*}
\Phi_i := \Phi_{i,a}\cup \Phi_{i,b}.
\end{equation*} 
Now, we define $\Supp_i(\DD)$ for $i=1,2$  as given below.
\begin{align*}
    \Supp_1(\DD) &:= \{(\xi, \eta; \lambda, \nu) \in (\Z/2\Z)^5 \times \C^5: \eqref{eq:condition_full_set_1} \text{ holds}\},\\
    \Supp_2(\DD) &:= \{(\xi, \eta; \lambda, \nu) \in (\Z/2\Z)^5 \times \C^5: \eqref{eq:condition_full_set_2} \text{ holds}\},
\end{align*}
\begin{align}
\label{eq:condition_full_set_1}
    (\lambda, \nu) \in \Phi_1, \quad \xi_1+\eta_1 \equiv \nu_1-\lambda_1,\, \xi_2+\eta_2 \equiv \nu_2-\lambda_2 \mod{2},\\
\label{eq:condition_full_set_2}
    (\lambda, \nu) \in \Phi_2, \quad \xi_2+\eta_1 \equiv \lambda_2-\nu_1,\, \xi_3+\eta_2 \equiv \lambda_3-\nu_2 \mod{2}.
\end{align}

\begin{theorem}
\label{thm:class1}
For $i=1,2$, 
the following conditions on the tuple $(\xi, \eta; \lambda, \nu) \in (\Z/2\Z)^5 \times \C^5$ are equivalent.
\begin{enumerate}[label = \normalfont{(\roman*)}]
    \item $\Diff_{G_i^\prime}\left(I(\xi, \lambda), J(\eta, \nu)\right) \neq \{0\}$.
    
    \item $\dim_\C \Diff_{G_i^\prime}\left(I(\xi, \lambda), J(\eta, \nu)\right) = 1$.

    \item $(\xi, \eta; \lambda, \nu) \in \Supp_i(\DD)$. 
\end{enumerate}
\end{theorem}

%%%%%%%%%%%%%%%%%%%%%%%%%%%%%%%%%%%%%%%%%
\subsection{Construction of DSBOs 
\texorpdfstring{$\DD_i(\lambda,\nu)$}{D_i(λ,ν)} for \texorpdfstring{$(GL(3, \R), G'_i)$}{(GL(3,R),G_i')} for \texorpdfstring{$i=1,2$}{i=1,2}} 
\label{sec:const12}

As in Section \ref{sec:const3}, we denote by 
$T^{\mathrm{sym}}_{(\alpha_1, \alpha_2)} q(D_1, D_2, D_3)$
and
$T^\ord_{(\alpha_1, \alpha_2)} q(D_1, D_2, D_3)$ 
the  symmetrized and ordered operators
associated to $q(t) \in \Pol[t]$, respectively.
Also, let $p_{a,b}^{(\alpha_1,\alpha_2)}(t)$ be the polynomial given by the Cayley 
continuants $\Cay_n(a,b)$  defined in \eqref{eqn:pab}.
In particular, if $b=\alpha_2$, then
\begin{equation}\label{eqn:p_alpha}
p_{a,\alpha_2}^{(\alpha_1, \alpha_2)}(t) = 
\sum_{n=0}^{\min(\alpha)} 
\frac{(-\alpha_1)_n}{2^n n! }\Cay_n(a,\alpha_2)t^n.
\end{equation}

We write 
\[
\Rest_1:=\Rest_{x_2=x_3=0} \quad \text{and}\quad \Rest_2:=\Rest_{x_1=x_3=0}
\]
for the restriction maps with respect to the coordinates
\eqref{eqn:coordinate1} and \eqref{eqn:coordinate2}, respectively.

%%%%%%%%%%%%%%%%%%%%%%%%%%%%%%%%%%%%%%%%%%%%%%%%
\begin{theorem}\label{thm:cons1}
For $i=1,2$, we have
\begin{equation*}
\Diff_{G_i^\prime}\left(I(\xi, \lambda), J(\eta, \nu)\right)
=
\begin{cases}
\CC\, \DD_i(\lambda, \nu) & 
\textnormal{if $(\xi, \eta; \lambda, \nu) \in \Supp_i(\DD)$},\\[3pt]
\{0\} & \textnormal{otherwise},
\end{cases}
\end{equation*}
where $\DD_i(\lambda, \nu)$ for $i=1,2$ are given as follows.
\vspace{5pt}
\begin{align*}
\DD_1(\lambda,\nu)
&=\Rest_1\circ
\big(T^{\mathrm{sym}}_{(\nu_1-\lambda_1,\,\nu_1+\nu_2-\lambda_1-\lambda_2)}
p_{\lambda_1+\nu_1-\lambda_2-\nu_2-2,\,\nu_1+\nu_2-\lambda_1-\lambda_2}^{(\nu_1-\lambda_1,\,\nu_1+\nu_2-\lambda_1-\lambda_2)}\big)(D_1,D_2,D_3)\\[5pt]
&=\Rest_1\circ 
T^\ord_{(\nu_1-\lambda_1,\,\nu_1+\nu_2-\lambda_1-\lambda_2)}{}_2F_0
[\lambda_1-\nu_1, \lambda_1-\nu_2-1;  D_1, D_2, D_3],\\[10pt]
\DD_2(\lambda,\nu)
&=\Rest_2\circ 
\big(T^{\mathrm{sym}}_{(\lambda_2+\lambda_3-\nu_1-\nu_2,\,\lambda_3-\nu_2)}
p_{\lambda_3+\nu_2-\lambda_2-\nu_1+2,\,\lambda_2+\lambda_3-\nu_1-\nu_2}^{(\lambda_3-\nu_2,\, \lambda_2+\lambda_3-\nu_1-\nu_2)}\big)(D_1,D_2,D_3) \\[5pt]
&=\Rest_2\circ  
T^\ord_{(\lambda_2+\lambda_3-\nu_1-\nu_2,\,\lambda_3-\nu_2)}
{}_2F_0[\nu_2-\lambda_3, \nu_2-\lambda_2+1;D_1, D_2, D_3].
\end{align*}
\end{theorem}

The order of the DSBO $\DD_i(\lambda,\nu)$
as a differential operator is
$2(\nu_1-\lambda_1) + \nu_2 - \lambda_2$ for $i=1$
and
$\lambda_2 - \nu_1 + 2(\lambda_3 - \nu_2)$ for $i=2$.

%%%%%%%%%%%%%%%%%%%%%%%%%%%%%%%%%%%%%%%%%
\subsection{Proof of  Theorems \texorpdfstring{\ref{thm:class1}}{9.7} and \texorpdfstring{\ref{thm:cons1}}{9.9}: Part 1} 
\label{sec:Thms12A}

As in Sections \ref{sec:proof1} and \ref{sec:proof2}, we follow the recipe of the F-method
in Section \ref{sec:recipe2}. Our goal is to determine
\begin{align*}
\Sol_i(\xi,\eta; \lambda,\nu) := \{ \psi \in \Hom_{M_i'A_i'}\big((\CC_\xi\boxtimes \CC_{\lambda})\big\rvert_{M_i'A_i'},\Pol(\frakn_+)\otimes (\CC_\eta\boxtimes \CC_{\nu})\big) : \eqref{PDE_SymmetryBreaking_12} \text{ holds}\},
\end{align*}
\begin{equation}\label{PDE_SymmetryBreaking_12}
\widehat{d\pi_{\lambda^*}}(N_{i}^+)\psi = 0,
\end{equation}
for $i=1,2$ and $(\xi,\eta; \lambda,\nu) \in (\ZZ/2\ZZ)^5\times \CC^5$.

We start with 
the $M_i'A_i'$-invariants 
\begin{equation*}
\Hom_{M_i'A_i'}\big((\CC_\xi\boxtimes \CC_{\lambda})\big\rvert_{M_i'A_i'},\Pol(\frakn_+)\otimes (\CC_\eta\boxtimes \CC_{\nu})\big).
\end{equation*}
For $\alpha=(\alpha_1,\alpha_2) \in \N^2$, we define $L_i(\alpha) \subset (\ZZ/2\ZZ)^5\times \CC^5$
 for $i=1,2$ as follows:
\begin{equation*}
L_1(\alpha):=\{(\xi, \eta; \lambda, \nu) \in (\ZZ/2\ZZ)^5\times \CC^5: 
\text{\eqref{eqn:ch11} and \eqref{eqn:ch12} hold}  \},
\end{equation*}
\begin{alignat}{3}
\xi_1+\eta_1&\equiv \alpha_1 \bmod{2}, \quad &&\xi_2+\eta_2&&\equiv \alpha_1-\alpha_2 \bmod{2},
\label{eqn:ch11}\\[3pt]
\nu_1-\lambda_1&=\alpha_1, \quad &&\lambda_2-\nu_2&&=\alpha_1-\alpha_2.
\label{eqn:ch12}
\end{alignat}
\vspace{-5pt}
\begin{equation*}
L_2(\alpha):=\{(\xi, \eta; \lambda, \nu) \in (\ZZ/2\ZZ)^5\times \CC^5: 
\text{\eqref{eqn:ch21} and \eqref{eqn:ch22} hold}  \},
\end{equation*}
\begin{alignat}{3}
\xi_2+\eta_1&\equiv\alpha_1-\alpha_2 \bmod{2}, \quad &&
\xi_3+\eta_2&&\equiv \alpha_2 \bmod{2},
\label{eqn:ch21}\\[3pt]
 \lambda_2-\nu_1&=\alpha_1-\alpha_2, \quad &&\lambda_3-\nu_2&&=\alpha_2.
\label{eqn:ch22}
\end{alignat}

\begin{proposition}\label{prop:Invariants_SB12} 
For $i=1,2$,
the following conditions on $(\xi, \eta; \lambda, \nu)\in (\ZZ/2\ZZ)^5\times \CC^5$ are equivalent.
\begin{enumerate}[label=\normalfont{(\roman*)}]
    \item $\Hom_{M_i'A_i'}\big((\CC_\xi\boxtimes \CC_{\lambda})\big\rvert_{M_i'A_i'},\Pol(\frakn_+)\otimes (\CC_\eta\boxtimes \CC_{\nu})\big)\neq \{0\}$.
    \item There exists some $\alpha\in \NN^2$ such that 
    $(\xi, \eta; \lambda, \nu) \in L_i(\alpha)$.
\end{enumerate}
\end{proposition}

\begin{proof}
The proof is identical to that of Proposition \ref{prop:Invariants_SB3} and is therefore omitted.
\end{proof}

The next step is to compute the T-saturation
$T_\alpha^\sharp \big(-\zeta_i\widehat{d\pi_{\lambda^*}}(N_i^+)\big)$
of $-\zeta_i\widehat{d\pi_{\lambda^*}}(N_i^+)$ ($i=1,2$)  in 
Proposition \ref{prop:Fdpi}.
For $(\beta_1,\beta_2) \in \CC^2$ and $a\in \CC$, let
\begin{equation}\label{eqn:Palpha}
\calP^{(\beta_1,\beta_2)}(a) :=\vartheta_t+\frac{a}{2}t(\beta_1-\vartheta_t)+\frac{1}{4}t^2(\beta_1-1-\vartheta_t)(\beta_1-\vartheta_t)(\beta_2-\vartheta_t).
\end{equation}
Further, for $\alpha_1, \alpha_2 \in \NN$,  we put
\begin{equation*}
    \alpha(i):=
    \begin{cases}
        (\alpha_1,\alpha_2), &\text{if $i=1$},\\
        (\alpha_2,\alpha_1), & \text{if $i=2$}.
    \end{cases}
\end{equation*}
Also, write
\begin{equation}\label{eqn:a012}
\begin{aligned}
a_{0,1}
&=a_{0,1}(\lambda;\alpha_1,\alpha_2)
:=2(\lambda_1-\lambda_2-1+\alpha_1-\tfrac{1}{2}\alpha_2),
\\
a_{0,2}
&=a_{0,2}(\lambda;\alpha_1,\alpha_2)
:= 2(\lambda_3-\lambda_2+1+\tfrac{1}{2}\alpha_1-\alpha_2).
\end{aligned}
\end{equation}

The T-saturation 
$T_\alpha^\sharp \big(-\zeta_i\widehat{d\pi_{\lambda^*}}(N_i^+)\big)$
are computed in \cite[Prop.\ 5.13]{KPV25} as follows.

\begin{proposition}[{\cite[Prop.\ 5.13]{KPV25}}]
\label{prop:pihat12}
For $\alpha=(\alpha_1,\alpha_2) \in \NN^2$,
we have
\begin{align}
tT_\alpha^\sharp(-\zeta_1\widehat{d\pi_{\lambda^*}}(N_1^+))
&=\calP^{\alpha(1)}(a_{0,1}),\label{eqn:P1}\\[3pt]
(-t)T_\alpha^\sharp(-\zeta_2\widehat{d\pi_{\lambda^*}}(N_2^+))
&=\calP^{\alpha(2)}(a_{0,2}).\label{eqn:P2}
\end{align}
\end{proposition}

Recall from \eqref{eqn:Solalpha} and \eqref{eqn:SolalphaAd} that,
for $i=1,2$,  we have
\begin{align*}
\Sol_i^{(\alpha)}(\xi,\lambda) &= \{ \psi \in \Hom_{M_i'A_i'}\big((\CC_\xi\boxtimes \CC_{\lambda})\big\rvert_{M_i'A_i'},\Pol(\alpha)\otimes (\CC_\eta\boxtimes \CC_{\nu})\big) : \eqref{PDE_SymmetryBreaking_12} \text{ holds}\},\\
\Sol_{\Ad(e^{J(\alpha)}),\,i}^{(\alpha)}(\xi,\lambda)
&=\{\varphi(\zeta) \in \Pol(\alpha) : 
\Ad(e^{J(\alpha)})(-\zeta_i\widehat{d\pi_{\lambda^*}}(N_i^+))\varphi(\zeta)=0\}.
\end{align*}
We then put
\begin{align}
\wSol_{\calP,\, i}^{(\alpha)}(\xi,\lambda) &:= \{p(t) \in \Pol_{\min(\alpha)}[t] : 
\calP^{\alpha(i)}(a_{0,i})p(t) = 0\},\label{eqn:Sol_tilde_i}\\
\wSol_{\Ad(e^{\wJ(\alpha)}) \calP,\, i}^{(\alpha)}(\xi,\lambda)
&:=\{q(t) \in \Pol_{\min(\alpha)}[t] : 
\Ad(e^{\wJ(\alpha)})\calP^{\alpha(i)}(a_{0,i})q(t)=0\}. \nonumber
\end{align}
Similar to \eqref{eqn:cdSol3} for $i=3$,
the commutative diagram for $i=1,2$
corresponding to \eqref{eqn:cdA} with $(\xi, \eta; \lambda, \nu) \in L_i(\alpha)$
is given as follows.
\begin{equation}\label{eqn:conjTsec9}
\begin{tikzcd}[column sep=1.5cm]
	{\wSol_{\Ad(e^{\wJ(\alpha)}) \calP,\, i}^{(\alpha)}(\xi,\lambda)} & 
	\Sol_{\Ad(e^{J(\alpha)}),\,i}^{(\alpha)}(\xi,\lambda)  	
	\arrow[dl, pos=0.5, phantom, "\circlearrowleft"]
	&{\Diff_{G_i^\prime}\left(I(\xi, \lambda), J(\eta, \nu)\right)} \\
	{\wSol_{\calP,\, i}^{(\alpha)}(\xi,\lambda)}  & {\Sol_i^{(\alpha)}(\xi,\lambda)} &
	\arrow[ul, pos=0.886, phantom, "\circlearrowleft"]
	\arrow["\Symb_{\ord}^{-1}",  from=1-2, to=1-3]
	\arrow[ "\Symb_0^{-1}"', from=2-2, to=1-3]
	\arrow["{T_\alpha}", "\sim"' , from=1-1, to=1-2]
	\arrow["\sim" sloped, "e^{-\wJ(\alpha)}"', from=1-1, to=2-1]
	\arrow["e^{-J(\alpha)}", "\sim"' sloped,  from=1-2, to=2-2]
	\arrow["\sim" , "{T_\alpha}"',  from=2-1, to=2-2]
\end{tikzcd}
\end{equation}

%%%%%%%%%%%%%%%%%%%%%%%%%%%%%%%%%%%%%%%%%
\subsection{Proof of  Theorems \texorpdfstring{\ref{thm:class1}}{9.7} and \texorpdfstring{\ref{thm:cons1}}{9.9}: Part 2} 
\label{sec:Thms12B}

We first consider the case $i=1$; the other case $i=2$ will be handled in 
Section \ref{sec:duality} via a certain duality.
As for $\wSol_3^{(\alpha)}(\xi,\lambda)$, 
we wish to understand
$\wSol_{\calP,\, 1}^{(\alpha)}(\xi,\lambda)$
and 
$\wSol_{\Ad(e^{\wJ(\alpha)}) \calP,\, 1}^{(\alpha)}(\xi,\lambda)$.

As opposed to Section \ref{sec:Step3} where we first solve the ordered F-system, we start with the classification of the original F-system
$\wSol_{\calP,\, 1}^{(\alpha)}(\xi,\lambda)$ here.
We define two conditions 
 (A1) and (A2) on $(\alpha_1, \alpha_2, a_{0,1}) \in \N^2 \times \C$ as follows:
\vspace{5pt}
\begin{enumerate}[label=\normalfont{(\arabic*)}]
\item[(A1)] $\alpha_1\leq \alpha_2$;
\vspace{5pt}
\item[(A2)] $\alpha_1>\alpha_2$ and 
$a_{0,1}\in \{\alpha_2 - 2k: k=0,1,2,\ldots, \alpha_2\}$.
\end{enumerate}
In the following,
we consider the case $(a,b)=(a_{0,1}, \alpha_2)$ for 
$p_{a,b}^{(\alpha_1, \alpha_2)}(t)$ in \eqref{eqn:p_alpha}.

\begin{proposition}\label{prop:SolD1}
For $(\xi, \eta; \lambda, \nu) \in L_1(\alpha)$,
the solution space $\wSol_{\calP,\, 1}^{(\alpha)}(\xi,\lambda)$ 
is classified as follows.
\begin{equation}\label{eqn:SolCay}
\wSol_{\calP,\, 1}^{(\alpha)}(\xi,\lambda)
=
\begin{cases}
\CC\,
p_{a_{0,1},\alpha_2}^{(\alpha)}(t) & 
\textnormal{if (A1) or (A2) holds,}\\[3pt]
\{0\} & \textnormal{otherwise.}
\end{cases}
\end{equation}
\end{proposition}

\begin{proof}
By
\eqref{eqn:Lalpha}
and
\eqref{eqn:Palpha},
one can easily check that 
\begin{equation}\label{eqn:LPapx1}
\calL^{(\alpha_1,\alpha_2+1)}(a_{0,1},\alpha_2)
=\calP^{(\alpha)}(a_{0,1})(\alpha_2+1-\vartheta_t).
\end{equation}
Then take $p(t) \in \Pol_{\min(\alpha)}[t]$ and 
assume $\calP^{(\alpha)}(a_{0,1})p(t)=0$.
Observe that, as $\Ker(\alpha_2+1-\vartheta_t)\big\vert_{\Pol_{\min(\alpha)}[t]}=\{0\}$, 
the operator
$(\alpha_2+1-\vartheta_t)^{-1}$ is well-defined on $\Pol_{\min(\alpha)}[t]$. 
Then put $q(t) := (\alpha_2+1-\vartheta_t)^{-1}p(t)$. 
By \eqref{eqn:LPapx1}, we have
\begin{equation}\label{eqn:PLapx2}
0=\calP^{(\alpha)}(a_{0,1})p(t)
=\calP^{(\alpha)}(a_{0,1})(\alpha_2+1-\vartheta_t)q(t)
=\calL^{(\alpha_1,\alpha_2+1)}(a_{0,1},\alpha_2)q(t).
\end{equation}
Now we consider the cases $\alpha_1 \leq \alpha_2$ and $\alpha_1 > \alpha_2$,
separately.

First, suppose $\alpha_1 \leq \alpha_2$. Then 
$\min(\alpha) = \alpha_1$ and so $\Pol_{\min(\alpha)}[t]= \Pol_{\alpha_1}[t]$.
By Proposition \ref{Prop:SolGeneralDiffEq1}, we have
$\Ker\calL^{(\alpha_1,\alpha_2+1)}(a_{0,1},\alpha_2)\big\vert_{\Pol_{\alpha_1}[t]}
=\CC p^{(\alpha_1,\alpha_2+1)}_{a_{0,1},\alpha_2}(t)$,
where
\begin{equation*}
p^{(\alpha_1,\alpha_2+1)}_{a_{0,1},\alpha_2}(t)
=\sum_{n=0}^{\alpha_1} 
\frac{(-\alpha_1)_n(-\alpha_2-1)_n}{2^nn!(-\alpha_2)_n}\Cay_n(a_{0,1},\alpha_2)t^n.
\end{equation*}

Thus the polynomial $q(t)$ in \eqref{eqn:PLapx2} is
$q(t) = p^{(\alpha_1,\alpha_2+1)}_{a_{0,1},\alpha_2}(t)$ up to scalar.
As $q(t) = (\alpha_2+1-\vartheta_t)^{-1}p(t)$, this shows that
\begin{align*}
p(t)
&=(\alpha_2+1-\vartheta_t)\sum_{n=0}^{\alpha_1} 
\frac{(-\alpha_1)_n(-\alpha_2-1)_n}{2^nn!(-\alpha_2)_n}\Cay_n(a_{0,1},\alpha_2)t^n\\
&=(\alpha_2+1)\sum_{n=0}^{\alpha_1} 
\frac{(-\alpha_1)_n}{2^nn!}
\Cay_n(a_{0,1},\alpha_2)t^n,
\end{align*}
concluding the assertion in the case (A1).

Next suppose $\alpha_1 > \alpha_2$.
In this case 
$\min(\alpha) = \alpha_2$ and so $\Pol_{\min(\alpha)}[t]= \Pol_{\alpha_2}[t]$.
By Proposition \ref{Prop:SolGeneralDiffEq1}, the following conditions on 
$a_{0,1}$ are equivalent:
\vspace{3pt}
\begin{enumerate}[label=\normalfont{(\roman*)}]
\item $\Ker\calL^{(\alpha_1,\alpha_2+1)}(a_{0,1},\alpha_2)\big \vert_{\Pol_{\alpha_2}[t]}
\neq \{0\}$;
\vspace{2pt}
\item $\Ker\calL^{(\alpha_1,\alpha_2+1)}(a_{0,1},\alpha_2)\big \vert_{\Pol_{\alpha_2}[t]}=\CC p^{(\alpha_1,\alpha_2+1)}_{a_{0,1},\alpha_2}(t)$;
\vspace{3pt}
\item $a_{0,1} \in \{\alpha_2-2k : k=0,1,\ldots, \alpha_2\}$.
\vspace{3pt}
\end{enumerate}
The same arguments for the case (A1) then draw the desired conclusion for (A2).
\end{proof}

\begin{corollary}\label{cor:Sol2F0}
For $(\xi, \eta; \lambda, \nu) \in L_1(\alpha)$,
the solution space
$\wSol_{\Ad(e^{\wJ(\alpha)}) \calP,\, 1}^{(\alpha)}(\xi,\lambda)$
is classified as follows.
\begin{equation*}
\wSol_{\Ad(e^{\wJ(\alpha)}) \calP,\, 1}^{(\alpha)}(\xi,\lambda)
=
\begin{cases}
\CC\, 
{}_2F_0[-\alpha_1,\tfrac{a_{0,1}-\alpha_2}{2};t]
& 
\textnormal{if (A1) or (A2) holds,}\\[3pt]
\{0\} & \textnormal{otherwise.}
\end{cases}
\end{equation*}
\end{corollary}

\begin{proof}
It follows from the commutative diagram \eqref{eqn:conjTsec9} that 
\begin{equation*}
\wSol_{\Ad(e^{\wJ(\alpha)}) \calP,\, 1}^{(\alpha)}(\xi,\lambda)
=\{e^{\wJ(\alpha)}p(t) \in \Pol_{\min(\alpha)}[t] : p(t) \in \wSol_{\calP,\, 1}^{(\alpha)}(\xi,\lambda)\}.
\end{equation*}
Then Proposition \ref{prop:SolD1} implies
\begin{equation*}
\wSol_{\Ad(e^{\wJ(\alpha)}) \calP,\, 1}^{(\alpha)}(\xi,\lambda)
=
\begin{cases}
\CC\,
e^{\wJ(\alpha)}p_{a_{0,1},\alpha_2}^{(\alpha)}(t) & 
\textnormal{if (A1) or (A2) holds,}\\[3pt]
\{0\} & \textnormal{otherwise.}
\end{cases}
\end{equation*}
By \eqref{eqn:p3F1}, we have 
\begin{equation*}
e^{\wJ(\alpha)}p_{a_{0,1},\alpha_2}^{(\alpha)}(t) 
= {}_3F_1
\left [\begin{matrix}
-\alpha_1& -\alpha_2 & \frac{a_{0,1}-\alpha_2}{2} \\
& -\alpha_2 &  
\end{matrix};t\right ].
\end{equation*}
We claim that 
$ {}_3F_1
\left [\begin{matrix}
-\alpha_1& -\alpha_2 & \frac{a_{0,1}-\alpha_2}{2} \\
& -\alpha_2 &  
\end{matrix};t\right ] 
= 
{}_2F_0[-\alpha_1,\tfrac{a_{0,1}-\alpha_2}{2};t]$
in the cases (A1) and (A2).
Indeed, by definition, we have $-\alpha_2 \leq -\alpha_1$ in the case (A1).
In the case (A2), the condition on $a_{0,1}$ implies that  
$-\alpha_1 < -\alpha_2 \leq - \tfrac{a_{0,1}-\alpha_2}{2} \in \N$.
Therefore, it follows from Remark \ref{rem:3F1_2F0} that
if (A1) or (A2) holds, then
\begin{equation*}
{}_3F_1
\left [\begin{matrix}
-\alpha_1& -\alpha_2 & \frac{a_{0,1}-\alpha_2}{2} \\
& -\alpha_2 &  
\end{matrix};t\right ]
=
{}_2F_0[-\alpha_1,\tfrac{a_{0,1}-\alpha_2}{2};t].
\end{equation*}
This completes the proof.
\end{proof}

\begin{remark}
One can relate the differential equation
$\Ad(e^{\wJ(\alpha)})\calP^{\alpha(i)}(a_{0,1})q(t)=0$
more directly to 
${}_2\calF_0(-\alpha_1,\tfrac{a_{0,1}-\alpha_2}{2};t)q(t)=0$
as in Proposition \ref{prop:AdL}.
Since some preliminary computations are needed,
we discuss it in Appendix \ref{appendix:2F0}
as an alternative proof of Corollary \ref{cor:Sol2F0}.
\end{remark}

We now classify the spaces 
$\Sol_1(\xi, \eta; \lambda, \nu)$ and 
$\Sol_{\Ad(e^{J(\alpha)}),\, 1}(\xi,\eta; \lambda,\nu)$  of solutions
in \eqref{eqn:Solalpha} and \eqref{eqn:SolalphaAd} for $i=1$, respectively.
We put
\begin{align*}
\psi_1(\zeta) &:= 
T_{\nu_1-\lambda_1,\,\nu_1+\nu_2-\lambda_1-\lambda_2}
\big(p_{\lambda_1+\nu_1-\lambda_2-\nu_2-2,\,\nu_1+\nu_2-\lambda_1-\lambda_2}^{(\nu_1-\lambda_1,\,\nu_1+\nu_2-\lambda_1-\lambda_2)}(t)\big)(\zeta),\\[5pt]
\varphi_1(\zeta)&:= 
T_{\nu_1-\lambda_1,\,\nu_1+\nu_2-\lambda_1-\lambda_2}
({}_2F_0[\lambda_1-\nu_1, \lambda_1-\nu_2-1; t])(\zeta).
\end{align*}

\begin{proposition}\label{prop:Sol_1_characterization}
The spaces $\Sol_1(\xi, \eta; \lambda, \nu)$ and 
$\Sol_{\Ad(e^{J(\alpha)}),\, 1}(\xi,\eta; \lambda,\nu)$ are classified as follows.
\begin{align*}
\Sol_1(\xi, \eta; \lambda, \nu) 
&= 
\begin{cases}
\C\psi_1(\zeta) & \textnormal{if $(\xi, \eta; \lambda, \nu) \in \Supp_1(\DD)$},\\[3pt]
\{0\}, & \textnormal{otherwise,}
\end{cases}\\[3pt]
\Sol_{\Ad(e^{J(\alpha)}),\, 1}(\xi,\eta; \lambda,\nu)
&= 
\begin{cases}
\C\varphi_1(\zeta) & \textnormal{if $(\xi, \eta; \lambda, \nu) \in \Supp_1(\DD)$},\\[3pt]
\{0\}, & \textnormal{otherwise}.
\end{cases}
\end{align*}
\end{proposition}

\begin{proof}
The idea of the proof is analogous to that of 
Proposition \ref{prop:Sol_3_characterization}.
Indeed,
it follows from Proposition \ref{prop:Invariants_SB12} that
the conditions (A1) and (A2) are equivalent to
the conditions for $\Phi_{1,a}$ and $\Phi_{1,b}$, respectively, through
the identifications
$\alpha_1= \nu_1-\lambda_1$, 
$\alpha_1-\alpha_2=\lambda_2-\nu_2$
and
$a_{0,1}=\lambda_1-\lambda_2+\nu_1-\nu_2-2$.
Since the remaining  arguments are exactly the same as those in the proof of Proposition \ref{prop:Sol_3_characterization}, we omit the details.
\end{proof}

%%%%%%%%%%%%%%%%%%%%%%%%%%%%%%%%%%%%%%%%%
\subsection{Duality \texorpdfstring{$\Sol_1(\xi,\eta; \lambda,\nu)
\leftrightarrow \Sol_2(\xi,\eta; \lambda,\nu)$}{Sol_1 -> Sol_2}}
\label{sec:duality}

In this subsection,  for $\alpha=(\alpha_1,\alpha_2) \in \N^2$,
we write $(\alpha)=(\alpha_1,\alpha_2)$ to emphasize $\alpha_1, \alpha_2 \in \N$.
For instance, we write $L_i(\alpha)=L_i(\alpha_1, \alpha_2)$,
$\Pol(\alpha)=\Pol(\alpha_1,\alpha_2)$ and 
$\Sol^{(\alpha)}_i(\xi,\lambda) = \Sol^{(\alpha_1,\alpha_2)}_i(\xi,\lambda)$.
We also write $T_\alpha = T_{\alpha_1,\alpha_2}$.

Our initial aim is to establish a certain duality between 
$\Sol_1(\xi,\eta; \lambda,\nu)$ and $\Sol_2(\xi,\eta; \lambda,\nu)$
to handle the case $i=2$ by using the results for $i=1$.
For the purpose, first define an involution $\varpi_m$ on $\C^m$ by
\begin{equation*}
\varpi_m: \C^m \rightarrow \C^m, \; 
(u_1, \ldots, u_m) \mapsto -(u_m, u_{m-1}, \dots, u_2, u_1).
\end{equation*}
Similarly, let  $\varpi'_n$ denote the involution on $(\Z/2\Z)^n$ defined by
\begin{equation*}
\varpi'_n: (\Z/2\Z)^n \rightarrow (\Z/2\Z)^n, \; 
 (v_1, \ldots, v_n) \mapsto (v_n, v_{n-1}, \dots, v_2, v_1).
\end{equation*}
For simplicity, for $i=1,2$, we write
\begin{equation*}
\Hom_i(\alpha_1,\alpha_2;\xi,\eta; \lambda,\nu):=
\Hom_{M_i'A_i'}\big((\CC_\xi\boxtimes \CC_{\lambda})\big\rvert_{M_i'A_i'},
\Pol(\alpha_1,\alpha_2)\otimes (\CC_\eta\boxtimes \CC_{\nu})\big).
\end{equation*}

\begin{proposition}\label{prop:Hom12}
The following conditions 
on $(\xi,\eta; \lambda,\nu) \in (\Z/2\Z)^5\times \C^5$ are equivalent.
\begin{enumerate}[label=\normalfont{(\roman*)}]
\item
$\Hom_2(\alpha_1,\alpha_2;\xi,\eta; \lambda,\nu)\neq \{0\}$,
\vspace{3pt}

\item 
$\Hom_1(\alpha_2,\alpha_1;(\varpi'_3(\xi), \varpi'_2(\eta); \varpi_3(\lambda), \varpi_2(\nu))\neq \{0\}$.
\end{enumerate}

\end{proposition}

\begin{proof}
A simple observation shows that 
the following conditions 
on $(\xi,\eta; \lambda,\nu) \in (\Z/2\Z)^5\times \C^5$ are equivalent.
\begin{enumerate}[label=\normalfont{(\roman*)}]
\item $(\xi,\eta; \lambda,\nu) \in L_2(\alpha_1,\alpha)$,
\vspace{3pt}
\item $(\varpi'_3(\xi), \varpi'_2(\eta); \varpi_3(\lambda), \varpi_2(\nu)) 
\in L_1(\alpha_2,\alpha_1)$.
\end{enumerate}
Proposition \ref{prop:Invariants_SB12} then yields the proposed assertion.
\end{proof}

It follows from Proposition \ref{prop:Hom12} that there exists 
a linear isomorphism
\begin{equation*}
\Theta \colon 
\Hom_2(\alpha_1,\alpha_2;\xi,\eta; \lambda,\nu)
\stackrel{\sim}{\To}
\Hom_1(\alpha_2,\alpha_1;(\varpi'_3(\xi), \varpi'_2(\eta); \varpi_3(\lambda), \varpi_2(\nu))
\end{equation*}
such that the following diagram commutes.
\begin{equation}\label{eqn:diagram12}
\begin{tikzcd}
\Hom_2(\alpha_1,\alpha_2;\xi,\eta; \lambda,\nu)
\ar[r, "\sim" sloped, "\Theta"'] 
\ar[draw=none]{d}{\bigcup}
& 
\Hom_1(\alpha_2,\alpha_1;(\varpi'_3(\xi), \varpi'_2(\eta); \varpi_3(\lambda), \varpi_2(\nu))
\ar[draw=none]{d}{\bigcup}
\\
\Sol^{(\alpha_1,\alpha_2)}_2(\xi,\lambda)
  \arrow[r, "\sim", "\Theta_{\Sol}"']
 & 
\Sol^{(\alpha_2,\alpha_1)}_1(\varpi'_3(\xi),\varpi_3(\lambda))\\
\wSol_{\calP,\, 2}^{(\alpha_1,\alpha_2)}(\xi,\lambda) 
  \arrow[r, "\sim", "\widetilde{\Theta}"']
  \arrow[u, "T_{\alpha_1,\alpha_2}", "\sim"' sloped]
  \arrow[ur,  pos=0.41, phantom, "\circlearrowleft"]
 & 
\wSol_{\calP,\, 1}^{(\alpha_2, \alpha_1)}(\varpi_3'(\xi),\varpi_3(\lambda))
\arrow[u, "\sim" sloped, "T_{\alpha_2,\alpha_1}"']\\
 \wSol_{\Ad(e^{\wJ(\alpha_1,\alpha_2)}) \calP,\, 2}^{(\alpha_1,\alpha_2)}(\xi,\lambda)
  \arrow[r, "\sim", "\widetilde{\Theta}_{\Ad}"']
  \arrow[u, "e^{-\wJ(\alpha_1,\alpha_2)}", "\sim"' sloped]
  \arrow[ur,  pos=0.3, phantom, "\circlearrowleft"]  
 & 
\wSol_{\Ad(e^{\wJ(\alpha_2,\alpha_1)}) \calP,\, 1}^{(\alpha_2, \alpha_1)}(\varpi_3'(\xi),\varpi_3(\lambda))
\arrow[u, "\sim" sloped, "e^{-\wJ(\alpha_2,\alpha_1)}"']
\end{tikzcd}
\end{equation}
Here, the maps $\Theta_{\Sol}$, 
$\widetilde{\Theta}$ and $\widetilde{\Theta}_{\Ad}$ are defined as follows.
\begin{align}
\Theta_{\Sol}&:=\Theta\big\vert_{\Sol^{(\alpha_1,\alpha_2)}_2(\xi,\lambda)},
\nonumber\\[3pt]
\widetilde{\Theta}&:= 
T^{-1}_{\alpha_2,\alpha_1} \circ \Theta_{\Sol} \circ T_{\alpha_1,\alpha_2},
\nonumber\\[3pt]
\widetilde{\Theta}_{\Ad}
&:=
e^{\wJ(\alpha_2,\alpha_1)}
\circ
\widetilde{\Theta}
\circ
e^{-\wJ(\alpha_1,\alpha_2)}.\label{eqn:wTheta_Ad}
\end{align}

We wish to determine the linear map $\Theta_{\Sol}$ to establish
the duality in concern.
For this purpose, we first determine $\widetilde{\Theta}$.
Recall from \eqref{eqn:Sol_tilde_i} that 
\begin{equation*}
\wSol_{\calP,\, i}^{(\alpha_1,\alpha_2)}(\xi,\lambda) = \{p(t) \in \Pol_{\min(\alpha)}[t] : 
\calP^{\alpha(i)}(a_{0,i})p(t) = 0\},
\end{equation*}
where 
\begin{align*}
\calP^{(\alpha_1,\alpha_2)}(a_{0,i}) 
&:=\vartheta_t+\frac{a_{0,i}}{2}t(\alpha_1-\vartheta_t)
+\frac{1}{4}t^2(\alpha_1-1-\vartheta_t)(\alpha_1-\vartheta_t)(\alpha_2-\vartheta_t),\\[3pt]
\alpha(i)
&:=\begin{cases}
        (\alpha_1,\alpha_2), &\text{if $i=1$},\\
        (\alpha_2,\alpha_1), & \text{if $i=2$},
\end{cases}
\end{align*}
and
\begin{alignat*}{2}
a_{0,1}
&=a_{0,1}(\lambda;\alpha_1,\alpha_2)
&&:=2(\lambda_1-\lambda_2-1+\alpha_1-\tfrac{1}{2}\alpha_2),
\\
a_{0,2}
&=a_{0,2}(\lambda;\alpha_1,\alpha_2)
&&:= 2(\lambda_3-\lambda_2+1+\tfrac{1}{2}\alpha_1-\alpha_2).
\end{alignat*}

\begin{proposition}\label{prop:Sol12}
We have
\begin{equation}\label{eqn:wTheta}
\widetilde{\Theta}
\colon
\wSol_{\calP,\, 2}^{(\alpha_1,\alpha_2)}(\xi,\lambda) 
\stackrel{\sim}{\To}
\wSol_{\calP,\, 1}^{(\alpha_2, \alpha_1)}(\varpi_3'(\xi),\varpi_3(\lambda)),
\quad
p(t)\longmapsto p(-t).
\end{equation}
\end{proposition}

\begin{proof}
One can readily check that the following conditions on 
$p(t) \in \Pol_{\min(\alpha_1,\alpha_2)}[t]$ are equivalent.
\begin{enumerate}[label=\normalfont{(\roman*)}]
\item $\calP^{(\alpha_2,\alpha_1)}(a_{0,2}(\lambda;\alpha_1,\alpha_2))p(t) = 0$,
\vspace{3pt}
\item $\calP^{(\alpha_2,\alpha_1)}(a_{0,1}(\varpi_3(\lambda);\alpha_2,\alpha_1))p(-t) = 0$.
\end{enumerate}
The proposed isomorphism
now follows immediately from the equivalence.
\end{proof}

As in \eqref{eqn:q_alpha_k}, we write
\begin{equation*}
\psi_{\alpha_1,\alpha_2,k}(\zeta) := \zeta_1^{\alpha_1-k}\zeta_2^{\alpha_2-k}\zeta_3^k.
\end{equation*}

\begin{corollary}\label{cor:duality}
We have 
\begin{align*}
\Theta_{\Sol}\colon
\Sol^{(\alpha_1,\alpha_2)}_2(\xi,\lambda) 
\stackrel{\sim}{\To} 
\Sol^{(\alpha_2,\alpha_1)}_1(\varpi'_3(\xi),\varpi_3(\lambda)),
    \;
    \psi_{\alpha_1,\alpha_2,k}(\zeta) \mapsto (-1)^k \psi_{\alpha_2,\alpha_1,k}(\zeta).
\end{align*}
\end{corollary}

\begin{proof}
By the diagram \eqref{eqn:diagram12}, we have 
$\Theta_{\Sol}=T_{\alpha_2,\alpha_1} \circ \widetilde{\Theta} \circ T^{-1}_{\alpha_1,\alpha_2}$.
Proposition \ref{prop:Sol12} and the saturation map $T_\alpha$ then conclude
the desired map.
\end{proof}

We are now going to determine 
$\wSol_{\calP,\, 2}^{(\alpha_1,\alpha_2)}(\xi,\lambda)$, 
$\wSol_{\Ad(e^{\wJ(\alpha_1,\alpha_2)}) \calP,\, 2}^{(\alpha_1,\alpha_2)}(\xi,\lambda)$
and
$\Sol^{(\alpha_1,\alpha_2)}_2(\xi,\lambda)$.
To do so, we first observe the polynomial 
$p_{a,\alpha_1}^{(\alpha_2,\alpha_1)}(t)$ for $a \in \C$.

It is known that the Cayley continuants $\Cay_n(a,b)$ for $a,b\in \C$ 
enjoy the following property:
\begin{equation*}
\Cay_n(-a,b)=(-1)^n\Cay_n(a,b).
\end{equation*}
(See, for instance, the generating function in \eqref{eqn:gen_Cay} 
in Appendix \ref{appendix:Cayely}.)
Therefore, we have
\begin{equation}\label{eqn:polynomial_p}
p_{-a,\alpha_1}^{(\alpha_2,\alpha_1)}(t) = 
\sum_{n=0}^{\min(\alpha_1,\alpha_2)} 
\frac{(-\alpha_2)_n}{2^n n! }\Cay_n(-a,\alpha_2)t^n
=p_{a,\alpha_1}^{(\alpha_2,\alpha_1)}(-t).
\end{equation}

We introduce two conditions 
 (B1) and (B2) on $(\alpha_1, \alpha_2, a_{0,2}) \in \N^2 \times \C$ as follows:
\vspace{5pt}
\begin{enumerate}[label=\normalfont{(\arabic*)}]
\item[(B1)] $\alpha_2\leq \alpha_1$;
\vspace{5pt}
\item[(B2)] $\alpha_2>\alpha_1$ and 
$a_{0,2}\in \{\alpha_1 - 2k: k=0,1,2,\ldots, \alpha_1\}$.
\end{enumerate}

\begin{proposition}\label{prop:SolD2}
Let $(\xi, \eta; \lambda, \nu) \in L_2(\alpha)$.
Then the solution spaces
$\wSol_{\calP,\, 2}^{(\alpha_1,\alpha_2)}(\xi,\lambda)$ and 
$\wSol_{\Ad(e^{\wJ(\alpha_1,\alpha_2)}) \calP,\, 2}^{(\alpha_1,\alpha_2)}(\xi,\lambda)$
are classified as follows.
\begin{align*}
\wSol_{\calP,\, 2}^{(\alpha_1,\alpha_2)}(\xi,\lambda)
&=
\begin{cases}
\CC\,
p_{a_{0,2},\alpha_1}^{(\alpha_2,\alpha_1)}(t) & 
\textnormal{if (B1) or (B2) holds,}\\[3pt]
\{0\} & \textnormal{otherwise.}
\end{cases}\\[3pt]
\wSol_{\Ad(e^{\wJ(\alpha_1,\alpha_2)}) \calP,\, 2}^{(\alpha_1,\alpha_2)}(\xi,\lambda)
&=
\begin{cases}
\CC\, 
{}_2F_0[-\alpha_2,\tfrac{a_{0,2}-\alpha_1}{2};t]
& 
\textnormal{if (B1) or (B2) holds,}\\[5pt]
\{0\} & \textnormal{otherwise,}
\end{cases}
\end{align*}
\end{proposition}

\begin{proof}
By the commutative diagram \eqref{eqn:diagram12}, it suffices to consider
$\wSol_{\calP,\, 1}^{(\alpha_2, \alpha_1)}(\varpi_3'(\xi),\varpi_3(\lambda))$
and
$\wSol_{\Ad(e^{\wJ(\alpha_2,\alpha_1)}) \calP,\, 1}^{(\alpha_2, \alpha_1)}(\varpi_3'(\xi),\varpi_3(\lambda))$. It follows from Proposition \ref{prop:SolD1} 
and Corollary \ref{cor:Sol2F0} that 
\begin{align}
\wSol_{\calP,\, 1}^{(\alpha_2, \alpha_1)}(\varpi_3'(\xi),\varpi_3(\lambda))
&=
\begin{cases}
\CC\,
p_{-a_{0,2},\alpha_1}^{(\alpha_2,\alpha_1)}(t) & 
\textnormal{if (B1) or (B2) holds,}\\[3pt]
\{0\} & \textnormal{otherwise,}
\end{cases}\label{eqn:Sol_P_2}\\[3pt]
\wSol_{\Ad(e^{\wJ(\alpha_2,\alpha_1)}) \calP,\, 1}^{(\alpha_2, \alpha_1)}(\varpi_3'(\xi),\varpi_3(\lambda))
&=
\begin{cases}
\CC\, 
{}_2F_0[-\alpha_2,-\tfrac{a_{0,2}+\alpha_1}{2};t]
& 
\textnormal{if (B1) or (B2) holds,}\\[5pt]
\{0\} & \textnormal{otherwise.}
\end{cases}\label{eqn:Sol_AdP_2}
\end{align}
Here, to obtain $a_{0,2}$ on the right-hand side, 
the following computation was carried out:
\begin{equation*}
a_{0,1}(\varpi_3(\lambda);\alpha_2,\alpha_1)
=2(\lambda_2-\lambda_3-1+\alpha_2-\tfrac{1}{2}\alpha_1)
=-a_{0,2}.
\end{equation*}
By \eqref{eqn:polynomial_p}, we have 
$p_{-a_{0,2},\alpha_1}^{(\alpha_2,\alpha_1)}(t) =
p_{a_{0,2},\alpha_1}^{(\alpha_2,\alpha_1)}(-t)$.
Now, the description of $\wSol_{\calP,\, 2}^{(\alpha_1,\alpha_2)}(\xi,\lambda)$ follows from \eqref{eqn:Sol_P_2} and from Proposition \ref{prop:Sol12}.

To show the assertion for the space $\wSol_{\Ad(e^{\wJ(\alpha_1,\alpha_2)}) \calP,\, 2}^{(\alpha_1,\alpha_2)}(\xi,\lambda)$, by \eqref{eqn:Sol_AdP_2} and \eqref{eqn:diagram12},
it suffices to apply
$\widetilde{\Theta}^{-1}_{\Ad}$
to ${}_2F_0[-\alpha_2,-\tfrac{a_{0,2}+\alpha_1}{2};t]$.
So, suppose that (B1) or (B2) holds. 
By Remark \ref{rem:3F1_2F0}, the polynomial
${}_2F_0[-\alpha_2,-\tfrac{a_{0,2}+\alpha_1}{2};t]$ can be given as
\begin{equation*}
{}_2F_0[-\alpha_2,-\tfrac{a_{0,2}+\alpha_1}{2};t]=
{}_3F_1
\left [\begin{matrix}
-\alpha_1& -\alpha_2 & -\tfrac{a_{0,2}+\alpha_1}{2} \\
& -\alpha_1 &  
\end{matrix};t\right ].
\end{equation*}
Thus, it follows from \eqref{eqn:p3F1} in Appendix \ref{appendix:Cayely} and 
\eqref{eqn:polynomial_p} that
\begin{align*}
e^{-\wJ(\alpha_2,\alpha_1)}
{}_2F_0[-\alpha_2,-\tfrac{a_{0,2}+\alpha_1}{2};t]
&=e^{-\wJ(\alpha_2,\alpha_1)}
{}_3F_1
\left [\begin{matrix}
-\alpha_1& -\alpha_2 & -\tfrac{a_{0,2}+\alpha_1}{2} \\
& -\alpha_1 &  
\end{matrix};t\right ]\\[3pt]
&=p_{a_{0,2},\alpha_1}^{(\alpha_2,\alpha_1)}(-t).
\end{align*}
Therefore, by \eqref{eqn:wTheta_Ad}, we have
\begin{align*}
\widetilde{\Theta}^{-1}_{\Ad}
\big({}_2F_0[-\alpha_2,-\tfrac{a_{0,2}+\alpha_1}{2};t]\big)
&= 
\big(e^{\wJ(\alpha_2,\alpha_1)}
\circ
\widetilde{\Theta}^{-1}
\circ
e^{-\wJ(\alpha_2,\alpha_1)}\big)
({}_2F_0[-\alpha_2,-\tfrac{a_{0,2}+\alpha_1}{2};t])\\[3pt]
&=
\big(e^{\wJ(\alpha_2,\alpha_1)}
\circ
\widetilde{\Theta}^{-1}\big)
(p_{a_{0,2},\alpha_1}^{(\alpha_2,\alpha_1)}(-t))\\[3pt]
&=
\big(e^{\wJ(\alpha_2,\alpha_1)}\big)
(p_{a_{0,2},\alpha_1}^{(\alpha_2,\alpha_1)}(t))\\[3pt]
&=
{}_2F_0[-\alpha_2,\tfrac{a_{0,2}-\alpha_1}{2};t],
\end{align*}
which concludes the description of $\wSol_{\Ad(e^{\wJ(\alpha_1,\alpha_2)}) \calP,\, 2}^{(\alpha_1,\alpha_2)}(\xi,\lambda)$.
\end{proof}

We now put
\begin{align*}
\psi_2(\zeta) &:= 
T_{\lambda_2+\lambda_3-\nu_1-\nu_2,\,\lambda_3-\nu_2}
\big(p_{\lambda_3+\nu_2-\lambda_2-\nu_1+2,\,\lambda_2+\lambda_3-\nu_1-\nu_2}^{(\lambda_3-\nu_2,\, \lambda_2+\lambda_3-\nu_1-\nu_2)}(t)\big)(\zeta),\\[5pt]
\varphi_2(\zeta)&:= 
T_{\lambda_2+\lambda_3-\nu_1-\nu_2,\,\lambda_3-\nu_2}
\big({}_2F_0[\nu_2-\lambda_3, \nu_2-\lambda_2+1;t]\big)(\zeta).
\end{align*}

\begin{proposition}\label{prop:Sol_2_characterization}
The spaces $\Sol_2(\xi, \eta; \lambda, \nu)$ and 
$\Sol_{\Ad(e^{J(\alpha)}),\, 2}(\xi,\eta; \lambda,\nu)$ are classified as follows.
\begin{align*}
\Sol_2(\xi, \eta; \lambda, \nu) 
&= 
\begin{cases}
\C\psi_2(\zeta) & \textnormal{if $(\xi, \eta; \lambda, \nu) \in \Supp_{2}(\DD)$},\\[3pt]
\{0\}, & \textnormal{otherwise,}
\end{cases}\\[3pt]
\Sol_{\Ad(e^{J(\alpha)}),\, 2}(\xi,\eta; \lambda,\nu)
&= 
\begin{cases}
\C\varphi_2(\zeta) & \textnormal{if $(\xi, \eta; \lambda, \nu) \in \Supp_{2}(\DD)$},\\[3pt]
\{0\}, & \textnormal{otherwise}.
\end{cases}
\end{align*}
\end{proposition}

\begin{proof}
It is easy to see that $(\xi,\eta; \lambda,\nu) \in \Supp_2(\DD)$
if and only if 
$(\varpi'_3(\xi), \varpi'_2(\eta); \varpi_3(\lambda), \varpi_2(\nu)) \in \Supp_1(\DD)$.
Thus, by Propositions \ref{prop:Sol_1_characterization} and
\ref{prop:Hom12}, the following two conditions on $(\xi, \eta; \lambda, \nu)$ 
are equivalent.
\begin{enumerate}[label=\normalfont{(\roman*)}]
\item $\Sol_2(\xi, \eta; \lambda, \nu),
\Sol_{\Ad(e^{J(\alpha)}),\, 2}(\xi,\eta; \lambda,\nu)\neq \{0\}$.
\item $(\xi,\eta; \lambda,\nu) \in \Supp_2(\DD)$.
\end{enumerate}
The explicit formulas of $\psi_2(\zeta)$ and $\varphi_2(\zeta)$ 
follow from Proposition \ref{prop:SolD2}. 
\end{proof}

%%%%%%%%%%%%%%%%%%%%%%%%%%%%%%%%%%%%%%%%%
\subsection{Proof of Theorems \texorpdfstring{\ref{thm:class1}}{9.7} and \texorpdfstring{\ref{thm:cons1}}{9.9}}
\label{sec:proof12}

We are now at the last step to complete 
a proof of Theorems \ref{thm:class1} and \ref{thm:cons1}.

\begin{proof}[Proof of  Theorems \ref{thm:class1} and \ref{thm:cons1}]
As for $\DD_{3}(\lambda,\nu)$, apply
 $\Symb_0^{-1}$ to $\psi_i(\zeta)$ to obtain 
the symmetrized form of $\DD_{i}(\lambda,\nu)$ for $i=1,2$.
Likewise, the ordered forms are obtained by applying 
 $\Symb_{\ord}^{-1}$ to $\varphi_i(\zeta)$ for $i=1,2$.
 This completes the proof. 
\end{proof}

%%%%%%%%%%%%%%%%%%%%%%%%%%%%%%%%%%%%%%%%%
\section{DIOs \texorpdfstring{$\calD(\lambda,\lambda')$}{D(λ,λ')} for \texorpdfstring{$GL(3,\R)$}{GL(3,R)}}
\label{sec:DIO}

The aim of this short section is to utilize the results for DSBOs $\DD_i(\lambda,\nu)$
 for $i=1,2$ to classify and construct DIOs 
 \begin{equation*}
\calD \in \Diff_{G}(I(\xi, \lambda),I(\xi', \lambda'))
\end{equation*}
for $G=GL(3,\RR)$.
Those results are obtained in Theorems \ref{thm:class_DIOs} and 
\ref{thm:cons_DIOs}. 

%%%%%%%%%%%%%%%%%%%%%%%%%%%%%%%%%%%%%%%%%
\subsection{Classification and construction of DIOs
\texorpdfstring{$\calD(\lambda,\lambda')$}{D(λ,λ')} for \texorpdfstring{$GL(3, \R)$}{GL(3,R)}} 

As for DSBOs $\DD$, we start with the classification of DIOs $\calD$.
First, for $j=0, 1,\ldots, 5$, we put
\begin{equation*}
\Lambda_j := \{(\lambda, \lambda') \in\C^6 : \text{(L$j$) holds}\},
\end{equation*}
\begin{equation}
\lambda=\lambda', \tag{L0}
\end{equation}
%\vspace{-20pt}
\begin{alignat}{4}
&\lambda_1 = \lambda_1', \quad 
&&\lambda_2' - \lambda_2 \in \N_+, \quad
&&\lambda_2 - \lambda_3' = 1, \quad
&&\lambda_3 - \lambda_2' = -1, \tag{L1} \\[3pt]
&\lambda_2 = \lambda_2', \quad 
&&\lambda_3 - \lambda_3' \in \N_+, \quad 
&&\lambda_1 - \lambda_3' = 2, \quad 
&&\lambda_3 - \lambda_1' = -2, \tag{L2}\\[3pt]
&\lambda_3 = \lambda_3', \quad 
&&\lambda_1' - \lambda_1 \in \N_+, \quad 
&&\lambda_2 - \lambda_1' = -1, \quad 
&&\lambda_1 - \lambda_2' = 1, \tag{L3}
\end{alignat}
%\vspace{-20pt}
\begin{alignat}{4}
&\lambda_1' - \lambda_1, \lambda_3 - \lambda_3' \in \N_+,\quad
&&\lambda_1'-\lambda_3 = 2,\quad
&&\lambda_1 - \lambda_3' \neq 2, \quad
&&(\lambda_2, \lambda_2') = (\lambda_3'+1, \lambda_1-1), \tag{L4}\\[3pt]
&\lambda_1' - \lambda_1, \lambda_3 - \lambda_3' \in \N_+,  \quad
&&\lambda_1'-\lambda_3 \neq 2, \quad
&&\lambda_1 - \lambda_3' = 2, \quad
&&(\lambda_2, \lambda_2') = (\lambda_1'-1, \lambda_3+1). \tag{L5}
\end{alignat}
\noindent
Furthermore, for $j=0,1,\ldots, 5$, we set
\begin{equation*}
\Supp(\calD)_j 
:= \{(\xi, \xi'; \lambda, \lambda')\in (\Z/2\Z)^6 \times \C^6: \text{\eqref{eqn:61} holds}\},
\end{equation*}
\vspace{-10pt}
\begin{equation}\label{eqn:61}
(\lambda, \lambda') \in  \Lambda_j
\quad \text{and} \quad 
\xi_k + \xi_k' \equiv \lambda_k' - \lambda_k \bmod{2} \enspace (k = 1, 2, 3).
\end{equation}
We define
\begin{equation*}
\Supp(\calD):=\bigcup_{j=0}^5\Supp(\calD)_j.
\end{equation*}

\begin{theorem}\label{thm:class_DIOs}
The following conditions on the tuple $(\xi, \xi'; \lambda, \lambda') \in (\Z/2\Z)^6 \times \C^6$ are equivalent.
\begin{enumerate}[label = \normalfont{(\roman*)}]
    \item $\Diff_{G}\left(I(\xi, \lambda), I(\xi', \lambda')\right) \neq \{0\}$.
    
    \item $\dim_\C \Diff_{G}\left(I(\xi, \lambda), I(\xi', \lambda')\right) = 1$.

    \item $(\xi, \xi'; \lambda, \lambda') \in \Supp(\calD)$.
\end{enumerate}
\end{theorem}

We next consider the construction of DIOs $\calD$.
Recall from Section \ref{sec:const3} that the operators
$(T^{\mathrm{sym}}_{(\alpha_1, \alpha_2)} q)(D_1, D_2, D_3)$
and
$(T^\ord_{(\alpha_1, \alpha_2)} q)(D_1, D_2, D_3)$ 
denote
the symmetrized and ordered operators
associated to 
$q(t) \in \Pol[t]$, respectively.
Also, we denote by
$p_{a,\alpha_2}^{(\alpha_1, \alpha_2)}(t)$
the polynomial given by the Cayley continuants $\Cay_n(a,\alpha_2)$
in \eqref{eqn:p_alpha}.

\begin{theorem}\label{thm:cons_DIOs}
We have
\begin{equation*}
\Diff_{G}\left(I(\xi, \lambda), I(\xi', \lambda')\right)
=
\begin{cases}
\CC\, \calD(\lambda, \lambda') & 
\textnormal{if $(\xi, \xi'; \lambda, \lambda') \in \Supp(\calD)$},\\[3pt]
\{0\} & \textnormal{otherwise},
\end{cases}
\end{equation*}
where $\calD(\lambda, \lambda')$ 
is given as follows.
\begin{align*}
\calD(\lambda,\lambda')
&=T^{\mathrm{sym}}_{(\lambda_1'-\lambda_1,\lambda_3-\lambda_3')}
p_{\lambda_1+\lambda_1'-\lambda_2-\lambda_2'-2,\lambda_3-\lambda_3'}^{(\lambda_1'-\lambda_1,\lambda_3-\lambda_3')}
(D_1, D_2, D_3)\\[3pt]
&=T^\ord_{(\lambda_1'-\lambda_1,\lambda_3-\lambda_3')}
{}_2F_0\Big [\lambda_1-\lambda_1', \lambda_1-\lambda_2'-1; D_1, D_2, D_3\Big ].
\end{align*}
Moreover, the differential operator $\calD(\lambda,\lambda')$ has 
the following form for each $(\xi, \xi'; \lambda, \lambda') \in \Supp(\calD)_j$.
\begin{enumerate}[label=\normalfont{(\arabic*)}]
\item 
Symmetrized form: We have
\begin{equation*}
\calD(\lambda,\lambda')
=
\begin{cases}
\id & 
\textnormal{if $(\xi, \xi'; \lambda, \lambda') \in \Supp(\calD)_0$},\\[5pt]
D_2^{\lambda_3-\lambda_3'} & 
\textnormal{if $(\xi, \xi'; \lambda, \lambda')  \in \Supp(\calD)_1$},\\[5pt]
T^{\mathrm{sym}}_{(\lambda_1'-\lambda_1,\lambda_3-\lambda_3')}
p_{\lambda_1+\lambda_1'-\lambda_2-\lambda_2'-2,\lambda_3-\lambda_3'}^{(\lambda_1'-\lambda_1,\lambda_3-\lambda_3')}
(D_1, D_2, D_3) & 
\textnormal{if $(\xi, \xi'; \lambda, \lambda')  \in \Supp(\calD)_2$},\\[5pt]
D_1^{\lambda_1-\lambda_1'} & 
\textnormal{if $(\xi, \xi'; \lambda, \lambda')  \in \Supp(\calD)_3$},\\[5pt]
T^{\mathrm{sym}}_{(\lambda_1'-\lambda_1,\lambda_3-\lambda_3')}
{}_2F_0\Big [\lambda_1-\lambda_1', \lambda_3-\lambda_3'; D_1, D_2, -D_3/2\Big ]
& 
\textnormal{if $(\xi, \xi'; \lambda, \lambda')  \in \Supp(\calD)_4$},\\[5pt]
T^{\mathrm{sym}}_{(\lambda_1'-\lambda_1,\lambda_3-\lambda_3')}
{}_2F_0\Big [\lambda_1-\lambda_1', \lambda_3-\lambda_3'; D_1, D_2, D_3/2\Big ]
& 
\textnormal{if $(\xi, \xi'; \lambda, \lambda')  \in \Supp(\calD)_5$}.
\end{cases}
\end{equation*}

\item 
Ordered form: We have
\begin{equation*}
\calD(\lambda,\lambda')
=
\begin{cases}
\id & 
\textnormal{if $(\xi, \xi'; \lambda, \lambda')  \in \Supp(\calD)_0$},\\[5pt]
D_2^{\lambda_3-\lambda_3'} &
\textnormal{if $(\xi, \xi'; \lambda, \lambda')  \in \Supp(\calD)_1$},\\[5pt]
T^\ord_{(\lambda_1'-\lambda_1,\lambda_3-\lambda_3')}
{}_2F_0\Big [\lambda_1-\lambda_1', \lambda_1-\lambda_2'-1; D_1, D_2, D_3\Big ]
& 
\textnormal{if $(\xi, \xi'; \lambda, \lambda')  \in \Supp(\calD)_2$},\\[5pt]
D_1^{\lambda_1-\lambda_1'} 
& 
\textnormal{if $(\xi, \xi'; \lambda, \lambda')  \in \Supp(\calD)_3$},\\[5pt]
D_1^{\lambda_1-\lambda_1'} D_2^{\lambda_3-\lambda_3'} 
& 
\textnormal{if $(\xi, \xi'; \lambda, \lambda')  \in \Supp(\calD)_4$},\\[5pt]
T^\ord_{(\lambda_1'-\lambda_1,\lambda_3-\lambda_3')}
{}_2F_0\Big [\lambda_1-\lambda_1', \lambda_3-\lambda_3'; D_1, D_2, D_3\Big ]
& 
\textnormal{if $(\xi, \xi'; \lambda, \lambda')  \in \Supp(\calD)_5$}.
\end{cases}
\end{equation*}

\end{enumerate}
\end{theorem}

As a differential operator,
the order of the DIO 
$\mathcal{D}(\lambda,\lambda')$ is 
\begin{equation*}
\lambda_1'-\lambda_1+\lambda_3-\lambda_3'
=
2(\lambda_1' - \lambda_1) + \lambda_2'-\lambda_2.
\end{equation*}

We shall show Theorems \ref{thm:class_DIOs} and \ref{thm:cons_DIOs} 
along with the recipe in Section \ref{sec:recipe2}.

%%%%%%%%%%%%%%%%%%%%%%%%%%%%%%%%%%%%%%%%%
\subsection{Proof of Theorems \texorpdfstring{\ref{thm:class_DIOs}}{10.2} and \texorpdfstring{\ref{thm:cons_DIOs}}{10.3}} 

We start with the classification of $MA$-invariants.
For $\alpha=(\alpha_1, \alpha_2) \in \N^2$, we put
\begin{equation*}
L_0(\alpha):=\{(\xi, \xi'; \lambda, \lambda') \in (\ZZ/2\ZZ)^6\times \CC^6: 
\text{\eqref{eqn:ch41} and \eqref{eqn:ch42} hold}  \},
\end{equation*}
\begin{alignat}{5}
\xi_1+\xi_1' &\equiv \alpha_1 \bmod{2}, 
\quad &&
\xi_2+\xi_2' &&\equiv \alpha_1-\alpha_2 \bmod{2},
\quad &&
\xi_3+\xi_3' &&\equiv \alpha_2 \bmod{2},
\label{eqn:ch41}\\[3pt]
\lambda_1'-\lambda_1&=\alpha_1, 
\quad &&
\lambda_2-\lambda_2'&&=\alpha_1-\alpha_2,
\quad &&
\lambda_3-\lambda_3' &&=\alpha_2.
\label{eqn:ch42}
\end{alignat}

\begin{proposition}\label{prop:Invariants_Intertwining} 
The following conditions on $(\xi, \xi'; \lambda, \lambda') \in (\ZZ/2\ZZ)^6\times \CC^6$ are equivalent.
\begin{enumerate}[label=\normalfont{(\roman*)}]
    \item $\Hom_{MA}\big(\CC_\xi\boxtimes \CC_{\lambda}, \Pol(\frakn_+)\otimes (\CC_{\xi'}\boxtimes \CC_{\lambda'})\big)\neq \{0\}$.
    \item There exists some $\alpha\in \NN^2$ such that 
    $(\xi,\xi';\lambda,\lambda')\in L_0(\alpha)$.
\end{enumerate}
\end{proposition}

\begin{proof}
This is an immediate consequence 
of Lemma \ref{lem:MApol} and \eqref{eqn:Homalpha2}.
\end{proof}

We next consider solving the F-system in concern.
Define
\begin{align}
\Sol(\xi, \xi^\prime; \lambda, \lambda^\prime) 
&:= \{ \psi \in \Hom_{MA}\big(\CC_\xi\boxtimes\CC_{\lambda}, \Pol(\frakn_+)\otimes (\CC_{\xi'}\otimes \CC_{\lambda'})\big) : \eqref{PDE_Intertwining} \text{ holds}\},
\nonumber\\[5pt]
\Sol^{(\alpha)}(\xi,\lambda) &:= \{ \psi \in \Hom_{MA}\big(\CC_\xi\boxtimes\CC_{\lambda}, \Pol(\alpha)\otimes (\CC_{\xi'}\otimes \CC_{\lambda'})\big) : \eqref{PDE_Intertwining} \text{ holds}\},
\label{SolSpace_zeta_Int}
\end{align}
\begin{equation}\label{PDE_Intertwining}
\widehat{d\pi_{(\xi,\lambda)^*}}(N_{j}^+)\psi = 0 \enspace \text{ for } j = 1, 2, 3,
\end{equation}
so that 
\begin{equation*}
\Sol(\xi, \xi^\prime; \lambda, \lambda^\prime) 
=\bigoplus_{\alpha \in \NN^2} \Sol^{(\alpha)}(\xi,\lambda).
\end{equation*}

Observe that,  as $N_3^+=[N_1^+, N_2^+]$,
we have 
\begin{equation*}
\Sol^{(\alpha)}(\xi,\lambda) = 
\Sol_1^{(\alpha)}(\xi,\lambda)
\cap
\Sol_2^{(\alpha)}(\xi,\lambda).
\end{equation*}
We put
\begin{align*}
\wSol_{\calP,\, \calP}^{(\alpha)}(\xi,\lambda)
&:=
\wSol_{\calP,\, 1}^{(\alpha)}(\xi,\lambda) \cap
\wSol_{\calP,\, 2}^{(\alpha)}(\xi,\lambda),\\
\vspace{5pt}
\wSol_{\Ad(e^{\wJ(\alpha)})(\calP,\, \calP)}^{(\alpha)}
(\xi,\lambda)
&:=
\wSol_{\Ad(e^{\wJ(\alpha)}) \calP,\, 1}^{(\alpha)}
(\xi,\lambda)
\cap
\wSol_{\Ad(e^{\wJ(\alpha)}) \calP,\, 2}^{(\alpha)}
(\xi,\lambda).
\end{align*}
The saturation map $T_\alpha$ 
together with the identities \eqref{eqn:P1} and \eqref{eqn:P2}
leads the following diagram for $(\xi,\xi';\lambda,\lambda')\in L_0(\alpha)$\begin{equation}\label{eqn:cdSol3b}
\begin{tikzcd}[column sep=1.5cm]
	{\wSol_{\Ad(e^{\wJ(\alpha)})(\calP,\, \calP)}^{(\alpha)}(\xi,\lambda)} & 
	\Sol_{\Ad(e^{J(\alpha)})}^{(\alpha)}(\xi,\lambda)  	
	\arrow[dl, pos=0.5, phantom, "\circlearrowleft"]
	&{\Diff_{G}\left(I(\xi, \lambda), I(\xi', \lambda')\right)} \\
	{\wSol_{\calP,\,\calP}^{(\alpha)}(\xi,\lambda)}  & {\Sol^{(\alpha)}(\xi,\lambda)} &
	\arrow[ul, pos=0.886, phantom, "\circlearrowleft"]
	\arrow["\Symb_{\ord}^{-1}", from=1-2, to=1-3]
	\arrow["\Symb_0^{-1}"', from=2-2, to=1-3]
	\arrow["{T_\alpha}", "\sim"' , from=1-1, to=1-2]
	\arrow["\sim" sloped, "e^{-\wJ(\alpha)}"', from=1-1, to=2-1]
	\arrow["e^{-J(\alpha)}", "\sim"' sloped,  from=1-2, to=2-2]
	\arrow["\sim" , "{T_\alpha}"',  from=2-1, to=2-2]
\end{tikzcd}
\end{equation}

Our initial goal is to determine 
the spaces 
$\wSol_{\calP,\, \calP}^{(\alpha)}(\xi,\lambda)$
and 
$\wSol_{\Ad(e^{\wJ(\alpha)})(\calP,\, \calP)}^{(\alpha)}(\xi,\lambda)$.
It follows from 
Propositions \ref{prop:SolD1} and \ref{prop:SolD2} and
Corollary \ref{cor:Sol2F0} that one needs to consider 
the following four cases:
\begin{enumerate}[label=\normalfont{(\arabic*)}]
\item[(A1B1)] Both (A1) and (B1) hold.
\vspace{5pt}
\item[(A1B2)] Both (A1) and (B2) hold.
\vspace{5pt}
\item[(A2B1)] Both (A2) and (B1) hold.
\vspace{5pt}
\item[(A2B2)] Both (A2) and (B2) hold.
\vspace{5pt}
\end{enumerate}
It is easy to see that (A2B2) cannot happen as it requires
$\alpha_1 > \alpha_2$ and $\alpha_1 < \alpha_2$. Therefore, it suffices to consider
the first three cases. 

We focus on $\wSol_{\Ad(e^{\wJ(\alpha)})(\calP,\, \calP)}^{(\alpha)}(\xi,\lambda)$;
the classification of $\wSol_{\calP,\, \calP}^{(\alpha)}(\xi,\lambda)$
is obtained by applying $e^{-J(\alpha)}$ to $\wSol_{\Ad(e^{\wJ(\alpha)})(\calP,\, \calP)}^{(\alpha)}(\xi,\lambda)$. 
By Corollary \ref{cor:Sol2F0} and Proposition \ref{prop:SolD2},
we wish to determine when the following identity holds:
\begin{equation}\label{eqn:2F0_2F0}
{}_2F_0[-\alpha_1, \tfrac{a_{0,1}-\alpha_2}{2};t]
=
{}_2F_0[-\alpha_2, \tfrac{a_{0,2}-\alpha_1}{2};t].
\end{equation}

For this purpose, we first show one lemma, which deals with
the equality \eqref{eqn:2F0_2F0} in a general situation.

\begin{lemma}\label{lem:FF}
For $n, m \in \N$ and $s, r \in \C$, the following two conditions 
on $(n,m, s,r)$ are equivalent.
\begin{enumerate}[label = \normalfont{(\roman*)}]
\item ${}_2F_0[-n,s;t] = {}_2F_0[-m,r;t]$.
\item The following two equalities hold:
\begin{enumerate}
\item[\textnormal{(ii-a)}] $sn=rm$;
\item[\textnormal{(ii-b)}] $sn(s-r+m-n)=0$.
\end{enumerate}
\end{enumerate}
\end{lemma}

\begin{proof}
Suppose that (i) holds. Then we have
\begin{equation}\label{eqn:nkmb}
(-n)_k(s)_k = (-m)_k(r)_k
\quad
\text{for all $k \in \N$}.
\end{equation}
Note that this equality holds trivially for $k>\max(n,m)$ as both sides are identically zero.
Thus, for $k=1$, we have $sn = rm$.
Similarly, for $k=2$, the equality \eqref{eqn:nkmb} amounts to
$sn(s-r+m-n)=0$, yielding the condition (ii).

In order to show the other implication, suppose that 
(ii-a) and (ii-b) hold. By (ii-b), we have $sn=0$ or $s+m=r+n$.
If $sn=0$, then, (ii-a) forces $rm=0$. Therefore, in this case, we have 
${}_2F_0[-n,s;t] =1 = {}_2F_0[-m,r;t]$. Next, assume that $s+m=r+n$.
Substitute $r=s+m-n$ into (ii-a) to obtain $sn= m(s+m-n)$, which is 
$(s+m)(n-m)=0$. Therefore, we have $s=-m$ or $n=m$.
Since the equality in (i) holds in either case, this concludes the lemma.
\end{proof}

Now we consider the equality \eqref{eqn:2F0_2F0} in each case. 
It follows from Lemma \ref{lem:FF} that the following conditions on 
$(\alpha_1, \alpha_2, a_{0,1}, a_{0,2})$ are equivalent.
\begin{enumerate}
\item[\textnormal{(i)}$'$] 
${}_2F_0[-\alpha_1, \tfrac{a_{0,1}-\alpha_2}{2};t]
=
{}_2F_0[-\alpha_2, \tfrac{a_{0,2}-\alpha_1}{2};t]
$.
\item[\textnormal{(ii)}$'$] 
The following two equalities hold:
\begin{enumerate}
\item[\textnormal{(ii-a)}$'$]
$\alpha_1(a_{0,1}-\alpha_2)=\alpha_2(a_{0,2}-\alpha_1)$;
\item[\textnormal{(ii-b)}$'$] 
$\alpha_1(a_{0,1}-\alpha_2)(a_{0,1}-a_{0,2}+\alpha_2-\alpha_1)=0$.
\end{enumerate}
\end{enumerate}

\begin{proposition}\label{prop:A1B1}
Suppose \textnormal{(A1B1)} holds.
Then, for $(\xi,\xi';\lambda,\lambda')\in L_0(\alpha)$,
we have 
\begin{equation*}
\wSol_{\Ad(e^{\wJ(\alpha)})(\calP,\, \calP)}^{(\alpha)}(\xi,\lambda)
=
\begin{cases}
\C &\textnormal{if $\alpha_1=\alpha_2=0$},\\[3pt]
\C\,{}_2F_0[-(\lambda_3-\lambda_1+2),\lambda_1-\lambda_2-1;t]
&\textnormal{if $\alpha_1=\alpha_2=\lambda_3-\lambda_1+2$},\\[3pt]
\{0\} & \textnormal{otherwise}.
\end{cases}
\end{equation*}
\end{proposition}

\begin{proof}
In this case we have $\alpha_1 \leq \alpha_2$ and 
$\alpha_2 \leq \alpha_1$, which implies
$\alpha_1=\alpha_2$. 
Then, condition (ii)$'$ is equivalent to $\alpha_1 (a_{0,1}-a_{0,2})=0$. Thus, $\alpha_1=0$ or 
$a_{0,1}=a_{0,2}$. By \eqref{eqn:a012}, 
the second condition is equivalent to 
$\alpha_1=\alpha_2 = \lambda_3-\lambda_1+2$.
This proves the proposition.
\end{proof}

\begin{proposition}\label{prop:A1B2_A2B1}
Suppose \textnormal{(A1B2)} holds.
Then, for $(\xi,\xi';\lambda,\lambda')\in L_0(\alpha)$,
we have 
\begin{equation*}
\wSol_{\Ad(e^{\wJ(\alpha)})(\calP,\, \calP)}^{(\alpha)}(\xi,\lambda)
=
\begin{cases}
\C&\textnormal{if $(\alpha_1,\alpha_2)=(0, \lambda_3-\lambda_2+1)$},\\[3pt]
\C&\textnormal{if $(\alpha_1,\alpha_2)=(a_{0,2},a_{0,1})$ with $\alpha_1 >0$},\\[3pt]
\C\,{}_2F_0[-\alpha_1,-\alpha_2;t]
&\textnormal{if $(\alpha_1,\alpha_2)=(-a_{0,2},-a_{0,1})$ with $\alpha_1 >0$},\\[3pt]
\{0\} & \textnormal{otherwise}.
\end{cases}
\end{equation*}
Similarly, if \textnormal{(A2B1)} holds, then  
\begin{equation*}
\wSol_{\Ad(e^{\wJ(\alpha)})(\calP,\, \calP)}^{(\alpha)}(\xi,\lambda)
=
\begin{cases}
\C&\textnormal{if $(\alpha_1,\alpha_2)=(\lambda_2-\lambda_1+1, 0)$},\\[3pt]
\C&\textnormal{if $(\alpha_1,\alpha_2)=(a_{0,2},a_{0,1})$ with $\alpha_2 >0$},\\[3pt]
\C\,{}_2F_0[-\alpha_1,-\alpha_2;t]
&\textnormal{if $(\alpha_1,\alpha_2)=(-a_{0,2},-a_{0,1})$ with $\alpha_2 >0$},\\[3pt]
\{0\} & \textnormal{otherwise}.
\end{cases}
\end{equation*}
\end{proposition}

\begin{proof}
We only show the case (A1B2); 
the other case can be shown similarly.
It is clear from (A1) and (B2) that the condition (A1B2) is the same as (B2), namely,
\begin{equation}\label{eqn:B2}
\alpha_2 > \alpha_1
\quad \text{and} \quad
 a_{0,2} \in \{\alpha_1-2k:k=0,1,2,\ldots, \alpha_1\}.
\end{equation}

Suppose that $\alpha_2 > \alpha_1=0$. In this case, (ii)$'$ is 
equivalent to $a_{0,2}=0$, which amounts to 
$\alpha_2=\lambda_3-\lambda_2+1$ by \eqref{eqn:a012}.

Next, suppose that $\alpha_2 > \alpha_1 > 0$.
In this case, (ii)$'$ is equivalent to satisfying the following two equalities:
\begin{enumerate}
\item[\textnormal{(ii-a)}$'$]
$a_{0,1}\alpha_1=a_{0,2}\alpha_2$;
\item[\textnormal{(ii-b)}$'$] 
$(a_{0,1}-\alpha_2)(a_{0,1}-a_{0,2}+\alpha_2-\alpha_1)=0$.
\end{enumerate}
By (ii-b)$'$, we have $a_{0,1}=\alpha_2$ or 
$a_{0,1}+\alpha_2=a_{0,2}+\alpha_1$.
If $a_{0,1}=\alpha_2$, then (ii-a)$'$ implies
$(\alpha_1,\alpha_2) = (a_{0,2}, a_{0,1})$.
Now assume that $a_{0,1}+\alpha_2=a_{0,2}+\alpha_1$.
As $a_{0,1}\alpha_1=a_{0,2}\alpha_2$, substitute 
$a_{0,1}=a_{0,2}\alpha_2/\alpha_1$ into 
$a_{0,1}+\alpha_2=a_{0,2}+\alpha_1$ to obtain 
$(a_{0,2}+\alpha_1)(\alpha_2-\alpha_1)=0$.
Since $\alpha_2>\alpha_1$, this shows that $a_{0,2}+\alpha_1=0$,
which yields $a_{0,1}+\alpha_2=0$. 
Therefore, in this case,
we obtain $(\alpha_1,\alpha_2)=(-a_{0,2},-a_{0,1})$.
This concludes the desired proposition.
\end{proof}

%%%%%%%%%%%%%%%%%%%%%%%%%%%%%%%%%%%%%%%%%
We now summarize the arguments to classify 
$\wSol_{\calP,\, \calP}^{(\alpha)}(\xi,\lambda)$
and 
$\wSol_{\Ad(e^{\wJ(\alpha)})(\calP,\, \calP)}^{(\alpha)}$.
To the end,
we introduce the following conditions on $(\alpha_1, \alpha_2, a_{0,1}, a_{0,2})$:
\begin{enumerate}[label=\normalfont{(\arabic*)}]
\item[(D0)] 
$\alpha_1=\alpha_2=0$;
\vspace{5pt}
\item[(D1)] 
$\alpha_2 \in \N_+$ and 
$(\alpha_1,\alpha_2)=(0,\lambda_3-\lambda_2+1)$;
\vspace{5pt}
\item[(D2)] 
$\alpha_1=\alpha_2=\lambda_3-\lambda_1+2 \in\N_+$;
\vspace{5pt}
\item[(D3)] 
$\alpha_1 \in \N_+$ and 
$(\alpha_1,\alpha_2)=(\lambda_2-\lambda_1+1,0)$;
\vspace{5pt}
\item[(D4)] 
$\alpha_1, \alpha_2 \in \N_+$ with $\alpha_1\neq \alpha_2$ and 
$(a_{0,1},a_{0,2})=(\alpha_2,\alpha_1)$;
\vspace{5pt}
\item[(D5)] 
$\alpha_1, \alpha_2 \in \N_+$ with $\alpha_1\neq \alpha_2$ 
and
$(a_{0,1},a_{0,2})=(-\alpha_2,-\alpha_1)$.
\end{enumerate}

\begin{proposition}\label{eqn:SolPP}
Let $(\xi,\xi';\lambda,\lambda')\in L_0(\alpha)$.
Then the solution spaces 
$\wSol_{\calP,\, \calP}^{(\alpha)}(\xi,\lambda)$
and
$\wSol_{\Ad(e^{\wJ(\alpha)})(\calP,\, \calP)}^{(\alpha)}
(\xi,\lambda)$
are classified as follows.
\begin{align}
\wSol_{\calP,\, \calP}^{(\alpha)}(\xi,\lambda)
&=
\begin{cases}
\CC\,
p_{a_{0,1},\alpha_2}^{(\alpha_1,\alpha_2)}(t) & 
\textnormal{if some (Dj) holds ($j=0,1,\ldots,5$),}\\[3pt]
\{0\} & \textnormal{otherwise,}
\end{cases}\label{eqn:SolCayPP}\\[5pt]
\wSol_{\Ad(e^{\wJ(\alpha)})(\calP,\, \calP)}^{(\alpha)}
(\xi,\lambda)
&=
\begin{cases}
\CC\,
{}_2F_0[-\alpha_1,\tfrac{a_{0,1}-\alpha_2}{2};t]
& 
\textnormal{if some (Dj) holds ($j=0,1,\ldots,5$),}\\[3pt]
\{0\} & \textnormal{otherwise.}
\end{cases}\label{eqn:Sol2F0PP}
\end{align}
More precisely, we have
\begin{align}
\wSol_{\calP,\, \calP}^{(\alpha)}(\xi,\lambda)
&=
\begin{cases}
\C & \textnormal{if (D0), (D1) or (D3) holds},\\[5pt]
\C p^{(\lambda_3-\lambda_1+2,\lambda_3-\lambda_1+2)}_{\lambda_1-2\lambda_2+\lambda_3,\lambda_3-\lambda_1+2}(t) & \textnormal{if (D2) holds},\\[5pt]
\C\,{}_2F_0[-\alpha_1,-\alpha_2;-\tfrac{t}{2}]& \textnormal{if (D4) holds},\\[5pt]
\C\,{}_2F_0[-\alpha_1,-\alpha_2;\tfrac{t}{2}]& \textnormal{if (D5) holds},\\[5pt]
\{0\} & \textnormal{otherwise},
\end{cases}\label{eqn:PP}\\[5pt]
\wSol_{\Ad(e^{\wJ(\alpha)})(\calP,\, \calP)}^{(\alpha)}
(\xi,\lambda)
&=
\begin{cases}
\C & \textnormal{if (D0), (D1) or (D3)  holds},\\[5pt]
\C\,{}_2F_0[-(\lambda_3-\lambda_1+2),\lambda_1-\lambda_2-1;t]& \textnormal{if (D2) holds},\\[5pt]
\C & \textnormal{if (D4) holds},\\[5pt]
\C\,{}_2F_0[-\alpha_1,-\alpha_2;t]& \textnormal{if (D5) holds},\\[5pt]
\{0\} & \textnormal{otherwise}.
\end{cases}\label{eqn:Ad_PP}
\end{align}
\end{proposition}

\begin{proof}
It follows from
Propositions \ref{prop:A1B1} and \ref{prop:A1B2_A2B1} 
that the solution spaces 
$\wSol_{\calP,\, \calP}^{(\alpha)}(\xi,\lambda)$
and
$\wSol_{\Ad(e^{\wJ(\alpha)})(\calP,\, \calP)}^{(\alpha)}
(\xi,\lambda)$
are non-zero if and only if  one of the conditions (Dj) for $j=0,1,\ldots, 5$ holds.
The formulas \eqref{eqn:SolCayPP} and \eqref{eqn:Sol2F0PP} follow
immediately from Proposition \ref{prop:SolD1} and Corollary \ref{cor:Sol2F0},
respectively, as 
$\wSol_{\calP,\, \calP}^{(\alpha)}(\xi,\lambda)
\subset 
\wSol_{\calP,\, 1}^{(\alpha)}(\xi,\lambda)$
and 
$\wSol_{\Ad(e^{\wJ(\alpha)})(\calP,\, \calP)}^{(\alpha)}
(\xi,\lambda)
\subset
\wSol_{\Ad(e^{\wJ(\alpha)}) \calP,\, 1}^{(\alpha)}(\xi,\lambda)$.

Regarding \eqref{eqn:PP} and \eqref{eqn:Ad_PP},
observe that 
\eqref{eqn:Ad_PP} is a direct consequence of 
Propositions \ref{prop:A1B1} and \ref{prop:A1B2_A2B1}.
The classification  \eqref{eqn:PP} is then obtained by
applying  $e^{-J(\alpha)}$ to \eqref{eqn:Ad_PP}.
Indeed, all the cases but (D5) in \eqref{eqn:PP} readily follow from
\eqref{eqn:p3F1} and Remark \ref{rem:3F1_2F0} (see also \eqref{eqn:eJ1} for (Dj) when $j=0,1,3,4$).
For (D5), since $(a_{0,1},a_{0,2})=(-\alpha_2,-\alpha_1)$, we have
\begin{equation*}
{}_2F_0[-\alpha_1, -\alpha_2;t]
={}_2F_0[-\alpha_1, \tfrac{a_{0,1}-\alpha_2}{2};t]
={}_2F_0[-\alpha_2, \tfrac{a_{0,2}-\alpha_1}{2};t].
\end{equation*}
It then follows from \eqref{eqn:p3F1}
and the identity 
$\Cay_n(-y,y)=(-y)_n$
of Cayley continuants
(see  Appendix \ref{appendix:Cayely})
that, 
for ${}_2F_0[-\alpha_1, -\alpha_2;t]
={}_2F_0[-\alpha_1, \tfrac{a_{0,1}-\alpha_2}{2};t]$, 
 we have 
\begin{align*}
e^{-J(\alpha)}{}_2F_0[-\alpha_1,-\alpha_2;t]
&=p_{-\alpha_2,\alpha_2}^{(\alpha)}(t)\\
&=\sum_{n=0}^{\min(\alpha)} 
\frac{(-\alpha_1)_n}{2^n n! }\Cay_n(-\alpha_2,\alpha_2)t^n\\
&=\sum_{n=0}^{\min(\alpha)} 
\frac{(-\alpha_1)_n(-\alpha_2)_n}{2^n n! }t^n\\
&={}_2F_0[-\alpha_1,-\alpha_2;\tfrac{t}{2}].
\end{align*}
It is easily seen that one can obtain the same identity even if 
we use the other equality
${}_2F_0[-\alpha_1, -\alpha_2;t]
={}_2F_0[-\alpha_2, \tfrac{a_{0,2}-\alpha_1}{2};t]$.
Now the proposition follows.
\end{proof}

We are now going to give a proof of
Theorems \ref{thm:class_DIOs} and \ref{thm:cons_DIOs}.

\begin{proof}[Proof of  Theorems \ref{thm:class_DIOs} and \ref{thm:cons_DIOs}]
The conditions (L0)--(L5) are easily obtained from (D0)--(D5) via the identity 
\eqref{eqn:ch42}. Then, by \eqref{eqn:ch41}, the equivalences in 
Theorem \ref{thm:class_DIOs} hold by Proposition \eqref{eqn:SolPP}. The explicit formulas in Theorem \ref{thm:cons_DIOs} (1) and (2) can be obtained by applying $\Symb_0^{-1}$
to \eqref{eqn:PP}, and $\Symb_{\ord}^{-1}$ to \eqref{eqn:Ad_PP}, respectively.
This completes the proof.
\end{proof}

%%%%%%%%%%%%%%%%%%%%%%%%%%%%%%%%%%%%%%%%%
\section{DSBOs \texorpdfstring{$\DD$}{D} for \texorpdfstring{$(SL(3,\R), SL(2,\R))$}{(SL(3,R),SL(2,R))}}
\label{sec:SL}

The aim of this section is to discuss the DSBOs $\DD$ for the pair 
$(SL(3,\R), SL(2,\R))$ by making use of 
the results for $(GL(3,\R), GL(2,\R))$. 
We achieve in Theorems \ref{thm:classSL12} and \ref{thm:classSL3} 
the classification of the parameters such that 
the space of DSBOs $\DD$ is non-zero.
We then show in Theorems \ref{thm:consSL12} and \ref{thm:consSL3} 
that any DSBOs $\DD$ for 
$(SL(3,\R), SL(2,\R))$ can be given as a linear combination of those for $(GL(3,\R), GL(2,\R))$.

%%%%%%%%%%%%%%%%%%%%%%%%%%%%%%%%%%%%%%%%%
\subsection{Notation for \texorpdfstring{$(SL(3,\R), SL(2,\R))$}{(SL(3,R),SL(2,R))}}
\label{sec:SL_notation}

Let $G_{SL}:=SL(3,\R)$ and $\frakg(\R)_{\fraksl}:=\fraksl(3,\R)$.
For any subgroup $H$ of $G=GL(3,\R)$, we write $H_{SL}:=H \cap G_{SL}$.
Recall that $P=MAN_+$ denotes the minimal parabolic subgroup for $G$
defined in Section \ref{sec:notation}.
Then $P_{SL}=M_{SL}A_{SL}N_+$ is 
a minimal  parabolic subgroup of $G_{SL}$.
Similarly, for any subalgebra $\frakh(\R) \subset \frakg(\R)=\mathfrak{gl}(3,\R)$, we write
$\frakh(\R)_{\fraksl}:=\frakh(\R)\cap \frakg(\R)_{\fraksl}$. Then
$\frakp(\R)_{\fraksl}=\fraka(\R)_{\fraksl}\oplus \frakn_+(\R)$
is the minimal parabolic subalgebra corresponding to $P_{SL}=M_{SL}A_{SL}N_+$.

For $\delta=(\delta_1,\delta_2) \in (\Z/2\Z)^2$ and $u=(u_1,u_2) \in \C^2$, 
write
\begin{equation}\label{eqn:delta_u}
\xi(\delta)
:=(\delta_1, 0, \delta_2)
\quad
\text{and}
\quad
\lambda(u)
:= \tfrac{1}{3}(2u_1+u_2,-u_1+u_2,-(u_1+2u_2)).
\end{equation}
We define characters $\chi_{M_{SL},\delta}$ and $\chi_{A_{SL},u}$
of $M_{SL}$ and $A_{SL}$ as 
\begin{equation*}
\chi_{M_{SL},\delta}:=\chi_{M,\xi(\delta)}\vert_{M_{SL}}
\quad
\text{and}
\quad
\chi_{A_{SL},u}:=\chi_{A,\lambda(u)}\vert_{A_{SL}},
\end{equation*}
where 
$\chi_{M,\xi}$ and $\chi_{A,\lambda}$ are 
the characters of $M$ and $A$ defined in 
\eqref{eqn:characterM}
and
\eqref{eqn:characterA}, respectively.
Then, for $M_{SL}=\{\diag(m_1,m_2, m_3) \in M: m_1m_2m_3=1\}$,
the character $\chi_{M_{SL},\delta}$ is given by
\begin{equation}\label{eqn:charSLM}
\chi_{M_{SL},\delta}\colon 
\diag(m_1, m_2, m_3)
\mapsto m_1^{\delta_1}m_3^{\delta_2}.
\end{equation}
For $A_{SL}=\exp(\fraka(\R)_{\fraksl})$, where 
$\fraka(\R)_{\fraksl}=\Span_\R\{H_{1,2}, H_{2,3}\}$ 
with $H_{i,j}:=E_{i,i}-E_{j,j}$, the character $\chi_{A_{SL},u}$ is given by
\begin{equation}\label{eqn:charSLA}
\chi_{A_{SL},u}\colon 
\exp(t_1H_{1,2} + t_2H_{2,3})
\mapsto
\exp(u_1t_1 + u_2t_2).
\end{equation}

For
$\C_\delta:=(\chi_{M_{SL},\delta},\C)$
and
$\C_u:=(\chi_{A_{SL},u},\C)$,  we write
\begin{equation*}
I(\delta, u)_{SL}:=\Ind_{P_{SL}}^{G_{SL}}(\C_{\delta}\boxtimes \C_u)
\end{equation*}
for the principal series representation of $G_{SL}$.

As for $G_{SL}=SL(3,\R)$, we put  $G'_{SL}:= SL(2,\R)$ and 
$\frakg'(\R)_{\fraksl}:=\mathfrak{sl}(2,\R)$.
For $H' \subset G'$ and $\frakh'(\R) \subset \frakg'(\R)$,
we write $H'_{SL}:=H' \cap G'_{SL}$ and $\frakh'(\R)_{\fraksl}:=\frakh'(\R) \cap \frakg'(\R)_{\fraksl}$.
Then $P'_{SL}=M'_{SL}A'_{SL}N_+'$ and $\frakp'(\R)_{\fraksl} =\fraka'(\R)_{\fraksl} \oplus \frakn'(\R)_{+}$ are the minimal parabolic subgroup of $G'_{SL}$ and its Lie algebra, respectively.

For $\sigma \in \Z/2\Z$ and $v \in \C$, we write
\begin{equation*}
\eta(\sigma):=(\sigma,0),
\quad
\text{and}
\quad
\nu(v):=\tfrac{1}{2}(v,-v).
\end{equation*}
We then define characters of $\chi_{M'_{SL},\sigma}$ and $\chi_{A'_{SL},v}$
of $M'_{SL}$ and $A'_{SL}$ as
\begin{equation*}
\chi_{M'_{SL},\sigma}:=\chi_{M',\eta(\sigma)}\vert_{M'_{SL}}
\quad
\text{and}
\quad
\chi_{A'_{SL},v}:=\chi_{A',\nu(v)}\vert_{A'_{SL}},
\end{equation*}
where $\chi_{M',\eta}$ and $\chi_{A',\nu}$ are the characters of 
$M'$ and $A'$ defined in \eqref{eqn:characterMprime} and 
\eqref{eqn:characterAprime}, respectively.
Then, for $M'_{SL}:=\{\diag(m,m):m \in \{\pm 1 \} \}$, the character 
$\chi_{M'_{SL},\sigma}$ is given by
\begin{equation}\label{eqn:charSLMprime}
\chi_{M'_{SL},\sigma}\colon \diag(m,m)\mapsto m^\sigma.
\end{equation}

Write $H'_{1,2}:=E_{1,1}-E_{2,2} \in \frakg'(\R)_{\fraksl}$ so that 
$A'_{SL}=\exp(\R H'_{1,2})$. Then the character
$\chi_{A',\nu}$ is given by
\begin{equation}\label{eqn:charSLAprime}
\chi_{A'_{SL},\nu}\colon 
\exp(tH'_{1,2})\mapsto \exp(tv).
\end{equation}

For 
$\C_\sigma:=(\chi_{M'_{SL},\sigma},\C)$
and
$\C_v: =(\chi_{A'_{SL}, v},\C)$,
we write 
\begin{equation*}
J(\sigma, v)_{SL}:=\Ind_{P'_{SL}}^{G'_{SL}}(\C_{\sigma}\boxtimes \C_v)
\end{equation*}
for the principal series representation of $G'_{SL}$.

For the embedding 
$\iota_i\colon G' \hookrightarrow G$ 
defined in \eqref{eqn:embedding} for $i=1,2,3$, 
we write $G'_{SL,i} :=\iota_i(G'_{SL})\simeq SL(2,\R)$
and $\frakg'(\R)_{\fraksl,i}:=\iota_i(\frakg'(\R)_{\fraksl})$.
The same convention is employed to the subgroups of $G'_{SL}$
and subalgebras of $\frakg'(\R)_{\fraksl}$.
Then $P'_{SL,i}=M'_{SL,i}A'_{SL,i}N_{i,+}'$ and 
$\frakp'(\R)_{\fraksl,i} =\fraka'(\R)_{\fraksl,i} \oplus \frakn'(\R)_{i,+}$
are the minimal parabolic subgroup of $G'_{SL,i}$
and minimal parabolic subalgebra of $\frakg'(\R)_{\fraksl,i}$,
respectively.

Via the embedding $\iota_i\colon G' \hookrightarrow G$, 
we apply $\chi_{M'_{SL},\sigma}$ and $\chi_{A'_{SL},v}$
to $M'_{SL,i}$ and $A'_{SL,i}$, respectively.
By slight abuse of notation, we regard $J(\sigma, v)_{SL}$ as 
a representation of $G'_{SL,i}$.
Then we wish to classify and construct DSBOs
$\DD \in \Diff_{G'_{SL,i}}(I(\delta,u)_{SL}, J(\sigma,v)_{SL})$.

%%%%%%%%%%%%%%%%%%%%%%%%%%%%%%%%%%%%%%%%%
\subsection{Classification and construction of DSBOs \texorpdfstring{$\DD$}{D} 
for \texorpdfstring{$(SL(3,\R), SL(2,\R))$}{(SL(3,R),SL(2,R))}}\label{sec:class-const-SL}

Now we are going to state the main results of this section.

As always, we start with the classification of the parameters 
$(\delta,\sigma; u, v) \in (\Z/2\Z)^3\times \C^3$ such that
$\Diff_{G'_{SL,i}}(I(\delta,u)_{SL}, J(\sigma,v)_{SL})\neq \{0\}$.
For $i=1,2,3$, define
\begin{equation*}
\Supp_{SL,i}(\DD):=\{(\delta,\sigma; u, v) \in (\Z/2\Z)^3\times \C^3: 
\text{(S$i$)  holds}  \},
\end{equation*}
\begin{alignat}{5}
&u_1-v &&\in \Z,
\qquad
&&\delta_1+\sigma
&&\equiv u_1-v &&\bmod{2},
\tag{S1}\\
&u_2-v &&\in \Z,
\qquad
&&\delta_2+\sigma &&\equiv u_2-v &&\bmod{2},
\tag{S2}\\
(u_1+&u_2)-v &&\in -\N,
\qquad
(\delta_1+&&\delta_2)+\sigma &&\equiv (u_1+u_2)-v &&\bmod{2}.
\tag{S3}
\end{alignat}

\begin{theorem}\label{thm:classSL12}
For $i=1,2$, the following conditions on $(\delta,\sigma; u, v)  
\in (\Z/2\Z)^3\times \C^3$ are equivalent.
\begin{enumerate}[label=\normalfont{(\roman*)}]
\item $\Diff_{G'_{SL,i}}(I(\delta,u)_{SL}, J(\sigma,v)_{SL})\neq \{0\}$.
\item $\dim_\C\Diff_{G'_{SL,i}}(I(\delta,u)_{SL}, J(\sigma,v)_{SL})=\infty$.
\item $(\delta,\sigma; u, v) \in \Supp_{SL,i}(\DD)$.
\end{enumerate}
\end{theorem}

\begin{theorem}\label{thm:classSL3}
For $i=3$, the following conditions on 
$(\delta,\sigma; u, v)  \in (\Z/2\Z)^3\times \C^3$ are equivalent.
\begin{enumerate}[label=\normalfont{(\roman*)}]
\item $\Diff_{G'_{SL,3}}(I(\delta,u)_{SL}, J(\sigma,v)_{SL})\neq \{0\}$.
\item $(\delta,\sigma; u, v) \in \Supp_{SL,3}(\DD)$.
\end{enumerate}
Moreover, for $(\delta,\sigma; u, v)  \in \Supp_{SL,3}(\DD)$ with $j=-(u_1+u_2-v)\in\N$,
we have
\begin{equation}\label{eqn:DIffDimSL3}
\dim_\C\Diff_{G'_{SL,3}}(I(\delta,u)_{SL}, J(\sigma,v)_{SL}) \geq j+1.
\end{equation}
In addition,  if $v \notin \N$, then 
\begin{equation}\label{eqn:DIffDimSL3b}
\dim_\C\Diff_{G'_{SL,3}}(I(\delta,u)_{SL}, J(\sigma,v)_{SL}) = j+1.
\end{equation}
\end{theorem}

\begin{remark}
In Proposition \ref{prop:dim_Diff_SL3}, 
we shall discuss the conditions for \eqref{eqn:DIffDimSL3b} to hold in more detail.
\end{remark}

We shall give a proof of Theorems \ref{thm:classSL12} and \ref{thm:classSL3} 
in Section \ref{sec:ProofSL}. 
Here is an immediate consequence of Theorem \ref{thm:classSL3}.

\begin{corollary}\label{cor:DiffDimSL3}
We have 
\begin{equation*}
\sup_{u\in \C^2}\sup_{v \in \C}
\dim_\C\Diff_{G'_{SL,3}}(I(\delta,u)_{SL}, J(\sigma,v)_{SL}) =\infty.
\end{equation*}
\end{corollary}

\begin{proof}
Since any $ k \in \N$ can be expressed as 
$k=-(u_1+u_2-v)$ for $(u,v)$ such that $(\delta,\sigma; u, v) \in \Supp_{SL,3}(\DD)$,
the desired assertion follows immediately from 
\eqref{eqn:DIffDimSL3}.
\end{proof}

\begin{remark}
Theorem 2.3 (2) of \cite{Kobayashi14} asserts that 
the space of DSBOs for the $SL$-pair
$(SL(3,\R), SL(2,\R))$
is uniformly bounded.
However, as shown in Theorem \ref{thm:classSL12} and 
Corollary \ref{cor:DiffDimSL3}, the space of DSBOs is in fact
not uniformly bounded. Indeed, the condition (ii) of the cited theorem
about the existence of the open orbit does not hold for the $SL$-pair.
\end{remark}

We next consider the construction of DSBOs
$\DD \in \Diff_{G'_{SL,i}}(I(\delta,u)_{SL}, J(\sigma,v)_{SL})$.
For this, define $\Gamma_i(\ell) \subset \N^2$ for $ \ell \in \Z$ as 
\begin{align}\label{eqn:Gamma_ell}
\Gamma_1(\ell)&:=\{(\alpha_1,\alpha_2) \in \N^2: -2\alpha_1+\alpha_2=\ell\},
\nonumber\\[3pt]
\Gamma_2(\ell)&:=\{(\alpha_1,\alpha_2) \in \N^2: \alpha_1-2\alpha_2=\ell\},\\[3pt]
\Gamma_3(\ell)&:=\{(\alpha_1,\alpha_2) \in \N^2: \alpha_1+\alpha_2=-\ell\}.
\nonumber
\end{align}

\begin{remark}\label{rem:GammaSL}
For $i=1,2$, we have 
\begin{equation}\label{eqn:GammaSL12}
\#\Gamma_i(\ell)=\infty \quad \text{for any $\ell \in \Z$}.
\end{equation}
On the other hand, for $i=3$, we have
\begin{equation}\label{eqn:GammaSL3}
\#\Gamma_3(\ell)=
\begin{cases}
0  &\text{if $\ell \in \N_+$},\\
\ell+1  &\text{if $\ell \in -\N$}.
\end{cases}
\end{equation}
\end{remark}

As above, we treat the cases $i=1,2$ and $i=3$, separately.

%%%%%%%%%%%%%%%%%%%%%%%%%%%%%%%%%%%%%%%%%
\subsubsection{Case \texorpdfstring{$i=1,2$}{i=1,2}}
\label{sec:SL_Case12}

As for $a_{0,1}, a_{0,2}$ in \eqref{eqn:a012},
for $u=(u_1, u_2) \in \C^2$ and $\lambda(u)$ in \eqref{eqn:delta_u}, 
we define
\begin{alignat*}{2}
a_{0,1}(u)
&:=a_{0,1}(\lambda(u);\alpha_1,\alpha_2)
&&=2(u_1-1+\alpha_1-\tfrac{1}{2}\alpha_2),
\\
a_{0,2}(u)
&:=a_{0,2}(\lambda(u);\alpha_1,\alpha_2)
&&= -2(u_2-1-\tfrac{1}{2}\alpha_1+\alpha_2).
\end{alignat*}
Then we consider the following conditions:

\vspace{5pt}

\begin{enumerate}[label=\normalfont{(\arabic*)}]
\item[(U1-1)] $\alpha_1\leq \alpha_2$;
\vspace{5pt}
\item[(U1-2)] $\alpha_1>\alpha_2$ and 
$a_{0,1}(u)\in \{\alpha_2 - 2k: k=0,1,2,\ldots, \alpha_2\}$.
\end{enumerate}

\vspace{5pt}

\begin{enumerate}[label=\normalfont{(\arabic*)}]
\item[(U2-1)] $\alpha_2\leq \alpha_1$;
\vspace{5pt}
\item[(U2-2)] $\alpha_2>\alpha_1$ and 
$a_{0,2}(u)\in \{\alpha_1 - 2k: k=0,1,2,\ldots, \alpha_1\}$.
\end{enumerate}

\vspace{5pt}

Define
\begin{align*}
\DD_{SL,1}(u;\alpha)
&:=\Rest_1\circ
\big(T^{\mathrm{sym}}_{(\alpha_1,\alpha_2)}
p_{a_{0,1}(u),\,\alpha_2}^{(\alpha_1,\alpha_2)}\big)(D_1,D_2,D_3)\\[5pt]
&=\Rest_1\circ 
T^\ord_{(\alpha_1,\alpha_2)}{}_2F_0
[-\alpha_1, \tfrac{a_{0,1}(u)-\alpha_2}{2};  D_1, D_2, D_3],\\[10pt]
\DD_{SL,2}(u;\alpha)
&:=\Rest_2\circ 
\big(T^{\mathrm{sym}}_{(\alpha_2\,\alpha_1)}
p_{a_{0,2}(u),\,\alpha_1}^{(\alpha_2,\alpha_1)}\big)(D_1,D_2,D_3) \\[5pt]
&=\Rest_2\circ  
T^\ord_{(\alpha_2,\,\alpha_1)}
{}_2F_0[-\alpha_2, \tfrac{a_{0,2}(u)-\alpha_1}{2};D_1, D_2, D_3].
\end{align*}

\begin{theorem}\label{thm:consSL12}
Let $i \in \{1, 2\}$ and suppose that $(\delta,\sigma; u, v) \in \Supp_{SL,i}(\DD)$.
If $\ell:=u_i-v \in \N$, then
\begin{equation*}
\Diff_{G'_{SL,i}}(I(\delta,u)_{SL}, J(\sigma,v)_{SL})
=\bigoplus_{\alpha \in\Gamma_i(\ell)}
\Diff_i(u;\alpha),
\end{equation*}
where
\begin{equation*}
\Diff_i(u;\alpha)
=
\begin{cases}
\C\, \DD_{SL,i}(u;\alpha)& \textnormal{if (U$i$-1) or (U$i$-2) holds},\\
0 & \textnormal{otherwise}.
\end{cases}
\end{equation*}
\end{theorem}

%%%%%%%%%%%%%%%%%%%%%%%%%%%%%%%%%%%%%%%%%
\subsubsection{Case \texorpdfstring{$i=3$}{i=3}}
\label{sec:SL_Case3}

As for $a_0, b_0$ in  \eqref{eqn:ab}, we define
\begin{align*}
a_0(u) &:=a_0(\lambda(u);\alpha_1,\alpha_2)= u_1-u_2+\alpha_1-\alpha_2, \\[3pt]
b_0(u) &:=b_0(\lambda(u);\alpha_1,\alpha_2)
= u_1+u_2-2+\alpha_1+\alpha_2.
\end{align*} 
Then we consider the following conditions.
\vskip 0.1in
\begin{enumerate}[label=\normalfont{(\arabic*)}]
\item[(V1)] $b_0 (u)\notin [0,\min(\alpha)-1]\cap \Z$;
\vspace{2pt}
\item[(V2)] $b_0(u) \in [0,\min(\alpha)-1]\cap \Z$ and 
$\frac{a_0(u)-b_0(u)}{2} \notin [-b_0(u),0]\cap \Z$;
\vspace{2pt}
\item[(V3)] $b_0(u)\in [0,\min(\alpha)-1]\cap \Z$ 
and $\frac{a_0(u)-b_0(u)}{2}\in [-b_0(u),0]\cap \Z$.
\end{enumerate}

\vspace{5pt}

Define
\vspace{5pt}
\begin{align*}
\DD^{(1)}_{SL,3}(u;\alpha)
&:=\Rest_3\circ
\big(T^{\mathrm{sym}}_{(\alpha_1,\alpha_2)} 
p^{(\alpha_1,\alpha_2)}_{a_0(u), \,b_0(u)}\big)
(D_1, D_2, D_3)\\[5pt]
&=\Rest_3\circ 
T^\ord_{(\alpha_1,\alpha_2)}{}_3F_1
\left[ \begin{matrix}
    -\alpha_1& -\alpha_2 & \tfrac{a_0(u)-b_0(u)}{2}\\
     &-b_0(u)&  
\end{matrix}; D_1, D_2, D_3\right]\\[10pt]
\DD^{(2)}_{SL,3}(u;\alpha)
&:=\Rest_3\circ \big(T^{\mathrm{sym}}_{(\alpha_1,\alpha_2)} 
t^{1+b_0(u)}p^{(\alpha_1-b_0(u)-1, \,\alpha_2-b_0(u)-1)}_{a_0(u),-b_0(u)-2}\big)(D_1, D_2, D_3) \\[5pt]
&=\Rest_3\circ (\star\star),
\end{align*}
where
\begin{equation*}
(\star\star):=
D_3^{1+b_0(u)}
T^\ord_{(\alpha_1,\,\alpha_2)}{}_3F_1\Big [\begin{matrix}
    b_0(u)+1-\alpha_1 & b_0(u)+1-\alpha_2 & \tfrac{a_0(u)+b_0(u)}{2}+1\\
    & 2+b_0(u)& 
\end{matrix}; D_1, D_2, D_3\Big ].
\end{equation*}

\begin{theorem}\label{thm:consSL3}
Let $i=3$ and suppose that $(\delta,\sigma; u, v) \in \Supp_{SL,3}(\DD)$.
If $-j:=u_1+u_2-v\in -\N$, then
\begin{equation*}
\Diff_{G'_{SL,3}}(I(\delta,u)_{SL}, J(\sigma,v)_{SL})
=\bigoplus_{\alpha\in\Gamma_3(-j)}
\Diff_3(u;\alpha),
\end{equation*}
where
\begin{equation*}
\Diff_3(u;\alpha)
=
\begin{cases}
\C\, \DD^{(1)}_{SL,3}(u;\alpha) & \textnormal{if (V1) holds},\\[3pt]
\C\, \DD^{(2)}_{SL,3}(u;\alpha) & \textnormal{if (V2) holds},\\[3pt]
\C\, \DD^{(1)}_{SL,3}(u;\alpha) \oplus 
\C\, \DD^{(2)}_{SL,3}(u;\alpha)
& \textnormal{if (V3) holds}.
\end{cases}
\end{equation*}
\end{theorem}

\begin{remark}
It follows from Theorems \ref{thm:consSL12} and \ref{thm:consSL3} that 
any DSBO for the $SL$-case can be given as a linear combination of
DSBOs for the $GL$-case.
\end{remark}

As for Theorems \ref{thm:classSL12} and \ref{thm:classSL3},
we shall prove Theorems \ref{thm:consSL12}  and \ref{thm:consSL3} 
in Section \ref{sec:ProofSL}.

%%%%%%%%%%%%%%%%%%%%%%%%%%%%%%%%%%%%%%%%%
\subsection{\texorpdfstring{$M_{SL,i}'A_{SL,i}'$}{M_{SL,i}'A_{SL,i}'}-representations on \texorpdfstring{$\Pol(\frakn_+)$}{Pol(n_+)}}
\label{sec:SL_notation1}

The rest of this section is devoted to proving Theorems \ref{thm:classSL12}, \ref{thm:classSL3},
\ref{thm:consSL12},
and
\ref{thm:consSL3}.
As for the $GL$-case,
we first compute the $M_{SL,i}'A_{SL,i}'$-decomposition of $\Pol(\frakn_+)$.

Write
\begin{equation*}
h_{1}:=\diag(t, t^{-1},1),
\quad
h_{2}:=\diag(1,t, t^{-1}),
\quad
h_{3}:=\diag(t, 1, t^{-1}).
\end{equation*}
Table \ref{table:Adzeta} summarizes 
the action of $h_{i}$ on $\zeta_k$ via $\Ad_{\#}$. 
\begin{table}[t]
\caption{}
\begin{center}
\renewcommand{\arraystretch}{1.3}
{
\begin{tabular}{c|c|c|c}
\hline
$\Ad_{\#}(h_i)\zeta_k$ & $\zeta_1$ & $\zeta_2$ & $\zeta_3$\\
\hline
$h_1$ & $t^{- 2}\zeta_1$ & $t\zeta_2$ & $t^{-1}\zeta_3$\\
\hline
$h_2$ & $t\zeta_1 $ & $t^{-2}\zeta_2$ & $t^{-1} \zeta_3$\\
\hline
$h_3$ & $t^{-1} \zeta_1$ & $t^{-1} \zeta_2$ & $t^{- 2}\zeta_3$\\
\hline
\end{tabular}
}
\end{center}\label{table:Adzeta}
\end{table}
For $\ell \in \Z$ and $i=1,2,3$, we put
\begin{equation*}
\Pol_{i}(\ell):=\{\psi(\zeta) \in \Pol(\frakn_+): \Ad_{\#}(h_i)\psi(\zeta) = t^\ell \psi(\zeta)\}.
\end{equation*}
Then, it follows from Table \ref{table:Adzeta} that we have
\begin{equation}\label{eqn:Pol_ell}
\Pol(\frakn_+)\big\vert_{M'_{SL,i}A'_{SL,i}}=\bigoplus_{\ell \in \Z} \Pol_{i}(\ell).
\end{equation}
Now, for $(\delta,\sigma; u, v) \in (\Z/2\Z)^3\times \C^3$ and $i=1,2,3$, we put
\begin{align*}
\Hom_{SL,i}(\delta,\sigma; u, v)
&:=\Hom_{M_{SL,i}'A_{SL,i}'}\big((\C_\delta\boxtimes \C_{u})
\big\rvert_{M_{SL,i}'A_{SL,i}'},\Pol(\frakn_+)\otimes (\C_\sigma\boxtimes \C_{v})\big),
\\
\Hom^{(\ell)}_{SL,i}(\delta; u)
&:=\Hom_{M_{SL,i}'A_{SL,i}'}\big((\C_\delta\boxtimes \C_{u})
\big\rvert_{M_{SL,i}'A_{SL,i}'},\Pol_i(\ell)\otimes (\C_\sigma\boxtimes \C_{v})\big).
\end{align*}
By the $M_{SL,i}'A_{SL,i}'$-decomposition \eqref{eqn:Pol_ell}, we have 
\begin{equation}\label{eqn:HomSL}
\Hom_{SL,i}(\delta,\sigma; u, v)
=\bigoplus_{\ell \in \Z} \Hom^{(\ell)}_{SL,i}(\delta; u).
\end{equation}
Further, for $\ell  \in \Z$ and $i=1,2,3$, we set
\begin{equation*}
L_{SL,i}(\ell):=\{(\delta,\sigma; u, v) \in (\Z/2\Z)^3\times \C^3: 
\text{(T$i$)  holds}  \},
\end{equation*}
\begin{alignat}{4}
&u_1-v &&= \ell,
\quad
&&\delta_1+\sigma
&&\equiv \ell \bmod{2},
\tag{T1}\\
&u_2-v &&= \ell,
\quad
&&\delta_2+\sigma &&\equiv \ell \bmod{2},
\tag{T2}\\
(u_1+&u_2)-v &&= \ell,
\quad
(\delta_1+&&\delta_2)+\sigma &&\equiv \ell \bmod{2}.
\tag{T3}
\end{alignat}

\begin{proposition}\label{prop:Invariants_SL} 
For $i=1,2,3$, 
the following conditions on 
$(\delta,\sigma; u, v) \in (\Z/2\Z)^3\times \C^3$
are equivalent.
\begin{enumerate}[label=\normalfont{(\roman*)}]
    \item $\Hom_{SL,i}(\delta,\sigma; u, v)\neq \{0\}$.
    \item There exists some $\ell\in \Z$ such that 
    $(\delta,\sigma; u, v) \in L_{SL, i}(\ell)$.
\end{enumerate}
\end{proposition}

\begin{proof}
We have 
\begin{equation*}
(\C_\delta\boxtimes \C_{u})
\big\rvert_{M_{SL,i}'A_{SL,i}'}
\simeq
\begin{cases}
\C_{\delta_1}\boxtimes \C_{u_1} &\text{if $i=1$}, \\[3pt]
\C_{\delta_2}\boxtimes \C_{u_2} &\text{if $i=2$}, \\[3pt]
\C_{\delta_1+\delta_2}\boxtimes \C_{u_1+u_2} &\text{if $i=3$}.
\end{cases}
\end{equation*}
Now the desired assertion follows.
\end{proof}

Recall from \eqref{eqn:q_alpha_k} that
we write
$\psi_{\alpha, k}(\zeta)=\zeta_1^{\alpha_1-k}\zeta_1^{\alpha_2-k}\zeta_3^{k}$
for $\alpha=(\alpha_1, \alpha_2) \in \N^2$ and 
$k \in [0,\min(\alpha_1,\alpha_2)] \cap \Z$.
By Table \ref{table:Adzeta}, we have 
\begin{equation*}
\Ad_{\#}(h_i)\psi_{\alpha, k}(\zeta)
=
\begin{cases}
t^{-2\alpha_1+\alpha_2}\psi_{\alpha, k}(\zeta) &\text{if $i=1$},\\[3pt]
t^{\alpha_1-2\alpha_2}\psi_{\alpha, k}(\zeta) &\text{if $i=2$},\\[3pt]
t^{-(\alpha_1+\alpha_2)}\psi_{\alpha, k}(\zeta) &\text{if $i=3$}.
\end{cases}
\end{equation*}
Therefore,
\begin{equation*}
\Pol(\alpha)\big\vert_{M_{SL,i}'A_{SL,i}'}
\simeq 
\begin{cases}
\bigoplus_{k =0}^{\min(\alpha)}
\C_{2\alpha_1+\alpha_2}\boxtimes
\C_{-2\alpha_1+\alpha_2} &\text{if $i=1$},\\[10pt]
\bigoplus_{k =0}^{\min(\alpha)}
\C_{\alpha_1+2\alpha_2}\boxtimes
\C_{\alpha_1-2\alpha_2} &\text{if $i=2$}, \\[10pt]
\bigoplus_{k =0}^{\min(\alpha)}
\C_{\alpha_1+\alpha_2}\boxtimes
\C_{-(\alpha_1+\alpha_2)} &\text{if $i=3$}.
\end{cases}
\end{equation*}

In the next subsection, we investigate $\Pol(\alpha)$ more carefully to understand
the polynomial solutions to the F-system for $(SL(3,\R), SL(2,\R))$.

%%%%%%%%%%%%%%%%%%%%%%%%%%%%%%%%%%%%%%%%%
\subsection{The F-system for \texorpdfstring{$(SL(3,\R), SL(2,\R))$}{(SL(3,R),SL(2,R))}}
\label{sec:DiffEqSL}

We next proceed to solve the F-system for $(SL(3,\R), SL(2,\R))$.
For $(\delta,\sigma; u, v) \in (\Z/2\Z)^3\times \C^3$, $\ell \in \Z$ and $i=1,2,3$, 
we define 
\begin{align*}
\Sol_{SL,i}(\delta,\sigma; u, v)
&:= \{ \psi \in \Hom_{SL,i}(\delta,\sigma; u, v): \text{\eqref{PDE_SL_123} holds}\},\\[3pt]
\Sol_{SL,i}^{(\ell)}(\delta,u) 
&:= \{ \psi \in \Hom^{(\ell)}_{SL,i}(\delta; u):\text{\eqref{PDE_SL_123} holds}\}.
\end{align*}
\begin{equation}\label{PDE_SL_123}
-\zeta_i\widehat{d\pi_{\lambda(u)^*}}(N_i^+)\psi = 0
\end{equation}
It follows from \eqref{eqn:HomSL} that 
\begin{equation}\label{eqn:SolDecompSL}
\Sol_{SL,i}(\delta,\sigma; u, v) 
= \bigoplus_{\ell\in \Z}\Sol_{SL,i}^{(\ell)}(\delta,u).
\end{equation}
Moreover, we have
\begin{equation}\label{eqn:PolSL}
\Pol_i(\ell)
=
\bigoplus_{\alpha \in \Gamma_i(\ell)}
\Pol(\alpha),
\end{equation}
where $\Gamma_i(\ell)$ is the set defined in \eqref{eqn:Gamma_ell}.
Then, via the identification $\Hom_{SL,i}^{(\ell)}(\delta,u) \subset \Pol_i(\ell)$, 
for $\alpha \in \Gamma_i(\ell)$, we define
\begin{equation*}
\Sol_{SL,i}^{(\ell)}(\delta, u;\alpha)
:=\{\psi \in \Hom_{SL,i}^{(\ell)}(\delta,u): 
\text{$\psi \in \Pol(\alpha)$ and \eqref{PDE_SL_123} holds}\}.
\end{equation*}

\begin{proposition}\label{prop:SolSLi}
For $i =1,2,3$, we have
\begin{equation*}
\Sol_{SL,i}^{(\ell)}(\delta,u)
=\bigoplus_{\alpha \in \Gamma_i(\ell)} 
\Sol_{SL,i}^{(\ell)}(\delta, u;\alpha).
\end{equation*}
\end{proposition}

\begin{proof}
By definition, it is clear that
$\Sol_{SL,i}^{(\ell)}(\delta,u)  \supset  \Sol_{SL,i}^{(\ell)}(\delta, u;\alpha)$
for each $\alpha \in \Gamma_i(\ell)$. To prove the converse,
take $\psi(\zeta) \in \Sol_{SL,i}^{(\ell)}(\delta,u)$. 
By \eqref{eqn:PolSL}, there exist 
$\psi^{(\alpha)}(\zeta) \in \Pol(\alpha)$ such that
\begin{equation*}
\psi(\zeta) = \sum_{\alpha \in \Gamma_i(\ell)}\psi^{(\alpha)}(\zeta).
\end{equation*}
Thus,
\begin{equation}\label{eqn:SolSLi}
0=-\zeta_i\widehat{d\pi_{\lambda(u)^*}}(N_i^+)\psi(\zeta) 
= \sum_{\alpha \in \Gamma_i(\ell)}
-\zeta_i\widehat{d\pi_{\lambda(u)^*}}(N_i^+)\psi^{(\alpha)}(\zeta).
\end{equation}
It follows from Proposition \ref{prop:frakn-mod} that 
$-\zeta_i\widehat{d\pi_{\lambda(u)^*}}(N_i^+)\psi^{(\alpha)}(\zeta) \in \Pol(\alpha)$ 
for all $\alpha \in \N^2$. Therefore, the equation \eqref{eqn:SolSLi} is 
equivalent to
\begin{equation*}
-\zeta_i\widehat{d\pi_{\lambda(u)^*}}(N_i^+)\psi^{(\alpha)}(\zeta)=0
\quad
\text{for all $\alpha \in \Gamma_i(\ell)$},
\end{equation*}
which is, $\psi^{(\alpha)}(\zeta) \in \Sol_{SL,i}^{(\ell)}(\delta, u;\alpha)$.
This completes the proof.
\end{proof}

For $u=(u_1, u_2) \in \C^2$ and $v \in \C$, we define 
\begin{equation*}
\ell(i;u,v)
:=
\begin{cases}
u_i-v & \text{if $i=1,2$},\\
(u_1+u_2)-v   & \text{if $i=3$}.
\end{cases}
\end{equation*}

\begin{proposition}\label{prop:DSBO_SL}
For $\ell(i;u,v) \in \Z$,  
the space $\Sol_{SL,i}(\delta,\sigma; u, v)$ can be decomposed as
\begin{equation}\label{SolSL2}
\Sol_{SL,i}(\delta,\sigma; u, v)
=\bigoplus_{\alpha\in\Gamma_i(\ell(i;u,v))}
\Sol_{SL,i}^{(\ell(i;u,v))}(\delta, u;\alpha).
\end{equation}
\end{proposition}

\begin{proof}
The proposed identity readily follows from Propositions
\ref{prop:Invariants_SL} and \ref{prop:SolSLi}.
\end{proof}

By applying $\Rest_i \circ \Symb_0^{-1}$ to both sides of \eqref{SolSL2},
we have 
\begin{equation}\label{eqn:DiffSolSL}
\Diff_{G'_{SL,i}}(I(\delta,u)_{SL}, J(\sigma,v)_{SL})
=\bigoplus_{\alpha\in\Gamma_i(\ell(i;u,v))}
(\Rest_i \circ \Symb_0^{-1})(\Sol_{SL,i}^{(\ell(i;u,v))}(\delta, u;\alpha)).
\end{equation}

In the next subsection, we utilize \eqref{eqn:DiffSolSL} to prove 
Theorems \ref{thm:classSL12} and \ref{thm:classSL3} for classification, 
and Theorems \ref{thm:consSL12} and\ref{thm:consSL3} for construction.

%%%%%%%%%%%%%%%%%%%%%%%%%%%%%%%%%%%%%%%%%
\subsection{Proofs of Theorems  
\texorpdfstring{\ref{thm:classSL12}}{11.6}, 
\texorpdfstring{\ref{thm:classSL3}}{11.7},
\texorpdfstring{\ref{thm:consSL12}}{11.17} 
and
\texorpdfstring{\ref{thm:consSL3}}{10.18}}\label{sec:ProofSL}

We start with a proof of Theorems \ref{thm:classSL12} and \ref{thm:classSL3}.

\begin{proof}[Proof of Theorem \ref{thm:classSL12}]
We only demonstrate the case $i=1$; the other case can be shown similarly.
We first show the equivalence of (i) and (iii).
By the definitions of $L_{SL,i}(\ell)$ and $\Supp_{SL,i}(\DD)$, 
and \eqref{eqn:DiffSolSL}, it suffices to show that,
for $\ell(1;u,v)=u_1-v \in \Z$,
there exists $\alpha \in \Gamma_1((\ell(1;u,v))$ such that
$\Sol_{SL,1}^{(\ell(1;u,v))}(\delta, u;\alpha)) \neq \{0\}$.
Indeed, for  $u_1-v \in \Z$, 
take $\alpha=(\alpha_1, \alpha_2)$ such that
$\alpha_1 :=|u_1-v|$ and 
$\alpha_2:=u_1-v+2\alpha_1$. 
Then we have
$\alpha \in \Gamma_1(\ell(1;u,v))$
and $\alpha_1 \leq \alpha_2$.
As the inequality $\alpha_1 \leq \alpha_2$ holds, 
the condition (A1) in Section \ref{sec:Thms12B} shows that 
$\Sol_{SL,1}^{(\ell(1;u,v))}(\delta, u;\alpha)) \neq \{0\}$.
This proves the equivalence of (i) and (iii).

Furthermore, for $j\in \N$, we define $\alpha(j):=(\alpha_1(j), \alpha_2(j))$
as $\alpha_1(j):=\alpha_1-j$ and $\alpha_2(j):=\alpha_2+2j$.
Then we have
$\alpha_1(j) \leq \alpha_2(j)$ and 
$\alpha(j) \in \Gamma_1(\ell(1;u,v))$ for any $j \in \N$.
By the same arguments as above, we have
$\Sol_{SL,1}^{(\ell(1;u,v))}(\alpha(j); \delta, u)) \neq \{0\}$
for all $j \in \N$, which yields the equivalence of (ii) and (iii).
Now the desired theorem follows.
\end{proof}

\begin{proof}[Proof of Theorem \ref{thm:classSL3}]
The equivalence of (i) and (ii) can be shown along the same lines as
Theorem \ref{thm:classSL12}. Indeed, it suffices to show that,
for $\ell(3;u,v)=(u_1+u_2)-v \in -\N$, there exists 
$\alpha \in \Gamma_3((\ell(3;u,v))$ such that
$\Sol_{SL,3}^{(\ell(3;u,v))}(\delta, u;\alpha)) \neq \{0\}$.
Suppose $v-(u_1+u_2) \in \N$ and 
take $\alpha=(\alpha_1,\alpha_2)$ such that 
$\alpha_1+\alpha_2=v-(u_1+u_2)$.
By the choice of $\alpha$, we have $\alpha \in \Gamma_3(\ell(3;u,v))$. 
It follows from the conditions (C1)--(C3) in Section \ref{sec:Step3} 
that $\Sol_{SL,1}^{(\ell(1;u,v))}(\delta, u;\alpha)) \neq \{0\}$,
which concludes the equivalence of (i) and (ii).

In order to prove \eqref{eqn:DIffDimSL3} and \eqref{eqn:DIffDimSL3b} observe that 
it follows from  \eqref{eqn:DiffSolSL} that, for $j=v-(u_1+u_2) \in\N$, we have 
\begin{equation}\label{eqn:DIffDimSL3c} 
\Diff_{G'_{SL,3}}(I(\delta,u)_{SL}, J(\sigma,v)_{SL})
=\bigoplus_{\alpha\in\Gamma_3(-j)}
(\Rest_3 \circ \Symb_0^{-1})(\Sol_{SL,3}^{(-j)}(\delta, u;\alpha)).
\end{equation}
By the aforementioned conditions (C1)--(C3), we have 
$\dim_\C\Sol_{SL,3}^{(-j)}(\delta, u;\alpha)\geq 1$ for 
all $\alpha \in \Gamma_3(-j)$.
The proposed inequality now follows from \eqref{eqn:GammaSL3}.
Moreover, for $\lambda(u) \in \C^2$ in \eqref{eqn:delta_u} and 
$\alpha_1+\alpha_2=v-(u_1+u_2)=j \in \N$,
the number $b_0$ appeared in (C1)--(C3) is given as $b_0=v-2$.
Thus, if $v \notin \N$, then $b_0$ satisfies (C1), which implies  
$\dim_\C\Sol_{SL,3}^{(-j)}(\delta, u;\alpha)= 1$. Since this holds for any 
$\alpha \in \Gamma_3(-j)$, the identity \eqref{eqn:GammaSL3} 
concludes the desired equality.
\end{proof}

We next prove Theorems \ref{thm:consSL12} and\ref{thm:consSL3}.

\begin{proof}[Proofs of Theorems \ref{thm:consSL12} and\ref{thm:consSL3}]
The polynomial solutions in $(\Sol_{SL,i}^{(\ell(i;u,v))}(\delta, u;\alpha))$
are the solutions to the differential equation
$-\zeta_i\widehat{d\pi_{\lambda(u)^*}}(N_i^+)\psi = 0$
on $\Pol(\alpha)$. Therefore, the proposed formulas follow simply from the 
arguments in Sections \ref{sec:proof1}--\ref{sec:Emb12} for 
the case $\lambda = \lambda(u)$.
\end{proof}

%%%%%%%%%%%%%%%%%%%%%%%%%%%%%%%%%%%%%%%%%
\section{Branching laws of Verma modules}
\label{sec:GVM}

The aim of this section is to discuss
the branching laws of  Verma modules
for the pairs $(\mathfrak{g}, \mathfrak{g}')
= (\mathfrak{gl}(3,\C), \mathfrak{gl}(2,\C)), 
(\mathfrak{sl}(3,\C), \mathfrak{sl}(2,\C))$
with respect to the three embeddings 
$\iota_{i}\colon \mathfrak{g}' \hookrightarrow \mathfrak{g}$.

We handle the cases
$(\mathfrak{gl}(3,\C), \mathfrak{gl}(2,\C))$ and $(\mathfrak{sl}(3,\C), \mathfrak{sl}(2,\C))$,
separately, as the behavior of the branching laws is significantly different.
We achieve the branching laws for the $\mathfrak{gl}$-case
in Theorem \ref{thm:BranchingEmb} and Corollary \ref{cor:BranchingEmb}.
The branchings laws for the $\fraksl$-case
are obtained in 
Theorems \ref{thm:BranchingEmb_SL12} and
\ref{thm:BranchingEmb_SL12}, and 
Corollary \ref{cor:BranchingEmb_SL3}.

%%%%%%%%%%%%%%%%%%%%%%%%%%%%%%%%%%%%%%%%%%
\subsection{Kobayashi's character identity}
\label{sec:character}

We start with the brief overview of Kobayashi's character identity
\cite{Kobayashi12} as it is our main machinery.
To discuss it in a general framework, we introduce some notation
 slightly different from the ones used  in Section \ref{sec:GL(3)};
we will resume it in Section \ref{sec:glgl}.

Let $\frakg$ be a complex reductive Lie algebra. 
Choose a Cartan subalgebra $\frakh$ and 
write $\Delta\equiv \Delta(\frakg, \frakh)$ for the set of roots of $\frakg$ with respect to $\frakh$.
Fix a positive system $\Delta^+$ and denote  by 
$\frakb$ the Borel subalgebra of $\frakg$ associated with $\Delta^+$, namely,
$\frakb = \frakh \oplus \fraku_+$ with $\fraku_+=\bigoplus_{\alpha \in \Delta^+} \frakg_\alpha$. 
Here $\frakg_\alpha$ is the root space
for $\alpha \in \Delta^+$. Let $\mathcal{O}$ denote the BGG category of $\frakg$-modules
whose objects are finitely generated $\frakg$-modules that are $\frakh$-semisimple
and locally $\fraku_+$-finite.

Let $\frakp \supset \frakb$ be a standard parabolic subalgebra of $\frakg$. Write 
$\frakp =\frakl \oplus \frakn_+$ for the Levi decomposition of $\frakp$ with $\frakh \subset \frakl$. 
We put $\Delta^+(\frakl):=\{\alpha \in \Delta^+: \frakg_\alpha \subset \frakl\}$. 
We denote by $\mathcal{O}^\frakp$ the parabolic BGG category, which is a full
subcategory of $\mathcal{O}$ whose objects are $\frakl$-semisimple and locally
$\frakn_+$-finite.

Let $\langle\cdot,\cdot\rangle$ denote the inner product on $\frakh^*$ induced from 
a non-degenerate symmetric bilinear form 
of $\frakg$. For $\alpha \in \Delta$, we write 
$\alpha^\vee = 2\alpha/\langle \alpha, \alpha \rangle$.
Then we put
\begin{equation*}
\Lambda^+(\frakl):=\{\lambda \in \frakh^* : 
\langle\lambda,\alpha^\vee\rangle 
\in \N
\;\;
\text{for all $\alpha \in \Delta^+(\frakl)$\}}.
\end{equation*}

For $\lambda \in \Lambda^+(\frakl)$, we denote by $F_\lambda$ the finite-dimensional
simple $\frakl$-module with highest weight $\lambda$. By letting $\frakn_+$ act trivially,
we regard $F_\lambda$ as a $\frakp$-module. We then define the generalized Verma
module $\Mp(\lambda)$ with highest weight $\lambda$ by 
\begin{equation*}
\Mp(\lambda) =U(\frakg)\otimes_{U(\frakp)}F_{\lambda}.
\end{equation*}
If $\frakp=\frakb$, then $\Mb(\lambda)$ is simply called a Verma module.

An element $H \in \frakg$ is called a hyperbolic element if
the eigenvalues of $\ad(H)$ on $\frakg$
are all real-valued. Thus, given a hyperbolic element $H$, 
one can define
subalgebras
\begin{equation*}
\frakn_-(H),
\quad
 \frakl(H),
\quad
\frakn_+(H)
\end{equation*}
as the sum of the eigenspaces of negative, zero and positive 
eiganvalue of $\ad(H)$, respectively.
In that case, we have
$\frakg=\frakn_-(H) \oplus \frakl(H) \oplus \frakn_+(H)$
and  $\frakp = \frakp(H):=\frakl(H)\oplus \frakn_+(H)$ is a Levi decomposition
of a parabolic subalgebra $\frakp$ of $\frakg$.

\begin{definition}\label{def:compatible}
Let $\frakg'$ and $\frakp$ be
a reductive subalgebra and a parabolic subalgebra of $\frakg$, respectively.
We say that 
$\frakp$ is \emph{$\frakg'$-compatible} if there exists a hyperbolic element 
$H \in \frakg'$ such that $\frakp = \frakp(H)$.
\end{definition}

Given a reductive subalgebra $\frakg'$ of $\frakg$,
suppose that  $\frakp=\frakp(H)$ is
a $\frakg'$-compatible parabolic subalgebra of $\frakg$.
The $\frakg'$-compatibility of $\frakp$ implies that
$\frakp':=\frakp \cap \frakg'$ is a parabolic subalgebra of $\frakg'$
with Levi decomposition
\begin{equation*}
\frakp'=\frakl' \oplus \frakn_+' :=(\frakl(H)\cap \frakg') \oplus (\frakn_+(H)\cap \frakg').
\end{equation*}

We choose a Cartan subalgebra
$\frakh' \subset \frakg'$ in such a way that  $H \in \frakh'$ and that
it extends to a Cartan subalgebra $\frakh$ of $\frakg$. Then we have 
$\frakh \subset \frakl(H)$
and $\frakh' \subset \frakg'$. For simplicity, we write 
$\frakl =\frakl(H)$ and $\frakn_{\pm}=\frakn_{\pm}(H)$.

Given a finite-dimensional vector space $V$, we write 
$S(V)=\bigoplus_{k=0}^\infty S^k(V)$ for the symmetric tensor algebra 
of $V$. 
For finite-dimensional simple $\frakl$- and $\frakl'$-modules
$F_\lambda$ and $F'_\nu$ with highest weights $\lambda \in \Lambda^+(\frakl)$
and $\nu \in \Lambda^+(\frakl')$, respectively, we put
\begin{equation*}
m(\nu;\lambda):=
\dim_\C
\Hom_{\frakl'}(F'_\nu, F_\lambda\vert_{\frakl'} \otimes S(\frakn_-/\frakn_-\cap \frakg')).
\end{equation*}
Let $[\Mp(\lambda)]$ and $[\Mpp(\nu)]$ denote the formal characters 
of $\Mp(\lambda)$ and $\Mpp(\nu)$, respectively.
For $\lambda \in \Lambda^+(\frakl)$,
we write  $\supp(\lambda) := \{\nu \in \Lambda^+(\frakl'):m(\nu;\lambda)\neq 0\}$.

\begin{theorem}[{\cite[Thm.\ 3.10]{Kobayashi12}}]\label{thm:bGVM}
Suppose that $\frakp=\frakl \oplus \frakn_+$ is a $\frakg'$-compatible parabolic subalgebra
of $\frakg$. Then, for $\lambda \in \Lambda^+(\frakl)$,  the following hold.

\begin{enumerate}
\item[\emph{(1)}]
$m(\nu;\lambda) < \infty$ for all $\nu \in \Lambda^+(\frakl')$.

\item[\emph{(2)}]
In the Grothendieck group of $\mathcal{O}^{\frakp'}$, we have
\begin{equation}\label{eqn:bGVM}
[\Mp(\lambda)\vert_{\frakg'}]
\simeq 
\bigoplus_{\nu \in \supp(\lambda)}
m(\nu; \lambda) [\Mpp(\nu)].
\end{equation}
\end{enumerate}
\end{theorem}

For later convenience we end this subsection by recalling 
an irreducibility criterion for Verma module $\Mb(\lambda)$
with highest weight $\lambda$.

\begin{theorem}[{cf.\ \cite[Thm.\ 4.4]{Hum08}}]\label{thm:VermaCriterion}
The following conditions on $\lambda \in \frakh^*$ are equivalent.
\begin{enumerate}[label=\normalfont{(\roman*)}]
\item
The Verma module $\Mb(\lambda)$ is irreducible. 
\vspace{3pt}

\item
The weight $\lambda \in \frakh^*$ satisfies the following condition:
\begin{equation}\label{eqn:VermaCriterion}
\langle\lambda+\rho,\beta^\vee\rangle 
\not\in \N_+
\;\;
\textnormal{for all $\beta \in \Delta^+$},
\end{equation}
where $\rho$ is half the sum of the positive roots $\beta \in \Delta^+$.
\end{enumerate}
\end{theorem}

\begin{remark}\label{rem:VermaCriterion}
In \cite[Thm.\ 4.4]{Hum08}, the Lie algebra $\frakg$ is assumed to be 
complex semisimple. However, one can apply the criterion for a complex 
reductive Lie algebra $\frakg$ as well since the center of $\frakg$ does not 
play any role in the irreducibility of the Verma module 
$\Mb(\lambda)$.
\end{remark}

In the following
we shall return to the setup from
Section \ref{sec:notation} to 
discuss the branching laws $\Mb(\lambda)\vert_{\frakg'}$
of the Verma modules $\Mb(\lambda)$
for $(\mathfrak{g}, \mathfrak{g}')= (\mathfrak{gl}(3,\CC),
\mathfrak{gl}(2,\CC)), (\mathfrak{sl}(3,\CC), \mathfrak{sl}(2,\CC))$.
For a real Lie algebra $\mathfrak{y}(\R)$, 
we write $\mathfrak{y}$ 
for its complexification.

%%%%%%%%%%%%%%%%%%%%%%%%%%%%%%%%%%%%%%%%%
\subsection{\texorpdfstring{$(\mathfrak{gl}(3,\C), \mathfrak{gl}(2,\C))$}{(gl(3,C),gl(2,C))}-case}
\label{sec:glgl}

We start with the $\mathfrak{gl}$-case.

%%%%%%%%%%%%%%%%%%%%%%%%%%%%%%%%%%%%%%%%%
\subsubsection{\texorpdfstring{$\frakg'$}{g}-compatibility of the Borel subalgebra \texorpdfstring{$\frakb$}{b}}
\label{sec:BorelGL}

Let $\frakg(\RR)=\mathfrak{gl}(3,\RR)$ so that 
$\frakg= \mathfrak{gl}(3,\CC)$.
Then the complexification $\frakb\equiv \frakp = \fraka \oplus \frakn_+$  
of the minimal parabolic subalgebra 
$\frakp(\RR)=\fraka(\RR)\oplus \frakn_+(\RR)$ of $\frakg(\RR)$ defined in 
\eqref{eqn:minimal_parabolic} is a Borel subalgebra of $\frakg$. 
For the complexificaion $\frakn_-$ of $\frakn_-(\RR)$, 
we have 
$\frakg = \frakn_- \oplus \frakb =\frakn_-\oplus \fraka \oplus \frakn_+$,
where 
\begin{equation*}
\fraka=\Span_\C\{H_1, H_2, H_3\}
\quad
\text{and}
\quad
\frakn_{\pm}=\Span_\C\{N_1^{\pm},N_2^{\pm},N_3^{\pm}\}.
\end{equation*}
Table \ref{table:adH}  summarizes
the adjoint action $\ad(H_j)$ on $N_k^\pm$ for $j, k \in \{1,2,3\}$.
\begin{table}[ht]
\caption{}
\begin{center}
\renewcommand{\arraystretch}{1.3}
{
\begin{tabular}{c|c|c|c}
\hline
$\ad(H_j)(N_k^\pm)$ & $N_1^\pm$ & $N_2^\pm$ & $N_3^\pm$\\
\hline
$H_1$ & $\pm N_1^\pm$ & $0$ & $\pm N_3^\pm$\\
\hline
$H_2$ & $\mp N_1^\pm $ & $\pm N_2^\pm$ & $0$\\
\hline
$H_3$ & $0$ & $\mp N_2^\pm$ & $\mp N_3^\pm$\\
\hline
\end{tabular}
}
\end{center}\label{table:adH}
\end{table}

Let $H_{j,k}:=E_{j,j}-E_{k,k}$.
For $\ell \in \ZZ$, we write $\frakg(\ell):=\{X \in \frakg: \ad(H_{1,3})X = \ell X\}$.
A simple observation shows that the subalgebras $\fraka$ and $\frakn_{\pm}$
can be characterized as 
\begin{equation*}
\fraka =\frakg(0)
\quad
\text{and}
\quad
\frakn_{\pm} = \frakg(\pm 1) \oplus \frakg(\pm 2).
\end{equation*}
Thus,
the Borel subalgebra $\frakb = \fraka \oplus \frakn_+$ is given by $\frakb = \frakb(H_{1,3})$.

Recall from Section \ref{sec:Gprime} that $\frakg_i'(\RR)$ $(i=1,2,3)$ is 
the image of the embedding 
$\iota\colon \mathfrak{gl}(2,\RR) \hookrightarrow \mathfrak{gl}(3,\RR)$
defined in \eqref{eqn:embedding2}. Then its complexification 
$\frakg_i'$ admits the decomposition
\begin{equation*}
\frakg_i' = \frakn_{i,-}' \oplus \frakb_i' =\frakn_{i,-}' \oplus \fraka_i' \oplus \frakn_{i,+}',
\end{equation*}
where $\frakn_{i,\pm}' = \frakn_{\pm} \cap \frakg_i'$,
$\frakb_i' =\frakb \cap \frakg_i'$ and $\fraka_i=\fraka'\cap \frakg_i'$.
The subalgebras $\fraka_i'$ and $\frakn_{i,-}'$ are given as
\begin{alignat}{6}
&\fraka_1' &&= \C H_1 \oplus \C H_2,
\quad
&\fraka_2' &= \C H_2 \oplus \C H_3,
\quad
&\fraka_3' &= \C H_1 \oplus \C H_3,\nonumber\\[3pt]
&\frakn_{1,\pm}' &&= \C N_1^{\pm},
\quad
&\frakn_{2,\pm}' &= \C N_2^{\pm},
\quad
&\frakn_{3,\pm}' &= \C N_3^{\pm}.\label{eqn:nilpotent}
\end{alignat}

\begin{proposition}\label{prop:compatibility}
The Borel subalgebra $\frakb$ is $\frakg_3'$-compatible, but not 
$\frakg_i'$-compatible for $i=1,2$.
\end{proposition}

\begin{proof}
As $\frakb = \frakb(H_{1,3})$,
the proposed statement follows that 
$H_{1,3} \in \frakg_i'$ if and only if $i=3$.
\end{proof}

\begin{remark}\label{rem:compatibility}
Although the Borel subalgebra $\frakb=\fraka \oplus \frakn_+$ is only $\frakg_3'$-compatible,
a careful reading of the proof of Theorem \ref{thm:bGVM} shows
that the character formula \eqref{eqn:bGVM} holds also for the other 
two subalgebras $\frakg_i'$ for $i=1,2$
as $\fraka_i'=\fraka \cap \frakg_i'$ and 
$\frakn_{i,\pm}' = \frakn_{\pm}\cap \frakg_i'$. 
Therefore, one can apply the 
character formula \eqref{eqn:bGVM} for all three cases.
\end{remark}

%%%%%%%%%%%%%%%%%%%%%%%%%%%%%%%%%%%%%%%%%
\subsubsection{The branching laws \texorpdfstring{$\Mb(\lambda)\vert_{\frakg_i'}$}{M(λ)|_{g_i}}}
\label{sec:BranchingLaw_GL}

As the Levi part $\frakl$ of the Borel subalgebra
$\frakb=\fraka \oplus \frakn_+$ is $\frakl=\fraka$, the weight space
$\Lambda^+(\frakl)$ in consideration is 
\begin{equation*}
\Lambda^+(\frakl) = \fraka^* \simeq \CC^3.
\end{equation*}
Similarly, the Levi part $\frakl_i'$ of  $\frakb_i'$ 
is $\frakl_i' = \fraka_i'$. Therefore,
\begin{equation*}
\Lambda^+(\frakl_i') = (\fraka_i')^* \simeq \CC^2.
\end{equation*}

As in \eqref{eqn:characterA}, we denote by $\C_{\lambda}$ for 
$\lambda=(\lambda_1,\lambda_2,\lambda_3) \in \C^3$
the one-dimensional $\fraka$-module defined by
\begin{equation*}
z_1 H_1 + z_2 H_2 + z_3H_3 
\longmapsto \lambda_1z_1 + \lambda_2z_2  +\lambda_3 z_3.
\end{equation*}
Likewise, we define the one-dimensional 
$\fraka_i'$-module $\C_{\nu}$ for $\nu=(\nu_1,\nu_2) \in \C^2$
as in \eqref{eqn:characterAprime} via the embeddings 
$\iota\colon \mathfrak{gl}(2,\RR) \hookrightarrow \mathfrak{gl}(3,\RR)$.
Then we have 
\begin{alignat*}{2}
z_1 H_1 + z_2 H_2  
&\longmapsto \nu_1z_1 + \nu_2 z_2 &&\quad \text{for $i=1$},\\[3pt]
z_1 H_2 + z_2 H_3  
&\longmapsto \nu_2 z_1 + \nu_3 z_2 &&\quad \text{for $i=2$},\\[3pt]
z_1 H_1 + z_2 H_3  
&\longmapsto \nu_1 z_1 + \nu_3 z_2 &&\quad \text{for $i=3$}.
\end{alignat*}
Then the restriction $\C_{(\lambda_1,\lambda_2,\lambda_3)}\vert_{\fraka_i'}$
of the $\fraka$-module $\C_{(\lambda_1,\lambda_2,\lambda_3)}$
to $\fraka_i'$ is given as
\begin{equation}\label{eqn:restriction_to_ai}
\C_{(\lambda_1, \lambda_2,\lambda_3)}\vert_{\fraka_1'} = \C_{(\lambda_1,\lambda_2)},
\quad
\C_{(\lambda_1, \lambda_2,\lambda_3)}\vert_{\fraka_2'} = \C_{(\lambda_2,\lambda_3)},
\quad
\C_{(\lambda_1, \lambda_2,\lambda_3)}\vert_{\fraka_3'} = \C_{(\lambda_1,\lambda_3)}.
\end{equation}

For $(\lambda,\nu) \in \C^5$, we write
\begin{align}
m_i(\nu;\lambda)&:=
\dim_\CC \Hom_{\fraka_i'}(\C_\nu, \C_\lambda \otimes 
S(\frakn_-/\frakn_{i,-}')),\label{eqn:mult_i}\\[3pt]
\supp_i(\lambda)&:=\{\nu \in \CC^2 : m_i(\nu,\lambda)\neq 0\}.\nonumber
\end{align}
Then, 
by Theorem \ref{thm:bGVM},
Proposition \ref{prop:compatibility} and
Remark \ref{rem:compatibility},
we have 
\begin{equation}\label{eqn:bGVM2}
[\Mb(\lambda)\vert_{\frakg_i'}]
\simeq 
\bigoplus_{\nu \in \supp_i(\lambda)}
m_i(\nu; \lambda) [M^{\frakg'_i}_{\frakb'_i}(\nu)].
\end{equation}

We now wish to determine $m_i(\nu; \lambda)$ and $\supp_i(\lambda)$. 
To the end, define
\begin{align}
\supp_{1} &:= \{(\lambda, \nu) \in \C^5: 
\nu_1-\lambda_1,\nu_2-\lambda_2 \in -\N\},\nonumber\\[3pt]
\supp_{2} &:= \{(\lambda, \nu) \in \C^5: 
\lambda_2-\nu_1,\lambda_3-\nu_2 \in -\N\},\label{eqn:supp_i}\\[3pt]
\supp_{3} &:= \{(\lambda, \nu) \in \C^5: 
\nu_1-\lambda_1, \lambda_3-\nu_2 \in -\N\}.\nonumber
\end{align}

%%%%%%%%%%%%%%%%%%%%%%%%%%%%%%%%%%%%%%%%%
\begin{proposition}\label{prop:BranchingEmb}
For $i=1,2,3$, the following conditions on 
$(\lambda, \nu) \in \CC^5$ are 
equivalent.

\begin{enumerate}[label=\normalfont{(\roman*)}]
\item
$m_i(\nu;\lambda)\neq 0$.
\vspace{3pt}

\item
$m_i(\nu;\lambda)=1$.
\vspace{3pt}

\item $(\lambda,\nu) \in \supp_{i}$.

\end{enumerate}

\end{proposition}

\begin{proof}
It follows from \eqref{eqn:nilpotent} that the quotient space
$\frakn_-/\frakn_{i,-}'$ is given as
\begin{equation*}
\frakn_-/\frakn_- \cap \frakg_i' = 
\Span_\C\{N_j^-, N_k^-: j,k\neq i\},
\end{equation*}
which yields
\begin{equation*}
S(\frakn_-/\frakn_{i,-}')
= \bigoplus_{\ell, r \in \N} \C(N_j^-)^\ell(N_k^-)^r.\qquad(j,k\neq i)
\end{equation*}
Thus, by Table \ref{table:adH}, we have
\begin{alignat*}{3}
\C_{\lambda}\otimes S(\frakn_-/\frakn_{1,-}')
&= \C_{\lambda}\otimes \bigoplus_{\ell, r \in \N} \C(N_2^-)^\ell(N_3^-)^r
&&\simeq \bigoplus_{\ell,r \in \N}\C_{(\lambda_1-r,\lambda_2-\ell)}
&&\quad \text{as $\fraka_1'$-modules},\\[3pt]
\C_{\lambda}\otimes S(\frakn_-/\frakn_{2,-}')  
&= \C_{\lambda}\otimes\bigoplus_{\ell, r \in \N} \C(N_1^-)^\ell(N_3^-)^r 
&&\simeq \bigoplus_{\ell,r \in \N}\C_{(\lambda_2+\ell,\lambda_3+r)}
&&\quad \text{as $\fraka_2'$-modules},\\[3pt]
\C_{\lambda}\otimes S(\frakn_-/\frakn_{3,-}')
&=\C_{\lambda}\otimes \bigoplus_{\ell,r \in \N} \C(N_1^-)^\ell(N_2^-)^r
&&\simeq \bigoplus_{\ell,r \in \N}\C_{(\lambda_1-\ell,\lambda_3+r)}
&&\quad \text{as $\fraka_3'$-modules}.
\end{alignat*}
Now the proposition follows.
\end{proof}

\begin{corollary}\label{cor:support}
For $\lambda=(\lambda_1,\lambda_2,\lambda_3) \in \C^3$, we have 
\begin{align*}
\supp_1(\lambda)
=\{(\lambda_1-r, \lambda_2-\ell): r, \ell \in \N\},\\[3pt]
\supp_2(\lambda)
=\{(\lambda_2+r, \lambda_3+\ell): r, \ell \in \N\},\\[3pt]
\supp_3(\lambda)
=\{(\lambda_1-r, \lambda_3+\ell): r, \ell \in \N\}.
\end{align*}
\end{corollary}

\begin{proof}
The proposed assertion follows immediately from Proposition \ref{prop:BranchingEmb}.
\end{proof}

\begin{theorem}\label{thm:BranchingEmb}
For $\lambda=(\lambda_1, \lambda_2,\lambda_3) \in \CC^3$,
we have
\begin{align}
[\Mb(\lambda)\vert_{\frakg_1'}]
&\simeq 
\bigoplus_{\ell, r \in \N}
[M^{\frakg'_1}_{\frakb'_1}(\lambda_1-r, \lambda_2-\ell)],\label{eqn:bGVM_Emb1}\\[3pt]
[\Mb(\lambda)\vert_{\frakg_2'}]
&\simeq 
\bigoplus_{\ell, r \in \N}
[M^{\frakg'_2}_{\frakb'_2}(\lambda_2+r, \lambda_3+\ell)],\label{eqn:bGVM_Emb2}\\[3pt]
[\Mb(\lambda)\vert_{\frakg_3'}]
&\simeq 
\bigoplus_{\ell, r \in \N}
[M^{\frakg'_3}_{\frakb'_3}(\lambda_1-r, \lambda_3+\ell)].\label{eqn:bGVM_Emb3}
\end{align}
\end{theorem}

\begin{proof}
The statement is an immediate consequence 
of \eqref{eqn:bGVM2}, Proposition \ref{prop:BranchingEmb}
and Corollary \ref{cor:support}.
\end{proof}

Theorem \ref{thm:BranchingEmb} provides the actual branching laws
for generic parameters for $i=1,2,3$. More precisely,
define 
\begin{equation}\label{eqn:Irreducible_Parameters}
\irred(\frakg,\frakb):=\{\lambda \in \C^3: \text{$\Mb(\lambda)$ is irreducible}\}.
\end{equation}
The integral condition  \eqref{eqn:VermaCriterion} yields
\begin{equation}\label{eqn:Irreducible_Parameters2}
\irred(\frakg,\frakb)=\{(\lambda_1,\lambda_2, \lambda_3) \in \C^3:
\lambda_1-\lambda_2,\,
\lambda_2-\lambda_3,\, 
\lambda_1-\lambda_3+1 \notin \N\}.
\end{equation}
For $i=1,2,3$, we then put
\begin{align*}
\irred_{i}(\frakg,\frakb)
&:=\{(\lambda_1,\lambda_2, \lambda_3)\in \irred(\frakg,\frakb): 
\lambda_i -\lambda_{i+1}\notin \Z \}\quad \text{for $i=1,2$},\\[3pt]
\irred_{3}(\frakg,\frakb)
&:=\{(\lambda_1,\lambda_2, \lambda_3)\in \irred(\frakg,\frakb): 
\lambda_1-\lambda_3 \notin \N \}.
\end{align*}

\begin{corollary}\label{cor:BranchingEmb}
For $i \in \{1,2,3\}$ and 
$\lambda=(\lambda_1, \lambda_2) \in \irred_{i}(\frakg,\frakb)$,
we have
\begin{align*}
\Mb(\lambda)\vert_{\frakg_1'}
&\simeq 
\bigoplus_{\ell, r \in \N}
M^{\frakg'_1}_{\frakb'_1}(\lambda_1-r, \lambda_2-\ell),\\[3pt]
\Mb(\lambda)\vert_{\frakg_2'}
&\simeq 
\bigoplus_{\ell, r \in \N}
M^{\frakg'_2}_{\frakb'_2}(\lambda_2+r, \lambda_3+\ell),\\[3pt]
\Mb(\lambda)\vert_{\frakg_3'}
&\simeq 
\bigoplus_{\ell, r \in \N}
M^{\frakg'_3}_{\frakb'_3}(\lambda_1-r, \lambda_3+\ell).
\end{align*}
\end{corollary}

\begin{proof}
It suffices to show that the Verma modules appeared 
in Theorem \ref{thm:BranchingEmb}
are all irreducible for $\lambda \in \irred_{i}(\frakg,\frakb)$.
Indeed, the Verma module $\Mb(\lambda)$ is irreducible 
for  $\lambda \in \irred(\frakg,\frakb)$ by the definition of $\irred(\frakg,\frakb)$.
For $i=1,2$,
if $\lambda \in \irred_{i}(\frakg,\frakb)$, then 
$(\lambda_i\pm r, \lambda_{i+1}\pm \ell)$ 
satisfy \eqref{eqn:VermaCriterion} for all $\ell,r \in \N$.
Similarly, if $\lambda \in \irred_{3}(\frakg,\frakb)$, then 
$(\lambda_1-r,\lambda_3+\ell)$ satisfies \eqref{eqn:VermaCriterion}
for all $\ell, r\in \N$.
Theorem \ref{thm:VermaCriterion} now concludes the desired decompositions.
\end{proof}

%%%%%%%%%%%%%%%%%%%%%%%%%%%%%%%%%%%%%%%%%
\subsection{\texorpdfstring{$(\mathfrak{sl}(3,\C), \mathfrak{sl}(2,\C))$}{(gl(3,C),gl(2,C))}-case}
\label{sec:slsl}

We next consider the $\fraksl$-case.

%%%%%%%%%%%%%%%%%%%%%%%%%%%%%%%%%%%%%%%%%
\subsubsection{\texorpdfstring{$\frakg'$}{g}-compatibility of the Borel subalgebra \texorpdfstring{$\frakb$}{b}}
\label{sec:BorelSL}

Let $\frakg_{\fraksl}=\fraksl(3,\CC)$. Then $\frakb_{\mathfrak{sl}} := \frakb \cap \frakg_{\fraksl}$ is a Borel subalgebra of $\frakg_{\fraksl}$.
The Lie algebra $\frakg_{\fraksl}$ can be decomposed as
\begin{equation*}
\frakg_{\fraksl}
=\frakn_{-} \oplus \frakb_{\fraksl} 
= \frakn_{-} \oplus \fraka_{\fraksl} \oplus \frakn_{+},
\end{equation*}
where 
$\fraka_{\fraksl} :=\fraka \cap \frakg_{\fraksl}$.
For $H_{j,k}=E_{j,j}-E_{k,k}$, the subalgebra $\fraka_{\fraksl}$ is given as
\begin{equation*}
\fraka_{\fraksl}=\Span_\C\{H_{1,2}, H_{2,3}\}.
\end{equation*}

Let $\frakg_{\fraksl,i}'$ denote the complexification of 
the image of the embedding
$\iota_i \vert_{\mathfrak{sl}(2,\RR)} \colon \mathfrak{sl}(2,\RR) \hookrightarrow 
\mathfrak{sl}(3,\RR)$ for $i=1,2,3$; 
in particular, we have $\frakg_{\fraksl,i}' \simeq \fraksl(2,\CC)$.
Then $\frakb_{\mathfrak{sl},i}':=\frakb_i'\cap \frakg_{\fraksl,i}'$ is a Borel
subalgebra of $\frakg_{\fraksl,i}'$.
The Lie algebra $\frakg_{\fraksl,i}'$ admits the decomposition
\begin{equation*}
\frakg_{\fraksl,i}'
=\frakn_{i,-}' \oplus \frakb_{\fraksl,i}'
= \frakn_{i,-}' \oplus \fraka_{\fraksl,i}' \oplus \frakn_{i,+}',
\end{equation*}
where
$\fraka_{\fraksl,i}' :=\fraka_i' \cap \frakg_{\fraksl,i}'$.
The subalgebras $\fraka_{\fraksl,i}'$ for $i=1,2,3$ are given as
\begin{equation*}
\fraka_{\fraksl,1}' = \C H_{1,2},
\quad
\fraka_{\fraksl,2}' = \C H_{2,3},
\quad
\fraka_{\fraksl,3}' = \C H_{1,3}.
\end{equation*}
Table \ref{table:adH2}  summarizes
the adjoint action $\ad(H_{j,k})$ on $N_\ell^\pm$.
\begin{table}[ht]
\caption{}
\begin{center}
\renewcommand{\arraystretch}{1.3}
{
\begin{tabular}{c|c|c|c}
\hline
$\ad(H_{j,k})(N_\ell^\pm)$ & $N_1^\pm$ & $N_2^\pm$ & $N_3^\pm$\\
\hline
$H_{1,2}$ & $\pm 2N_1^\pm$ & $\mp N_2^\pm$ & $\pm N_3^\pm$\\
\hline
$H_{2,3}$ & $\mp N_1^\pm $ & $\pm 2N_2^\pm$ & $\pm N_3^\pm$\\
\hline
$H_{1,3}$ & $\pm N_1^\pm$ & $\pm N_2^\pm$ & $\pm 2N_3^\pm$\\
\hline
\end{tabular}
}
\end{center}\label{table:adH2}
\end{table}

\begin{corollary}\label{cor:compatibility}
The Borel subalgebra $\frakb_{\mathfrak{sl}}$ is 
$\frakg_{\fraksl,3}'$-compatible, but not 
$\frakg_{\fraksl,i}'$-compatible for $i=1,2$.
\end{corollary}

\begin{proof}
This readily follows from the observation that 
$H_{1,3} \in \frakg_{\fraksl,i}'$ if and only if $i=3$.
\end{proof}

\begin{remark}
By the same arguments given in Remark \ref{rem:compatibility}, 
one can apply the character formula \eqref{eqn:bGVM} 
to $(\frakg_{\fraksl}, \frakg_{\fraksl,i}')$ for all $i=1,2,3$.
However, it will turn out that
the finite multiplicity phenomenon in Theorem \ref{thm:bGVM} (1) does not hold
in the non-compatible cases. 
(See Theorems \ref{thm:BranchingEmb_SL12} and \ref{thm:BranchingEmb_SL3}.)
\end{remark}

%%%%%%%%%%%%%%%%%%%%%%%%%%%%%%%%%%%%%%%%%
\subsubsection{The branching laws \texorpdfstring{$\Mbsl(\lambda)\vert_{\frakg_{\fraksl,i}'}$}{M(λ)|_{g_{sl,i}}}}
\label{sec:BranchingLaw_SL}

Since the Levi part $\frakl$ of $\frakb_{\fraksl}$ and 
$\frakl'$ of $\frakb_{\fraksl,i}'$
are $\frakl=\fraka_{\fraksl}$ and 
$\frakl'=\fraka_{\fraksl,i}'$, we have 
\begin{equation*}
\Lambda^+(\frakl) = \fraka_{\fraksl}^* \simeq \CC^2
\quad
\text{and}
\quad
\Lambda^+(\frakl') = \fraka_{\fraksl,i}^* \simeq \CC.
\end{equation*}

\vspace{10pt}

As in \eqref{eqn:charSLA},
for $u=(u_1,u_2) \in \C^2$,
let $\C_{u}$ denote  the one-dimensional $\fraka_{\fraksl}$-module 
defined by
\begin{equation*}
z_1 H_{1,2} + z_2 H_{2,3}  
\longmapsto 
u_1z_1  + u_2z_2.
\end{equation*}
Likewise, as in \eqref{eqn:charSLAprime}, we denote by
$\C_{v}$ for $v\in \C$
the one-dimensional $\fraka_{\fraksl,i}'$-module defined by
\begin{alignat*}{2}
z H_{1,2}  &\longmapsto vz  &&\quad \text{for $i=1$},\\[3pt]
z H_{2,3}  &\longmapsto vz  &&\quad \text{for $i=2$},\\[3pt]
z H_{1,3}  &\longmapsto vz  &&\quad \text{for $i=3$}.
\end{alignat*}
Then the restriction $\C_{(u_1,u_2)}\vert_{\fraka_{\fraksl,i}'}$ is given as
\begin{equation}\label{eqn:restriction_to_ai2}
\C_{(u_1, u_2)}\vert_{\fraka_{\fraksl,1}'} = \C_{u_1},
\quad
\C_{(u_1, u_2)}\vert_{\fraka_{\fraksl,2}'} = \C_{u_2},
\quad
\C_{(u_1, u_2)}\vert_{\fraka_{\fraksl,3}'} = \C_{u_1+u_2}.
\end{equation}

For $(u,v) \in \C^3$, we write
\begin{align}
m_{\fraksl,i}(v;u)&:=
\dim_\CC \Hom_{\fraka_{\fraksl,i}'}(\C_v, \C_u \otimes 
S(\frakn_{\fraksl,-}/\frakn_{\fraksl, i,-}')),\label{eqn:mult_i_sl}\\[3pt]
\supp_{\fraksl,i}(u)&:=\{v \in \CC : m_{\fraksl,i}(v,u)\neq 0\}.\nonumber
\end{align}
Then, as in \eqref{eqn:bGVM2}, we have 
\begin{equation}\label{eqn:bGVM_SL2}
[\Mbsl(u)\vert_{\frakg_{\fraksl,i}'}]
\simeq 
\bigoplus_{v \in \supp_{\fraksl,i}(u)}
m_{\fraksl,i}(v; u) 
[M^{\frakg'_{\fraksl,i}}_{\frakb'_{\fraksl,i}}(v)].
\end{equation}

We wish to determine $m_{\fraksl,i}(v; u)$ 
and $\supp_{\fraksl,i}(u)$. 
First we consider $i=1,2$.
Define
\begin{equation}\label{eqn:supp_i_sl}
\supp_{\fraksl,i} := \{(u, v) \in \C^3: 
v-u_i \in \N\}.
\end{equation}

\begin{proposition}\label{prop:BranchingSLEmb12}
For $i=1,2$, the following conditions on 
$(u, v) \in \CC^3$ are 
equivalent.
\begin{enumerate}[label=\normalfont{(\roman*)}]
\item
$m_{\fraksl,i}(v;u) \neq 0$.
\vspace{3pt}

\item
$m_{\fraksl,i}(v;u) =\infty$.
\vspace{3pt}

\item $(u,v) \in \supp_{\fraksl,i}$.
\end{enumerate}
\end{proposition}

\begin{proof}
It follows from Table \ref{table:adH2} that, for $\ell, r \in \N$, we have
\begin{equation*}
\ad(H_{1,2})(N_2^-)^\ell(N_3^-)^r=(\ell-r)(N_2^-)^\ell(N_3^-)^r.
\end{equation*}
\begin{equation*}
\ad(H_{2,3})(N_1^-)^\ell(N_3^-)^r=(\ell-r)(N_1^-)^\ell(N_3^-)^r.
\end{equation*}
Therefore,
\begin{alignat*}{3}
\C_{u}\otimes S(\frakn_-/\frakn_{\fraksl, 1,-}') 
&= \C_{u}\otimes \bigoplus_{j \in \N} \bigoplus_{\ell \geq j}
\C(N_2^-)^\ell(N_3^-)^{\ell-j}
&&\simeq \bigoplus_{j \in \N} \bigoplus_{\ell \geq j} \C_{u_1+j}
\quad &&\text{as $\fraka_{\fraksl,1}'$-modules},\\[3pt]
\C_{u}\otimes S(\frakn_-/\frakn_{\fraksl, 2,-}') 
&= \C_{u}\otimes \bigoplus_{j \in \N} \bigoplus_{\ell \geq j}
\C(N_1^-)^\ell(N_3^-)^{\ell-j}
&&\simeq \bigoplus_{j \in \N} \bigoplus_{\ell \geq j} \C_{u_2+j}
\quad &&\text{as $\fraka_{\fraksl,2}'$-modules}.
\end{alignat*}
This proves the proposition.
\end{proof}

\begin{corollary}\label{cor:support_SL12}
For $u=(u_1,u_2) \in \C^2$, we have 
\begin{equation*}
\supp_{\fraksl,i}(u)=\{u_i+j : j \in \N\}.
\end{equation*}
\end{corollary}

\begin{proof}
This is a direct consequence of 
Proposition \ref{prop:BranchingSLEmb12}.
\end{proof}

\begin{theorem}\label{thm:BranchingEmb_SL12}
Let $i \in \{1,2\}$, 
Then, for $u=(u_1, u_2) \in \CC^2$,
we have
\begin{equation*}
[\Mbsl(u)\vert_{\frakg_{\fraksl,i}'}]
\simeq 
\bigoplus_{j \in \N}
m_{\fraksl,i}(u_i+j;u)
[M^{\frakg'_{\fraksl,i}}_{\frakb'_{\fraksl,i}}(u_i + j)].
\end{equation*}
with $m_{\fraksl,i}(u_i+j;u)=\infty$.
\end{theorem}

\begin{proof}
The theorem follows simply from 
\eqref{eqn:bGVM_SL2}, 
Proposition \ref{prop:BranchingSLEmb12} and
Corollary \ref{cor:support_SL12}.
\end{proof}

We next consider the case $i=3$.
Define
\begin{equation}\label{eqn:supp_3_sl}
\supp_{\fraksl,3} 
:= \{(u, v) \in \C^3: v -(u_1+u_2) \in -\N\}.
\end{equation}

\begin{proposition}\label{prop:BranchingSLEmb3}
The following conditions on 
$(u, v) \in \CC^3$ are 
equivalent.
\begin{enumerate}[label=\normalfont{(\roman*)}]
\item
$m_{\fraksl,3}(v;u) \neq 0$.
\vspace{3pt}
\item 
$(u,v) \in \supp_{\fraksl,3}$.
\end{enumerate}
Further, for $(u,v) \in \supp_{\fraksl,3}$, we have 
\begin{equation*}
m_{\fraksl,3}(v;u) = u_1+u_2-\nu+1.
\end{equation*}
\end{proposition}

\begin{proof}
It follows from Table \ref{table:adH2} that, for $\ell, r \in \N$, we have
\begin{equation*}
\ad(H_{1,3})(N_1^-)^\ell(N_2^-)^r=-(\ell+r)(N_1^-)^\ell(N_2^-)^r.
\end{equation*}
Therefore, as $\fraka_{\fraksl,3}'$-modules, we have 
\begin{equation*}
\C_u \otimes
S(\frakn_-/\frakn_{\fraksl, 3,-}') 
%&= \C_u \otimes \bigoplus_{\ell, r \in \N} \C(N_1^-)^\ell(N_2^-)^r\\[3pt]
= \C_u \otimes \bigoplus_{j \in \N} \bigoplus_{\ell=0}^j
\C(N_1^-)^\ell(N_3^-)^{j-\ell}\\[3pt]
\simeq
\bigoplus_{j \in \N} \bigoplus_{\ell=0}^j
\C_{u_1+u_2-j}.
\end{equation*}
The desired proposition now follows.
\end{proof}

\begin{corollary}\label{cor:support_SL3}
For $u=(u_1,u_2) \in \C^2$, we have 
\begin{equation*}
\supp_{\fraksl,3}(u)
=\{u_1+u_2-j: j \in \N\}.
\end{equation*}
\end{corollary}

\begin{proof}
This is a straightforward consequence of Proposition \ref{prop:BranchingSLEmb3}.
\end{proof}

\begin{theorem}\label{thm:BranchingEmb_SL3}
For $u=(u_1, u_2) \in \CC^2$,
we have
\begin{equation}\label{eqn:bGVM_Emb3_sl}
[\Mbsl(u)\vert_{\frakg_{\fraksl,3}'}]
\simeq 
\bigoplus_{j\in \N}
(j+1)
[M^{\frakg'_{\fraksl,3}}_{\frakb'_{\fraksl,3}}(u_1+u_2-j)].
\end{equation}
\end{theorem}

\begin{proof}
Proposition \ref{prop:BranchingSLEmb3} and
Corollary \ref{cor:support_SL3}  conclude the desired theorem.
\end{proof}

We are now ready to 
give actual branching laws for generic parameters for $i=1,2,3$.
To the end, define
\begin{equation*}
\irred(\frakg_{\fraksl},\frakb_{\fraksl})
:=\{(u \in \C^2:\text{$\Mbsl(u)$ is irreducible}\}.
\end{equation*}
By Theorem \ref{thm:VermaCriterion}, one has
\begin{equation}\label{eqn:irred_sl}
\irred(\frakg_{\fraksl},\frakb_{\fraksl})=\{(u_1,u_2) \in \C^2:
2u_1-u_2,\,
2u_2-u_1,\, 
u_1+u_2+1 \notin \N\}.
\end{equation}
For $i=1,2,3$, we then put
\begin{align*}
\irred_i(\frakg_{\fraksl},\frakb_{\fraksl})
&:=\{(u_1,u_2)\in \irred(\frakg_{\fraksl},\frakb_{\fraksl}): 
u_i \notin \Z \}\quad \text{for $i=1,2$},\\[3pt]
\irred_3(\frakg_{\fraksl},\frakb_{\fraksl})
&:=\{(u_1,u_2)\in \irred(\frakg_{\fraksl},\frakb_{\fraksl}): 
u_1+u_2 \notin \N \}.
\end{align*}

\begin{corollary}\label{cor:BranchingEmb_SL3}
For $i \in \{1,2,3\}$ and 
$u=(u_1, u_2) \in \irred_i(\frakg_{\fraksl},\frakb_{\fraksl})$,
we have
\begin{align*}
\Mbsl(u)\vert_{\frakg_{\fraksl,i}'}
&\simeq 
\bigoplus_{j \in \N}
m_{\fraksl,i}(u_i+j;u)
M^{\frakg'_{\fraksl,i}}_{\frakb'_{\fraksl,i}}(u_i + j)
\quad
\textnormal{for $i=1,2$,}\\[3pt]
\Mbsl(u)\vert_{\frakg_{\fraksl,3}'}
&\simeq 
\bigoplus_{j\in \N}
(j+1)
M^{\frakg'_{\fraksl,3}}_{\frakb'_{\fraksl,3}}(u_1+u_2-j)
\end{align*}
with $m_{\fraksl,i}(u_i+j;u)=\infty$.
\end{corollary}

\begin{proof}
Since the proof is analogous to that of Corollary \ref{cor:BranchingEmb},
we omit the proof.
\end{proof}

%%%%%%%%%%%%%%%%%%%%%%%%%%%%%%%%%%%%%%%%%
\subsection{The relation to DSBOs \texorpdfstring{$\DD$}{D}}
\label{sec:BranchingDSBO}

To conclude this section, we remark on 
the relation of the branching laws of 
Verma modules to DSBOs $\DD$.
As in the previous subsections, we consider the $GL$-case and 
$SL$-case, separately.

%%%%%%%%%%%%%%%%%%%%%%%%%%%%%%%%%%%%%%%%%
\subsubsection{The \texorpdfstring{$GL$}{GL}-case}
\label{sec:BranchingDSBO_GL}

We start with the $GL$-case.
Let $\C_{\xi}$ denote the one-dimensional representation of $M$
defined in \eqref{eqn:characterM}. 
Then we write 
\begin{equation*}
\Mb(\xi, \lambda):=U(\frakg)\otimes_{U(\frakb)}(\C_{\xi}\boxtimes \C_{\lambda})
\end{equation*}
for the Verma module induced from the $P$-representation 
$\C_{\xi}\boxtimes \C_{\lambda}$. 
The Verma module $\Mb(\xi, \lambda)$ is a $(\frakg, P)$-module
via the diagonal action of $P$.

Likewise, we denote by
$\C_{\eta}$ the one-dimensional representation of $M_i'$ 
defined in \eqref{eqn:characterMprime} via the embedding
$\iota_i \colon GL(2,\R) \hookrightarrow GL(3,\R)$.
Then write
\begin{equation*}
M^{\frakg'_i}_{\frakb'_i}(\eta,\nu)
:=U(\frakg_i')\otimes_{U(\frakb_i')}(\C_{\eta}\boxtimes \C_{\nu})
\end{equation*}
for the  $(\frakg_i', P_i')$-module induced from 
$\C_{\eta}\boxtimes \C_{\nu}$.

By the duality theorem \cite{KoPe16a} 
between DSBOs and homomorphisms between generalized 
Verma modules, there exists a linear isomorphism
\begin{equation}\label{eqn:duality}
\Hom_{\frakg_i', P_i'}
(M^{\frakg'_i}_{\frakb'_i}(\eta,-\nu), \Mb(\xi, -\lambda))
\simeq 
\Diff_{G_i'}(I(\xi, \lambda), J(\eta,\nu)).
\end{equation}
Define
\begin{equation*}
\irred(\frakg_{i}',\frakb_{i}')
:=\{\nu \in \C^2 : \text{$M^{\frakg'_i}_{\frakb'_i}(\nu)$ is irreducible}\}.
\end{equation*}
Theorem \ref{thm:VermaCriterion} yields
\begin{equation}\label{eqn:irred_prime}
\irred(\frakg_{i}',\frakb_{i}')=\{\nu \in \C^2 : \nu_1-\nu_2 \notin \N\}.
\end{equation}
Then, for $\supp_i$ in \eqref{eqn:supp_i} and $\irred(\frakg,\frakb)$ in
\eqref{eqn:Irreducible_Parameters2}, we put
\begin{equation*}
\irredsupp_i:=
\big(\irred(\frakg,\frakb)\times \irred(\frakg_{i}',\frakb_{i}')\big) \cap \supp_i.
\end{equation*}
For $(-\lambda, -\nu) \in \irredsupp_i$,
we have
\begin{align*}
\dim_\C\Hom_{\frakg_i', P_i'}
(M^{\frakg'_i}_{\frakb'_i}(\eta,-\nu), \Mb(\xi, -\lambda))
&\leq 
\dim_\C\Hom_{\frakg_i'}
(M^{\frakg'_i}_{\frakb'_i}(-\nu), \Mb(-\lambda))\\[3pt]
&=m_i(-\nu;-\lambda),
\end{align*}
where $m_i(\nu;\lambda)$ is the multiplicity for the branching law
$[\Mb(\lambda)\vert_{\frakg_i'}]$ defined in \eqref{eqn:mult_i}.
It follows from Proposition \ref{prop:BranchingEmb} that 
$m_i(-\nu;-\lambda)=1$ for $(-\lambda, -\nu) \in \irredsupp_i$.
Thus, via \eqref{eqn:duality}, we obtain
\begin{equation}\label{eqn:DSBO_inequality}
\dim_\C\Diff_{G_i'}(I(\xi, \lambda), J(\eta,\nu)) \leq 1
\quad
\text{for $(-\lambda,-\nu) \in \irredsupp_i$}.
\end{equation}

In the following, we show 
that the inequality \eqref{eqn:DSBO_inequality} indeed holds for
$(-\lambda, -\nu) \in \irredsupp_i$ from the results in Sections \ref{sec:Emb3}--\ref{sec:Emb12}.
We start with the case $i=1,2$.

\begin{proposition}
Let $i \in \{1,2\}$. 
If $(-\lambda, -\nu) \in \irredsupp_i$, then
\begin{equation*}
\dim_\C\Diff_{G_i'}(I(\xi, \lambda), J(\eta,\nu)) \leq 1.
\end{equation*}
In addition, if $\xi, \eta$ satisfy the parity condition in 
\eqref{eq:condition_full_set_1}, then
\begin{equation*}
\dim_\C\Diff_{G_i'}(I(\xi, \lambda), J(\eta,\nu)) =1.
\end{equation*}
\end{proposition}

\begin{proof}
We only discuss the case $i=1$; the other case can be handled similarly.
The first inequality follows immediately from Theorem \ref{thm:class1}.
To show the second equality,
recall from \eqref{eqn:Phi1a} and \eqref{eqn:supp_i} that we have 
\begin{align*}
\Phi_{1,a} &:= \{(\lambda, \nu) \in \C^5: \nu_1-\lambda_1,\, \nu_2-\lambda_2 \in \N\},\\
\supp_{1} &= \{(\lambda, \nu) \in \C^5: \nu_1-\lambda_1,\nu_2-\lambda_2 \in -\N\}.
\end{align*}
Thus, the condition $(-\lambda, -\nu) \in \supp_{1}$ is equivalent to
$(\lambda, \nu) \in \Phi_{1,a}$. 
Therefore, by the definition of $\Supp_1(\DD)$ 
in Section \ref{sec:ClassThmsDSBOs2}, 
if $(-\lambda, -\nu) \in \supp_{1}$ and 
$\xi,\eta$ satisfy the parity condition in \eqref{eq:condition_full_set_1}, 
then $(\xi,\eta;\lambda,\nu) \in \Supp_i(\DD)$.
Now Theorem \ref{thm:class1} concludes the desired equality.
\end{proof}

We next discuss the case $i=3$. For this purpose, we first show the following lemma.

\begin{lemma}\label{lem:nunu}
Suppose that $(\lambda_1,\lambda_2,\lambda_3)\in \C^3$ 
and $(\nu_1,\nu_2) \in \C^2$ satisfy the following conditions.
\vskip 0.1in

\begin{itemize}

\item $\nu_1-\lambda_1, \lambda_3-\nu_2 \in \N$.
\vskip 0.1in

\item $\lambda_1-\lambda_3-1,\nu_1-\nu_2 \notin -\N$.
\vskip 0.1in
\end{itemize}
Then we have 
\begin{equation}\label{eqn:nunu}
\nu_1-\nu_2-1 \notin [1, \min(\nu_1-\lambda_1,\lambda_3-\nu_2)]\cap \Z.
\end{equation}

\end{lemma}

\begin{proof}
Write 
$a:=\nu_1-\nu_2 \notin -\N$ and 
$b:=\nu_1-\lambda_1, c:=\lambda_3-\nu_2 \in \N$.
Then we have
\begin{equation}\label{eqn:abc}
a-b-c-1=
\lambda_1-\lambda_3-1 \notin -\N.
\end{equation}
Now consider the following two cases for $a \notin -\N$:
(1) $a \notin \Z$ and (2) $a \in \Z$. 

If $a \notin \Z$, then the condition \eqref{eqn:nunu} clearly holds,
as $\nu_1-\nu_2-1=a-1 \notin \Z$.
Next  suppose that $a \in \Z$.
By \eqref{eqn:abc}, we have
$a-b-c-1 \in \N_+$.
Therefore, there exists $d \in \N$ such that 
$a-b-c-1=d+1$, yielding
\begin{equation*}
\nu_1-\nu_2-1=a-1=b+c+d+1>b, c.
\end{equation*}
Since $b=\nu_1-\lambda_1, c=\lambda_3-\nu_2$, this shows that 
$\nu_1-\nu_2-1 > \nu_1-\lambda_1, \lambda_3-\nu_2$.
This completes the proof.
\end{proof}

\begin{proposition}
For $(-\lambda, -\nu) \in \irredsupp_3$, we have 
\begin{equation*}
\dim_\C\Diff_{G_3'}(I(\xi, \lambda), J(\eta,\nu)) \leq 1.
\end{equation*}
In addition, if $\xi, \eta$ satisfy \eqref{eq:condition_full_set_3}, then 
\begin{equation*}
\dim_\C\Diff_{G_3'}(I(\xi, \lambda), J(\eta,\nu)) =1.
\end{equation*}
\end{proposition}

\begin{proof}
Recall from \eqref{eqn:Phi3} and \eqref{eqn:supp_i} that we have 
\begin{align*}
\Phi_{3}&:= \{(\lambda, \nu) \in \C^5: \nu_1-\lambda_1,\, \lambda_3-\nu_2 \in \N\},\\
\supp_{3} &:= \{(\lambda, \nu) \in \C^5: \nu_1-\lambda_1,\, \lambda_3-\nu_2 \in -\N\}.
\end{align*}
Thus, the condition $(-\lambda, -\nu) \in \supp_{3}$ is equivalent to
$(\lambda, \nu) \in \Phi_{3}$. 
Moreover, it follows from
\eqref{eqn:Irreducible_Parameters2} and
 \eqref{eqn:irred_prime} that 
if $(-\lambda, -\nu) \in \irredsupp_3$,
then $(\lambda,\nu)$ satisfies the conditions in Lemma \ref{lem:nunu}.
Thus, by the lemma, we have 
\begin{equation*}
\nu_1-\nu_2-1 \notin [1, \min(\nu_1-\lambda_1,\lambda_3-\nu_2)]\cap \Z,
\end{equation*}
which shows that $(-\lambda,-\nu) \in \Phi_{3,a}$.
Now, the proposed inequality and equality follow from Theorem \ref{thm:class3}.
\end{proof}

%%%%%%%%%%%%%%%%%%%%%%%%%%%%%%%%%%%%%%%%%
\subsubsection{The \texorpdfstring{$SL$}{SL}-case}
\label{sec:BranchingDSBO_SL}

We next consider the $SL$-case.
As for the $GL$-case, let 
$\C_{\delta}$ and $\C_{\sigma}$ denote the 
one-dimensional representations $M_{SL}$ and $M'_{SL,i}$
defined in \eqref{eqn:charSLM}
and
\eqref{eqn:charSLMprime}, 
respectively.
We then write
\begin{align*}
\Mbsl(\delta,u)
&:=U(\frakg_{\fraksl})\otimes_{U(\frakb_{\fraksl})}(\C_{\delta}\boxtimes \C_{u}),\\[3pt]
M^{\frakg'_{\fraksl,i}}_{\frakb'_{\fraksl,i}}(\sigma,v)
&:=U(\frakg_{\fraksl,i}')\otimes_{U(\frakb'_{\fraksl,i})}(\C_{\sigma}\boxtimes \C_{v})
\end{align*}
for the $(\frakg_{\fraksl},P_{SL})$-module and 
$(\frakg'_{\fraksl, i},P'_{SL,i})$-module induced from 
$\C_{\delta}\boxtimes \C_{u}$ and 
$\C_{\sigma}\boxtimes \C_{v}$, respectively.

Define
\begin{equation*}
\irred(\frakg_{\fraksl,i}',\frakb_{\fraksl,i}')
:=\{v \in \C : \text{$M^{\frakg'_{\fraksl,i}}_{\frakb'_{\fraksl,i}}(v)$ is irreducible}\}.
\end{equation*}
It follows from Theorem \ref{thm:VermaCriterion} that 
\begin{equation}\label{eqn:irred_prime_sl}
\irred(\frakg_{\fraksl,i}',\frakb_{\fraksl,i}')=\{v \in \C : v \notin \N\}.
\end{equation}
Then, for $\supp_{\fraksl,i}$ in \eqref{eqn:supp_i_sl} and 
\eqref{eqn:supp_3_sl}, and 
$\irred(\frakg_{\fraksl},\frakb_{\fraksl})$ in \eqref{eqn:irred_sl}, we put
\begin{equation*}
\irredsupp_{\fraksl,i}:=
\big(\irred(\frakg_{\fraksl},\frakb_{\fraksl})\times 
\irred(\frakg_{\fraksl,i}',\frakb_{\fraksl, i}')\big) \cap \supp_{\fraksl,i}.
\end{equation*}
The same argument in Section \ref{sec:BranchingDSBO_GL} shows that,
for $(-u, -v) \in \irredsupp_{\fraksl,i}$,we have
\begin{equation}\label{eqn:DSBO_inequality_SL}
\dim_\C\Diff_{G_{SL,i}'}(I(\delta, u)_{SL}, J(\sigma,v)_{SL}) \leq m_{\fraksl,i}(-v;-u),
\end{equation}
where $m_{\fraksl,i}(v;u)$ is the multiplicity of the branching law 
$[\Mbsl(u)\vert_{\frakg_{\fraksl,i}'}]$ defined in \eqref{eqn:mult_i_sl}.

We now show the inequality \eqref{eqn:DSBO_inequality_SL} indeed holds for
$(-u, -v) \in \irredsupp_{\fraksl,i}$ from the results in Section \ref{sec:SL}.

As for the $GL$-case, we start with the case $i=1,2$.
It follows from 
Proposition \ref{prop:BranchingSLEmb12} that, in this case, we have 
$m_{\fraksl,i}(-v;-u)=\infty$. Thus the inequality \eqref{eqn:DSBO_inequality_SL}
trivially holds. We then show that the equality could be attained
for appropriate $\delta, \sigma$.

\begin{proposition}
Let $i \in \{1,2\}$. If $(-\lambda, -\nu) \in \irredsupp_{\fraksl,i}$ and $\delta, \sigma$
satisfy the condition (Si) in Section \ref{sec:class-const-SL}, then
\begin{equation*}
\dim_\C\Diff_{G_{SL,i}'}(I(\delta, u)_{SL}, J(\sigma,v)_{SL}) =\infty.
\end{equation*}
\end{proposition}

\begin{proof}
It follows from \eqref{eqn:supp_i_sl} that 
if $(-\lambda, -\nu) \in \irredsupp_{\fraksl,i}$ and $\delta, \sigma$
satisfy the condition (Si), then $(\delta,\sigma; u, v) \in \Supp_{SL,i}(\DD)$.
Then Theorem \ref{thm:classSL12} concludes the proposition.
\end{proof}

We next discuss the case $i=3$. 
As for the $GL$-case, we first show one technical lemma.

\begin{lemma}\label{lem:nunuSL}
Let $\alpha_1, \alpha_2 \in \N$,
and suppose that $(u_1,u_2)\in \C^2$ and $v \in \C$ satisfy the following conditions.

\begin{itemize}

\item $v-(u_1+u_2)=\alpha_1+\alpha_2 \in \N$.
\vskip 0.1in

\item $u_1+u_2-1, v \notin -\N$.
\vskip 0.1in
\end{itemize}
Then we have 
\begin{equation}\label{eqn:nunuSL}
v-2 \notin [0, \min(\alpha_1,\alpha_2)-1]\cap \Z.
\end{equation}

\end{lemma}

\begin{proof}
If $v \notin \Z$, then \eqref{eqn:nunuSL} clearly holds.
So, assume that $v \in \Z$. As $v \notin -\N$, we have
$v \in \N_+$. Since $v-(u_1+u_2)$ is assumed to be $v-(u_1+u_2) \in \N$,
it implies that $u_1+u_2 \in \Z$. Moreover, since $u_1+u_2-1 \notin -\N$,
it shows that $u_1+u_2\in 2+ \N$. Thus, there exists $a\in \N$ such that 
$u_1+u_2 = a + 2$.
Then we have 
\begin{equation*}
v-2 
= (\alpha_1+\alpha_2) + (u_1+u_2)-2
=(\alpha_1+\alpha_2) + a 
\geq \min(\alpha_1,\alpha_2);
\end{equation*}
that is, $v-2$ satisfies \eqref{eqn:nunuSL}.
This concludes the lemma.
\end{proof}

To state the assertion for the case $i=3$, observe that,
by definition \eqref{eqn:supp_3_sl}, 
if $(-u, -v) \in \supp_{\fraksl,3}$, then $(u_1+u_2)-v \in -\N$.

\begin{proposition}\label{prop:dim_Diff_SL3}
For $(-u, -v) \in \irredsupp_{\fraksl,3}$ with
$j=-(u_1+u_2-v)\in\N$, we have 
\begin{equation*}
\dim_\C\Diff_{G_{SL,i}'}(I(\delta, u)_{SL}, J(\sigma,v)_{SL}) \leq j+1.
\end{equation*}
In addition, if $\delta, \sigma$ satisfy (S3) in Section \ref{sec:class-const-SL}, then
\begin{equation*}
\dim_\C\Diff_{G_{SL,i}'}(I(\delta, u)_{SL}, J(\sigma,v)_{SL}) =j+1.
\end{equation*}
\end{proposition}

\begin{proof}
Observe that, for $(-u,-v) \in \supp_{\fraksl,3}$, we have 
$(\delta, \sigma;u,v) \in \Supp_{SL,3}(\DD)$ if and only if 
$\delta, \sigma$ satisfy (S3) in Section \ref{sec:class-const-SL}.
Thus, by Theorem \ref{thm:classSL3}, it suffices to show the second assertion. 
Suppose that $\delta, \sigma$ satisfy (S3).
Since $(\delta, \sigma;u,v) \in \Supp_{SL,3}(\DD)$,
it follows from \eqref{eqn:GammaSL3} and 
\eqref{eqn:DIffDimSL3c} that
for $j=-(u_1+u_2-v)\in\N$, we have 
\begin{equation*}
\Diff_{G'_{SL,3}}(I(\delta,u)_{SL}, J(\sigma,v)_{SL})
=\bigoplus_{\alpha\in\Gamma_3(-j)}
(\Rest_3 \circ \Symb_0^{-1})(\Sol_{SL,3}^{(-j)}(\delta, u;\alpha))
\end{equation*}
with $\# \Gamma_3(-j) = j+1$. Since $(-u, -v) \in \irredsupp_{\fraksl,3}$,
it follows from \eqref{eqn:irred_sl} and \eqref{eqn:irred_prime_sl} that
$u_1, u_2, v$ satisfy the conditions in Lemma \ref{lem:nunuSL}.
Thus, by the lemma, we have $v-2 \notin [0,\min(\alpha_1,\alpha_2)-1]\cap \Z$.
As in the proof of Theorem \ref{thm:classSL3} in Section \ref{sec:ProofSL}, 
the number $b_0$ for 
the condition (C1) in Section \ref{sec:Step3} is given as $b_0=v-2$.
Thus, we have $b_0 \notin  [0, \min(\alpha_1,\alpha_2)-1]\cap \Z$,
which implies that 
$\dim_\C\Sol_{SL,3}^{(-j)}(\delta, u;\alpha)= 1$ for all $\alpha \in \Gamma_3(-j)$.
Therefore,
\begin{equation*}
\dim_\C \Diff_{G'_{SL,3}}(I(\delta,u)_{SL}, J(\sigma,v)_{SL})
=\sum_{\alpha \in \Gamma_3(-j)} \dim_\C\Sol_{SL,3}^{(-j)}(\delta, u;\alpha) = j+1.
\end{equation*}
This completes the proof.
\end{proof}

%%%%%%%%%%%%%%%%%%%%%%%%%%%%%%%%%%%%%%%%%
\begin{appendices}

%%%%%%%%%%%%%%%%%%%%%%%%%%%%%%%%%%%%%%%%%
\section{Symmetrization operator for the Heisenberg Lie algebra}
\label{appendix:Heis}

The aim of this appendix is to generalize the arguments in Sections \ref{sec:Upn} and \ref{sec:T-saturation} for the three-dimensional Heisenberg Lie algebra.
In particular, in Proposition \ref{appendix:eJU}, we give the symmetrization operator $S_{\ord}$ for the Heisenberg Lie algebra $\frakh_{2d+1}$ of dimension $2d+1$. The T-saturation of this operator is then obtained in Proposition \ref{appendix:DD}.

Throughout this appendix, let $\F:=\R$ or $\C$.

%%%%%%%%%%%%%%%%%%%%%%%%%%%%%%%%%%%%%%%%%
\subsection{Symmetrization operator}
\label{appendix:Heis1}

For $j=1,2,\ldots, d$, we put $\bar{j}:=j+d$. Then we write
\begin{equation*}
\frakh_{2d+1}:=\Span_\F\{X_j,X_{\bar{j}}, X_{2d+1}: j=1,\ldots, d\}
\end{equation*}
for the Heisenberg Lie algebra of dimension $2d+1$ 
with commutation relations
\begin{equation*}
[X_j, X_k] = 0,
\quad
[X_{\bar{j}}, X_{\bar{k}}] = 0,
\quad
[X_j, X_{\bar{k}}] = -\delta_{j,k} X_{2d+1},
\quad
[X_{2d+1},X_j]=0
\end{equation*}
for all $j,k=1,2,\ldots, d$, where $\delta_{j,k}$ is the Kronecker delta.

Fix an ordered basis $\ord$ of $\frakh$ as
\begin{equation}\label{eqn:ord-Heisenberg}
\ord:=\{X_1,\ldots, X_d, X_{\bar{1}}, \ldots, X_{\bar{d}}, X_{2d+1}\}.
\end{equation}
Thus, the corresponding PBW-basis $\mathrm{PBW}_\ord$ of 
$U(\frakh)$ is 
\begin{equation*}
\mathrm{PBW}_\ord
=\{X_1^{r_1}\cdots X_d^{r_d}
X_{\bar{1}}^{r_{\bar{1}}}\cdots X_{\bar{d}}^{r_{\bar{d}}} X_{2d+1}^{r_{2d+1}}
: r_j, r_{\bar{j}}, r_{2d+1} \geq 0\}.
\end{equation*}
As in \eqref{eqn:X-ord}, we write
\begin{equation*}
(X_{j_1}X_{j_2}\cdots X_{j_m})_\ord
=
X_{j_{\tau(1)}}X_{j_{\tau(2)}} \cdots X_{j_{\tau(m)}}
\in \mathrm{PBW}_\ord
\end{equation*}
for $\tau \in \mathfrak{S}_m$ such that $j_{\tau(k)} \leq j_{\tau(k+1)}$
for all $k=1,2,\ldots, m-1$.

We write 
$(\zeta_1, \dots, \zeta_d, \zeta_{\bar{1}},\ldots, \zeta_{\bar{d}}, \zeta_{2d+1})$
for the coordinates 
on $\frakh^\vee$ with respect to the 
dual basis of $\ord$.
Then the ordering map 
$\Upsilon_{\ord}$ from $\Pol(\frakh^\vee)$ to $U(\frakh)$ 
is given by
\begin{equation}\label{eqn:upsilon_ord3}
\Upsilon_{\ord} \colon \Pol(\frakh^\vee) \stackrel{\sim}{\to} U(\frakh),
\quad
\zeta_{j_1}\zeta_{j_2}\cdots \zeta_{j_m} 
\longmapsto
\big(X_{j_1}X_{j_2}\cdots X_{j_m}\big)_\ord.
\end{equation}
If $d=1$, then \eqref{eqn:upsilon_ord3} recovers the ordering map in 
\eqref{eqn:upsilon_ord2}.

The symmetrization $\sym_{\ord}$ on $U(\frakh_{2d+1})$ is computed by the following identity.

\begin{proposition}\label{appendix:SymmetrizationMapIdentity}
Let $k_j, k_{\bar{j}} ,r \in \N$ for  $j=1,\ldots, d$. 
Then, for $n_j:=k_j+k_{\bar{j}}$ and $n:=r+\sum_{j=1}^d n_j$,
we have 
\begin{align}\label{appendix:HeisSym1}
\sym_n\big(X_{2d+1}^r\prod_{j=1}^{d}X_j^{k_j}X_{\bar{j}}^{k_{\bar{j}}} \big)
&= X_{2d+1}^r\prod_{j=1}^d \sym_{n_j}(X_j^{k_j}X_{\bar{j}}^{k_{\bar{j}}}),
\end{align}
where
\begin{equation}\label{appendix:HeisSym2}
\sym_{n_j}(X_j^{k_j}X_{\bar{j}}^{k_{\bar{j}}})
=
\sum_{\ell=0}^{\min(k_j,k_{\bar{j}})}
\frac{(-k_{j})_\ell(-k_{\bar{j}})_\ell}{2^\ell \ell!}
X_j^{k_j-\ell}X_{\bar{j}}^{k_{\bar{j}}-\ell}X_{2d+1}^{\ell}.
\end{equation}
\end{proposition}

\begin{proof}
Both identities follow immediately from a property of $\sym_n$ and Proposition \ref{Prop:SymmetrizationMapIdentity}.
Indeed,  the first identity \eqref{appendix:HeisSym1} is given by
the fact that all basis elements except the pair $X_j, X_{\bar{j}}$ commute,
while 
the second identity \eqref{appendix:HeisSym2} can be shown 
exactly in the same way as 
Proposition \ref{Prop:SymmetrizationMapIdentity}.
\end{proof}

Now we compute the symmetrization operator 
$S_{\ord}$ on  
$\Pol(\frakh_{2d+1}^\vee)$ associated with the symmetrization
$\sym_\ord$ on $U(\frakh_{2d+1})$. That is, we aim to compute
a linear isomorphism
\begin{equation*}
S_{\ord} \colon 
\Pol(\frakh_{2d+1}^\vee) \stackrel{\sim}{\To} \Pol(\frakh_{2d+1}^\vee)
\end{equation*}
 such that
\begin{equation*}
\Upsilon_{\ord} \circ S_{\ord}=
\sym_{\ord} \circ \Upsilon_{\ord}.
\end{equation*}
In other words, we wish to find a linear map
$S_{\ord}$ on $\Pol(\frakh^\vee)$
such that
\begin{equation*}
(\Upsilon_{\ord} 
\circ
S_{\ord} )
\Big(
\zeta_{2d+1}^r
\prod_{j=1}^d
\zeta_j^{k_j}\zeta_{\bar{j}}^{k_{\bar{j}}}
\Big)
=
\sym_{\ord}\big(X_{2d+1}^r\prod_{j=1}^{d}X_j^{k_j}X_{\bar{j}}^{k_{\bar{j}}} \big)
\end{equation*}
for any $\zeta_{2d+1}^r
\prod_{j=1}^d
\zeta_j^{k_j}\zeta_{\bar{j}}^{k_{\bar{j}}} \in \Pol(\frakh^\vee)$.

To the end, for $i=1.\ldots, d$, we put
\begin{equation*}
J_i'
:=
\frac{\zeta_{2d+1}}{2}\frac{\partial^2}{\partial\zeta_{i}\partial\zeta_{\bar{i}}}
\end{equation*}
and 
formally define
\begin{equation*}
e^{J_i'}:=\sum_{\ell=0}^\infty \frac{(J_{i}')^\ell}{\ell!}.
\end{equation*}
As in \eqref{eqn:ijk}, a direct computation shows that,
for $\alpha_i, \alpha_{\bar{i}} \in \N^2$, we have 
\begin{equation}\label{appendix:ijk}
e^{J_{i}'}
\zeta_i^{\alpha_i}\zeta_{\bar{i}}^{\alpha_{\bar{i}}}
=\sum_{\ell=0}^{\min(\alpha_i,\alpha_{\bar{i}})}
 \frac{(-\alpha_i)_\ell(-\alpha_{\bar{i}})_\ell}{2^\ell \ell!}
\zeta_i^{\alpha_i-\ell}\zeta_{\bar{i}}^{\alpha_{\bar{i}}-\ell}\zeta_{2d+1}^{\ell}.
\end{equation}
Since
\begin{equation}\label{appendix:eJi}
e^{J_{i}'}
\Big(
\zeta_{2d+1}^r
\prod_{j=1}^d
\zeta_j^{\alpha_j}\zeta_{\bar{j}}^{\alpha_{\bar{j}}}
\Big)
 = 
\zeta_{2d+1}^r
\prod_{j\neq i}
\zeta_j^{\alpha_j}\zeta_{\bar{j}}^{\alpha_{\bar{j}}}
\cdot
e^{J_{i}'}
\big(\zeta_i^{\alpha_i}\zeta_{\bar{i}}^{\alpha_{\bar{i}}}\big),
\end{equation}
the operator $e^{J_{i}'}$ is well-defined  on $\Pol(\frakh_{2d+1}^\vee)$.

Now we write
\begin{equation*}
J_{2d+1}
:=
\sum_{i=1}^d J_i'
\end{equation*}
and put
\begin{equation*}
e^{J_{2d+1}}
:=\sum_{\ell=0}^\infty \frac{(J_{2d+1})^\ell}{\ell!}.
\end{equation*}
Since the $J_i'$ commute with each other, we have 
\begin{equation}\label{appendix:eJd}
e^{J_{2d+1}}=\prod_{i=1}^d e^{J_i'}.
\end{equation}

\begin{proposition}\label{appendix:eJU}
With respect to the ordered basis 
$\ord$ in \eqref{eqn:ord-Heisenberg}, 
we have
\begin{equation*}
S_{\ord} =e^{J_{2d+1}}.
\end{equation*}
\end{proposition}

\begin{proof}
It follows from \eqref{appendix:eJi} and \eqref{appendix:eJd} that 
\begin{equation*}
e^{J_{2d+1}}
\Big(
\zeta_{2d+1}^r
\prod_{j=1}^d
\zeta_j^{\alpha_j}\zeta_{\bar{j}}^{\alpha_{\bar{j}}}
\Big)
 = 
\zeta_{2d+1}^r
\prod_{j=1}^d
e^{J_{j}'}
\big(\zeta_j^{\alpha_j}\zeta_{\bar{j}}^{\alpha_{\bar{j}}}\big).
\end{equation*}
Now \eqref{appendix:HeisSym2} and \eqref{appendix:ijk} yield the desired assertion.
\end{proof}

For $j=1,2,\ldots, d$, 
write $\valpha_j:=(\alpha_j, \alpha_{\bar{j}}) \in \N^2$.
Then, for $k_j =0,1,\ldots, \min(\valpha_j)$, we put 
\begin{equation*}
\psi_{\valpha_j,k_j}
:=\zeta_j^{\alpha_j-k_j}\zeta_{\bar{j}}^{\alpha_{\bar{j}}-k_{j}}\zeta_{2d+1}^{k_j}
\end{equation*}
and
\begin{equation*}
\Pol(\valpha_j):=
\Span_\F\{\psi_{\valpha_j,k_j} : k_j = 0, 1,\ldots, \min(\valpha_j)\}.
\end{equation*}
For $\valpha:=(\valpha_1,\valpha_2,\ldots, \valpha_d, \alpha_{2d+1}) \in \N^{2d+1}$,
we define
\begin{equation*}
\Pol(\valpha)
:=\Span_{\C}\Big\{\prod_{j=1}^d \psi_{\valpha_j,k_j}: 
\psi_{\valpha_j,k_j} \in \Pol(\valpha_j)\Big\}.
\end{equation*}

\begin{proposition}\label{appendix:symU}
The subspace $\Pol(\valpha)$ is $e^{J_{2d+1}}$-invariant.
\end{proposition}

\begin{proof}
This simply follows from \eqref{appendix:ijk}.
\end{proof}

We write
\begin{align*}
J'_{j}(\valpha)
&:=J'_{j} \big\vert_{\Pol(\valpha)},\\[3pt]
J_{2d+1}(\valpha)
&:=J_{2d+1} \big\vert_{\Pol(\valpha)}.
\end{align*}
Then, by Propositions \ref{appendix:eJU} and \ref{appendix:symU}, we have 
\begin{equation}\label{eqn:symU1}
\sym_{\ord} \circ \Upsilon_{\ord}\vert_{\Pol(\valpha)} 
= \Upsilon_{\ord} \circ e^{J_{2d+1}(\valpha)}.
\end{equation}
Equivalently, the following diagram commutes.
\begin{equation}\label{appendix:diagram0}
\begin{tikzcd}[row sep=1cm, column sep=1cm]
\Pol(\valpha)
\arrow[r, "\Upsilon_{\ord}", hook]
& 
U(\frakh)
\arrow[dl, pos=0.5, phantom, "\circlearrowleft"]
\\
\Pol(\valpha)
\arrow[u, "e^{J_{2d+1}(\valpha)}","\sim"' sloped]
\arrow[r, "\Upsilon_{\ord}"', hook]
& U(\frakh)
\arrow[u, "\sim" sloped, "\sym_{\ord}"']
\end{tikzcd}
\end{equation}

%%%%%%%%%%%%%%%%%%%%%%%%%%%%%%%%%%%%%%%%%
\subsection{T-saturation}
\label{appendix:T-saturation}

As in \eqref{eqn:Tsat}, the space $\Pol(\valpha_j)$ can be identified with 
$\Pol_{\min(\valpha_j)}[t_j]$ via a T-saturation map
\begin{equation}\label{appendix:Tsat}
T_{\valpha_j}\colon
\Pol_{\min(\valpha_j)}[t_j]\xrightarrow{\sim}\Pol(\valpha_j),\quad 
p(t_j)\mapsto 
T_{\valpha_j}(p)(\zeta),
\end{equation}
where
\begin{equation*}
T_{\valpha_j}(p)(\zeta)
:=
\zeta_j^{\alpha_j}\zeta_{\bar{j}}^{\alpha_{\bar{j}}}p\Big(\frac{\zeta_{2d+1}}{\zeta_j\zeta_{\bar{j}}}\Big).
\end{equation*}
Define a subspace 
$\Pol_{\valpha}[t] \subset \Pol[t]:=\Pol[t_1,t_2,\ldots, t_d]$
 as
\begin{equation*}
\Pol_{\valpha}[t]:=
\Span_\F\big\{\prod_{j=1}^{d} p_j(t_j): p_j(t_j) \in \Pol_{\min(\valpha_j)}[t_j]\big\}.
\end{equation*}
Then the map
\begin{equation}\label{appendix:Tsat2}
T_{\valpha}\colon
\Pol_{\valpha_j}[t]\xrightarrow{\sim}\Pol(\valpha),\quad 
\prod_{j=1}^{d} p_j(t_j)\mapsto
\prod_{j=1}^{d} T_{\valpha_j}(p_j)(\zeta)
\end{equation}
is a T-saturation map for $\Pol(\valpha)$.

We wish to compute the T-saturation 
$T_{\valpha}^\sharp (e^{ J_{2d+1}(\valpha)})$ of $e^{J_{2d+1}(\valpha)}$
with respect to $T_{\valpha}$ in \eqref{appendix:Tsat2}.
To that end, 
for $\valpha_j=(\alpha_j,\alpha_{\bar{j}}) \in \NN^2$, we put
\begin{equation}\label{appendix:Da}
\wJ'_j(\valpha_j):=\frac{t_j}{2}(\vartheta_j-\alpha_j)(\vartheta_{\bar{j}}-\alpha_{\bar{j}}),
\end{equation}
where $\vartheta_j:=t_j \frac{d}{d t_j}$.
Then, as in \eqref{eqn:D1}, we have
\begin{equation}\label{appendix:D1}
\wJ'_j(\valpha)^\ell t_j^{k_j} 
=\frac{(k_j-\alpha_j)_\ell(k_j-\alpha_{\bar{j}})_\ell}{2^\ell}
t_j^{k_j+\ell},
\end{equation}
so, $\wJ_j'(\valpha) \in \End_\CC(\Pol_{\min(\valpha_j)}[t_j])$.
Then define 
$e^{\wJ_j'(\valpha_j)} \in GL(\Pol_{\min(\valpha_j)}[t_j])$ by
\begin{equation*}
e^{\wJ_j'(\valpha_j)}
:=\sum_{\ell=0}^\infty \frac{\wJ_j'(\valpha_j)^\ell}{\ell!}.
\end{equation*}
We write
\begin{equation*}
\wJ_{2d+1}(\valpha):=\sum_{j=1}^{d}\wJ'_j(\valpha_j)
\end{equation*}
and put
\begin{equation*}
e^{\wJ_{2d+1}(\valpha)}
:=\sum_{\ell=0}^\infty \frac{\wJ_{2d+1}(\valpha)^\ell}{\ell!}.
\end{equation*}
As for $e^{J_{2d+1}}$ in \eqref{appendix:eJd}, 
since the $\wJ'_j(\valpha_j)$ commute with each other, we have 
\begin{equation}\label{appendix:e_wJ_j}
e^{\wJ_{2d+1}(\valpha)}=\prod_{i=1}^d e^{\wJ_j'(\valpha_j)}.
\end{equation}

\begin{proposition}\label{appendix:DD}
We have 
\begin{equation*}
e^{\wJ_{2d+1}(\valpha)}=T_{\valpha}^\sharp (e^{ J_{2d+1}(\valpha)}).
\end{equation*}
\end{proposition}

\begin{proof}
It suffices to show
\begin{equation*}
\wJ_{2d+1}(\valpha)=T_{\valpha}^\sharp J_{2d+1}(\valpha).
\end{equation*}
The lemma then immediately follows from \eqref{appendix:ijk} and \eqref{appendix:D1}.
\end{proof}

By Proposition \ref{appendix:DD},
the following commutative diagram  holds.
\begin{equation}\label{appendix:diagram01}
\begin{tikzcd}[row sep=1cm, column sep=1cm]
\Pol_{\valpha}[t]
\arrow[dr, pos=0.2, phantom, "\circlearrowleft"]
\arrow[r, "T_{\valpha}", "\sim"']
& 
\Pol(\valpha)
\arrow[r, "\Upsilon_{\ord}", hook]
& 
U(\frakh_{2d+1})
\arrow[dl, pos=0.5, phantom, "\circlearrowleft"]
\\
\Pol_{\valpha}[t]
\arrow[u, "e^{\wJ_{2d+1}(\valpha)}","\sim"' sloped]
\arrow[r,  "\sim", "T_{\valpha}"']
& 
\Pol(\valpha)
\arrow[u, "e^{ J_{2d+1}(\valpha)}","\sim"' sloped]
\arrow[r, "\Upsilon_{\ord}"', hook]
& U(\frakh_{2d+1})
\arrow[u, "\sim" sloped, "\sym_{\ord}"']
\end{tikzcd}
\end{equation}

\begin{remark}\label{appendix:Heis_2F0}
For each $j=1,\ldots, d$,  the images of $1$ under the maps
 in \eqref{appendix:diagram01} are traced as 
follows.
\begin{equation*}
\begin{tikzcd}[row sep=1cm, column sep=1cm]
{}_2F_0[-\alpha_j, -\alpha_{\bar{j}};\tfrac{t_j}{2}]
\arrow[r, "T_{\valpha_j}", mapsto]
& 
T_{\valpha_j}({}_2F_0[-\alpha_j, -\alpha_{\bar{j}};\tfrac{t_j}{2}])
\arrow[dl, pos=0.5, phantom, "\circlearrowleft"]
\arrow[r, "\Upsilon_{\ord}", mapsto]
& 
\sym_{\alpha_j+\alpha_{\bar{j}}}(X_j^{\alpha_j}X_{\bar{j}}^{\alpha_{\bar{j}}})
\arrow[dl, pos=0.5, phantom, "\circlearrowleft"]
\\
1
\arrow[u, "e^{\wJ'_j(\valpha_j)}", mapsto]
\arrow[r, "T_{\valpha_j}"', mapsto]
& 
\zeta_j^{\alpha_j}\zeta_{\bar{j}}^{\alpha_{\bar{j}}}
\arrow[u, "e^{J_j'(\valpha)}",mapsto]
\arrow[r, "\Upsilon_{\ord}"', mapsto]
& X_j^{\alpha_j}X_{\bar{j}}^{\alpha_{\bar{j}}}
\arrow[u,  "\sym_{\ord}"', mapsto]
\end{tikzcd}
\end{equation*}
Moreover, by \eqref{appendix:D1} and \eqref{appendix:e_wJ_j}, we have
\begin{equation*}
e^{\wJ_{2d+1}(\valpha)} 1
= \prod_{j=1}^d
\sum_{\ell=0}^{\min(\valpha_j)}
\frac{(-\alpha_j)_\ell(-\alpha_{\bar{j}})_\ell}{2^\ell \ell!}t_j^\ell
=\prod_{j=1}^d {}_2F_0[-\alpha_j, -\alpha_{\bar{j}};\tfrac{t_j}{2}].
\end{equation*}
\end{remark}

%%%%%%%%%%%%%%%%%%%%%%%%%%%%%%%%%%%%%%%%%
\section{Cayley continuants and Proof of Proposition \texorpdfstring{\ref{Prop:SolGeneralDiffEq1}}{8.12}}
\label{appendix:Cayely}

The aim of this section is to give a quick overview of a certain family of polynomials 
$\Cay_n(x,y)$ called the Cayley continuants. We also give a proof of 
Proposition \ref{Prop:SolGeneralDiffEq1}, which concerns these polynomials.

%%%%%%%%%%%%%%%%%%%%%%%%%%%%%%%%%%%%%%%%%
\subsection{Cayley continuants}
\label{appendix:CayleyCont}

We start with the definition of the Cayley continuants $\Cay_n(x,y)$. 
The \emph{Cayley continuants} $\Cay_n(x,y)$ for $m\in \NN_0$ are
the following polynomials of two variables $x,y$
(cf.\ \cite{Cayley58,MT05}):
\begin{itemize}
\item $n=0$: $\Cay_0(x,y)=1$, \vspace{3pt}
\item $n=1$: $\Cay_1(x,y)=x$, \vspace{3pt}
\item $n\geq 2$: 
\begin{equation}\label{eqn:Cayley}
\Cay_n(x,y)=\det \begin{pmatrix}
    x & 1 &     &    & \\
    y &  x & 2 &    &\\
        & y-1 & x & 3 & \\
    & & \ddots & \ddots & \ddots &\\
    & & & y-n+3 & x & n-1\\
    & & & &y-n+2 & x
\end{pmatrix}.
\end{equation}
\end{itemize}

It is known that the Cayley continuants $\Cay_n(x,y)$ enjoy the following 
generating function (cf.\ \cite[(4)]{MT05}):
\begin{equation}\label{eqn:gen_Cay}
(1+t)^{\frac{y+x}{2}}(1-t)^{\frac{y-x}{2}} = 
\sum_{m=0}^\infty\Cay_n(x,y) \frac{t^n}{n!}.
\end{equation}
It follows from \eqref{eqn:gen_Cay} that 
$\Cay_n(x,y)$ can be expressed in terms of the Pochhammer symbol $(x)_n=x(x+1)\cdots(x+n-1)=\frac{\Gamma(x+n)}{\Gamma(x)}$ as
\begin{equation}\label{eq:CayleyContExpansion}
\Cay_n(x,y)=\sum_{j=0}^n(-1)^j\binom{n}{j}\Big(\frac{x-y}{2}\Big)_{n-j}\Big(-\frac{x+y}{2}\Big)_{j}.
\end{equation}
In particular, we have
\begin{equation*}
\Cay_n(y,y)= (-1)^n(-y)_n
\quad
\text{and}
\quad
\Cay_n(-y,y)=(-y)_n.
\end{equation*}
For more details, see, for instance, \cite{MT05}.

%%%%%%%%%%%%%%%%%%%%%%%%%%%%%%%%%%%%%%%%%
\subsection{Proof of Proposition \texorpdfstring{\ref{Prop:SolGeneralDiffEq1}}{8.12}}
\label{appendix:CayleyCont3}

We next prove Proposition \ref{Prop:SolGeneralDiffEq1}.
Recall from Section \ref{sec:T-saturation} that,
for $\alpha=(\alpha_1,\alpha_2) \in \NN^2$, we write
\begin{equation}\label{eqn:Da2}
\wJ(\alpha)=
\wJ(\alpha_1, \alpha_2):=\frac{t}{2}(\vartheta_t-\alpha_1)(\vartheta_t-\alpha_2),
\end{equation}
and the operator
$e^{\wJ(\alpha)} \in GL(\Pol_{\min(\alpha)}[t])$ is defined by
\begin{equation*}
e^{\wJ(\alpha)}=
e^{\wJ(\alpha_1, \alpha_2)}
:=\sum_{\ell=0}^\infty \frac{\wJ(\alpha_1,\alpha_2)^\ell}{\ell!}.
\end{equation*}
For $k,\ell\in \N$, we have
\begin{equation}\label{eqn:wJtk}
\wJ(\alpha_1, \alpha_2)^\ell t^k 
=\frac{(k-\alpha_1)_\ell(k-\alpha_2)_\ell}{2^\ell}
t^{k+\ell},
\end{equation}
where $(x)_\ell:=x(x+1) \cdots (x+\ell-1) = \frac{\Gamma(x+\ell)}{\Gamma(x)}$.
Thus, for $k=0, 1\ldots, \min(\alpha)$, we have
\begin{equation}\label{eqn:eJk}
e^{\tilde{J}(\alpha_1, \alpha_2)}t^k
= \sum_{n=0}^{\infty}\frac{(k-\alpha_1)_n(k-\alpha_2)_n}{2^nn!}t^{n+k}
= \sum_{n=0}^{\min(\alpha)-k}\frac{(k-\alpha_1)_n(k-\alpha_2)_n}{2^nn!}t^{n+k}.
\end{equation}

\begin{proposition}\label{prop:Cayley_and_3F1_relation}
Let $a, b, c \in \C$, $\ell \in \N$ and $\alpha = (\alpha_1, \alpha_2), \beta = (\beta_1, \beta_2) \in \N^2$.
Then, the following identities hold whenever ${}_3F_1$ is 
well-defined.
\begin{align}\label{eq:Cayley_and_3F1_relation}
e^{-\wJ(\alpha_1, \alpha_2)}{}_3F_1\left[\begin{matrix}
    -\alpha_1, -\alpha_2, \frac{a-b}{2}\\
    -b
\end{matrix}; t\right] = \sum_{n=0}^{\min(\alpha)} \frac{(-\alpha_1)_n(-\alpha_2)_n}{2^n n! (-b)_n}\Cay_n(a,b)t^n,\\
\label{eq:Cayley_and_3F1_relation_2}
e^{-\wJ(\beta_1 + \ell, \beta_2 + \ell)}t^{\ell}{}_3F_1\left[\begin{matrix}
    -\beta_1, -\beta_2, \frac{a-c}{2}\\
    -c
\end{matrix}; t\right] = t^\ell \sum_{n=0}^{\min(\beta)} \frac{(-\beta_1)_n(-\beta_2)_n}{2^n n! (-c)_n}\Cay_n(a,c)t^n.
\end{align}
\end{proposition}

\begin{proof} We start with \eqref{eq:Cayley_and_3F1_relation}.
Firstly, it follows from \eqref{eq:CayleyContExpansion} that
$\Cay_n(a,b)$ can be rewritten as
\begin{equation*}
\Cay_n(a,b) = \left(\frac{a-b}{2}
\right)_n {}_2F_1\left[\begin{matrix}
     -n, \frac{-a-b}{2}\\
    1-\frac{a-b}{2}-n
\end{matrix}; -1 \right].
\end{equation*}
Therefore, the right-hand side of \eqref{eq:Cayley_and_3F1_relation} is given as
\begin{align*}
\text{(RHS)}
&=\sum_{n=0}^{\min(\alpha)} \frac{(-\alpha_1)_n(-\alpha_2)_n}{2^n n! (-b)_n}\Cay_n(a,b)t^n\\[3pt]
&=\sum_{n=0}^{\min(\alpha)} \frac{(-\alpha_1)_n(-\alpha_2)_n}{2^n n! (-b)_n}\left(\frac{a-b}{2}
\right)_n {}_2F_1\left[\begin{matrix}
     -n, \frac{-a-b}{2}\\
    1-\frac{a-b}{2}-n
\end{matrix}; -1 \right]t^n.
\end{align*}
On the other hand, by \eqref{eqn:eJk} and the property 
$(x)_r(x+r)_k = (x)_{r+k}$ of the Pochhammer symbol, 
the left-hand side can be computed as
\begin{align*}
\text{(LHS)}
&=e^{-\wJ(\alpha_1, \alpha_2)}{}_3F_1\left[\begin{matrix}
    -\alpha_1, -\alpha_2, \frac{a-b}{2}\\
    -b
\end{matrix}; t\right]\\[3pt]
&\sum_{k=0}^{\min(\alpha)}\frac{(-\alpha_1)_k(-\alpha_2)_k(\frac{a-b}2)_k}{(-b)_k k!}\left(\sum_{r=0}^{\min(\alpha)-k}\frac{(k-\alpha_1)_r(k-\alpha_2)_r}{2^r r!}(-1)^r t^{k+r}\right) \\[3pt]
=& 
\sum_{n=0}^{\min(\alpha)} \sum_{k=0}^{n} \frac{(-\alpha_1)_k(-\alpha_2)_k(k-\alpha_1)_{n-k}(k-\alpha_2)_{n-k}(\frac{a-b}2)_k (-1)^{n-k}}{2^{n-k}k!(-b)_k(n-k)!}t^n\\[3pt]
=& \sum_{n=0}^{\min(\alpha)}\frac{(-\alpha_1)_n(-\alpha_2)_n}{2^n n!} (-1)^n {}_2F_1\left[\begin{matrix}
-n, \frac{a-b}{2}\\
-b
\end{matrix}; 2
\right]t^n.
\end{align*}
The lemma follows now from the following transformation formula for ${}_2F_1$ 
for $n \in \NN$ and $x, y \in \C$ with $y, 1-x-n <-n$
(cf.\ \cite[Eq.\ (15.8.6)]{DLMF}):
\begin{equation*}
{}_2F_1 \left[\begin{matrix}
-n, x\\
y
\end{matrix}; t
\right] = \frac{(x)_n}{(y)_n}(1-t)^n
{}_2F_1 \left[\begin{matrix}
-n, y-x\\
1-x-n
\end{matrix}; \frac{1}{1-t}
\right].
\end{equation*}

To show the second identity
\eqref{eq:Cayley_and_3F1_relation_2},
observe that, for $p(t) \in \Pol_{\min(\beta)}[t]$,
we have 
\begin{equation*}
(-\wJ(\beta_1+\ell, \beta_2+\ell)^nt^\ell)p(t) 
=t^\ell(-\wJ(\beta_1, \beta_2)^n)p(t) 
\end{equation*}
for any $n, \ell \in \NN$ by \eqref{eqn:wJtk}, which implies that
\begin{equation}\label{eqn:etn}
\big(e^{-\wJ(\beta_1+\ell, \beta_2+\ell)} t^\ell)
p(t)
= \big(t^\ell e^{-\wJ(\beta_1, \beta_2)}\big)p(t).
\end{equation}
The identity \eqref{eq:Cayley_and_3F1_relation_2}
is then a direct consequence of the first one and \eqref{eqn:etn}.
This completes the proof.
\end{proof}

Recall from \eqref{eqn:pab} that we define a polynomial $p^{(\alpha_1, \alpha_2)}_{a,b}(t) \in \Pol_{\min(\alpha)}[t]$ as 
\begin{equation*}
p^{(\alpha_1, \alpha_2)}_{a,b}(t)
:=\sum_{n=0}^{\min(\alpha)}\frac{(-\alpha_1)_n(-\alpha_2)_n}{2^n n!(-b)_n}
\Cay_n(a,b)\,t^n.
\end{equation*}
In terms of $p^{(\alpha_1, \alpha_2)}_{a,b}(t)$, the identities
\eqref{eq:Cayley_and_3F1_relation} and \eqref{eq:Cayley_and_3F1_relation_2}
are given as
\begin{align}
p^{(\alpha_1, \alpha_2)}_{a,b}(t)&=
e^{-\wJ(\alpha_1, \alpha_2)}{}_3F_1\left[\begin{matrix}
    -\alpha_1, -\alpha_2, \frac{a-b}{2}\\
    -b
\end{matrix}; t\right],\label{eqn:p3F1}\\[3pt]
t^\ell p^{(\beta_1, \beta_2)}_{a,c}(t)&=e^{-\wJ(\beta_1 + \ell, \beta_2 + \ell)}t^{\ell}{}_3F_1\left[\begin{matrix}
    -\beta_1, -\beta_2, \frac{a-c}{2}\\
    -c
\end{matrix}; t\right]. \nonumber
\end{align}

%%%%%%%%%%%%%%%%%%%%%%%%%%%%%%%%%%%%%%%%%
\section{Alternative proof of Corollary \texorpdfstring{\ref{cor:Sol2F0}}{9.27}}
\label{appendix:2F0}

The aim of this section is to give an alternative proof of Corollary \ref{cor:Sol2F0},
that is, the space
$\wSol_{\Ad(e^{\wJ(\alpha)}) \calP,\, 1}^{(\alpha)}(\xi,\lambda)$
of polynomial solutions
in $\Pol_{\min(\alpha)}[t]$ 
is classified as follows.
\begin{equation}\label{eqn:CorSol}
\wSol_{\Ad(e^{\wJ(\alpha)}) \calP,\, 1}^{(\alpha)}(\xi,\lambda)
=
\begin{cases}
\CC\, 
{}_2F_0[-\alpha_1,\tfrac{a_{0,1}-\alpha_2}{2};t]
& 
\textnormal{if (A1) or (A2) holds,}\\[3pt]
\{0\} & \textnormal{otherwise,}
\end{cases}
\end{equation}
where (A1) and (A2) are some conditions on 
 $(\alpha_1, \alpha_2, a_{0,1})$ given in Section \ref{sec:Thms12B}.

We start with a lemma playing a key role in the alternative proof.

\begin{lemma}\label{lem:eDD}
As elements of $GL(\Pol_{\min(\alpha)}[t])$ we have the following identity:
\begin{equation*}
e^{\tilde{J}(\alpha_1,\alpha_2+1)}e^{-\tilde{J}(\alpha_1,\alpha_2)}=1+\tfrac{t}{2}(\alpha_1-\vartheta_t).
\end{equation*}
\end{lemma}

\begin{proof}
Take $t^k \in \Pol_{\min(\alpha)}[t]$. We wish to prove
\begin{equation*}
e^{\tilde{J}(\alpha_1,\alpha_2+1)}e^{-\tilde{J}(\alpha_1,\alpha_2)}t^k
=\big(1+\tfrac{t}{2}(\alpha_1-\vartheta_t)\big)t^k.
\end{equation*}
Clearly, we have 
\begin{equation*}
\big(1+\tfrac{t}{2}(\alpha_1-\vartheta_t)\big)t^k 
= \big(1+\tfrac{t}{2}(\alpha_1- k)\big)t^k.
\end{equation*}
Thus we aim to show
\begin{equation*}
e^{\tilde{J}(\alpha_1,\alpha_2+1)}e^{-\tilde{J}(\alpha_1,\alpha_2)}t^k
= \big(1+\tfrac{t}{2}(\alpha_1- k)\big)t^k.
\end{equation*}

In the following, we simply write finite sums as infinite sums as in \eqref{eqn:eJk}.
Then, by \eqref{eqn:eJk}, we have 
\begin{equation*}
e^{-\tilde{J}(\alpha_1, \alpha_2)}t^k
= \sum_{n=0}^{\infty}(-1)^n\frac{(k-\alpha_1)_n(k-\alpha_2)_n}{2^nn!}t^{n+k}.
\end{equation*}
Thus, $e^{\tilde{J}(\alpha_1,\alpha_2+1)}e^{-\tilde{J}(\alpha_1,\alpha_2)}t^k$ 
is computed as
\begin{multline}
e^{\tilde{J}(\alpha_1,\alpha_2+1)}e^{-\tilde{J}(\alpha_1,\alpha_2)}t^k
=\\
\sum_{m=0}^{\infty}\sum_{n=0}^{\infty}
(-1)^n\frac{(k-\alpha_1)_n(k-\alpha_2)_n(n+k-\alpha_1)_m(n+k-\alpha_2-1)_m}{2^{n+m}n!m!} t^{n+m+k}. \label{eqn:eJeJ}
\end{multline}
By the property $(x)_\ell(x+r)_k = (x)_{r+k}$,
the expression \eqref{eqn:eJeJ} can be reduced to
\begin{align}
\eqref{eqn:eJeJ}
&=\sum_{m=0}^{\infty}\sum_{n=0}^{\infty}
(-1)^n\frac{(k-\alpha_1)_n(k-\alpha_2)_n(n+k-\alpha_1)_m(n+k-\alpha_2-1)_m}{2^{n+m}n!m!} t^{n+m+k} \nonumber \\[3pt]
&=t^k \sum_{\ell=0}^\infty \frac{(k-\alpha_1)_\ell}{2^\ell}t^\ell
\left(\sum_{n+m=\ell} (-1)^n\frac{(k-\alpha_2)_n(n+k-\alpha_2-1)_m}{n!m!}\right)\nonumber\\[3pt]
&=\frac{t^k}{k-\alpha_2-1}
 \sum_{\ell=0}^\infty \frac{(k-\alpha_1)_\ell(k-\alpha_2-1)_\ell}{2^\ell}t^\ell
\left(\sum_{n+m=\ell} \frac{(-1)^n}{n!m!}(n+k-\alpha_2-1)\right)\nonumber\\[3pt]
&=\frac{t^k}{k-\alpha_2-1}
 \sum_{\ell=0}^\infty \frac{(k-\alpha_1)_\ell(k-\alpha_2-1)_\ell}{2^\ell \ell!}t^\ell
\left(\sum_{n=0}^\ell(-1)^n\binom{\ell}{n}(n+k-\alpha_2-1)\right).\label{eqn:eJeJ2}
\end{align}
For the inner sum $\sum_{n=0}^\ell(-1)^n\binom{\ell}{n}(n+k-\alpha_2-1)$ 
in \eqref{eqn:eJeJ2}, observe that 
\begin{alignat*}{2}
\sum_{n=0}^\ell(-1)^n\binom{\ell}{n}
&= (1-x)^\ell\big\vert_{x=1} 
&&= 
\begin{cases}
1 & \text{if $\ell = 0$},\\
0 & \text{otherwise},
\end{cases}\\[5pt]
\sum_{n=1}^\ell(-1)^n\binom{\ell}{n}n
&= \frac{d}{dx}\bigg\vert_{x=1}(1-x)^\ell 
&&= 
\begin{cases}
-1 & \text{if $\ell = 1$},\\
0 & \text{otherwise}.
\end{cases}\\[3pt]
\end{alignat*}
Thus, the sum $\sum_{n=0}^\ell(-1)^n\binom{\ell}{n}(n+k-\alpha_2-1)$ is 
evaluated to
\begin{align*}
\sum_{n=0}^\ell(-1)^n\binom{\ell}{n}(n+k-\alpha_2-1)
=
\begin{cases}
(k-\alpha_2-1) & \text{if $\ell=0$},\\[5pt]
-1 & \text{if $\ell=1$},\\[3pt]
0 & \text{otherwise}.
\end{cases}
\end{align*}
Hence, we have 
\begin{align*}
\eqref{eqn:eJeJ2}
&=\frac{t^k}{k-\alpha_2-1}
\sum_{\ell=0}^\infty \frac{(k-\alpha_1)_\ell(k-\alpha_2-1)_\ell}{2^\ell \ell!}t^\ell
\left(\sum_{n=0}^\ell(-1)^n\binom{\ell}{n}(n+k-\alpha_2-1)\right)\\[5pt]
&=\frac{t^k}{k-\alpha_2-1}
\big( (k-\alpha_2-1)-\tfrac{t}{2} (k-\alpha_2-1)(k-\alpha_1)\big)\\[5pt]
&=\big(1+\tfrac{t}{2}(\alpha_1-k)\big)t^k.
\end{align*}
This completes the proof.
\end{proof}

\begin{proposition}\label{prop:eDsec9}
For any $a\in \C$, 
as elements of $\End_\CC(\Pol_{\min(\alpha)}[t])$, we have
\begin{equation*}
(1+\tfrac{t}{2}(\alpha_1-\vartheta_t))\Ad(e^{\wJ(\alpha)})\calP^{(\alpha)}(a)
={}_2\calF_0(-\alpha_1,\tfrac{a-\alpha_2}{2};t).
\end{equation*}
\end{proposition}

\begin{proof}
As in \eqref{eqn:LPapx1}, we have 
\begin{equation*}
\calP^{(\alpha)}(a)(\alpha_2+1-\vartheta_t)
=
\calL^{(\alpha_1,\alpha_2+1)}(a,\alpha_2),
\end{equation*}
where $\calL^{(\alpha_1,\alpha_2)}(a,b)$ is the differential operator defined in 
\eqref{eqn:Lalpha}. By applying $\Ad(e^{\widetilde{J}(\alpha_1,\alpha_2+1)})$
to both sides of the identity, we have 
\begin{equation}\label{eqn:AdLP}
\Ad(e^{\widetilde{J}(\alpha_1,\alpha_2+1)}) 
\big(\calP^{(\alpha)}(a)(\alpha_2+1-\vartheta_t)\big)
=
\Ad(e^{\widetilde{J}(\alpha_1,\alpha_2+1)})
\calL^{(\alpha_1,\alpha_2+1)}(a,\alpha_2).
\end{equation}
It follows from Proposition \ref{prop:AdL} that the right-hand side \eqref{eqn:AdLP}
evaluates to 
\begin{align}
\Ad(e^{\widetilde{J}(\alpha_1,\alpha_2+1)})\calL^{(\alpha_1,\alpha_2+1)}(a,\alpha_2)
&=
- {}_3\calF_1(-\alpha_1,-\alpha_2-1,\tfrac{a-\alpha_2}{2};-\alpha_2;t) \nonumber\\[3pt]
&=
{}_2\calF_0(-\alpha_1,\tfrac{a-\alpha_2}{2};t)(\alpha_2+1-\vartheta_t).
\label{eqn:Ad2F0}
\end{align}
On the other hand, the left-hand side of \eqref{eqn:AdLP} can be given as
\begin{align}
&\Ad(e^{\widetilde{J}(\alpha_1,\alpha_2+1)}) 
\big(\calP^{(\alpha)}(a)(\alpha_2+1-\vartheta_t)\big) \nonumber\\
&
\hspace{100pt}
=
\Big(\Ad(e^{\tilde{J}(\alpha_1,\alpha_2+1)})\calP^{(\alpha)}(a)\Big)\Big(\Ad(e^{\tilde{J}(\alpha_1,\alpha_2+1)})(\alpha_2+1-\vartheta_t)\Big),\label{eqn:AdLP2}
\end{align}
where, by Lemma \ref{lem:eD} and \ref{lem:eDD}, 
the expression $\Ad(e^{\tilde{J}(\alpha_1,\alpha_2+1)})(\alpha_2+1-\vartheta_t)$
simplifies to
\begin{align*}
\Ad(e^{\tilde{J}(\alpha_1,\alpha_2+1)})(\alpha_2+1-\vartheta_t)&=\alpha_2+1-\vartheta_t+\tilde{J}(\alpha_1,\alpha_2+1)\\
&=(1+\tfrac{t}{2}(\alpha_1-\vartheta_t))(\alpha_2+1-\vartheta_t)\\
&=e^{\tilde{J}(\alpha_1,\alpha_2+1)}e^{-\tilde{J}(\alpha_1,\alpha_2)}(\alpha_2+1-\vartheta_t).
\end{align*}
Therefore, the right-hand side of \eqref{eqn:AdLP2} evaluates to
\begin{align}
&\Big(\Ad(e^{\tilde{J}(\alpha_1,\alpha_2+1)})\calP^{(\alpha)}(a)\Big)
\Big(\Ad(e^{\tilde{J}(\alpha_1,\alpha_2+1)})(\alpha_2+1-\vartheta_t)\Big)
\nonumber\\[3pt]
&=
\Big(\Ad(e^{\tilde{J}(\alpha_1,\alpha_2+1)})\calP^{(\alpha)}(a)\Big)
e^{\tilde{J}(\alpha_1,\alpha_2+1)}e^{-\tilde{J}(\alpha_1,\alpha_2)}(\alpha_2+1-\vartheta_t)\nonumber\\[3pt]
&=e^{\tilde{J}(\alpha_1,\alpha_2+1)}e^{-\tilde{J}(\alpha_1,\alpha_2)}
\Big(\Ad(e^{\tilde{J}(\alpha_1,\alpha_2)})\calP^{(\alpha)}(a)\Big)(\alpha_2+1-\vartheta_t)
\nonumber \\[3pt]
&=(1+\tfrac{t}{2}(\alpha_1-\vartheta_t))\Big(\Ad(e^{\tilde{J}(\alpha_1,\alpha_2)})\calP^{(\alpha)}(a)\Big)(\alpha_2+1-\vartheta_t),\label{eqn:AdLP3}
\end{align}
where Lemma \ref{lem:eDD} is applied from line three to line four.
By \eqref{eqn:Ad2F0} and \eqref{eqn:AdLP3}, we have 
\begin{equation*}
(1+\tfrac{t}{2}(\alpha_1-\vartheta_t))\Big(\Ad(e^{\tilde{J}(\alpha_1,\alpha_2)})\calP^{(\alpha)}(a)\Big)(\alpha_2+1-\vartheta_t)
={}_2\calF_0(-\alpha_1,\tfrac{a-\alpha_2}{2};t)(\alpha_2+1-\vartheta_t).
\end{equation*}
As $(\alpha_2+1-\vartheta_t) \in GL(\Pol_{\min(\alpha)}[t])$, 
the inverse $(\alpha_2+1-\vartheta_t)^{-1}$ is well-defined in 
$\End_{\CC}(\Pol_{\min(\alpha)})$.
Then the desired identity is obtained by 
applying $(\alpha_2+1-\vartheta_t)^{-1}$  from the right to both sides.
\end{proof}

We define
\begin{equation*}
\Sol^{(\alpha)}_{2\calF0}(a)
:=\{q(t) \in \Pol_{\min{(\alpha)}}[t]: 
{}_2\calF_0(-\alpha_1,\tfrac{a-\alpha_2}{2};t)q(t)=0\}.
\end{equation*}

\begin{corollary}\label{cor:Sol2F02}
We have 
\begin{equation*}
\wSol_{\Ad(e^{\wJ(\alpha)}) \calP,\, 1}^{(\alpha)}(\xi,\lambda) 
= \Sol^{(\alpha)}_{2\calF0}(a_{0,1}).
\end{equation*}
\end{corollary}

\begin{proof}
As $\Ker\big( (\alpha_2+1-\vartheta_t)\vert_{\Pol_{\min(\alpha)}}\big)=\{0\}$,
it follows from Proposition \ref{prop:eDsec9} that
the following conditions on $q(t) \in \Pol_{\min(\alpha)}[t]$ are equivalent:
\vspace{3pt}
\begin{enumerate}[label=\normalfont{(\roman*)}]
\item ${}_2\calF_0(-\alpha_1,\tfrac{a_{0,1}-\alpha_2}{2};t)q(t)=0$;
\vspace{2pt}
\item 
$(1+\tfrac{t}{2}(\alpha_1-\vartheta_t))\Ad(e^{\wJ(\alpha)})\calP^{(\alpha)}(a_{0,1})q(t)=0$;
\vspace{3pt}
\item $\Ad(e^{\wJ(\alpha)})\calP^{(\alpha)}(a_{0.1})q(t)=0$,
\end{enumerate}
which concludes the proposed equality.
\end{proof}

We are now ready to give an alternative proof of Corollary \ref{cor:Sol2F0}.
Observe that, for any $a\in \C$,  we have 
\begin{equation*}
\Sol^{(\alpha)}_{2\calF0}(a) \subseteq \C 
{}_2F_0[-\alpha_1, \frac{a-\alpha_2}{2};t].
\end{equation*}
The equality holds if and only if 
${}_2F_0[-\alpha_1, \frac{a-\alpha_2}{2};t] \in \Pol_{\min(\alpha)}[t]$.

\begin{proof}[Alternative proof of Corollary \ref{cor:Sol2F0}]
We wish to show \eqref{eqn:CorSol}. 
By a direct observation on the generalized hypergeometric polynomial
${}_2F_0[-\alpha_1, \frac{a_{0,1}-\alpha_2}{2};t]$ that 
\begin{equation}\label{eqn:CorSol2}
\Sol^{(\alpha)}_{2\calF0}(a_{0,1})
=
\begin{cases}
\CC\, 
{}_2F_0[-\alpha_1,\tfrac{a_{0,1}-\alpha_2}{2};t]
& 
\textnormal{if (A1) or (A2) holds,}\\[3pt]
\{0\} & \textnormal{otherwise,}
\end{cases}
\end{equation}
Now Corollary \ref{cor:Sol2F02} and \eqref{eqn:CorSol2}
conclude \eqref{eqn:CorSol}.
\end{proof}

%%%%%%%%%%%%%%%%%%%%%%%%%%%%%%%%%%%%%%%%%

\end{appendices}

%%%%%%%%%%%%%%%%%%%%%%%%%%%%%%%%%%%%%%%%%
\noindent
\textbf{Acknowledgments.}
The first author was supported by the Carlsberg Foundation (grant no. CF24-045).
The second author was partially supported by 
Grant-in-Aid for Scientific Research(C) (JP22K03362 and JP26K06853) and
the third author was supported by Grant-in-Aid for 
JSPS International Research Fellows (JP24KF0075).

%%%%%%%%%%%%%%%%%%%%%%%%%%%%%%%%%%%%%%%%%
\bibliographystyle{amsplain}

\bibliography{bibdb}

%%%%%%%%%%%%%%%%%%%%%%%%%%%%%%%%%%%%%%%%%

\end{document}